\documentclass[10pt,reqno]{amsart} 

\usepackage{appendix}
\usepackage{float} 

\usepackage{soul}

\usepackage{hhline}

\usepackage{yhmath}\usepackage{epsfig,amssymb,amsmath,version}
\usepackage{amssymb,version,graphicx,fancybox,mathrsfs}
\usepackage{subcaption}
\usepackage{enumitem}
\usepackage{bbm}
\usepackage{url,hyperref}
\usepackage{color}
\usepackage{stmaryrd}
\usepackage{multirow}
\usepackage{booktabs,siunitx}
\usepackage{booktabs}
\usepackage{multicol,epstopdf}
\usepackage{xypic}
\usepackage[]{graphicx}
\usepackage[all]{xy}
\usepackage{tikz,ifthen}
\usetikzlibrary{positioning, math, decorations.markings, arrows.meta}
\usetikzlibrary{calc,shapes,cd}
\usetikzlibrary{knots} 
\usepgfmodule{decorations}
\allowdisplaybreaks
\tikzset{knotarrow/.pic={ \draw[edge, <-] (0,0) -- +(-.001,0);}}

\tikzset{edge/.style={line width=0.8}}
\tikzset{wall/.style={very thick}}

\tikzset{->-/.style n args={2}{decoration={markings, mark=at position #1 with {\arrow{#2}}}, postaction={decorate}}} 

\ExplSyntaxOn
\cs_new_eq:NN \ifstreqF \str_if_eq:nnF
\cs_new_eq:NN \ifstreqTF \str_if_eq:nnTF
\ExplSyntaxOff

\tikzset{-o-/.code 2 args={\ifstreqF{#2}{} 
{\ifstreqTF{#2}{>}
   {\pgfkeysalso{decoration={markings,mark=at position #1 with {\arrow[scale=0.8]{#2}}}
                    ,postaction={decorate}}
    }
   {\ifstreqTF{#2}{<}
       {\pgfkeysalso{decoration={markings,mark=at position #1 with {\arrow[scale=0.8]{#2}}}
                    ,postaction={decorate}}
        }
       {\pgfkeysalso{decoration={markings,
                    mark=at position #1 with
                    {\draw[black, fill={#2}] circle[radius=2pt];}}
                    ,postaction={decorate}}
        }
     }
  }}}

\allowdisplaybreaks[4]

\newtheorem{theorem}{Theorem}[section]
\newtheorem{lemma}[theorem]{Lemma}
\newtheorem{definition}[theorem]{Definition}
\newtheorem{corollary}[theorem]{Corollary}
\newtheorem{proposition}[theorem]{Proposition}
\newtheorem{remark}[theorem]{Remark}
\newtheorem{conjecture}[theorem]{Conjecture}

\newcommand{\Rlabel}[1]{\hyperref[R#1]{(R#1)}}

\newcommand{\gaa}{\overline{\mathsf{g}}}
\newcommand{\ga}{\mathsf{g}}

\newcommand{\bp}{\begin{proposition}}
\newcommand{\ep}{\end{proposition}}
\newcommand{\bpr}{\begin{proof}}
\newcommand{\epr}{\end{proof}}

\newcommand{\bt}{\begin{theorem}}
\newcommand{\et}{\end{theorem}}

\newcommand{\bl}{\begin{lemma}}
\newcommand{\el}{\end{lemma}}

\newcommand{\bcr}{\begin{corollary}}
\newcommand{\ecr}{\end{corollary}}

\newcommand{\be}{\begin{equation}}
\newcommand{\ee}{\end{equation}}

\newcommand{\bes}{\begin{equation*}}
\newcommand{\ees}{\end{equation*}}

\newcommand{\ba}{\begin{align}}
\newcommand{\ea}{\end{align}}

\newcommand{\bas}{\begin{align*}}
\newcommand{\eas}{\end{align*}}

\DeclareMathOperator{\skeleton}{\textbf{sk}}
\DeclareMathOperator{\im}{\mathrm{Im}}

\DeclareMathOperator{\Int}{\mathrm{Int}}

\DeclareMathOperator{\sgn}{\mathrm{sgn}}

\def\huang#1{(\textcolor{red}{Huang: #1})}

\newcommand{\todo}[1]{\textcolor{blue}{Need to do: #1}}

\graphicspath{{./Figures/} }

\begin{document}
\bibliographystyle{alpha}

\title{Quantum cluster algebra realization for stated ${\rm SL}_n$-skein algebras and rotation-invariant bases for polygons}

\author[Peigen Cao]{Peigen Cao}
\address{Peigen Cao, School of Mathematical Sciences, University of Science and Technology of China, Hefei, China}
\email{peigencao@126.com}

\author[Min Huang]{Min Huang}
\address{Min Huang, School of Mathematics (Zhuhai), Sun Yat-sen University, Zhuhai, China}
\email{huangm97@mail.sysu.edu.cn}

\author[Zhihao Wang]{Zhihao Wang}
\address{Zhihao Wang, School of Mathematics, Korea Institute for Advanced Study (KIAS), 85 Hoegi-ro, Dongdaemun-gu, Seoul 02455, Republic of Korea}
\email{zhihaowang@kias.re.kr}


 \maketitle

\begin{abstract}
We construct a quantum cluster structure on the skew-field of fractions ${\rm Frac}({\mathscr S}_\omega(\mathfrak{S}))$ of the stated ${\rm SL}_n$-skein algebra ${\mathscr S}_\omega(\mathfrak{S})$, where $\mathfrak{S}$ is a triangulable pb surface without interior punctures. This work complements the construction for the projected stated skein algebra $\widetilde{\mathscr S}_\omega(\mathfrak{S})$ given in \cite{huang2025quantum}, and to the best of our knowledge, establishes the first such structure for any $n$.

A key feature of our setup is that frozen variables are not invertible in ${\mathscr S}_\omega(\mathfrak{S})$. Accordingly, we introduce ${\mathscr A}_\omega(\mathfrak{S})$ and ${\mathscr U}_\omega(\mathfrak{S})$ to denote the associated quantum cluster algebra and quantum upper cluster algebra inside ${\rm Frac}({\mathscr S}_\omega(\mathfrak{S}))$, respectively, in the non-invertible frozen variable setting. We prove the inclusion
\[
{\mathscr S}_\omega(\mathfrak{S}) \subset {\mathscr A}_\omega(\mathfrak{S}).
\]
Moreover, when $n=2$, we show that
\[
{\mathscr S}_\omega(\mathfrak{S})
=
{\mathscr A}_\omega(\mathfrak{S})
=
{\mathscr U}_\omega(\mathfrak{S}).
\]

We write ${\mathscr S}_\omega^{\rm fr}(\mathfrak{S})$, ${\mathscr A}_\omega^{\rm fr}(\mathfrak{S})$, and ${\mathscr U}_\omega^{\rm fr}(\mathfrak{S})$ for the localizations of ${\mathscr S}_\omega(\mathfrak{S})$, ${\mathscr A}_\omega(\mathfrak{S})$, and ${\mathscr U}_\omega(\mathfrak{S})$, respectively, at the multiplicative set generated by all frozen variables. Let $\overline{\mathscr A}_\omega(\mathfrak{S})$ and $\overline{\mathscr U}_\omega(\mathfrak{S})$ denote the quantum cluster algebra and quantum upper cluster algebra associated to ${\rm Frac}(\widetilde{\mathscr S}_\omega(\mathfrak{S}))$, in the invertible frozen variable setting.
We prove the equalities
\[
\widetilde{\mathscr S}_\omega(\mathfrak{S})
=
\overline{\mathscr A}_\omega(\mathfrak{S})
=
\overline{\mathscr U}_\omega(\mathfrak{S})
\quad \text{and} \quad
{\mathscr S}_\omega^{\rm fr}(\mathfrak{S})
=
{\mathscr A}_\omega^{\rm fr}(\mathfrak{S})
=
{\mathscr U}_\omega^{\rm fr}(\mathfrak{S})
\]
when $\mathfrak{S}$ is a polygon. In this setting, the projected stated ${\rm SL}_n$-skein algebra $\widetilde{\mathscr S}_\omega(\mathfrak{S})$ coincides with the reduced stated ${\rm SL}_n$-skein algebra $\overline{\mathscr S}_\omega(\mathfrak{S})$.

As a consequence, when $\mathfrak{S}$ is a polygon, we show that the theta basis of $\overline{\mathscr U}_\omega(\mathfrak{S})$ (resp.\ ${\mathscr U}_\omega^{\rm fr}(\mathfrak{S})$) yields a rotation-invariant basis of $\overline{\mathscr S}_\omega(\mathfrak{S})$ (resp.\ ${\mathscr S}_\omega^{\rm fr}(\mathfrak{S})$) with several desirable properties, including positivity and a natural parametrization.

In the special case where $n=3$ and $\mathfrak{S}$ is the bigon, we prove that
\[
{\mathscr S}_\omega(\mathfrak{S})
=
{\mathscr A}_\omega(\mathfrak{S})
=
{\mathscr U}_\omega(\mathfrak{S}).
\]
We further give a complete web-theoretic characterization of cluster variables and clusters in ${\mathscr S}_\omega(\mathfrak{S}) = \mathcal O_q({\rm SL}_3)$. Consequently, inspired by the work of L{\^e} and Sikora \cite{le_sikora2025}, we establish a web interpretation of the dual canonical basis of $\mathcal O_q({\rm SL}_3)$.
\end{abstract}

\tableofcontents

\newcommand{\ca}{{\cev{a}  }}

\def\fS{\mathfrak{S}}
\def\dS{\widetilde{\cS}_\omega(\fS)}

\def\BZ{\mathbb Z}
\def\Id{\mathrm{Id}}
\def\Mat{\mathrm{Mat}}
\def\BN{\mathbb N}

\def \cb {\color{blue}}
\def \cred {\color{red}}
\def \cbf {\color{blue}\bf}
\def \credf {\color{red}\bf}
\definecolor{ligreen}{rgb}{0.0, 0.3, 0.0}
\def \cg {\color{ligreen}}
\def \cgf {\color{ligreen}\bf}
\definecolor{darkblue}{rgb}{0.0, 0.0, 0.55}
\def \dbf {\color{darkblue}\bf}
\definecolor{anti-flashwhite}{rgb}{0.55, 0.57, 0.68}
\def \afw {\color{anti-flashwhite}}
\def\cF{\mathbb F}
\def\cP{\mathcal P}
\def\embed{\hookrightarrow}
\def\pr{\mathrm{pr}}
\def\cV{\mathcal V}
\def\ot{\otimes}
\def\buu{{\mathbf u}}


\def \ri {{\rm i}}
\newcommand{\bs}[1]{\boldsymbol{#1}}
\newcommand{\cev}[1]{\reflectbox{\ensuremath{\vec{\reflectbox{\ensuremath{#1}}}}}}
\def\bS{\bar \fS}
\def\cE{\mathcal E}
\def\fB{\mathbb{P}_2}
\def\cR{\mathcal R}
\def\cY{\mathcal Y}
\def\cS{\mathscr S}
\def\rS{\overline{\cS}_\omega}

\def\fS{\mathfrak{S}}

\def\MN {(M)}
\def\cN {\mathcal{N}}
\def\SL{{\rm SL}_n}

\def\bP{\mathbb P}
\def\bR{\mathbb R}

\def\SS{\cS_{\omega}(\fS)}
\def\rdS{\overline \cS_{\omega}(\fS)}
\def\rdP{\overline \cS_{\omega}(\mathbb{P}_4)}

\def\Vm{\mathcal V_{\text{mut}}}
\def\Vc{\mathcal V}

\def\bT{\mathbb T}

\def\tr{{\rm tr}_{\lambda}}

\def\bk{{\bf k}}
\def\V{\overline{V}_\lambda}

\def\barV{\overline{V}}
\def\barK{\overline{K}}
\def\barVl{\overline{V}_\lambda}
\def\barKt{\overline{K}_\tau}

\def\barVt{\overline{V}_\tau}

\def\Y{A}

\def\a{\overline{A}}

\def\A{\overline{\mathcal{A}}_{\omega}(\fS,\lambda)}

\def\Ap{\overline{\mathcal{A}}_{\omega}^{+}(\fS,\lambda)}

\def\bZ{\mathbb Z}

\def\trA{\overline{{\rm tr}}_{\lambda}}
\def\SK{\overline{\cS}_\omega(\fS)}

\def\dS{\widetilde{\cS}_\omega(\fS)}

\def\bN{\mathbb N}
\def\VA{V'_\lambda}
\def\OVA{\overline{V}_{\lambda^\ast}}
\def\Qast{\overline{Q}_{\lambda^\ast}}
\def\Past{\overline{P}_{\lambda^\ast}}
\def\Ql{Q_{\lambda}}
\def\Pl{P_{\lambda}}

\def\sfQ{Q^{\ast}}
\def\bZ{\mathbb Z}
\def\barQ{Q}

\def\cX{\mathcal Z^{\ast}}
\def\cbX{\mathcal Z^{\ast{\rm bl}}}
\def\barX{\mathcal Z}
\def\sfC{C}

\def\bJ{\mathbb{J}}

\def\As{\mathcal A_{\omega}(\fS,\lambda)}
\def\Aa{\overline{\mathcal A}_{\omega}(\fS^\ast,\lambda^\ast)}
\def\Asp{\mathcal A^{+}_{\omega}(\fS,\lambda)}

\def\oa{\overline{\mathcal{A}}_{\omega}}
\def\cA{\mathcal{A}_{\omega}}

\def\TB{\bT_\omega(\overline{Q}_\lambda^\ast;B_\lambda)}
\def\bZs{\mathcal Z_\omega^{\ast{\rm bl}}(\fS,\lambda)}

\def\sgn{\text{sgn}}

\def\Exc{\mathsf{Exch}_\mathsf{S}}

\def\Fr{\text{Frac}}
\def\OvA{\overline{\mathscr{A}}_{\omega}(\fS)}
\def\OvU{\overline{\mathscr{U}}_{\omega}(\fS)}
\def\OvS{\overline{\mathsf{S}}(\fS)}

\def\RVlast{\mathring{\overline{V}}_{\lambda^\ast}}
\def\SSa{\cS_\omega(\fS^\ast)}
\def\dSS{\widetilde{\cS}_\omega(\fS^\ast)}
\def\lll{(\lambda')^\ast}
\def\last{\lambda^\ast}

\newcommand{\beq}{\begin{equation}}
	\newcommand{\eeq}{\end{equation}}

\section{Introduction}\label{sec-intro}
Throughout the paper, we fix a positive integer $n$ and a ground ring $R$ with an invertible element $\omega^{\frac{1}{2}}\in R$. Set  $\xi = \omega^{n}$ and $q=\omega^{n^2}$ with
$\xi^{\frac{1}{2n}} = \omega^{\frac{1}{2}}$
and $q^{\frac{1}{2n^2}} = \omega^{\frac{1}{2}}$.  
Define the following constants:
\begin{align}\label{intro-constants}
\mathbbm{c}_{i}= (-q)^{i-n} q^{\frac{1-n}{2n}},\quad
\mathbbm{t}= (-1)^{n-1} q^{\frac{1-n^2}{n}},\quad 
\mathbbm{a} =   q^{-\frac{n+1-2n^2}{4}}.
\end{align}
Unless otherwise specified, all algebras in this paper are assumed to be $R$-algebras.

We use $\mathbb{N}$ and $\mathbb{Z}$ to denote the sets of nonnegative integers and integers, respectively.

\subsection{(Quantum) cluster algebras}
Cluster algebras, introduced by Fomin and Zelevinsky \cite{FZ}, are a class of commutative algebras equipped with a distinguished set of generators known as cluster variables. These variables are grouped into collections called clusters, which are related by certain birational transformations known as mutations. In \cite{BFZ}, Berenstein, Fomin, and Zelevinsky introduced the notion of upper cluster algebras, defined as the intersection of the Laurent polynomial rings associated with each cluster. From a geometric perspective, upper cluster algebras are often more natural objects to consider. Cluster variables are rational functions in initial cluster variables by construction. The remarkable Laurent phenomenon \cite{FZ} says that all the cluster variables are actually Laurent polynomials in initial cluster variables. Moreover, the coefficients of these Laurent polynomials are non-negative, cf. \cite{LS,GHKK,D,H2,H3}.

The quantum (upper) cluster algebras  (Definition \ref{def-quan-cluster-algebra}) were introduced by Berenstein and Zelevinsky in \cite{BZ} and it is proved that the Laurent phenomenon also holds in the quantum setting.

For a (quantum) cluster algebra $\mathscr A$ and its upper cluster algebra $\mathscr U$, the Laurent phenomenon ensures that $\mathscr A\subseteq \mathscr U$. However, in general, $\mathscr A\neq \mathscr U$. The problem of whether a cluster algebra coincides with its upper cluster algebra was posed by Berenstein, Fomin, and Zelevinsky and has been studied in \cite{BFZ}. This question has attracted considerable attention since its inception; see, for example, \cite{M,M1,CLS,GY,SW21,L,MW,IOS,CGGLS,QY} and references therein. If $\mathscr A=\mathscr U$, then the cluster algebra $\mathscr A$ enjoys several desirable properties, including the existence of a generic basis and a theta basis \cite{CKQ,Q,GLS,GHKK,GLS1}. 

The search for (quantum) cluster algebra and upper (quantum) cluster algebra structures on important algebraic and geometric objects has drawn significant interest; see, for example, 
\cite{FZ,S,SW21,CGGLS,GLSB,QY} 
and the references therein. Cluster algebras and cluster varieties have also been constructed in the context of Teichm\"{u}ller theory \cite{FG06,FST,GSV,GS19}, and are closely related to skein theory \cite{muller2016skein,ishibashi2023skein,LY22,LY23,KimWang,Kim21,IY1}. 

\subsection{(Reduced) stated ${\rm SL}_n$-skein algebras and quantum trace maps}\label{intro-sec-reduced}

A  \emph{pb surface} $\fS$ is obtained from a compact oriented surface $\overline{\fS}$ by removing finitely many points, which are called \emph{punctures}, such that every boundary component of $\fS$ is diffeomorphic to an open interval. 
An embedding $c:(0,1)\rightarrow \fS$ is called an \emph{ideal arc} if both
$\bar c(0)$ and $\bar c(1)$ are punctures, where $\bar c\colon [0,1] \to \overline{\fS}$ is the `closure' of $c$.
For any positive integer $k$, we use $\mathbb P_k$ to denote the pb surface obtained from the closed disk by removing $k$ punctures from the boundary. 

The \emph{stated ${\rm SL}_n$–skein algebra} $\cS_\omega(\fS)$ of a pb surface $\fS$  
is the quotient of the $R$–module freely generated by the set of isotopy classes of stated $n$–webs  
(Definition~\ref{def-n-web}) in $\fS \times (-1,1)$, subject to the relations \eqref{w.cross}–\eqref{wzh.eight}.  
For two stated $n$-webs $\alpha,\beta$, their product $\alpha\beta$ is defined by stacking $\alpha$ over $\beta$.
The stated ${\rm SL}_n$–skein algebra was introduced by L{\^e} and Sikora \cite{LS21} as a generalization of the ${\rm SL}_n$–skein algebra  
\cite{Sik05} to the stated setting, extending the stated ${\rm SL}_2$– and ${\rm SL}_3$–skein algebras  \cite{le2018triangular,higgins2020triangular} to the ${\rm SL}_n$ case.  
The \emph{reduced stated ${\rm SL}_n$–skein algebra} $\overline{\cS}_\omega(\fS)$ \cite{LY23} is defined as the quotient of $\cS_\omega(\fS)$  
by the two–sided ideal generated by all bad arcs (see Figure~\ref{Fig;badarc}).

\vspace{0.2cm}

Let $\fS$ be a triangulable pb surface (see \S\ref{sec-traceX}),  
and let $\lambda$ be a triangulation of $\fS$—that is, a maximal collection of pairwise disjoint, non-isotopic ideal arcs in $\fS$  
(note that self-folded triangles are not allowed; see \S\ref{sec-traceX}).  
Cutting $\fS$ along all ideal arcs that are not isotopic to any component of $\partial\fS$ yields a collection of triangles (copies of $\mathbb P_3$),  
denoted by $\mathbb F_\lambda$.  
For each $\tau=\mathbb P_3\in \mathbb F_\lambda$, there is a weighted quiver $\Gamma_\tau$ inside $\tau$ (see Figure~\ref{Fig;coord_ijk})  
whose arrows have weight $1$, except those lying on the boundary, which have weight $\tfrac{1}{2}$.  
The vertices of $\Gamma_\tau$ are called \emph{small vertices}.  
By gluing the triangles in $\bigsqcup_{\tau\in \mathbb F_\lambda}\tau$ back together to recover $\fS$,  
we obtain a quiver $\Gamma_\lambda$ on $\fS$: whenever two edges are identified, the small vertices on these edges are identified as well,  
and any pair of arrows of equal weight but opposite direction between two vertices cancel each other.
We denote by $\overline V_\lambda$ the vertex set of $\Gamma_\lambda$; each element of $\overline V_\lambda$ is called a \emph{small vertex}.
Let $\overline Q_\lambda\colon \overline V_\lambda\times\overline V_\lambda\rightarrow \tfrac{1}{2}\mathbb Z$ be the signed adjacency matrix of $\Gamma_\lambda$ (see \eqref{eq-def-Q-lambda-re}).

Assume that $\fS$ has no interior punctures.  
There is another antisymmetric matrix \eqref{eq-anti-matric-P-def}
\begin{align*}
  \overline  P_\lambda\colon\overline V_\lambda\times \overline V_\lambda\rightarrow n\mathbb Z
\end{align*}
associated with the triangulation $\lambda$ of $\fS$, which equals $-\overline{\mathsf P}_\lambda$ defined in \cite[Equations~(163) and (205)]{LY23}.
Put $\overline\Pi_\lambda=\frac{1}{n}\overline P_\lambda$.
The \emph{$\mathcal A$-version quantum torus} of $(\fS,\lambda)$ is defined by
\begin{equation*}
\overline{\mathcal{A}}_{\omega}(\fS,\lambda)
= R \langle 
\overline A_v^{\pm 1},\, v \in V_\lambda \rangle \big/ \bigl(
\overline A_v\overline A_{v'}= \xi^{\overline\Pi_\lambda(v,v')} \overline A_{v'}\overline A_v
\text{ for } v,v'\in\overline V_\lambda \bigr).
\end{equation*}

There is an algebra homomorphism \cite{LY23} (see Theorem~\ref{thm-trace-A})
\[
\overline{\rm tr}_\lambda\colon
\overline{\cS}_\omega(\fS)\longrightarrow \overline{\mathcal{A}}_{\omega}(\fS,\lambda)
\]
with the following properties:
\begin{enumerate}[label={\rm (\alph*)}]\itemsep0.3em
    \item\label{intro-thm-trace-A-a}
    $\overline{\mathcal{A}}_{\omega}^{+}(\fS,\lambda)\subset\im \overline{\rm tr}_\lambda\subset \overline{\mathcal{A}}_{\omega}(\fS,\lambda)$,  
    where $\overline{\mathcal{A}}_{\omega}^{+}(\fS,\lambda)$ is the $R$-subalgebra of $\overline{\mathcal{A}}_{\omega}(\fS,\lambda)$ generated by $\overline A^{\bf k}$ for ${\bf k}\in\mathbb N^{V_\lambda}$.  
    Here $\overline A^{\bf k}$ is the Laurent monomial defined using the Weyl-ordered product (see \eqref{def-monomial-for-A}).

    \item \label{intro-thm-trace-A-b}
    If $n=2,3$, or $n>3$ and $\fS$ is a polygon (i.e., $\fS=\mathbb P_k$ for some $k>2$), then $\overline{\rm tr}_\lambda$ is injective.

\end{enumerate}

Define the \emph{projected $\SL$-skein algebra} by
\[
\widetilde{\cS}_\omega(\fS):=
\overline{\cS}_\omega(\fS)\big/ \ker \overline{\rm tr}_\lambda.
\]
It is shown in \cite{huang2025quantum} that 
$\widetilde{\cS}_\omega(\fS)$ is independent of the choice of the triangulation $\lambda$.
Then property \ref{intro-thm-trace-A-a} of 
$\overline{\rm tr}_\lambda$ implies that 
\begin{align}\label{intro-identity-Frac}
    {\rm Frac}\bigl(\widetilde{\cS}_\omega(\fS)\bigr)
= {\rm Frac}\bigl(\overline{\mathcal{A}}_{\omega}(\fS,\lambda)\bigr),
\end{align}
where ${\rm Frac}(-)$ denotes the skew-field of fractions of the corresponding Ore domain.

\vspace{2mm}

A pb surface is called \emph{generalized triangulable} if each of its connected components is either a triangulable pb surface or a bigon. The \emph{generalized triangulation} of a bigon is the set of its two boundary edges.

Let $\fS$ be a generalized triangulable pb surface without interior punctures.
Attach a triangle $\bP_3$ to each boundary edge of $\fS$, and let $\fS^\ast$ denote the resulting triangulable pb surface; see Figure~\ref{Fig;attaching}(A).
We label the boundary edges of $\bP_3$ by $e_1$, $e_2$, and $e_3$ (see Figure~\ref{Fig;attaching}(A)).
For any generalized triangulation $\lambda$ of $\fS$, let $\lambda^\ast$ be the ideal triangulation of $\fS^\ast$ whose restriction to $\fS$ is $\lambda$.
Define the extended $A$-vertex set $\VA\subset \overline{V}_{\lambda^\ast}$ to be the subset of all small vertices not on
the edge $e_2$ in the attached triangles.

There is an  antisymmetric matrix (see \eqref{def-matrix-Pl})  \cite{LY23}
$$P_\lambda \colon \VA\times\VA\to\mathbb Z$$
which is the extended version of $\overline P_\lambda$.
Put $\Pi_\lambda=\frac{1}{n} P_\lambda$.
The \emph{extended $\mathcal A$-version quantum torus} of $(\fS,\lambda)$ is defined by
\begin{equation*}
{\mathcal{A}}_{\omega}(\fS,\lambda)
= R \langle 
 A_v^{\pm 1},\, v \in V_\lambda' \rangle \big/ \bigl(
 A_v A_{v'}= \xi^{\Pi_\lambda(v,v')}  A_{v'} A_v
\text{ for } v,v'\in V_\lambda' \bigr).
\end{equation*}

There is an algebra embedding  \cite{LY23} (see Theorem~\ref{thm-trace-A})
\[
{\rm tr}_\lambda\colon
{\cS}_\omega(\fS)\longrightarrow {\mathcal{A}}_{\omega}(\fS,\lambda)
\]
such that 
\begin{equation}\label{eq-p-s-a}
\Asp \;\subset\; \im\tr \;\subset\; \As,
\end{equation}
where $\Asp$ is defined similarly as $\overline{\mathcal{A}}_{\omega}^{+}(\fS,\lambda)$.

\vspace{2mm}

As we will see in the next subsection, the quantum trace maps relate the (projected) stated $\SL$-skein algebras to quantum cluster algebras, which are two quantizations of the same underlying object \cite{IOS}.

\subsection{Quantum cluster structures arising from the stated ${\rm SL}_n$-skein algebra}
Let $\fS$ be a triangulable pb surface without interior punctures, and let $\lambda$ be a triangulation of $\fS$.
The last two authors constructed a quantum seed (see Definition~\ref{def-quantum-seed}) inside the skew-field of fractions ${\rm Frac}\bigl(\widetilde{\cS}_\omega(\fS)\bigr)$ of 
$\widetilde{\cS}_\omega(\fS)$ in \cite{huang2025quantum} (see also \cite{muller2016skein,ishibashi2023skein,LY22} for the cases $n=2,3$) as follows.

To construct a quantum seed, we require a set of vertices $\mathcal V$ and a subset of mutable vertices $\mathcal V_{\rm mut}$.
Set $\mathcal V = \overline V_\lambda$, and let $\mathcal V_{\rm mut} = \mathring{\overline V}_\lambda$ be the subset of vertices lying in the interior of $\fS$.  
Using the identification in \eqref{intro-identity-Frac}, it is shown in \cite{huang2025quantum} that the triple 
$\overline{\mathsf s}_\lambda = (\overline Q_\lambda,\overline\Pi_\lambda,\overline M_\lambda)$ 
forms a quantum seed (Definition~\ref{def-quantum-seed}) in the skew-field ${\rm Frac}\bigl(\widetilde{\cS}_\omega(\fS)\bigr)$ (Theorem~\ref{intro-thm-skein-inclusion-A}), where  
\begin{align*}
    \overline M_\lambda \colon \mathbb Z^{\mathcal V} \longrightarrow {\rm Frac}\bigl(\widetilde{\cS}_\omega(\fS)\bigr), 
    \qquad {\bf t} \longmapsto \overline A^{\bf t}.
\end{align*}
Moreover, it is proved in \cite{huang2025quantum} that the quantum seeds 
$\overline{\mathsf s}_{\lambda}$ and $\overline{\mathsf s}_{\lambda'}$ are mutation-equivalent (see Definition~\ref{def-quantum-class}) for any two triangulations $\lambda,\lambda'$ of $\fS$.

Consequently, one obtains a quantum cluster algebra and a quantum upper cluster algebra (see Definition~\ref{def-quan-cluster-algebra})
\[
\overline{\mathscr A}_\omega(\fS)
\quad \text{and} \quad 
\overline{\mathscr U}_\omega(\fS),
\]
associated with the quantum seed $\overline{\mathsf s}_{\lambda}$; these are independent of the choice of triangulation $\lambda$.
It is further shown 
(see Theorem~\ref{intro-thm-skein-inclusion-A}(c)) \cite{huang2025quantum} that 
\begin{align}\label{intro-aaaa}
    \widetilde{\cS}_\omega (\fS) \subset \overline{\mathscr A}_\omega(\fS)
\end{align}
whenever each connected component of $\fS$ contains at least two punctures.

\vspace{2mm}

Let $\fS$ be a generalized triangulable pb surface without interior punctures.
In this paper, we construct a quantum cluster structure in the skew-field of fractions ${\rm Frac}\bigl({\cS}_\omega(\fS)\bigr)$ of the stated $\SL$-skein algebra ${\cS}_\omega(\fS)$.

Let $\lambda$ be a generalized triangulation of $\fS$.
Set $\mathcal V = \VA$ and 
$\mathcal V_{\mathrm{mut}} = \mathring{\overline V}_{\lambda^\ast},$
where $\mathring{\overline V}_{\lambda^\ast} \subset \VA$ denotes the subset of vertices lying in the interior of $\fS^\ast$.
Let 
\[
Q_\lambda \colon \mathcal V \times \mathcal V \to \bZ
\]
be the restriction of 
\[
\overline Q_{\lambda^\ast} \colon \OVA \times \OVA \to \bZ.
\]

Equation~\eqref{eq-p-s-a} implies that
\[
{\rm Frac}\bigl({\cS}_\omega(\fS)\bigr)
=
{\rm Frac}\bigl(\mathcal A_\omega(\fS,\lambda)\bigr).
\]
Finally, define
\begin{align*}
    M_\lambda \colon \mathbb Z^{\mathcal V} \longrightarrow {\rm Frac}\bigl({\cS}_\omega(\fS)\bigr),
    \qquad 
    {\bf t} \longmapsto A^{\bf t}.
\end{align*}
We prove that the triple 
\[
{\mathsf s}_\lambda = ( Q_\lambda,\Pi_\lambda, M_\lambda)
\]
forms a quantum seed (Definition~\ref{def-quantum-seed}) in the skew-field ${\rm Frac}\bigl({\cS}_\omega(\fS)\bigr)$ (Lemma~\ref{lem-seed-skein}).
Let $\mathsf{S}_{\fS,\lambda}$ denote the quantum mutation class (Definition~\ref{def-quantum-class}) containing the quantum seed ${\mathsf s}_\lambda$.

Let $\lambda$ and $\lambda'$ be two generalized triangulations of $\fS$. 
It is shown in Proposition~\ref{prop-naturality} that
\[
\mathsf{S}_{\fS,\lambda}=\mathsf{S}_{\fS,\lambda'}.
\]
This establishes the naturality of the constructed quantum cluster structure inside 
${\rm Frac}\bigl({\cS}_\omega(\fS)\bigr)$.

We note that frozen variables are not invertible in ${\cS}_\omega(\fS)$. 
Therefore, we introduce the following definitions:
\[
\mathsf{S}(\fS):=\mathsf{S}_{\fS,\lambda},\quad
\mathscr{A}_{\omega}(\fS):=
\mathscr{A}_{\mathsf{S}(\fS)}^+,\quad
\mathscr{U}_{\omega}(\fS):=
\mathscr{U}_{\mathsf{S}(\fS)}^+,\quad
\mathscr{A}_{\omega}^{\rm fr}(\fS):=
\mathscr{A}_{\mathsf{S}(\fS)},\quad
\mathscr{U}_{\omega}^{\rm fr}(\fS):=
\mathscr{U}_{\mathsf{S}(\fS)},
\]
where the superscript $+$ indicates that frozen variables are not inverted (see Definition~\ref{def-quan-cluster-algebra}).

\vspace{2mm}

Note that the weighted quiver associated with the exchange matrix $Q_\lambda$ is obtained from that associated with $\overline Q_{\lambda^*}$ by deleting  certain frozen vertices. 
We then define a quantum cluster algebra 
$\mathscr A_{\omega}^{\rm qc}(\fS^\ast)$ obtained from
$\mathscr A_{\omega}(\fS)$ by localizing all frozen variables and then adjoining the same number of isolated invertible frozen variables (see~\S\ref{sub-quas-sl2}). 
We prove that $\mathscr A_{\omega}^{\rm qc}(\fS^\ast)$ is quasi-isomorphic (Definition~\ref{def:QQH}) to
$\mathscr A_{\omega}(\fS^\ast)$ (see Proposition~\ref{Prop-quasi-iso}).

The following sequence highlights the close relationship between the quantum cluster algebra $\mathscr A_{\omega}(\fS)$ constructed here and the quantum cluster algebra $\overline{\mathscr A}_{\omega}(\fS^\ast)$ constructed in \cite{huang2025quantum}:
\begin{align}\label{intro-eq-AA}
    \mathscr A_{\omega}(\fS)\lhook\joinrel\longrightarrow
    \mathscr A_{\omega}^{\rm qc}(\fS^\ast)\xrightarrow{\text{quasi-isomorphism}} \overline{\mathscr A}_{\omega}(\fS^\ast).
\end{align}

As will be shown in the following theorem, the skein-theoretic counterpart of \eqref{intro-eq-AA} is the algebra embedding
\begin{align}\label{intro-eq-iota}
    \iota_*\colon {\cS}_\omega(\fS) \lhook\joinrel\longrightarrow
    \widetilde{\cS}_\omega(\fS^{*}),
    \qquad \text{(see Corollary~\ref{cor-injectivity})},
\end{align}
which is induced by a natural embedding 
$\iota\colon \fS\hookrightarrow \fS^*$.

A properly embedded oriented arc in $\fS$ is called an \emph{essential arc}  
if its endpoints lie on two distinct boundary components of $\partial \fS$.  

We now state the first main result, which establishes the inclusion 
\[
\SS \subset \mathscr A_{\omega}(\fS)
\]
whenever every component of $\fS$ contains at least two punctures, as well as the compatibility between the inclusions
$\SS \subset \mathscr A_{\omega}(\fS)$ and 
$\widetilde{\cS}_\omega(\fS^*) \subset \overline{\mathscr A}_{\omega}(\fS^*)$ via the algebra embeddings in 
\eqref{intro-eq-AA} and \eqref{intro-eq-iota}.

\begin{theorem}[Theorem~\ref{thm-skein-inclusion-A}]
\label{intro-A-skein}
Let $\fS$ be a generalized triangulable pb surface without interior punctures, such that every connected component of $\fS$ contains at least two punctures. 
    Then we have $$\SS\subset \mathscr A_{\omega}(\fS)\subset \mathscr{U}_{\omega}(\fS).$$
Moreover, every stated essential arc is a cluster variable
in $\mathscr A_{\omega}(\fS)$.

The inclusion $\SS\subset \mathscr A_{\omega}(\fS)$ is compatible with the inclusion $\widetilde{\cS}_\omega(\fS^*) \subset \overline{\mathscr A}_{\omega}(\fS^*)$ in \eqref{intro-aaaa}. Specifically, the following diagram commutes:
\begin{equation}\label{intro-eq-com}
     \begin{tikzcd}[
    row sep=0.7cm,       
    column sep=1cm,   ]
    \SS \arrow[d, " "', hook] \arrow[rrrr, "\iota_\ast",hook] &&&& \dSS \arrow[d, " ", hook] \\  
    {\mathscr A}_{\omega}(\fS) \arrow[rr, 
        "{\begin{minipage}{3cm}\centering \scriptsize adding isolated \\ frozen variables \end{minipage}}"',hook] 
        &&  {\mathscr A}_{\omega}^{\rm qc}(\fS^\ast) \arrow[rr, "\text{quasi-iso}"'] &&  \overline {\mathscr A}_{\omega}(\fS^\ast)
\end{tikzcd}.
\end{equation}
\end{theorem}

The inclusion $\mathscr A_{\omega}(\fS)\subset \mathscr{U}_{\omega}(\fS)$ was established in \cite[Theorem 2.5]{goodearl2017quantum}.

Lemma~\ref{lem-essential} shows that the algebra ${\cS}_\omega(\fS)$ is generated by finitely many stated essential arcs.
To prove the inclusion $\SS\subset \mathscr A_{\omega}(\fS)$, it therefore suffices to show that each stated essential arc lies in 
$\mathscr A_{\omega}(\fS)$.

The algebra embedding $\iota_*$ in \eqref{intro-eq-iota}
 sends each stated essential arc in ${\cS}_\omega(\fS)$ to a stated essential arc in $\widetilde{\cS}_\omega(\fS^{*})$.
Moreover, every stated essential arc in $\widetilde{\cS}_\omega(\fS^{*})$ coincides with a cluster variable in $\overline{\mathscr A}_{\omega}(\fS^{*})$, up to multiplication by frozen variables \cite{huang2025quantum} (see \eqref{intro-eq-Cij}, \eqref{intro-eq-bar-Cij}, and \eqref{into-eq-Dij}).

For a generalized triangulation $\lambda$ of $\fS$, we show that $\iota_*$ extends to an algebra embedding
\begin{align}\label{intro-eq-L}
    L_\lambda\colon \mathcal{A}_\omega (\fS,\lambda)
\lhook\joinrel\longrightarrow \overline{\mathcal{A}}_\omega (\fS,\lambda)
\qquad \text{(see Lemma~\ref{lem-com-reduced-stated})}.
\end{align}
We then prove that the embedding $L_\lambda$ behaves well with respect to certain special mutation sequences; see Proposition~\ref{prop-flips-L} and Lemma~\ref{lem-gL}.
Combining this with the formulas in \eqref{intro-eq-Cij}, \eqref{intro-eq-bar-Cij}, and \eqref{into-eq-Dij}, we obtain explicit cluster realizations of each stated essential arc in ${\cS}_\omega(\fS)$ (see \eqref{pro-inclusion-3}, \eqref{pro-inclusion-4}, and \eqref{eq-alpha-counterclock}).
This shows that every stated essential arc is a cluster variable in $\mathscr A_{\omega}(\fS)$.

The commutative diagram \eqref{intro-eq-com} follows directly from the construction of $L_\lambda$ (see \eqref{eq-def-L-lamda}), the compatibility between $\iota_*$ and $L_\lambda$ (Lemma~\ref{lem-com-reduced-stated}), and the construction of the quasi-isomorphism in \eqref{intro-eq-AA} (see Proposition~\ref{Prop-quasi-iso}).

\vspace{2mm}

For any triangulation $\lambda$ of $\fS$, Equation~\eqref{eq-p-s-a} implies that the cluster variables $A_v$ (see \S\ref{sec-mutation-quantum}), $v \in V_{\lambda}'$, are contained in ${\cS}_\omega(\fS)$. In particular, all frozen variables (see \S\ref{sec-mutation-quantum}) lie in ${\cS}_\omega(\fS)$. We denote by $\mathcal{F}$ the multiplicative subset of ${\cS}_\omega(\fS)$ generated by all frozen variables.
We use ${\cS}_\omega^{\rm fr}(\fS)$ to denote the localization of ${\cS}_\omega(\fS)$ at the multiplicative subset $\mathcal{F}$, called the \emph{localized stated ${\rm SL}_n$-skein algebra}.

We formulate the following conjecture concerning the full equivalence between the localized stated $\SL$-skein algebra and the quantum (upper) cluster algebra. In the next subsection, we present several confirming cases.

\begin{conjecture}\label{con-equality-Skein-A-U}
      Under the same assumption of Theorem~\ref{intro-A-skein},
     we have 
     $${\cS}_\omega^{\rm fr}(\fS)= \mathscr{A}_{\omega}^{\rm fr}(\fS)
     =\mathscr{U}_{\omega}^{\rm fr}(\fS).$$
     
\end{conjecture}

\subsection{Equivalence of quantum cluster algebras with localized or reduced  stated ${\rm SL}_n$-skein algebras}

The following theorem verifies Conjecture~\ref{con-equality-Skein-A-U} in the case $n=2$, providing the first evidence for a quantum cluster algebra structure in the (localized) stated skein theory.

\begin{theorem}[Theorem~\ref{thm-skein-eq-A-two}]\label{intro-thm-sl2}
Let $\fS$ be a pb surface without interior punctures. We require that every component of $\fS$ contains at least two punctures.
    When $n=2$, we have 
 $${\cS}_\omega(\fS)= \mathscr{A}_{\omega}(\fS)=\mathscr{U}_{\omega}(\fS).$$
    
\end{theorem}

Note that ${\cS}_\omega(\fS)= \mathscr{A}_{\omega}(\fS)=\mathscr{U}_{\omega}(\fS)$ implies 
${\cS}_\omega^{\rm fr}(\fS)= \mathscr{A}_{\omega}^{\rm fr}(\fS)=\mathscr{U}_{\omega}^{\rm fr}(\fS)$.

We briefly outline the proof of Theorem~\ref{thm-skein-eq-A-two}. Using the algebra embeddings $\iota_*$ and $L_\lambda$ in \eqref{intro-eq-iota} and \eqref{intro-eq-L}, together with Lemma~\ref{lem-com-reduced-stated}, which establishes the compatibility between $\iota_*$ and $L_\lambda$, and Proposition~\ref{prop-flips-L}, which shows the compatibility of $L_\lambda$ with certain mutation sequences, we prove that each cluster variable in $\mathscr A_{\omega}(\fS)$ is represented by a stated arc in $\SS$. This yields the inclusion $\mathscr A_{\omega}(\fS)\subset \SS$. Combining this with Theorem~\ref{intro-A-skein}, we obtain
\begin{equation}\label{intro-eq-skein=a}
    {\cS}_\omega(\fS)= \mathscr{A}_{\omega}(\fS).
\end{equation}

It was shown in \cite{muller2016skein} that
\[
\overline{\mathscr A}_\omega(\fS^*)
=
\overline{\mathscr U}_\omega(\fS^*).
\]
The quasi-isomorphism in Proposition~\ref{Prop-quasi-iso} (see Definition~\ref{def:QQH}) then implies
\[
\mathscr{A}_{\omega}^{\rm fr}(\fS)
=
\mathscr{U}_{\omega}^{\rm fr}(\fS).
\]
Combined with \eqref{intro-eq-skein=a}, this yields
\begin{equation}\label{intro-eq-skein=a=U}
    {\cS}_\omega^{\rm fr}(\fS)
    =
    \mathscr{A}_{\omega}^{\rm fr}(\fS)
    =
    \mathscr{U}_{\omega}^{\rm fr}(\fS).
\end{equation}
We further show that each frozen variable is a prime element in ${\cS}_\omega(\fS)$; see Lemma~\ref{lem-prime-ele}. 
Together with \eqref{intro-eq-skein=a=U}, this implies, by \cite[Theorem~C]{QY}, that
\begin{equation}\label{intro-eq-skein=U}
    {\cS}_\omega(\fS)
    =
    \mathscr{U}_{\omega}(\fS).
\end{equation}
This completes the proof of Theorem~\ref{intro-thm-sl2}.

\vspace{2mm}

It is conjectured in \cite{huang2025quantum} that
\begin{align}\label{into-eq-reduced}
    \widetilde\cS_{\omega}(\fS)
    =
    \overline{\mathscr A}_\omega(\fS)
    =
    \overline{\mathscr U}_\omega(\fS)
\end{align}
under the same assumptions as in Theorem~\ref{intro-A-skein}, except that ``generalized triangulable'' is replaced by ``triangulable.''
Results in \cite{IOS} verify \eqref{into-eq-reduced} in the classical case, i.e., when $\omega^{\frac{1}{2}}=1$. However, the stated case—namely, Conjecture~\ref{con-equality-Skein-A-U}—is not addressed in \cite{IOS}.

The following theorem verifies Conjecture~\ref{con-equality-Skein-A-U} and \eqref{into-eq-reduced} in the case where the pb surface $\fS$ is a polygon. 
As we will see in the next subsection, this result further implies the existence of a rotation-invariant and positive basis of the localized or reduced stated 
$\SL$-skein algebra of a polygon.
Note that the property \ref{intro-thm-trace-A-b} of 
$\overline{\rm tr}_\lambda$ implies
$\widetilde\cS_{\omega}(\mathbb P_{k+2})=\overline\cS_{\omega}(\mathbb P_{k+2})$
for any $k\geq 1$.

\begin{theorem}[Theorems~\ref{thm:poly1} and \ref{thm:poly2}]\label{intro-thm-equa}
    (a) For any $k\geq 1$,
    we have 
    $$\overline\cS_{\omega}(\mathbb P_{k+2})=\overline{\mathscr A}_\omega(\mathbb P_{k+2})=\overline{\mathscr U}_\omega(\mathbb P_{k+2}).$$

    (b) For any $k\geq 0$,  
    we have 
    $$\cS_{\omega}^{\rm fr}(\mathbb P_{k+2})={\mathscr A}_\omega^{\rm fr}(\mathbb P_{k+2})={\mathscr U}_\omega^{\rm fr}(\mathbb P_{k+2}).$$
\end{theorem}


The proofs of Theorem~\ref{intro-thm-equa}(a) and (b) follow the same strategy; here we briefly outline the argument for Theorem~\ref{intro-thm-equa}(a).

There is a cluster algebra $\overline{\mathscr A}^{\rm qc}_\omega(\mathbb P_{k+2})$ (see \S\ref{sub-quasi-polygon}) studied in \cite{SW21,Q2024}, which is quasi-isomorphic (Definition~\ref{def:QQH}) to 
$\overline{\mathscr A}_\omega(\mathbb P_{k+2})$ (see Proposition~\ref{prop-quasi-polygon}).
In \cite{SW21}, a collection of so-called \emph{standard cluster variables} is constructed using certain maximum green sequences (see \eqref{eq:mustandard}).
It is proved in \cite{Q2024} that, as an $R$-algebra,
\[
\overline{\mathscr A}^{\rm qc}_\omega(\mathbb P_{k+2})
=
\overline{\mathscr U}^{\rm qc}_\omega(\mathbb P_{k+2})
\]
is generated by these standard cluster variables together with the frozen variables in $\overline{\mathscr A}_\omega(\mathbb P_{k+2})$ (see Proposition~\ref{pro:standard}).

We establish technical Lemmas~\ref{lem:B-d} and~\ref{lem:B4}, which employ (extended)
$g$-vectors to realize standard cluster variables via different mutation sequences and to analyze the equivalences of full subquivers.
As mentioned above, \cite{huang2025quantum} shows that every stated essential arc in $\overline{\cS}_\omega(\mathbb P_{k+2})$ coincides with a cluster variable in $\overline{\mathscr A}_{\omega}(\mathbb P_{k+2})$, up to multiplication by frozen variables (see \eqref{intro-eq-Cij}, \eqref{intro-eq-bar-Cij}, and \eqref{into-eq-Dij}).
Combining these formulas with Lemmas~\ref{lem:B-d} and~\ref{lem:B4}, we prove that each standard cluster variable coincides with a stated essential arc in $\overline{\cS}_\omega(\mathbb P_{k+2})$, again up to multiplication by frozen variables (Theorem~\ref{thm:standardgenerators}).

Moreover, it follows from \eqref{intro-eq-Cij} that all frozen variables in $\overline{\mathscr A}_\omega(\mathbb P_{k+2})$ also lie in 
$\overline\cS_{\omega}(\mathbb P_{k+2})$. 
Combining this with \eqref{intro-aaaa}, Propositions~\ref{prop-quasi-polygon}, and \ref{pro:standard}, we conclude that
\[
\overline\cS_{\omega}(\mathbb P_{k+2})
=
\overline{\mathscr A}_\omega(\mathbb P_{k+2})
=
\overline{\mathscr U}_\omega(\mathbb P_{k+2}).
\]



\subsection{Rotation-invariant bases for the reduced and localized stated ${\rm SL}_n$-skein algebras of a polygon}\label{intro-sub-basis}

One important motivation for studying the equivalence between the quantum cluster algebra and the (reduced) stated ${\rm SL}_n$-skein algebra is the construction of ``good"  bases for the latter. 
The existence of such  good bases is crucial for the study of the (reduced) stated ${\rm SL}_n$-skein algebra, with applications including the injectivity of the quantum trace map, the description of the image of general webs under the Frobenius map constructed in \cite{kim2025frobenius} (see Remark~\ref{rem-Fro}), and the structure of the center of the algebra.

At present, a ``good'' basis is known only in limited cases. 
For ${\rm SL}_2$, the bracelet basis provides a ``good'' basis of the (reduced) stated ${\rm SL}_2$-skein algebra \cite{Q,thurston2014positive}.
For ${\rm SL}_3$, a so-called reduced non-elliptic basis was constructed in \cite{frohman20223,higgins2020triangular}. 
This basis behaves well with respect to the quantum trace map \cite{douglas2024tropical,Kim20}, but it does not satisfy a positivity property.
For $n>3$, the existence of ``good'' bases of the (reduced) stated ${\rm SL}_n$-skein algebra remains an important open problem.
When $n=4$ and the surface is a polygon, a rotation-invariant basis was constructed in \cite{gaetz2025rotation}; however, as far as we know, its further structural properties are not yet understood.

As an application of Theorem~\ref{intro-thm-equa}, we show that the theta basis (see \cite{GHKK,davison2021strong}) provides a ``good" basis for both the reduced and the localized stated ${\rm SL}_n$-skein algebras of a polygon, enjoying the following properties.
The theta basis is rotation-invariant (Theorems~\ref{thm:poly1} and \ref{thm:poly2}) and satisfies a  positivity property \cite{davison2021strong}. Moreover, it admits a natural parametrization: by $\mathbb{Z}^{\overline V_\lambda}$ in the reduced case and by $\mathbb{Z}^{V_\lambda}$ in the localized case \cite{GHKK,davison2021strong,Q}. In particular, each basis element has a unique highest monomial term, whose exponent vector coincides with the corresponding coordinate.
Finally, the theta basis is compatible with the Frobenius map \cite{mandel2021scattering}.

At present, however, a web-theoretic interpretation of the theta basis is not known.
We plan to investigate this question in future work.

\subsection{A web interpretation of the dual canonical basis of the quantum group $\mathcal O_q({\rm SL}_3)$.}

It was proved in \cite{LS21} that the stated
$\mathrm{SL}_n$-skein algebra of the bigon $\cS_\omega(\mathbb P_2)$ is isomorphic to the quantum group $\mathcal O_q(\mathrm{SL}_n)$ (see \S\ref{sec:bigon}).
For the special case $n=3$, explicit formulas for cluster variables and clusters are computed in 
\S\ref{sec:clustervar} for ${\mathscr{A}}_\omega(\mathbb{P}_{2})={\cS}_{\omega}(\mathbb{P}_{2})=\mathcal O_q({\rm SL}_3)$ (see Lemma~\ref{lem:A=U=skein}). 
In particular, each cluster variable of ${\cS}_{\omega}(\mathbb{P}_{2})=\mathcal O_q({\rm SL}_3)$ is of one of the following forms:
\begin{align}\label{intro-eq-cluster-va}
 \raisebox{-.20in}{
\begin{tikzpicture}
\tikzset{->-/.style={
    decoration={markings,mark=at position #1 with {\arrow{latex}}},
    postaction={decorate}
}}
\filldraw[draw=white,fill=gray!20] (0,0) rectangle (1, 1);
\draw [line width=0.8pt,decoration={markings, mark=at position 0.5 with {\arrow{<}}},postaction={decorate}] (0,0.5)--(1,0.5);
\draw[line width=1pt] (0,0)--(0,1);
\draw[line width=1pt] (1,0)--(1,1);
\node[left] at(0,0.5) {$i$};
\node[right] at(1,0.5) {$j$};
\end{tikzpicture}
},\qquad
    \raisebox{-.20in}{
\begin{tikzpicture}
\tikzset{->-/.style={
    decoration={markings,mark=at position #1 with {\arrow{latex}}},
    postaction={decorate}
}}
\filldraw[draw=white,fill=gray!20] (0,0) rectangle (1, 1);
\draw [line width=0.8pt,decoration={markings, mark=at position 0.5 with {\arrow{>}}},postaction={decorate}] (0,0.5)--(1,0.5);
\draw[line width=1pt] (0,0)--(0,1);
\draw[line width=1pt] (1,0)--(1,1);
\node[left] at(0,0.5) {$i$};
\node[right] at(1,0.5) {$j$};
\end{tikzpicture}
},\qquad
\left[\raisebox{-.25in}{

\begin{tikzpicture}
\tikzset{->-/.style=

{decoration={markings,mark=at position #1 with

{\arrow{latex}}},postaction={decorate}}}

\filldraw[draw=white,fill=gray!20] (0,0) rectangle (1.5, 1.5);
\draw [line width =1pt,decoration={markings, mark=at position 0.2 with {\arrow{<}}},postaction={decorate}] (0,0)--(0,1.5);
\draw [line width =1pt,decoration={markings, mark=at position 0.2 with {\arrow{<}}},postaction={decorate}] (1.5,0)--(1.5,1.5);
\draw [line width =0.8pt,decoration={markings, mark=at position 0.5 with {\arrow{<}}},postaction={decorate}](0,0.5)--(0.75,0.5);
\draw [line width =0.8pt,decoration={markings, mark=at position 0.5 with {\arrow{<}}},postaction={decorate}](1.5,0.5)--(0.75,0.5);
\draw [line width =0.8pt,decoration={markings, mark=at position 0.5 with {\arrow{>}}},postaction={decorate}](0,1)--(0.75,1);
\draw [line width =0.8pt,decoration={markings, mark=at position 0.5 with {\arrow{>}}},postaction={decorate}](1.5,1)--(0.75,1);
\draw [line width =0.8pt,decoration={markings, mark=at position 0.5 with {\arrow{>}}},postaction={decorate}](0.75,0.5)--(0.75,1);
\node [left] at(0,0.5) {$3$};
\node [right] at(1.5,0.5) {$1$};
\node [left] at(0,1) {$1$};
\node [right] at(1.5,1) {$3$};
\end{tikzpicture}
}\right]_{\rm norm},\qquad
\left[\raisebox{-.25in}{

\begin{tikzpicture}
\tikzset{->-/.style=

{decoration={markings,mark=at position #1 with

{\arrow{latex}}},postaction={decorate}}}

\filldraw[draw=white,fill=gray!20] (0,0) rectangle (1.5, 1.5);
\draw [line width =1pt,decoration={markings, mark=at position 0.2 with {\arrow{<}}},postaction={decorate}] (0,0)--(0,1.5);
\draw [line width =1pt,decoration={markings, mark=at position 0.2 with {\arrow{<}}},postaction={decorate}] (1.5,0)--(1.5,1.5);
\draw [line width =0.8pt,decoration={markings, mark=at position 0.5 with {\arrow{>}}},postaction={decorate}](0,0.5)--(0.75,0.5);
\draw [line width =0.8pt,decoration={markings, mark=at position 0.5 with {\arrow{>}}},postaction={decorate}](1.5,0.5)--(0.75,0.5);
\draw [line width =0.8pt,decoration={markings, mark=at position 0.5 with {\arrow{<}}},postaction={decorate}](0,1)--(0.75,1);
\draw [line width =0.8pt,decoration={markings, mark=at position 0.5 with {\arrow{<}}},postaction={decorate}](1.5,1)--(0.75,1);
\draw [line width =0.8pt,decoration={markings, mark=at position 0.5 with {\arrow{<}}},postaction={decorate}](0.75,0.5)--(0.75,1);
\node [left] at(0,0.5) {$3$};
\node [right] at(1.5,0.5) {$1$};
\node [left] at(0,1) {$1$};
\node [right] at(1.5,1) {$3$};
\end{tikzpicture}
}\right]_{\rm norm},
\end{align}
where $i,j\in\{1,2,3\}$ and $[\;\cdot\;]_{\rm norm}$ denotes the normalization defined in \S\ref{sub-sec-invariant}.
Theorem~\ref{thm:webint} then provides a geometric characterization of all clusters of 
${\mathscr{A}}_\omega(\mathbb{P}_{2})={\cS}_{\omega}(\mathbb{P}_{2})=\mathcal O_q({\rm SL}_3)$.

It is well known that the set of all cluster monomials (Definition~\ref{def-quan-cluster-algebra}(d)) constitutes the dual canonical basis of ${\cS}_{\omega}(\mathbb{P}_{2})=\mathcal O_q({\rm SL}_3)$ (cf. \cite[Theorem 4.26]{BR}), noting that the definition of (quantum) cluster monomials may differ slightly.
In \cite{le_sikora2025}, L{\^e} and Sikora constructed a web interpretation of this basis.  
Motivated by their work, we provide an alternative web interpretation using the description of cluster variables in \eqref{intro-eq-cluster-va} and the clusters in Theorem~\ref{thm:webint}.  
We emphasize that their results and ideas were circulated prior to our construction.

To describe our web interpretation, we fix three points on each boundary component of $\mathbb{P}_2$, labeled $1,2,3$ in counterclockwise order, as illustrated in Figure~\ref{fig:P2-labeled-points}.  
A \emph{labeled arc} $C$ is a properly embedded oriented curve in $\mathbb{P}_2$ connecting the left point $i$ to the right point $j$ for some $i,j \in \{1,2,3\}$.  
Two labeled arcs are said to be \emph{compatible} if they satisfy certain intersection requirements; see Definition~\ref{def:comp}.  
A \emph{system} is a maximal set of pairwise compatible labeled arcs.  
A \emph{weighted system} is a system in which each labeled arc is assigned a nonnegative integer, called its weight.  
We use $\mathcal{WS}$ to denote the set of all weighted systems.

Let $S$ be a weighted system. We associate to $S$ a stated $3$-web diagram $W(S)$ in $\mathbb P_2$ as follows (see \S\ref{sub-dual-basis-web}).  
We place $S$ in general position so that any two labeled arcs in $S$ realize their minimal intersection numbers.  
First, we replace each labeled arc in $S$ with $m$ parallel copies of that arc, where $m$ is its weight, to obtain a collection $S'$ of labeled arcs.  
Next, we obtain $S''$ from $S'$ by replacing each crossing $\raisebox{-.10in}{
	
	\begin{tikzpicture}[scale=0.6, rotate=90]
		\tikzset{->-/.style=
			
			{decoration={markings,mark=at position #1 with
					
					{\arrow{latex}}},postaction={decorate}}}
		\filldraw[draw=white,fill=gray!20] (-0,-0.2) rectangle (1, 1.2);
		\draw [line width =0.6pt,decoration={markings, mark=at position 0.5 with {\arrow{>}}},postaction={decorate}](0.6,0.6)--(1,1);
		\draw [line width =0.6pt,decoration={markings, mark=at position 0.5 with {\arrow{>}}},postaction={decorate}](0.6,0.4)--(1,0);
		\draw[line width =0.6pt] (0,0)--(0.6,0.6);
		\draw[line width =0.6pt] (0,1)--(0.4,0.6);
		\draw[line width =0.6pt] (0.4,0.6)--(0.6,0.4);
	\end{tikzpicture}
}$ with $\raisebox{-.12in}{
	\begin{tikzpicture}[scale=0.6, rotate=90]
		\tikzset{->-/.style=
			
			{decoration={markings,mark=at position #1 with
					
					{\arrow{latex}}},postaction={decorate}}}
		\filldraw[draw=white,fill=gray!20] (0,-0.2) rectangle (1.2, 1.2);
		\draw [line width =0.6pt,decoration={markings, mark=at position 0.7 with {\arrow{>}}},postaction={decorate}](0,0)--(0.4,0.5);
		\draw [line width =0.6pt,decoration={markings, mark=at position 0.7 with {\arrow{>}}},postaction={decorate}](0,1)--(0.4,0.5);
		\draw[line width =0.6pt] (0.4,0.5)--(0.8,0.5);
		\draw [line width =0.6pt,decoration={markings, mark=at position 0.6 with {\arrow{>}}},postaction={decorate}](0.8,0.5)--(1.2,0);
        \draw [line width =0.6pt,decoration={markings, mark=at position 0.6 with {\arrow{>}}},postaction={decorate}](0.8,0.5)--(1.2,1);
	\end{tikzpicture}
}$.  
Then, we construct a crossingless $3$-web diagram $\widetilde S$ in $\mathbb P_2$ from $S''$ by applying the procedure shown in Figure~\ref{fig:moves}(B) for each $i=1,2,3$.  
Finally, we assign the state $i$ to the endpoints of $\widetilde S$ in Figure~\ref{fig:moves}(B) for each $i=1,2,3$. 
See the following example, where the two red numbers $2$ indicate the weights of the corresponding labeled arcs.

\begin{align}\label{eq-intro-exam}
   W\left( \begin{array}{c}\includegraphics[scale=0.6]{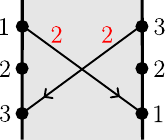}\end{array} \right)
   =\begin{array}{c}\includegraphics[scale=0.6]{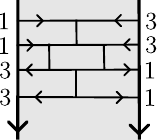}\end{array}.
\end{align}

We prove Lemmas~\ref{lem-basis5}--\ref{lem-basis8}, which establish relations in ${\cS}_{\omega}(\mathbb{P}_{2})$ involving $W(S)$.  
Using these lemmas, we show that $W(S)$ is reflection-normalizable (see \S\ref{sub-sec-invariant}), and that its normalization $[W(S)]_{\rm norm}$ is a cluster monomial in 
${\mathscr{A}}_\omega(\mathbb P_2) = {\cS}_{\omega}(\mathbb{P}_{2})$ (see Proposition~\ref{prop-MS-cluster-monomial}).  
Combining Proposition~\ref{prop-MS-cluster-monomial} with Theorem~\ref{thm:webint}, we obtain the following result, which provides a web interpretation of the dual canonical basis of ${\cS}_{\omega}(\mathbb{P}_{2})=\mathcal O_q({\rm SL}_3)$.

\begin{theorem} [Theorem~\ref{thm-basis-bigon}]
The set $\{[W(S)]_{\rm norm} \mid S \in \mathcal{WS}\}$ is the dual canonical basis of ${\cS}_{\omega}(\mathbb{P}_{2})=\mathcal O_q({\rm SL}_3)$.
\end{theorem}

The web interpretation of the dual canonical basis in \cite{le_sikora2025} uses webs whose orientations point from left to right at each boundary component of $\mathbb P_2$.  
In contrast, the orientations of our webs at each boundary component of $\mathbb P_2$ are mixed.  
At present, it is not clear to us how to relate these two constructions.

\vspace{2mm}

{\bf Acknowledgments:}
 This work was partially supported by the National Natural Science Foundation
of China (No.12471023) (M.H.). 
Z.W. was supported by a KIAS Individual Grant (MG104701) at the Korea Institute for Advanced Study.
We thank Tsukasa Ishibashi for helpful discussions.

\section{Quantum trace maps for (reduced) stated $\SL$-skein algebras}

The quantum trace maps for the (reduced) stated $\SL$-skein algebras were constructed in \cite{LY23}. These maps are algebra homomorphisms from the (reduced) stated $\SL$-skein algebras to quantum tori.
In this section, we review the definition of the (reduced) stated $\SL$-skein algebras and the construction of the quantum trace maps.

\subsection{Stated $\SL$-skein algebras}\label{sub-def-reduced-skein}

\def\Si{\fS}

Let $\fS$ be a pb surface (see \S\ref{intro-sec-reduced}).
Consider the 3-manifold $\Si \times (-1,1)$, the thickened surface. For a point $(x,t) \in \Si \times (-1,1)$, the value $t$ is called the {\bf height} of this point. 
\begin{definition}[\cite{LS21}]\label{def-n-web}
    An {\bf $n$-web} $\alpha$ in $\Si\times(-1,1)$ is a disjoint union of oriented closed curves and a directed finite graph properly embedded into $\Si\times(-1,1)$, satisfying the following requirements:
\begin{enumerate}
    \item[$(1)$] $\alpha$ only contains $1$-valent or $n$-valent vertices. Each $n$-valent vertex is a source or a  sink. The set of $1$-valent vertices is denoted as $\partial \alpha$, which are called \textbf{endpoints} of $\alpha$. For any boundary component $c$ of $\Si$, we require that the points of  $\partial\alpha\cap (c\times(-1,1))$ have mutually distinct heights.
    \item[$(2)$] Every edge of the graph is an embedded oriented  closed interval  in $\Si\times(-1,1)$.
    \item[$(3)$] $\alpha$ is equipped with a transversal \textbf{framing}. 
    \item[$(4)$] The set of half-edges at each $n$-valent vertex is equipped with a  cyclic order. 
    \item[$(5)$] $\partial \alpha$ is contained in $\partial\Si\times (-1,1)$ and the framing at these endpoints is given by the positive direction of $(-1,1)$.
\end{enumerate}
We will consider $n$-webs up to (ambient) \textbf{isotopy} which are continuous deformations of $n$-webs in their class. 
The empty $n$-web, denoted by $\emptyset$, is also considered as an $n$-web, with the convention that $\emptyset$ is only isotopic to itself. 

A {\bf state} for $\alpha$ is a map $s\colon\partial\alpha\rightarrow \{1,2,\cdots,n\}$. A {\bf stated $n$-web} in $\Si\times(-1,1)$ is an $n$-web equipped with a state.
\end{definition}

By regarding $\fS$ as $\fS\times\{0\}$, there is a projection $\text{pr}\colon\fS\times(-1,1)\rightarrow \fS.$
We say the (stated) $n$-web $\alpha$ is in {\bf vertical position} if 
\begin{enumerate}
    \item the framing at everywhere is given by the positive direction of $(-1,1)$,
    \item $\alpha$ is in general position with respect to the projection  $\text{pr}\colon \Si\times(-1,1)\rightarrow \Si\times\{0\}$,
    \item at every $n$-valent vertex, the cyclic order of half-edges as the image of $\text{pr}$ is given by the positive orientation of $\Si$ (drawn counter-clockwise in pictures).
\end{enumerate}

For every (stated) $n$-web $\alpha$, we can isotope $\alpha$ to be in vertical position. For each boundary component $c$ of $\fS$, the heights of $\partial\alpha\cap (c\times(-1,1))$ determine a linear order on  $c\cap \text{pr}(\alpha)$.
Then a {\bf (stated) $n$-web diagram} of $\alpha$ is $\text{pr}(\alpha)$ equipped with the usual over/underpassing information at each double point (called a crossing) and a linear order on $c\cap \text{pr}(\alpha)$ for each boundary component $c$ of $\fS$.

The orientation of $\partial \Si$ induced by the orientation of $\Si$ is called the {\bf positive orientation} of $\partial \Si$. 
The one opposite to the positive orientation of $\partial\Si$ is called the {\bf negative orientation} of $\partial \Si$.
In this paper, we always assume that the orientations of depicted 
surfaces point to the readers, i.e., the surfaces are equipped with the orientations in counterclockwise.  
The orientations of boundary components of a surface also have positive orientations, i.e., these are induced from the orientation of the surface.

A stated $n$-web diagram $\alpha$ is called {\bf negatively ordered}  if the linear order on $\alpha\cap c$, for each boundary component $c$ of $\Si$, is indicated by the negative orientation of $c$.

Let $S_n$ denote the permutation group on the set $\{1,2,\cdots,n\}$. 
For an integer $i\in\{1,2,\cdots,n\}$, we use $\bar{i}$ to denote $n+1-i$.

The \textbf{stated $\SL$-skein algebra} $\cS_{\omega}(\fS)$ of $\fS$ is
the quotient module of the $R$-module (see \S\ref{sec-intro}) freely generated by the set 
 of all isotopy classes of stated 
$n$-webs in $\fS\times (-1,1)$ subject to  relations \eqref{w.cross}-\eqref{wzh.eight} (see \eqref{intro-constants} for involved constants in $R$).

\beq\label{w.cross}
q^{-\frac{1}{n}} 
\raisebox{-.20in}{

\begin{tikzpicture}
\tikzset{->-/.style=

{decoration={markings,mark=at position #1 with

{\arrow{latex}}},postaction={decorate}}}
\filldraw[draw=white,fill=gray!20] (-0,-0.2) rectangle (1, 1.2);
\draw [line width =1pt,decoration={markings, mark=at position 0.5 with {\arrow{>}}},postaction={decorate}](0.6,0.6)--(1,1);
\draw [line width =1pt,decoration={markings, mark=at position 0.5 with {\arrow{>}}},postaction={decorate}](0.6,0.4)--(1,0);
\draw[line width =1pt] (0,0)--(0.4,0.4);
\draw[line width =1pt] (0,1)--(0.4,0.6);
\draw[line width =1pt] (0.4,0.6)--(0.6,0.4);
\end{tikzpicture}
}
- q^{\frac{1}{n}}
\raisebox{-.20in}{
\begin{tikzpicture}
\tikzset{->-/.style=

{decoration={markings,mark=at position #1 with

{\arrow{latex}}},postaction={decorate}}}
\filldraw[draw=white,fill=gray!20] (-0,-0.2) rectangle (1, 1.2);
\draw [line width =1pt,decoration={markings, mark=at position 0.5 with {\arrow{>}}},postaction={decorate}](0.6,0.6)--(1,1);
\draw [line width =1pt,decoration={markings, mark=at position 0.5 with {\arrow{>}}},postaction={decorate}](0.6,0.4)--(1,0);
\draw[line width =1pt] (0,0)--(0.4,0.4);
\draw[line width =1pt] (0,1)--(0.4,0.6);
\draw[line width =1pt] (0.6,0.6)--(0.4,0.4);
\end{tikzpicture}
}
= (q^{-1}-q)
\raisebox{-.20in}{

\begin{tikzpicture}
\tikzset{->-/.style=

{decoration={markings,mark=at position #1 with

{\arrow{latex}}},postaction={decorate}}}
\filldraw[draw=white,fill=gray!20] (-0,-0.2) rectangle (1, 1.2);
\draw [line width =1pt,decoration={markings, mark=at position 0.5 with {\arrow{>}}},postaction={decorate}](0,0.8)--(1,0.8);
\draw [line width =1pt,decoration={markings, mark=at position 0.5 with {\arrow{>}}},postaction={decorate}](0,0.2)--(1,0.2);
\end{tikzpicture}
},
\eeq 
\beq\label{w.twist}
\raisebox{-.15in}{
\begin{tikzpicture}
\tikzset{->-/.style=
{decoration={markings,mark=at position #1 with
{\arrow{latex}}},postaction={decorate}}}
\filldraw[draw=white,fill=gray!20] (-1,-0.35) rectangle (0.6, 0.65);
\draw [line width =1pt,decoration={markings, mark=at position 0.5 with {\arrow{>}}},postaction={decorate}](-1,0)--(-0.25,0);
\draw [color = black, line width =1pt](0,0)--(0.6,0);
\draw [color = black, line width =1pt] (0.166 ,0.08) arc (-37:270:0.2);
\end{tikzpicture}}
= \mathbbm{t}
\raisebox{-.15in}{
\begin{tikzpicture}
\tikzset{->-/.style=
{decoration={markings,mark=at position #1 with
{\arrow{latex}}},postaction={decorate}}}
\filldraw[draw=white,fill=gray!20] (-1,-0.5) rectangle (0.6, 0.5);
\draw [line width =1pt,decoration={markings, mark=at position 0.5 with {\arrow{>}}},postaction={decorate}](-1,0)--(-0.25,0);
\draw [color = black, line width =1pt](-0.25,0)--(0.6,0);
\end{tikzpicture}}
,  
\eeq
\beq\label{w.unknot}
\raisebox{-.20in}{
\begin{tikzpicture}
\tikzset{->-/.style=
{decoration={markings,mark=at position #1 with
{\arrow{latex}}},postaction={decorate}}}
\filldraw[draw=white,fill=gray!20] (0,0) rectangle (1,1);
\draw [line width =1pt,decoration={markings, mark=at position 0.5 with {\arrow{>}}},postaction={decorate}](0.45,0.8)--(0.55,0.8);
\draw[line width =1pt] (0.5 ,0.5) circle (0.3);
\end{tikzpicture}}
= (-1)^{n-1} [n]\ 
\raisebox{-.20in}{
\begin{tikzpicture}
\tikzset{->-/.style=
{decoration={markings,mark=at position #1 with
{\arrow{latex}}},postaction={decorate}}}
\filldraw[draw=white,fill=gray!20] (0,0) rectangle (1,1);
\end{tikzpicture}}
,\ \text{where}\ [n]=\frac{q^n-q^{-n}}{q-q^{-1}},
\eeq
\beq\label{wzh.four}
\raisebox{-.30in}{
\begin{tikzpicture}
\tikzset{->-/.style=
{decoration={markings,mark=at position #1 with
{\arrow{latex}}},postaction={decorate}}}
\filldraw[draw=white,fill=gray!20] (-1,-0.7) rectangle (1.2,1.3);
\draw [line width =1pt,decoration={markings, mark=at position 0.5 with {\arrow{>}}},postaction={decorate}](-1,1)--(0,0);
\draw [line width =1pt,decoration={markings, mark=at position 0.5 with {\arrow{>}}},postaction={decorate}](-1,0)--(0,0);
\draw [line width =1pt,decoration={markings, mark=at position 0.5 with {\arrow{>}}},postaction={decorate}](-1,-0.4)--(0,0);
\draw [line width =1pt,decoration={markings, mark=at position 0.5 with {\arrow{<}}},postaction={decorate}](1.2,1)  --(0.2,0);
\draw [line width =1pt,decoration={markings, mark=at position 0.5 with {\arrow{<}}},postaction={decorate}](1.2,0)  --(0.2,0);
\draw [line width =1pt,decoration={markings, mark=at position 0.5 with {\arrow{<}}},postaction={decorate}](1.2,-0.4)--(0.2,0);
\node  at(-0.8,0.5) {$\vdots$};
\node  at(1,0.5) {$\vdots$};
\end{tikzpicture}}=(-q)^{-\frac{n(n-1)}{2}}\cdot \sum_{\sigma\in S_n}
(-q^{-\frac{1-n}n})^{\ell(\sigma)} \raisebox{-.30in}{
\begin{tikzpicture}
\tikzset{->-/.style=
{decoration={markings,mark=at position #1 with
{\arrow{latex}}},postaction={decorate}}}
\filldraw[draw=white,fill=gray!20] (-1,-0.7) rectangle (1.2,1.3);
\draw [line width =1pt,decoration={markings, mark=at position 0.5 with {\arrow{>}}},postaction={decorate}](-1,1)--(0,0);
\draw [line width =1pt,decoration={markings, mark=at position 0.5 with {\arrow{>}}},postaction={decorate}](-1,0)--(0,0);
\draw [line width =1pt,decoration={markings, mark=at position 0.5 with {\arrow{>}}},postaction={decorate}](-1,-0.4)--(0,0);
\draw [line width =1pt,decoration={markings, mark=at position 0.5 with {\arrow{<}}},postaction={decorate}](1.2,1)  --(0.2,0);
\draw [line width =1pt,decoration={markings, mark=at position 0.5 with {\arrow{<}}},postaction={decorate}](1.2,0)  --(0.2,0);
\draw [line width =1pt,decoration={markings, mark=at position 0.5 with {\arrow{<}}},postaction={decorate}](1.2,-0.4)--(0.2,0);
\node  at(-0.8,0.5) {$\vdots$};
\node  at(1,0.5) {$\vdots$};
\filldraw[draw=black,fill=gray!20,line width =1pt]  (0.1,0.3) ellipse (0.4 and 0.7);
\node  at(0.1,0.3){$\sigma_{+}$};
\end{tikzpicture}},
\eeq
where the ellipse enclosing $\sigma_+$  is the minimum crossing positive braid representing a permutation $\sigma\in S_n$ and $\ell(\sigma)=\#\{(i,j)\mid 1\leq i<j\leq n,\ \sigma(i)>\sigma(j)\}$ is the length of $\sigma\in S_n$.

\beq\label{wzh.five}
   \raisebox{-.30in}{
\begin{tikzpicture}
\tikzset{->-/.style=
{decoration={markings,mark=at position #1 with
{\arrow{latex}}},postaction={decorate}}}
\filldraw[draw=white,fill=gray!20] (-1,-0.7) rectangle (0.2,1.3);
\draw [line width =1pt](-1,1)--(0,0);
\draw [line width =1pt](-1,0)--(0,0);
\draw [line width =1pt](-1,-0.4)--(0,0);
\draw [line width =1.5pt](0.2,1.3)--(0.2,-0.7);
\node  at(-0.8,0.5) {$\vdots$};
\filldraw[fill=white,line width =0.8pt] (-0.5 ,0.5) circle (0.07);
\filldraw[fill=white,line width =0.8pt] (-0.5 ,0) circle (0.07);
\filldraw[fill=white,line width =0.8pt] (-0.5 ,-0.2) circle (0.07);
\end{tikzpicture}}
   = 
   \mathbbm{a} \sum_{\sigma \in S_n} (-q)^{-\ell(\sigma)}\,  \raisebox{-.30in}{
\begin{tikzpicture}
\tikzset{->-/.style=
{decoration={markings,mark=at position #1 with
{\arrow{latex}}},postaction={decorate}}}
\filldraw[draw=white,fill=gray!20] (-1,-0.7) rectangle (0.2,1.3);
\draw [line width =1pt](-1,1)--(0.2,1);
\draw [line width =1pt](-1,0)--(0.2,0);
\draw [line width =1pt](-1,-0.4)--(0.2,-0.4);
\draw [line width =1.5pt,decoration={markings, mark=at position 1 with {\arrow{>}}},postaction={decorate}](0.2,1.3)--(0.2,-0.7);
\node  at(-0.8,0.5) {$\vdots$};
\filldraw[fill=white,line width =0.8pt] (-0.5 ,1) circle (0.07);
\filldraw[fill=white,line width =0.8pt] (-0.5 ,0) circle (0.07);
\filldraw[fill=white,line width =0.8pt] (-0.5 ,-0.4) circle (0.07);
\node [right] at(0.2,1) {$\sigma(n)$};
\node [right] at(0.2,0) {$\sigma(2)$};
\node [right] at(0.2,-0.4){$\sigma(1)$};
\end{tikzpicture}},
\eeq
\beq \label{wzh.six}
\raisebox{-.20in}{
\begin{tikzpicture}
\tikzset{->-/.style=
{decoration={markings,mark=at position #1 with
{\arrow{latex}}},postaction={decorate}}}
\filldraw[draw=white,fill=gray!20] (-0.7,-0.7) rectangle (0,0.7);
\draw [line width =1.5pt,decoration={markings, mark=at position 1 with {\arrow{>}}},postaction={decorate}](0,0.7)--(0,-0.7);
\draw [color = black, line width =1pt] (0 ,0.3) arc (90:270:0.5 and 0.3);
\node [right]  at(0,0.3) {$i$};
\node [right] at(0,-0.3){$j$};
\filldraw[fill=white,line width =0.8pt] (-0.5 ,0) circle (0.07);
\end{tikzpicture}}   = \delta_{\bar j,i }\,  \mathbbm{c}_{i} \raisebox{-.20in}{
\begin{tikzpicture}
\tikzset{->-/.style=
{decoration={markings,mark=at position #1 with
{\arrow{latex}}},postaction={decorate}}}
\filldraw[draw=white,fill=gray!20] (-0.7,-0.7) rectangle (0,0.7);
\draw [line width =1.5pt](0,0.7)--(0,-0.7);
\end{tikzpicture}},
\eeq
\beq \label{wzh.seven}
\raisebox{-.20in}{
\begin{tikzpicture}
\tikzset{->-/.style=
{decoration={markings,mark=at position #1 with
{\arrow{latex}}},postaction={decorate}}}
\filldraw[draw=white,fill=gray!20] (-0.7,-0.7) rectangle (0,0.7);
\draw [line width =1.5pt](0,0.7)--(0,-0.7);
\draw [color = black, line width =1pt] (-0.7 ,-0.3) arc (-90:90:0.5 and 0.3);
\filldraw[fill=white,line width =0.8pt] (-0.55 ,0.26) circle (0.07);
\end{tikzpicture}}
= \sum_{i=1}^n  (\mathbbm{c}_{\bar i})^{-1}\, \raisebox{-.20in}{
\begin{tikzpicture}
\tikzset{->-/.style=
{decoration={markings,mark=at position #1 with
{\arrow{latex}}},postaction={decorate}}}
\filldraw[draw=white,fill=gray!20] (-0.7,-0.7) rectangle (0,0.7);
\draw [line width =1.5pt,decoration={markings, mark=at position 1 with {\arrow{>}}},postaction={decorate}](0,0.7)--(0,-0.7);
\draw [line width =1pt](-0.7,0.3)--(0,0.3);
\draw [line width =1pt](-0.7,-0.3)--(0,-0.3);
\filldraw[fill=white,line width =0.8pt] (-0.3 ,0.3) circle (0.07);
\filldraw[fill=black,line width =0.8pt] (-0.3 ,-0.3) circle (0.07);
\node [right]  at(0,0.3) {$i$};
\node [right]  at(0,-0.3) {$\bar{i}$};
\end{tikzpicture}},
\eeq
\beq\label{wzh.eight}
\raisebox{-.20in}{

\begin{tikzpicture}
\tikzset{->-/.style=

{decoration={markings,mark=at position #1 with

{\arrow{latex}}},postaction={decorate}}}
\filldraw[draw=white,fill=gray!20] (-0,-0.2) rectangle (1, 1.2);
\draw [line width =1.5pt,decoration={markings, mark=at position 1 with {\arrow{>}}},postaction={decorate}](1,1.2)--(1,-0.2);
\draw [line width =1pt](0.6,0.6)--(1,1);
\draw [line width =1pt](0.6,0.4)--(1,0);
\draw[line width =1pt] (0,0)--(0.4,0.4);
\draw[line width =1pt] (0,1)--(0.4,0.6);
\draw[line width =1pt] (0.4,0.6)--(0.6,0.4);
\filldraw[fill=white,line width =0.8pt] (0.2 ,0.2) circle (0.07);
\filldraw[fill=white,line width =0.8pt] (0.2 ,0.8) circle (0.07);
\node [right]  at(1,1) {$i$};
\node [right]  at(1,0) {$j$};
\end{tikzpicture}
} =q^{\frac{1}{n}}\left(\delta_{{j<i} }(q^{-1}-q)\raisebox{-.20in}{

\begin{tikzpicture}
\tikzset{->-/.style=

{decoration={markings,mark=at position #1 with

{\arrow{latex}}},postaction={decorate}}}
\filldraw[draw=white,fill=gray!20] (-0,-0.2) rectangle (1, 1.2);
\draw [line width =1.5pt,decoration={markings, mark=at position 1 with {\arrow{>}}},postaction={decorate}](1,1.2)--(1,-0.2);
\draw [line width =1pt](0,0.8)--(1,0.8);
\draw [line width =1pt](0,0.2)--(1,0.2);
\filldraw[fill=white,line width =0.8pt] (0.2 ,0.8) circle (0.07);
\filldraw[fill=white,line width =0.8pt] (0.2 ,0.2) circle (0.07);
\node [right]  at(1,0.8) {$i$};
\node [right]  at(1,0.2) {$j$};
\end{tikzpicture}
}+q^{\delta_{i,j}}\raisebox{-.20in}{

\begin{tikzpicture}
\tikzset{->-/.style=

{decoration={markings,mark=at position #1 with

{\arrow{latex}}},postaction={decorate}}}
\filldraw[draw=white,fill=gray!20] (-0,-0.2) rectangle (1, 1.2);
\draw [line width =1.5pt,decoration={markings, mark=at position 1 with {\arrow{>}}},postaction={decorate}](1,1.2)--(1,-0.2);
\draw [line width =1pt](0,0.8)--(1,0.8);
\draw [line width =1pt](0,0.2)--(1,0.2);
\filldraw[fill=white,line width =0.8pt] (0.2 ,0.8) circle (0.07);
\filldraw[fill=white,line width =0.8pt] (0.2 ,0.2) circle (0.07);
\node [right]  at(1,0.8) {$j$};
\node [right]  at(1,0.2) {$i$};
\end{tikzpicture}
}\right),
\eeq
where   
$\delta_{j<i}= 
\begin{cases}
1  & j<i\\
0 & \text{otherwise}
\end{cases},\ 
\delta_{i,j}= 
\begin{cases} 
1  & i=j\\
0  & \text{otherwise}
\end{cases}$, and small white dots represent an arbitrary orientation of the edges (left-to-right or right-to-left), consistent for the entire equation. The black dot represents the opposite orientation. When a boundary edge of a shaded area is directed, the direction indicates the height order of the endpoints of the diagrams on that directed line, where going along the direction increases the height, and the involved endpoints are consecutive in the height order. The height order outside the drawn part can be arbitrary.

Our parameter $\omega^{\frac{1}{2}}$ corresponds to $\hat q^{-1}$ in \cite{LY23} (our setting fits well with the quantum cluster theory). 
Note that $R$ carries a natural $\mathbb{Z}[\omega^{\pm \frac{1}{2}}]$-module structure, making it a 
$\mathbb{Z}[\omega^{\pm \frac{1}{2}}]$-algebra. 
For the remainder of this paper, we will assume $R = \mathbb{Z}[\omega^{\pm \frac{1}{2}}]$. 
All results remain valid for a general $R$ by applying the functor 
$-\otimes_{\mathbb{Z}[\omega^{\pm \frac{1}{2}}]} R$.

 The algebra structure for $\cS_{\omega}(\fS)$ is given by stacking the stated $n$-webs, i.e. for any two stated $n$-webs $\alpha,\alpha'
 \subset\fS\times(-1,1)$, the product $\alpha\alpha'$ is defined by stacking $\alpha$ above $\alpha'$. That is, if $\alpha \subset \fS\times(0,1)$ and $\alpha' \subset \fS\times (-1,0)$, we have $\alpha \alpha' = \alpha \cup \alpha'$.

 Let $f$ be a proper embedding $\fS_1\rightarrow\fS_2$ between two pb surfaces. There is an $R$-module homomorphism $f_*\colon\cS_\omega(\Si_1)\rightarrow \cS_\omega(\Si_2)$ defined as following.
Let $\alpha$ be a negatively ordered stated $n$-web diagram in $\Si_1$, define $f_*(\alpha)\in \cS_\omega(\Si_2)$ represented by the negatively ordered stated $n$-web diagram $f(\alpha)$ \cite{LY23}.

For a boundary puncture $p$ of a pb surface $\fS$, 
corner arcs $C(p)_{ij}$ and $\overline{C}(p)_{ij}$ are stated arcs depicted as in Figure \ref{Fig;badarc}.
For a boundary puncture $p$ which is not on a $\mathbb P_1$ component of $\fS$, set 
$$C_p=\{C(p)_{ij}\mid i<j\},\quad\overline{C}_p=\{\overline{C}(p)_{ij}\mid i<j\}.$$  
Each element of $C_p\cup \overline{C}_p$ is called a \textbf{bad arc} at $p$.

\begin{figure}[h]
    \centering
    \includegraphics[width=150pt]{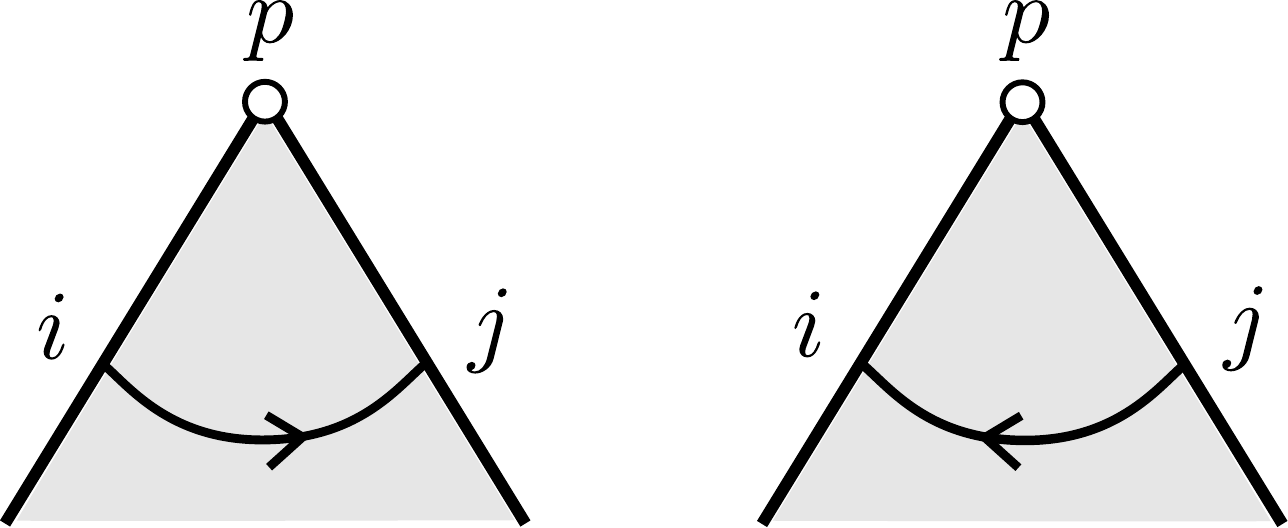}
    \caption{The left is $C(p)_{ij}$ and the right is $\overline{C}(p)_{ij}$.}\label{Fig;badarc}
\end{figure}

For a pb surface $\fS$,  $$\overline \cS_{\omega}(\fS) = \cS_{\omega}(\fS)/I^{\text{bad}}$$ 
is called the \textbf{reduced stated $\SL$-skein algebra}, defined in \cite{LY23}, where $I^{\text{bad}}$ is the two-sided ideal of $\cS_{\omega}(\fS)$ generated by all bad arcs.

\def\al{\alpha}

\def\bT{\mathbb T}

\def\bZ{\mathbb Z}

\def\Oq{\mathcal{O}_q(\mathrm{SL}_n)}
\def\cO{\mathcal{O}}

\subsection{The bigon and $\Oq$}\label{sec:bigon}

 Here we refer to \cite{KS,LS21,LY23} for the definition of $\Oq$.

 We call $\fB$ the {\bf bigon}. We can label the two boundary components of the bigon  by $e_l$ and $e_r$. A bigon with this labeling is called a {\bf directed bigon},
 see an example
$
\raisebox{-.20in}{

\begin{tikzpicture}
\tikzset{->-/.style=

{decoration={markings,mark=at position #1 with

{\arrow{latex}}},postaction={decorate}}}

\filldraw[draw=black,fill=gray!20] (0.5 ,0.5) circle (0.5);
\filldraw[draw=black,fill=white] (0.5,0) circle (0.05);
\filldraw[draw=black,fill=white] (0.5,1) circle (0.05);
\node [left] at(0,0.5) {$e_{l}$};
\node [right] at(1,0.5) {$e_{r}$};
\end{tikzpicture}
}
$. 
We can draw $\fB$ like $
\raisebox{-.20in}{

\begin{tikzpicture}
\tikzset{->-/.style=

{decoration={markings,mark=at position #1 with

{\arrow{latex}}},postaction={decorate}}}

\filldraw[draw=white,fill=gray!20] (0,0) rectangle (1, 1);
\draw[line width =1pt] (0,0)--(0,1);
\draw[line width =1pt] (1,0)--(1,1);
\end{tikzpicture}
}
$, and use $b_{ij}$
to denote 
$
\raisebox{-.20in}{

\begin{tikzpicture}
\tikzset{->-/.style=

{decoration={markings,mark=at position #1 with

{\arrow{latex}}},postaction={decorate}}}

\filldraw[draw=white,fill=gray!20] (0,0) rectangle (1, 1);
\draw [line width =1pt,decoration={markings, mark=at position 0.5 with {\arrow{>}}},postaction={decorate}](0,0.5)--(1,0.5);
\draw[line width =1pt] (0,0)--(0,1);
\draw[line width =1pt] (1,0)--(1,1);
\node [left] at(0,0.5) {$i$};
\node [right] at(1,0.5) {$j$};
\end{tikzpicture}
}
$, where $i,j\in\{1,2,\cdots,n\}$.




For
$i,j,k,l\in \{1,2,\cdots,n\}$,
we have the following coefficients
\beq
 \cR^{ij}_{lk} = q^{\frac 1n} \left(    q^{ -\delta_{i,j}} \delta_{jk} \delta_{il} + (q^{-1}-q)
    \delta_{j<i} \delta_{jl} \delta_{ik}\right),
 \label{R}
\eeq
where $\delta_{j<i}=1$ if $j<i$ and $\delta_{j<i}=0$ otherwise.

Let $\cO_q(M(n))$ be the associative algebra generated by  $u_{ij}$,   $i,j\in\{1,2,\cdots,n\},$
subject to the relations 
\beq
(\buu \ot \buu) \cR = \cR (\buu \ot \buu),  
\eeq
where $\cR$ is the $n^2\times n^2$ matrix given by equation \eqref{R}, and $\buu \ot \buu$ is the $n^2\times n^2$ matrix with entries $(\buu \ot \buu)^{ik}_{jl} = u_{i,j} u_{k,l}$ for $i,j,k,l\in \{1,2,\cdots,n\}$. 

Define  the element 
$$ {\det}_q(\buu):= \sum_{\sigma\in S_n} (-q)^{-\ell(\sigma)}u_{1\sigma(1)}\cdots u_{n\sigma(n)} = \sum_{\sigma\in S_n} (-q)^{-\ell(\sigma)}u_{\sigma(1)1}\cdots u_{\sigma(n)n}.$$

Define $\Oq$ to be  $\cO_q(M(n))/(\det_q(\buu)-1).$ Then 
$\Oq$ is a Hopf algebra with the Hopf algebra structure given by
\begin{align*}
\Delta(u_{ij})  = \sum_{k=1}^n u_{ik} \ot u_{kj}, \quad  \epsilon(u_{ij})= \delta_{ij},\quad 
S({u}_{ij} )
= (-q)^{j-i} {\det}_q(\buu^{ji}), 
\end{align*}
where $\Delta,\ \epsilon$ and $S$ denote the comultiplication, the counit and the antipode respectively, and 
$\buu^{ji}$ is the result of removing the $j$-th row and $i$-th column from $\buu$.

\begin{theorem}[\cite{LS21}]\label{Hopf} 
 There is an algebra isomorphism  $g_{big}\colon \cO_{q}(\SL) \rightarrow \cS_\omega(\fB) $  defined by
  $ g_{big} (u_{ij}) = b_{ij}$.
\end{theorem}

\subsection{Reflection}\label{sub-sec-invariant}
Recall that $R = \bZ[\omega^{\pm\frac{1}{2}}]$. 
An $R$-algebra with reflection is
an $R$-algebra $A$ equipped with a $\mathbb Z$-linear anti-involution ${\bf Rf}$, called the reflection, such that ${\bf Rf}(\omega^{\frac{1}{2}}) = \omega^{-\frac{1}{2}}$. In other words, ${\bf Rf}\colon A\rightarrow A$ is a $\mathbb Z$-linear map such that for all
$x,y\in A$,
$${\bf Rf}(xy) = {\bf Rf}(y){\bf Rf}(x),\quad
{\bf Rf}(\omega^{\frac{1}{2}}x)
=\omega^{-\frac{1}{2}}x,\quad
{\bf Rf}^2={\rm Id}_A.$$
An element $z\in A$ is called reflection invariant if ${\bf Rf}(z)=z$. If $B$ is another $R$-algebra
with reflection ${\bf Rf}'$, then a map $f\colon A\rightarrow B$ is {\bf reflection invariant} if $f\circ {\bf Rf} = {\bf Rf}'\circ f$.

For a pb surface $\fS$, 
it is known that there is a unique reflection $$* \colon \cS_{\omega} (\fS)\to \cS_{\omega} (\fS)$$ such that, for a stated n-web diagram $\alpha$, $*(\alpha)$ is defined from $\alpha$ by switching all the crossings and reversing the height order on each boundary edge \cite[Theorem 4.9]{LS21}. 

A stated $n$-web diagram $\alpha$ in a pb surface $\fS$ is \textbf{reflection-normalizable} if $*(\alpha) = \omega^{k}\alpha$ for $k \in \bZ$. Note that such $k$ is unique if each connected component of $\fS$ contains non-empty boundaries \cite{LS21}. We define the \textbf{reflection-normalization} of $\alpha$ by
$$[\alpha]_{\rm norm} := \omega^{\frac{k}{2}}\alpha.$$
Then we have $*([\alpha]_{\rm norm}) = [\alpha]_{\rm norm}$, i.e., $[\alpha]_{\rm norm}$ is reflection invariant.

\subsection{The $n$-triangulation}
\label{sec-traceX}

\begin{figure}[h]
    \centering
    \includegraphics[width=220pt]{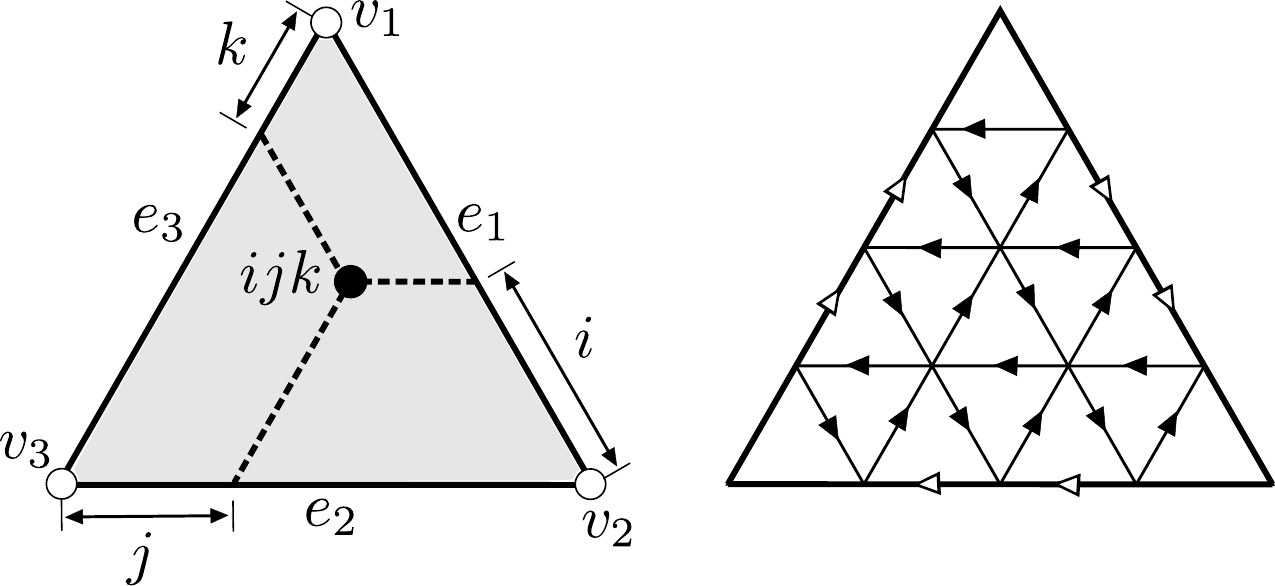}
    \caption{Barycentric coordinates $ijk$ and a $4$-triangulation with its quiver.}\label{Fig;coord_ijk}
\end{figure}

Consider barycentric coordinates for an ideal triangle $\bP_3$ so that
\begin{equation*}
\bP_3=\{(i,j,k)\in\bR^3\mid i,j,k\geq 0,\ i+j+k=n\}\setminus\{(0,0,n),(0,n,0),(n,0,0)\}, 
\end{equation*}
where $(i,j,k)$ (or $ijk$ for simplicity) are the barycentric coordinates. 
Let $v_1=(n,0,0)$, $v_2=(0,n,0)$, $v_3=(0,0,n)$. 
Let $e_i$ denote the edge on $\partial \bP_3$ whose endpoints are $v_i$ and $v_{i+1}$. 
See Figure \ref{Fig;coord_ijk}.

The \textbf{$n$-triangulation} of $\bP_3$ is the subdivision of $\bP_3$ into $n^2$ small triangles using lines $i,j,k=\text{constant integers}$. 
For the $n$-triangulation of $\bP_3$, the vertices and edges of all small triangles, except for $v_1,v_2,v_3$ and the small edges adjacent to them, form a quiver $\Gamma_{\bP_3}$.
An \textbf{arrow} is the direction of a small edge defined as follows. If a small edge $e$ is in the boundary $\partial\bP_3$ then $e$ has the clockwise direction of $\partial \bP_3$. If $e$ is interior then its direction is the same with that of a boundary edge of $\bP_3$ parallel to $e$. Assign weight $\frac{1}{2}$ to any boundary arrow and weight $1$ to any interior arrow.

\def\bZ{\mathbb Z}

Let $\overline{V}_{\bP_3}$ be the set of all
points with integer barycentric coordinates of $\bP_3$:
\begin{align}\label{def-V-P3}
\overline{V}_{\bP_3} = \{ijk \in \bP_3 \mid i, j, k \in \bZ\}.
\end{align}
Note that $\overline{V}_{\bP_3}$ does not contain
$v_1$, $v_2$, or $v_3$.
Elements of $\overline{V}_{\bP_3}$ are called \textbf{small vertices}, and small vertices on the boundary of $\bP_3$ are called the \textbf{edge vertices}. 

A pb surface $\fS$ is said to be {\bf triangulable} if no connected component of $\fS$ is of the following types:  
\begin{enumerate}[label={\rm (\arabic*)}]
    \item the monogon $\mathbb{P}_1$,  
    \item the bigon $\mathbb{P}_2$,  
    \item a sphere with one or two punctures,  
    \item a closed surface (i.e., a pb surface without punctures).  
\end{enumerate}

A {\bf triangulation} $\lambda$ of $\fS$ is a collection of disjoint ideal arcs in $\fS$ with the following properties: (1) any two arcs in $\lambda$ are not isotopic; (2) $\lambda$ is maximal under condition (1); (3) every puncture is adjacent to at least two ideal arcs.
Our definition of triangulation excludes the so-called self-folded triangles.
We will call each ideal arc in $\lambda$ an {\bf edge} of $\lambda$.
If an edge is isotopic to a boundary component of $\fS$, we call such an edge a {\bf boundary edge}.
Let $\partial\lambda$ denote the set of all the boundary edges. Note that $\partial\lambda$ is independent of the choice of $\lambda$.
It is well-known that any triangulable surface admits a triangulation.

Suppose that $\fS$ is a triangulable pb surface with a triangulation $\lambda$.
By cutting $\fS$ along all edges in $\lambda\setminus\partial\lambda$, we have a disjoint union of ideal triangles. Each triangle is called a \textbf{face} of $\lambda$. 
We use $\mathbb F_{\lambda}$ to denote the set of faces of $\lambda$.
Then
\begin{equation}\label{eq.glue}
\fS = \Big( \bigsqcup_{\tau\in\mathbb F_\lambda} \tau \Big) /\sim,
\end{equation}
where each face $\tau$ is regarded as a copy of $\bP_3$, and $\sim$ is the identification of edges of the faces to recover $\lambda$. 
Each face $\tau$ is characterized by a \textbf{characteristic map} 
\begin{align}\label{eq-character-map}
    f_\tau\colon \mathbb P_3 \to \fS,
\end{align}
which is a homeomorphism when we restrict $f_\tau$ to $\Int\mathbb P_3$ or the interior of each edge of $\mathbb P_3$.

An \textbf{$n$-triangulation} of $\lambda$ is a collection of $n$-triangulations of the faces $\tau$ which are compatible with the gluing $\sim$,  where compatibility means, for any edges $b$ and $b'$ glued via $\sim$, the vertices on $b$ and $b'$ are identified. Define
$$\overline{V}_\lambda=\bigcup_{\tau\in\mathbb F_\lambda} \overline{V}_\tau, \quad \overline{V}_\tau=f_\tau(\overline V_{\mathbb P_3}).$$
The images of the weighted quivers $\Gamma_{\mathbb P_3}$ by $f_\tau$ form a quiver $\Gamma_\lambda$ on $\fS$.
Note that when edges $b$ and $b'$ are glued, a small edge on $b$ and the corresponding small edge of $b'$ have opposite directions, i.e. the resulting arrows are of weight $0$. 

For any two $v,v'\in \overline{V}_\lambda$, define
$$
a_\lambda(v,v') = \begin{cases} w \quad & \text{if there is an arrow from $v$ to $v'$ of weight $w$},\\
0 &\text{if there is no arrow between $v$ and $v'$.} 
\end{cases}$$
Let $\overline{Q}_\lambda\colon \overline{V}_\lambda\times \overline{V}_\lambda \to \frac{1}{2}\bZ$ be the signed adjacency matrix of the weighted quiver $\Gamma_\lambda$ defined by 
\begin{equation}\label{eq-def-Q-lambda-re}
\overline{Q}_\lambda(v,v') = a_\lambda(v,v') - a_\lambda(v',v).
\end{equation}
Especially we use $\overline Q_{\bP_3}$ to denote $\overline{Q}_\lambda$ when $\fS=\bP_3$.

\begin{remark}\label{rem-Q-LY}
    Note that the arrows in the quiver of Figure~\ref{Fig;coord_ijk} are oriented oppositely to those in \cite[Figure~10]{LY23}, 
and that $\overline{Q}_\lambda$ in this paper equals 
$-\tfrac{1}{2}\,\overline{\mathsf{Q}}_\lambda$ as defined in \cite{LY23}.  
\end{remark}

\subsection{The $\mathcal A$-version quantum trace maps} \label{sec;A_tori}
In this subsection, assume $\fS$ contains interior punctures.

For any ideal triangle $\tau\in \mathbb F_{\lambda}$, we define a  map $\skeleton_\tau: \overline{V}_{\lambda}\to \mathbb Z[\barV_\tau]$ by the following steps:

Step (i): For a small vertex $v\in \overline{V}_{\lambda}$, we choose an ideal triangle $\nu\in\mathbb F_{\lambda}$ containing $v$, say $v=(ijk)\in \barV_\nu$. 

Step (ii): Draw a weighted directed graph $Y_v$ properly embedded into $\nu$ as in the left of Figure~\ref{Fig;skeleton}, where an edge of $Y_v$ has weight $i$, $j$ or $k$ according as the endpoint lying on the edge $e_1$, $e_2$ or $e_3$ respectively.


\begin{figure}[h]
    \centering
    \includegraphics[width=320pt]{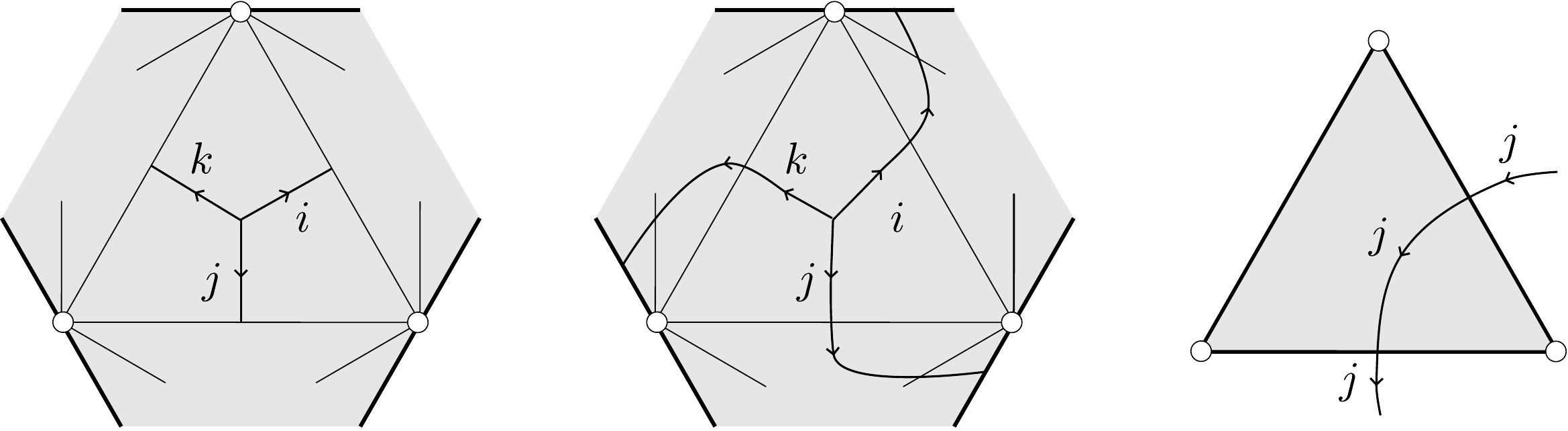}
    \caption{Left: Weighted graph $Y_v$,\quad Middle: Elongation $\widetilde{Y}_v$,\quad Right: Turning left}\label{Fig;skeleton}
\end{figure}

Step (iii): Elongate the nonzero-weighted edges of $Y_v$ to have an embedded weighted directed graph $\widetilde{Y}_v$ as drawn in the middle of Figure~\ref{Fig;skeleton}. Here, each edge is elongated by turning left whenever it enters a triangle. 
The part of the elongated edge in a triangle $\tau$ is called a \textbf{(arc) segment} of $\widetilde{Y}_v$ in $\tau$. In addition, we also regard $Y_v$ as a segment of $\widetilde{Y}_v$, called the \textbf{main segment}.

Step (iv): For the main segment $s=Y_v$, define $Y(s) = v \in \barV_\nu$. For any other arc segment $s$ in an ideal triangle $\sigma \in \mathbb F_{\lambda}$, define $Y(s)\in \overline{V}_\sigma$ to be the small vertex of the following weighted graph. 
$$
s=\begin{array}{c}\includegraphics[scale=0.27]{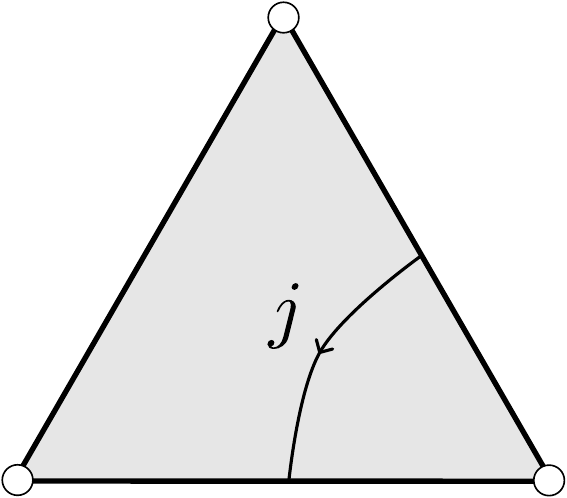}\end{array}
\longrightarrow\quad Y(s):=\begin{array}{c}\includegraphics[scale=0.27]{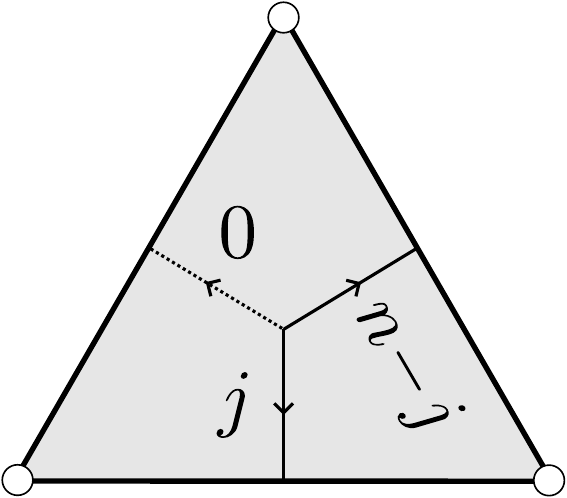}\end{array}
$$
Step (v): Now we focus on the segments contained in the given triangle $\tau$ and define  $\skeleton_\tau(v)\in \bZ [\barVt]$ by
\begin{equation}
\skeleton_\tau(v) = \sum_{s \subset \tau\cap\tilde{Y}_v} Y(s) \in \bZ [\barVt].
\end{equation}
It is known that $\skeleton_\tau(v)$ is well-defined \cite[Lemma~11.4]{LY23}, i.e.,  $\skeleton_\tau(v)$ does not depend on the choice of trianlge $\nu$ in Step (i).

\def\barP{\overline{P}}

Define the $\bZ_3$-invariant the anti-symmetric integer matrix 
\[
\barP_{\mathbb P_3} \colon \barV_{\mathbb P_3}\times\barV_{\mathbb P_3} \to n\mathbb{Z}
\]
as follows. Let \(v = ijk\) and \(v' = i'j'k'\) be two small vertices in \(\overline V_{\mathbb P_3}\).  
If the pair \((v,v')\) satisfies
\[
 \text{either } i \le i' \text{ and } j \le j', \quad \text{or } i \ge i' \text{ and } j \ge j',
\]
then
\begin{equation}\label{def-PP3}
\barP_{\bP_3}(v,v') = -n
\begin{vmatrix}
i & j \\
i' & j'
\end{vmatrix}
= -n(ij' - ji').
\end{equation}
It is known that $\barP_{\mathbb P_3}$ is well-defined \cite[Page~63]{LY23}. 


Recall that $\cF_\lambda$ denotes the set of all the faces of the triangulation $\lambda$. 
Define the anti-symmetric integer matrix $\barP_{\lambda}\colon \overline V_\lambda\times \overline V_\lambda\rightarrow n\mathbb Z$ by 
\begin{equation}\label{eq-anti-matric-P-def}
\overline P_{\lambda} (u,v)=\sum_{\tau\in \mathbb F_\lambda}\overline P_{\tau}(\skeleton_\tau(u),\skeleton_\tau(v)).
\end{equation}


\begin{remark}\label{rem-P-P}
    Our $\overline P_\lambda$ is the matrix $-\overline{\mathsf P}_\lambda$ defined in \cite[Equation~(163) and (205)]{LY23}.
\end{remark}

\cite[Equation~(214)]{LY23} implies that 
\begin{align}\label{eq-prod-PQ}
    \overline P_\lambda \overline{Q}_\lambda
    =\begin{pmatrix}
-2n^2(\text{Id}_{\mathring{\overline{V}}_\lambda\times \mathring{\overline{V}}_\lambda}) & * \\
O    &  * \\
\end{pmatrix},
\end{align}
where $\mathring{\overline{V}}_\lambda\subset \overline{V}_\lambda$ consists of all small vertices contained in the interior of $\fS$.

The following is the \textbf{$\mathcal A$-version quantum torus} of $(\fS,\lambda)$ \cite{LY23}:
\begin{equation}
\overline{\mathcal{A}}_{\omega}(\fS,\lambda) = R \langle 
\overline\Y_v^{\pm 1}, v \in \overline{V}_\lambda \rangle \Big/ (
\overline\Y_v 
\overline\Y_{v'}= \omega^{\overline P_\lambda(v,v')} 
\overline\Y_{v'} 
\overline\Y_v \text{ for } v,v'\in \overline{V}_\lambda ).
\end{equation}
For any $v_1,\ldots,v_r \in \overline{V}_\lambda$ and $a_1,\ldots,a_r \in \mathbb{Z}$,
\begin{align}
\label{Weyl_ordering-Y}
\left[ \overline\Y_{v_1}^{a_1} \overline\Y_{v_2}^{a_2} \cdots \overline\Y_{v_r}^{a_r} \right] := \omega^{-\frac{1}{2}\sum_{i<j}a_i a_j \overline P_\lambda(v_i,v_j)} \overline\Y_{v_1}^{a_1} \overline\Y_{v_2}^{a_2} \cdots \overline\Y_{v_r}^{a_r}
\end{align}
In particular, for ${\bf t} = (t_v)_{v\in \overline{V}_\lambda} \in \mathbb{Z}^{\overline{V}_\lambda}$, define
\begin{align}\label{def-monomial-for-A}
    \overline\Y^{\bf t} := \left[ \prod_{v\in \overline{V}_\lambda} \overline\Y_v^{t_v}\right].
\end{align}


For $v=(ijk) \in  \barV_\nu\subset \overline{V}_\lambda$ with a triangle $\nu$ of $\lambda$,  consider the graph $\widetilde{Y}_v$ defined in \S\ref{sec;A_tori}. 
By replacing a $k$-labeled edge of $\widetilde{Y}_v$ with $k$-parallel edges, we obtain a stated $n$-web $g''_v$, adjusted by a sign: 
$$\widetilde{Y}_v=
\begin{array}{c}\includegraphics[scale=0.38]{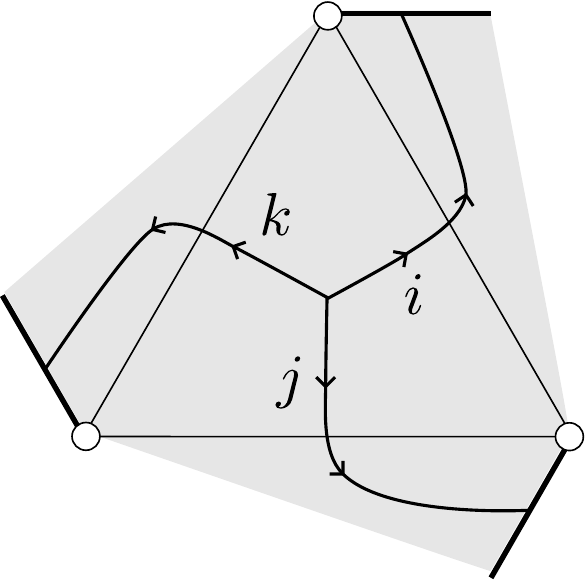}\end{array}\longrightarrow g''_v:=(-1)^{\binom{n}{2}}\begin{array}{c}\includegraphics[scale=0.40]{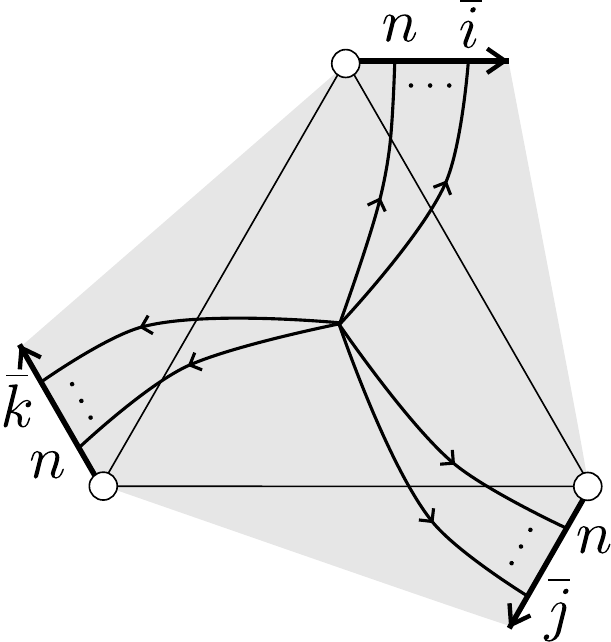}\end{array}. $$
It is known that $g''_v$ is reflection-normalizable \cite[Lemma~4.12]{LY23}.

Define $\ga_v\in \cS_{\omega}(\fS)$ as the reflection invariant of $g_v''$. 
We use $\gaa_v$ to denote the image of $\ga_v$ in $\overline{\cS}_{\omega}(\fS)$ by the projection $\cS_{\omega}(\fS)\to \overline{\cS}_{\omega}(\fS)$. 
The following lemma is shown in \cite{LY23}.

\begin{lemma}\label{gaa-com}
    For any $v,v'\in \overline{V}_\lambda$, we have
    \begin{align}
    \gaa_v\gaa_{v'} = \omega^{\overline P_\lambda(v,v')
    } \gaa_{v'}\gaa_v.
\end{align}
\end{lemma}

We use $\Ap$ to denote the $R$-subalgebra of $\A$ generated by $\overline A^{\bf k}$ for ${\bf k}\in\mathbb N^{\overline{V}_\lambda}$.

\begin{theorem}\cite{LY23}\label{thm-trace-A}
    Let $\fS$ be a triangulable pb surface without interior punctures, and let $\lambda$ be an ideal triangulation of $\fS$. 
    There exists a unique algebra homomorphism 
    $$\trA\colon
    \overline{\cS}_\omega(\fS)\rightarrow\A$$
    with the following properties
    \begin{enumerate}[label={\rm (\alph*)}]\itemsep0,3em

    \item\label{thm-trace-A-a} For each $v\in \overline{V}_\lambda$, we have $\trA(\gaa_v) =\overline  A_v$.

    \item\label{thm-trace-A-b} We have $\Ap\subset\im\trA\subset\A$.

    \item If $n=2,3$, or $n>3$ and $\fS$ is a polygon, then $\trA$ is injective.
    \end{enumerate}
\end{theorem}

\begin{definition}\label{def-key-algebra}
    Let $\fS$ be a triangulable pb surface without interior punctures, and let $\lambda$ be an ideal triangulations of $\fS$. Define 
    $\dS:=\rdS/(\ker \tr)$, called the {\bf projected stated ${\rm SL}_n$-skein algebra}.
    \cite[Lemma~2.12]{huang2025quantum} shows that 
    $\dS$ is independent of the choice of $\lambda$.
\end{definition}

\subsection{Extended quantum trace maps}\label{sub-extended}
A pb surface is called {\bf generalized triangulable} if each of its connected components is either a triangulable pb surface or a bigon. The {\bf generalized triangulation} of a bigon is the set of its two boundary edges.

Let $\fS$ be a generalized triangulable pb surface without interior punctures.
Attach a triangle $\bP_3$ to each boundary edge of $\fS$, and let $\fS^\ast$ denote the resulting triangulable pb surface; see Figure~\ref{Fig;attaching}(A).
Suppose that the attaching edge in each copy of $\bP_3$ is $e_1$.
For any generalized triangulation $\lambda$ of $\fS$, let $\lambda^\ast$ be the ideal triangulation of $\fS^\ast$ whose restriction to $\fS$ is $\lambda$.
Define the extended $A$-vertex set $\VA\subset \overline{V}_{\lambda^\ast}$ to be the subset of all small vertices not on
the $e_2$ edge in the attached triangles.

Note that $\fB^\ast=\mathbb P_4$, and the induced triangulation is illustrated as in Figure~\ref{Fig;attaching}(B).

\begin{figure}[htbp]
    \centering
    \begin{minipage}{0.48\textwidth}
        \centering
        \includegraphics[width=90pt]{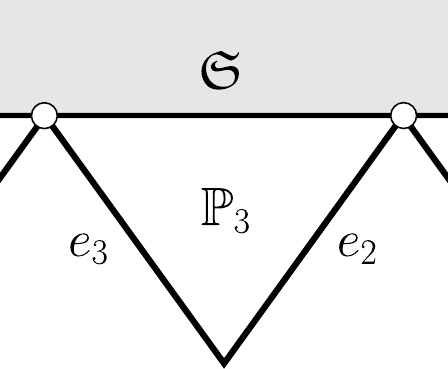}
        \vspace{4pt}
        
        (A)
    \end{minipage}%
    \hspace{-0.05\textwidth} 
    \begin{minipage}{0.48\textwidth}
        \centering
        \includegraphics[width=70pt]{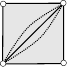}
        \vspace{4pt}
        
        (B)
    \end{minipage}  
    \caption{(A) Attaching triangles. \quad (B) The triangulation for $\fB^\ast=\mathbb P_4$.}
    \label{Fig;attaching}
\end{figure}

Define a map 
\[
p\colon \OVA\setminus \V \;\longrightarrow\; \OVA\setminus V'_\lambda
\]
as follows (see \cite[Figure~18]{LY23} for an illustration).
Each vertex $v\in \OVA\setminus\V$ has coordinates $ijk$ in an attached triangle with $k\neq 0$, while $\OVA\setminus V'_\lambda$ consists of vertices $i'j'k'$ in attached triangles with $i'=0$.
We set
\begin{equation}\label{eq-cov-pdef}
p(v)=(0,n-k,k)
\qquad\text{in the same attached triangle.}
\end{equation}

Define the matrix $\sfC\colon V'_\lambda\times \OVA\to\bZ$ by
\begin{equation}\label{def-matrix-C}
    \begin{aligned}
        \sfC(v,v) &= 1, && v\in V'_\lambda,\\
\sfC(v,p(v)) &= -1, && v\in V'_\lambda\setminus\V,\\
\sfC(v,v') &= 0, && \text{otherwise}.
    \end{aligned}
\end{equation}

The extended matrix $P_\lambda \colon \VA\times\VA\to\mathbb Z$ is then given by
\begin{align}\label{def-matrix-Pl}
    P_\lambda := C \, \Past \, C^{T}.
\end{align}

Define the {\bf extended $A$-version quantum tori} associated to $(\fS,\lambda)$ by
\[
\mathcal A_{\omega}(\fS,\lambda)
=
R\langle 
\Y_v^{\pm1},\; v\in \VA
\rangle \Big/ 
\big(
\Y_v \Y_{v'} = \omega^{P_\lambda(v,v')} \Y_{v'} \Y_v,\;
v,v'\in\VA
\big).
\]

\begin{remark}
Our convention for $P_\lambda$ coincides with the matrix $-\mathsf{P}_\lambda$ in \cite[\S11.4]{LY23}.
\end{remark}

Recall from \S\ref{sec;A_tori} that an element $\ga_v\in\SS$ was defined for every $v\in \V\subset\VA$.
We now extend this definition to all vertices $v\in\VA\setminus\V$ \cite[Page~89]{LY23}.

Let $\binom{\bJ}{k}$ denote the set of all $k$-element subsets of $\bJ=\{1,\dots,n\}$.
For any $I\subset\bJ$, set
\[
\bar{I}=\{\bar{i}\mid i\in I\},\qquad
I^c=\bJ\setminus I,\qquad
\bar{I}^c = (\bar{I})^c.
\]
For $I,J\in\binom{\bJ}{k}$, let $M^I_J(\mathbf{u})\in\Oq$ be the quantum determinant of the $I\times J$ submatrix of~$\mathbf{u}$.
Identifying $\cS_\omega(\fB)=\Oq$, and identifying each generator $u_{ij}$ with the stated arc described in \S\ref{sec:bigon}, we proceed as follows.

Let $a$ be a properly embedded oriented arc in $\fS$, and let $N(a)$ be a small open tubular neighborhood of $a$.
Using an identification $\fB\cong N(a)$ inducing $\cS_\omega(\fB)\to\SS$, denote by $M^I_J(a)$ the image of $M^I_J(\mathbf{u})$.
We depict this schematically as
\begin{align}
\begin{array}{c}\includegraphics[scale=0.43]{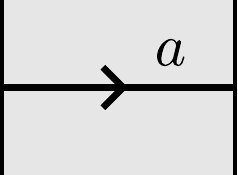}\end{array}
\longrightarrow
I\!\begin{array}{c}\includegraphics[scale=0.43]{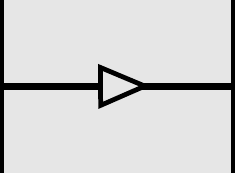}\end{array}\!J
:= M^I_J(a).
\label{eq:quantum_minor}
\end{align}

For an attached triangle $\nu=\bP_3$ and a vertex $v=(ijk)\in \barV_{\nu}\subset V'_\lambda\setminus \barV_\lambda$, let $c$ be the oriented corner arc of $\fS$ starting on $e$ and always turning left:
\[
\begin{array}{c}\includegraphics[scale=0.35]{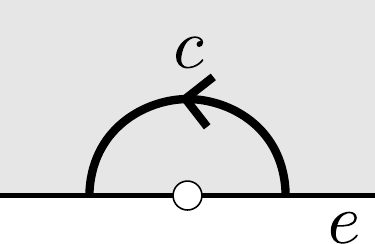}\end{array}
\longrightarrow
g''_v := 
M^{[\,j+1,\; j+i\,]}_{[\bar{i},\,n]}(c)
\begin{array}{c}\includegraphics[scale=0.47]{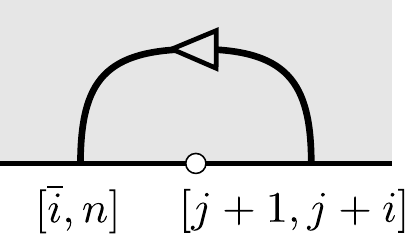}\end{array}.
\]
It is known that $g''_v$ is reflection-normalizable \cite[Lemmas 4.13 and 4.10]{LY23}.
Define $\ga_v$ to be the reflection invariant of $g''_v$.

\begin{theorem}[\cite{LY23}]\label{thm-stated-trace}
Let $\fS$ be a generalized triangulable pb surface with triangulation $\lambda$, and assume that $\fS$ has no interior punctures.
\begin{enumerate}[label={\rm (\alph*)}]\itemsep0.3em

\item For all $v,v'\in\VA$, we have
\[
\ga_v \ga_{v'} \;=\; \omega^{P_\lambda(v,v')}\,\ga_{v'}\ga_v.
\]

\item There exists a unique injective algebra homomorphism
\[
\tr\colon \cS_\omega(\fS)\longrightarrow \As,\qquad
\ga_v \longmapsto A_v,\quad v\in\VA.
\]

\item We have the inclusions
\begin{equation}\label{eq-inclusion-A-s}
\Asp \;\subset\; \im\tr \;\subset\; \As,
\end{equation}
where $\Asp$ is the subalgebra of $\As$ generated by $\{A_v\mid v\in V'_\lambda\}$.

\item $\SS$ is an Ore domain.
\end{enumerate}
\end{theorem}

\section{The quantum cluster algebras}

In this section, we first review the notion of quantum cluster algebras introduced in \cite{BZ}. 
We then recall the quantum cluster structure on the skew field of $\dS$ constructed in \cite{huang2025quantum}. 
This structure will be used in \S\ref{sec-cluster-skein} to construct a quantum cluster structure on the skew field of $\SS$.

\subsection{Quantum cluster algebras}\label{sec-mutation-quantum}

Fix a non-empty set $\mathcal{V}$, and a non-empty subset $\mathcal{V}_{\rm mut}$ of $\mathcal{V}$. The elements of $\mathcal{V}$, $\mathcal{V}_{\rm mut}$ and $\mathcal{V}\setminus \mathcal{V}_{\rm mut}$ are called {\bf vertices}, {\bf mutable vertices} and {\bf frozen vertices}, respectively. By an (ice) quiver $\Gamma$ we mean a directed graph whose set of vertices is $\mathcal{V}$ 
and equipped with weight on the edges so that the weight on each edge is $1$ unless the edge connects two frozen vertices, in which case the weight is $\frac{1}{2}$. An edge of weight $1$ will be called an {\bf arrow}, and an edge of weight $\frac{1}{2}$ a {\bf half-arrow}. As before, we denote by $Q = (Q(u,v))_{u,v\in \mathcal{V}}$ to denote the signed adjacency matrix of $\Gamma$, which is called an {\bf exchange matrix}. 
We will identify the weighted quiver with its signed adjacency matrix.

Suppose that $k\in\mathcal V_{
\rm mut}$. The {\bf mutation}
$\mu_{k}$ for the quiver $\Gamma$ at the 
mutable vertex $k\in\mathcal V_{
\rm mut}$ is defined by 
 the following procedures:
\begin{enumerate}
    \item for each pair of arrows $i\rightarrow k$ and $k\rightarrow j$ draw an arrow $i\rightarrow j$,
    
    \item reverse all the arrows incident to the vertex $k$,

    \item delete a maximal pairs of arrows $i\rightarrow j$ and $i\leftarrow j$ going in the opposite directions.
\end{enumerate}
We use $Q' = (Q'(u,v))_{u,v\in\mathcal V}$ to denote the signed adjacency matrix of $\Gamma'=\mu_k(\Gamma)$. Then we have 
\begin{align}\label{eq-mutation-Q}
Q'(u,v) = \begin{cases}
    - Q(u,v) & k\in\{u,v\},\\
    Q(u,v) +\frac{1}{2}\Big(Q(u,k)|Q(k,v)|+
    |Q(u,k)|Q(k,v)\Big) & k\notin\{u,v\}.
\end{cases}
\end{align}

\def\Fr{{\rm Frac}}

Let $\mathcal F$ be a skew-field over $R=\mathbb Z[\omega^{\pm\frac{1}{2}}]$. Recall that $\xi=\omega^n$.
\begin{definition}[\cite{BZ}]\label{def-quantum-seed}
    A {\bf quantum seed} in $ \mathcal F$ is a triple $\mathsf s=(Q,\Pi,M)$, where 
\begin{enumerate}
    \item[$(1)$] $Q=(Q(u,v))_{u,v\in\mathcal V}$ is an exchange matrix;

    \item[$(2)$] $\Pi=(\Pi(u,v))_{u,v\in\mathcal V}$ is a anti-symmetric matrix with integral entries satisfying the compatibility relation
    $$\sum_{k\in\mathcal V}Q(k,u) \Pi(k,v)=\delta_{u,v} d_u$$
    for all $u\in\mathcal V_{\text{mut}}$
    and $v\in\mathcal V$, where $d_u$ is a positive integer for $u\in\mathcal V_{\text{mut}}$;

    \item[$(3)$] $M\colon \mathbb Z^{\mathcal V}\rightarrow\mathcal F\setminus\{0\}$ is a function such that 
    \begin{align}\label{eq-M-relation}
        M({\bf k}) M({\bf t}) = \xi^{\frac{1}{2} {\bf k} \Pi {\bf t}^T} M({\bf k} + {\bf t})
    \end{align}
    for row vectors ${\bf k},{\bf t}\in \mathbb Z^{\mathcal V}$, and $M(\mathbb Z^{\mathcal V})$ is a basis of a quantum torus over $R$ whose skew-field is $\mathcal F$.
\end{enumerate}
\end{definition}

For each $i\in\mathcal V$, define ${\bf e}_i\in\mathbb Z^{\mathcal V}$ with
\begin{align}\label{eq-def-vector-ei-ele}
    {\bf e}_i(v)=\delta_{i,v} \text{ for $v\in\mathcal V$}.
\end{align}
Then the function $M$ is uniquely determined by the values $A_i:=M({\bf e}_i)$, which we call {\bf (quantum) cluster variables}.
A (quantum) cluster variable $A_i$ is called a {\bf frozen (quantum) cluster variable} if $i\in\mathcal V\setminus\mathcal V_{\text{mut}}$, and an {\bf exchangeable (quantum) cluster variable} otherwise.

We label the vertex set $\mathcal V$ as 
$\{v_1,\cdots,v_r\}$.
For any ${\bf k}=(k_v)_{v\in\mathcal V}\in \mathbb Z^{\mathcal V}$, we have 
\begin{align}\label{eq-A-power}
    M({\bf k}) = A^{\bf k}:= \left[\prod_{v\in \mathcal V}A_v^{k_v}\right],
\end{align}
where $\left[\prod_{v\in \mathcal V}A_v^{k_v}\right]$ is the Weyl-ordered product 
\begin{align}\label{Weyl-A}
    \left[\prod_{v\in \mathcal V}A_v^{k_v}\right]= \xi^{-\frac{1}{2}\sum_{i<j} k_{v_i} k_{v_j}\Pi(i,j)} A_{v_1}^{k_{v_1}} A_{v_2}^{k_{v_2}} \cdots A_{v_r}^{k_{v_r}}.
\end{align}
Note that the Weyl-ordered product in \eqref{Weyl-A} is independent of how we label the vertex set $\mathcal V$.

Note that the anti-symmetric matrix $\Pi$ is uniquely determined by the function $M$ because of \eqref{eq-M-relation}.

Given a quantum seed $\mathsf s=(Q,\Pi,M)$ in $\mathcal F$, the {\bf mutation} at vertex $k\in\mathcal V_{\text{mut}}$ produces a new quantum seed $(Q',\Pi',M')=\mu_{k}(\mathsf s)$ according to the following rules:
\begin{itemize}
    \item $Q'=\mu_k(Q)$;

    \item the mutation on the quantum cluster variables is
    \begin{equation*}
        A_i':=M'({\bf e}_i)=
        \begin{cases}
            M({\bf e}_i) & i\neq k,\\
            M(-{\bf e}_k + \sum_{j\in\mathcal V}[Q(j,k)]_{+}{\bf e}_j)
            + M(-{\bf e}_k + \sum_{j\in\mathcal V}[-Q(j,k)]_{+}{\bf e}_j)
            & i=k,
        \end{cases}
    \end{equation*}
    where $[a]_+ := \max\{a,0\} = {\textstyle \frac{1}{2}(a+|a|)}$ for $a\in\mathbb R$.

    \item $\Pi'=(\Pi'(i,j))_{i,j\in\mathcal V}$, where, for each pair $i,j\in\mathcal V$, the entry $\Pi'(i,j)$ is the unique integer such that 
    $$M'({\bf e}_i)M'({\bf e}_j)
    =\xi^{\Pi'(i,j)} M'({\bf e}_j)M'({\bf e}_i).$$
\end{itemize}

\begin{definition}\label{def-quantum-class}
Two quantum seeds in $\mathcal F$ are said to be
{\bf mutation-equivalent} if they are transformed to each other by a finite sequence of quantum 
mutations. An equivalence class of quantum seeds is called a {\bf quantum 
mutation class}.
\end{definition}


Let $\mathsf{S}$ be a quantum mutation class of quantum seeds. 
For any $\mathsf s=(Q, \Pi, M)\in \mathsf{S}$, define 
\begin{align}
\textbf{A}^{+}(\mathsf s)&:=\{M({\bf k})\mid {\bf k}\in\mathbb Z_{\geq 0}^{\mathcal V}\}\subset
\textbf{A}(\mathsf s):=\{M({\bf k})\mid {\bf k}\in\mathbb Z^{\mathcal V}\}\subset \mathcal F
\label{eq-A-w}\\
\textbf{A}^{f}(\mathsf s)&:=\{M({\bf k})\mid
{\bf k}\in\mathbb Z^{\mathcal V} \text{ such that }
{\bf k}|_{\mathcal V\setminus \mathcal V_{\rm mut}}\in\mathbb Z_{\geq 0}^{\mathcal V\setminus \mathcal V_{\rm mut}}\} \subset \mathcal F \label{eq-A-f}\\
\mathbb T^f(\mathsf s)&:= \text{span}_R(\textbf{A}^f(\mathsf s))\subset    \mathbb T(\mathsf s):= \text{span}_R(\textbf{A}(\mathsf s))\subset \mathcal F.
    \label{eq-T-w}
\end{align}
We call $\mathbb T(\mathsf s)$ the {\bf based quantum torus} associated to the quantum seed $\mathsf s$.
We use ${\rm Frac}(\mathbb T(\mathsf s))$ to denote 
 the skew-field of fractions of $\mathbb T(\mathsf s)$; for the existence of ${\rm Frac}(\mathbb T(\mathsf s))$, see \cite{Cohn}. 

\def\Exc{\mathsf{S}}

We will define two versions of quantum (upper) cluster algebras, depending on whether the frozen variables are inverted or not.

\begin{definition}\label{def-quan-cluster-algebra}
Let $\mathsf{S}$ be a quantum mutation class of quantum seeds.  

\begin{enumerate}[label={\rm (\alph*)}]\itemsep0.3em
\item
The associated {\bf quantum cluster algebra} is the $R$-subalgebra
$\mathscr{A}_{\mathsf S}\subset \mathcal F$
generated by the set $\bigcup_{\mathsf s\in\Exc} {\bf A}^{+}(\mathsf s)$ (see~\eqref{eq-A-w}) together with the inverses of all frozen quantum cluster variables.  
We also denote by
$\mathscr{A}_{\mathsf S}^+\subset \mathcal F$
the $R$-subalgebra generated by $\bigcup_{\mathsf s\in\Exc} {\bf A}^{+}(\mathsf s)$.

\item Each exchangeable quantum cluster variable of any seed $\mathsf s\in\mathsf S$ is called an {\bf exchangeable quantum cluster variable} of $\mathscr{A}_{\mathsf S}$.

\item The associated {\bf quantum upper cluster algebras} are defined as
\[
\mathscr{U}_{\mathsf S}^+
:= \bigcap_{\mathsf s\in\Exc} \mathbb T^f(\mathsf s)
\;\subset\;
\mathscr{U}_{\mathsf S}
:= \bigcap_{\mathsf s\in\Exc} \mathbb T(\mathsf s)
\;\subset\; \mathcal F,
\]
where $\mathbb T(\mathsf s)$ and $\mathbb T^f(\mathsf s)$ are as in~\eqref{eq-T-w}.

\item Let $\mathsf s=(Q, \Pi, M)\in \mathsf{S}$ be a quantum seed. Let ${\bf k}\in\mathbb Z^{\mathcal V}$ such that 
${\bf k}|_{\mathcal V_{\rm mut}}\in\mathbb Z_{\geq 0}^{\mathcal V_{\rm mut}}$. Then $M({\bf k})$ is called a \textbf{cluster monomial} in $\mathscr{A}_{\mathsf S}$. If, in addition, ${\bf k}|_{\mathcal V}\in\mathbb Z_{\geq 0}^{\mathcal V}$, then $M({\bf k})$ is called a \textbf{cluster monomial} in $\mathscr{A}^+_{\mathsf S}$.
\end{enumerate}
\end{definition}

Given a quantum seed $\mathsf s\in \mathsf S$, we also denote $\mathscr A_{\mathsf S}$ (resp. $\mathscr A^+_{\mathsf S}$, $\mathscr U_{\mathsf S}$, $\mathscr U^+_{\mathsf S}$) as $\mathscr A_{\mathsf s}$ (resp. $\mathscr A^+_{\mathsf s}$, $\mathscr U_{\mathsf s}$, $\mathscr U^+_{\mathsf s}$).


    The following theorem is the so-called Laurent phenomenon. 

\begin{theorem}\cite{BZ,goodearl2017quantum}\label{thm-inclusion-quantum}
    We have the inclusions
    $$\mathscr{A}^{+}_{\mathsf S}\subset \mathscr{U}^{+}_{\mathsf S}\text{ and }\mathscr{A}_{\mathsf S}\subset \mathscr{U}_{\mathsf S}.$$
\end{theorem}

Let $\mathsf s=(Q,\Pi,M)$ be a quantum seed. Suppose $\mathcal V'=\mathcal V'_1\sqcup \mathcal V'_2$ is a subset of $\mathcal V$ such that there are no arrows between vertices in $\mathcal V'_1$ and those in $\mathcal V\setminus \mathcal V'$. Let $Q'$ be the subquiver of $Q$ induced by the vertices in $\mathcal V'$, with the vertices in $\mathcal V'_2$ frozen. Let $\Pi'$ and $M'$ be the restrictions of $\Pi$ and $M$ to $\mathcal V'$, respectively. The following is immediate.

\begin{lemma}
   Using the above notation, we have:
   \begin{enumerate}[label={\rm (\alph*)}]\itemsep0.3em
       \item The triple $\mathsf s'=(Q',\Pi',M')$ constitutes a quantum seed.
       \item The quantum cluster algebra $\mathscr A_{\mathsf s'}$ is a subalgebra of $\mathscr A_{\mathsf s}$. Specifically, for any sequence of mutable vertices $v_1,\dots,v_m$ in $Q'$, it holds that
       $\widetilde A'_{v_m} = \widetilde A_{v_m}.$ 
       Here, $\widetilde A'_{v_m}$ and $\widetilde A_{v_m}$ are the cluster variable associated with the vertex $v_k$ of $\mu_{v_m}\cdots\mu_{v_1}(\mathsf s')$ and $\mu_{v_m}\cdots\mu_{v_1}(\mathsf s)$, respectively.
    \end{enumerate}
\end{lemma}

\begin{definition}\label{def:subseed}
With the above notation, we call the seed $\mathsf s'=(Q',\Pi',M')$ a \textbf{sub-seed} of $\mathsf s$ supported on the vertex set $\mathcal V'$, with the vertices in $\mathcal V'_2$ frozen.
\end{definition}

\subsection{Quasi-isomorphism of quantum cluster algebras}
Let $\mathsf s=(Q,\Pi,M)$ and  $\mathsf s'=(Q',\Pi',M')$ be two quantum seeds. Assume that $Q$ and $Q'$ have the same mutable vertices $\mathcal V_{\rm mut}$. Denote the vertices of $Q$ and $Q'$ by $\mathcal V$ and $\mathcal V'$, respectively. 

\begin{definition}\label{def:QQH}\cite{KWQ,CHL}
We call an $R$-algebra homomorphism $\varphi: \mathscr A_{\mathsf s} \to \mathscr A_{\mathsf s'}$ a quantum quasi-homomorphism (with respect to $\mathsf s$ and $\mathsf s'$) if the following conditions are satisfied:
\begin{enumerate}

\item[$(1)$] The mutable subquiver of $Q$ and $Q'$ are the same, i.e, $Q\mid_{\mathcal V_{\rm mut}}=Q'\mid_{\mathcal V_{\rm mut}}$;

\item[$(2)$] For $v\in \mathcal V_{\rm mut}$, $\varphi(M(\mathbf e_v))=M'({\bf f}_v)$ for some ${\bf f}_v=(f_{v,w})_{w\in \mathcal V'}$ with $f_{v,v}=1$ and $f_{v,w}=0$ for all $w(\neq v)\in \mathcal V_{\rm mut}$. 

\item[$(3)$] For $v\in \mathcal V\setminus \mathcal V_{\rm mut}$, $\varphi(M(\mathbf e_v))=M'({\bf f}_v)$ for some ${\bf f}_v=(f_{v,w})_{w\in \mathcal V'}$ with $f_{v,w}=0$ for all $w\in \mathcal V_{\rm mut}$. 

\item[$(4)$] For all $v\in \mathcal V_{\rm mut}$, $\varphi(M(Q(-,v)))=M'(Q'(-,v))$, where $Q(-,v)$ denotes the column vector of $Q$ indexed by $v$.
\end{enumerate}
\end{definition}

\begin{proposition}\label{prop:quasi-iso}\cite[Proposition 2.7]{CHL}
Assume that $Q|_{\mathcal{V}_{\mathrm{mut}}} = Q'|_{\mathcal{V}_{\mathrm{mut}}}$. 
A map 
$$\varphi: \{M(\mathbf{e}_v) \mid v \in \mathcal{V}\} \to \mathbb{T}(\mathsf{s}')$$
induces a quasi-homomorphism $\varphi: \mathscr{A}_{\mathsf{s}} \to \mathscr{A}_{\mathsf{s}'}$ if and only if the following conditions hold:
\begin{enumerate}
    \item[$(1)$] For each $v \in \mathcal{V}_{\mathrm{mut}}$, we have $\varphi(M(\mathbf{e}_v)) = M'(\mathbf{f}_v)$ for some vector $\mathbf{f}_v = (f_{v,w})_{w \in \mathcal{V}'}$ satisfying $f_{v,v} = 1$ and $f_{v,w} = 0$ for all $w \in \mathcal{V}_{\mathrm{mut}} \setminus \{v\}$.

    \item[$(2)$] For each $v \in \mathcal{V} \setminus \mathcal{V}_{\mathrm{mut}}$, we have $\varphi(M(\mathbf{e}_v)) = M'(\mathbf{f}_v)$ for some vector $\mathbf{f}_v = (f_{v,w})_{w \in \mathcal{V}'}$ such that $f_{v,w} = 0$ for all $w \in \mathcal{V}_{\mathrm{mut}}$.

    \item[$(3)$] For each $v \in \mathcal{V}_{\mathrm{mut}}$, the following two equivalent conditions hold:
    \[
        \sum_{w \in \mathcal{V}} Q(w,v)\mathbf{f}_w = \sum_{w' \in \mathcal{V}'} Q'(w',v)\mathbf{e}_{w'},
    \]
      \[
        \varphi([\prod_{w \in \mathcal{V}}A_w^{Q(w,v)}]) = [\prod_{w \in \mathcal{V'}}A_{w'}^{Q'(w',v)}],
    \]
    where $\mathbf{e}_{w'}$ denotes the standard basis vector in $\mathbb{Z}^{\mathcal{V}'}$ associated with $w'$, defined by $(\mathbf{e}_{w'})_{w''} = \delta_{w', w''}$ (i.e., $1$ if $w''=w'$ and $0$ otherwise).

    \item[$(4)$] For all $v_1, v_2 \in \mathcal{V}$, $\Pi(\mathbf{e}_{v_1}, \mathbf{e}_{v_2}) = \Pi'(\mathbf{f}_{v_1}, \mathbf{f}_{v_2})$.
\end{enumerate}
\end{proposition}

\subsection{Quantum cluster structure inside $\text{Frac}(\dS)$}\label{sec-seed-structure}
Let $\mathcal A$ be an Ore domain over $R$. We use 
$\Fr(\mathcal A)$ to denote the skew-field of $\mathcal A$.

Let $\fS$ be a triangulable pb surface without interior punctures.
It was shown in \cite[Lemma~2.17(b)]{huang2025quantum}
that $\dS$ is an Ore domain.

Let $\lambda$ be a triangulation of $\fS$.  
Set 
$\mathcal V = \V$ and $\mathcal V_{\text{mut}} = \mathring{\overline V}_\lambda$, 
where $\mathring{\overline V}_\lambda$ denotes the set of small vertices lying in the interior of $\fS$.  

Recall the anti-symmetric matrices $\overline Q_\lambda$ (see \S\ref{sec-traceX}) and $\overline P_\lambda$ (see \S\ref{sec;A_tori}).  
Define  
\begin{align}\label{eq-prod-Qpi}
   \overline \Pi_\lambda := \tfrac{1}{n} \overline P_\lambda.
\end{align}
Equations~\eqref{def-PP3} and \eqref{eq-anti-matric-P-def} imply that $\overline \Pi_\lambda$ is an integral matrix.

For each $v\in \mathcal V$, we defined an element $\gaa_v \in \rdS$ (see \S\ref{sec;A_tori}), and we use the same notation $\gaa_v$ for its image under the projection $\rdS \twoheadrightarrow \dS$.  
Recall that $\xi = \omega^n$.  
Then Lemma~\ref{gaa-com} shows that  
\begin{align}\label{eq-gaa-com}
    \gaa_v \gaa_{v'} = \xi^{\overline \Pi_\lambda(v,v')} \gaa_{v'} \gaa_v
    \quad\in \dS.
\end{align}

For any $v_1,\ldots,v_r \in \mathcal V$ and $a_1,\ldots,a_r \in \mathbb{Z}$, define the Weyl-ordered product by  
\begin{align}\label{Weyl_ordering-gaa}
    \left[ \gaa_{v_1}^{a_1} \gaa_{v_2}^{a_2} \cdots \gaa_{v_r}^{a_r} \right] 
    := \xi^{-\frac{1}{2}\sum_{i<j} a_i a_j \overline \Pi_\lambda(v_i,v_j)} 
       \gaa_{v_1}^{a_1} \gaa_{v_2}^{a_2} \cdots \gaa_{v_r}^{a_r}.
\end{align}
For ${\bf t} = (t_v)_{v\in V_\lambda} \in \mathbb{Z}^{\mathcal V}$, set  
\begin{align}\label{def-gaa-monomial}
    \gaa^{\bf t} := \left[ \prod_{v\in V_\lambda} \gaa_v^{t_v}\right].
\end{align}

Finally, define the map  
\begin{align}\label{def-M-lambda}
   \overline M_\lambda \colon \mathbb Z^{\mathcal V} \longrightarrow \Fr(\dS), 
    \quad {\bf t} \longmapsto \gaa^{\bf t}.
\end{align}

A properly embedded oriented arc in $\fS$ is called an {\bf essential arc}  
if its endpoints lie on two distinct components of $\partial \fS$.

\begin{theorem}[\cite{huang2025quantum}]\label{intro-thm-skein-inclusion-A}
Let $\fS$ be a triangulable pb surface without interior punctures, and let $\lambda$ be a triangulation of~$\fS$.

\begin{enumerate}[label={\rm (\alph*)}]\itemsep0.3em

\item The triple $\overline {\mathsf s}_\lambda=(\overline Q_\lambda, \overline \Pi_\lambda, \overline M_\lambda)$ is a quantum seed (Definition~\ref{def-quantum-seed}) inside the skew-field $\Fr(\dS)$.

\item Let $\overline{\mathsf{S}}_{\fS,\lambda}$ denote the quantum mutation class (Definition~\ref{def-quantum-class}) containing the quantum seed $(\overline Q_\lambda, \overline \Pi_\lambda, \overline M_\lambda)$.  
Then for any other triangulation $\lambda'$ of~$\fS$ we have
\[
\overline{\mathsf{S}}_{\fS,\lambda}
=
\overline{\mathsf{S}}_{\fS,\lambda'}.
\]

\item Set 
\begin{align}\label{eq-def-OA-cluster}
    \OvS := \overline{\mathsf{S}}_{\fS,\lambda},\qquad
\OvA := \mathscr{A}_{\OvS},\qquad
\OvU := \mathscr{U}_{\OvS}.
\end{align}
Then
\[
\dS \subset \OvU .
\]
If each connected component of $\fS$ contains at least two punctures, then also
\[
 \dS \subset \OvA.
\]

  \item Let $1 \le j \le i \le n$, and let $C_{ij}$ be the stated corner arc represented by the red arc in Figure~\ref{Fig;tau-v1}(A), oriented counterclockwise around $v_1$.  
Let $\lambda$ be a triangulation of $\fS$ that contains the ideal arcs $e_1,e_2,e_3$ shown in Figure~\ref{Fig;tau-v1}(A).  
For any $j,k$ with $1 \le j < k$, define  
\begin{align}\label{intro-eq-mukj}
\mu_{(k;j)}=\mu_{k j}\cdots\mu_{k1},
\end{align}
where the small vertices $i j \in \overline V_\lambda$ are labeled as in Figure~\ref{Fig;tau-v1}(B).
With the convention that $\overline A_{n0}=\overline A_{00}=1$,
 we have 
     \begin{align}\label{intro-eq-Cij}
        C_{ij}=\begin{cases}
[\overline A_{i1}\cdot\overline A_{i0}^{-1}\cdot \overline A_{11}^{-1}] & \mbox{ if $j=1$},\\
[\overline A_{i0}^{-1}\cdot \overline A_{i-1,0}] & \mbox{ if $j=i$},\\
[\mu_{(j;j-1)}\cdots \mu_{(i-2;j-1)} \mu_{(i-1;j-1)} (\overline A_{j,j-1})\cdot  \overline A_{i0}^{-1}\cdot \overline A_{jj}^{-1}] & \mbox{ if $1<j<i$}.
\end{cases}
     \end{align}
   \item Let $1 \le j \le i \le n$, and let $\overline C_{ij}$ be the stated corner arc represented by the red arc in Figure~\ref{Fig;tau-v1}(A), oriented clockwise around $v_1$.  
Let $\lambda$ be a triangulation of $\fS$ that contains the ideal arcs $e_1,e_2,e_3$ shown in Figure~\ref{Fig;tau-v1}(A).  
For any $k,t$ such that $1\leq t$ and $k+t<n$, we denote 
\begin{align}\label{intro-eq-bar-mukj}
\overline\mu_{(k;t)}=\mu_{\overline{k,k+1}} \cdots \mu_{\overline{k,k+t}},
\end{align}
where the small vertices $\overline{ij} \in\overline V_\lambda$ are labeled as in Figure~\ref{Fig;tau-v1}(C).
With the convention that $\overline A_{\overline {nn}}=\overline A_{\overline {00}}=\overline A_{\overline {0n}}=1$,
 we have
     \begin{align}\label{intro-eq-bar-Cij}
         \overline C_{ij}=\begin{cases}
[\overline A_{\overline{jj}}^{-1}\cdot \overline A_{\overline{j-1,j}}] & \mbox{ if $i=n$},\\
 [\overline A_{\overline{in}}^{-1}\cdot \overline A_{\overline{i-1,n}}] & \mbox{ if $i=j$},\\
[\overline \mu_{(j;n-i)}\cdots \overline \mu_{(i-2;n-i)} \overline \mu_{(i-1;n-i)} (\overline A_{\overline{j,j+1}})\cdot \overline A_{\overline{jj}}^{-1}\cdot \overline A_{\overline{in}}^{-1}] & \mbox{ if $1\leq j<i<n$}.
\end{cases}
     \end{align}

\item For any $1\leq i,j\leq n$, let $D_{ij}$ be the stated essential arc represented by the red arc in Figure~\ref{Fig;essential-P4}(A).
Let $\lambda$ be a triangulation of $\fS$ that contains the ideal arcs $c_1,c_2,c_3,c_4,c_5$ shown in Figure~\ref{Fig;essential-P4}(A) (we allow $c_1=c_3$).
We label the small vertices in $\overline V_\lambda$
contained in the quadrilateral bounded by $c_1\cup c_2\cup c_3\cup c_4$ as Figure~\ref{Fig;essential-P4}(B).

For any $s,t$ with $t<s<n$, denote
\begin{align}\label{def-must}
    \mu^r_{(s;t)}=\mu_{s,s-1}\mu_{s,s-2}\cdots \mu_{s,s-t}.
\end{align}

For any $s,t$ with $s+t<n$, denote 
$$\mu^l_{(s;t)}=\mu_{s,n-t}\cdots \mu_{s,n-2}\mu_{s,n-1}.$$

For any $1\leq i,j\leq n-1$ denote
$$\mu^{\diamondsuit}_j:= \mu^r_{(n-1;j-1)}\mu^r_{(n-2;j-1)}\cdots \mu^r_{(j;j-1)},$$
and 
$$\mu_i^\triangle:=\mu^l_{(n-2;1)}\mu^l_{(n-3;2)}\cdots \mu^l_{(i;n-i-1)},$$
With the convention that $\mu^{\diamondsuit}_j=id$ in case $j=1$, $\mu_i^\triangle=id$ in case $i=n-1$.

For any $i,j$ with $j\geq i\geq 1$, denote $$\mu^{\diamondsuit}_{(i;j)}=\mu_j^\triangle\circ \left(\mu^l_{(j-1;n-j-1)}\mu^l_{(j-2;n-j-1)}\cdots \mu^l_{(i;n-j-1)}\right),$$
where $\mu_{(k;j)}$ is defined as in
\eqref{intro-eq-mukj}, with the convention that in case $i=j$, 
\[\left(\mu^l_{(j-1;n-j-1)}\mu^l_{(j-2;n-j-1)}\cdots \mu^l_{(i;n-j-1)}\right)=id.\]
Then we have
   \begin{align}\label{into-eq-Dij}
       D_{ij}=
    \begin{cases}
        [\overline A_{in}^{-1}\cdot \overline A_{i-1,n-1}]
        & \mbox{ if $j=n$,}\\
        [ \mu^{\diamondsuit}_{j}(\overline A_{n-1,n-2}) \cdot \overline A_{j0}^{-1}] & \mbox{ if $i=n$,}\vspace{1.5mm}\\
        [\left(\mu_{n-1,n-1}\cdots \mu_{i+1,i+1}\mu_{ii}\right)\circ \mu^{\diamondsuit}_j\circ  \mu^{\triangle}_{i}(\overline A_{n-1,n-1}) \cdot \overline A_{in}^{-1}\cdot\overline A_{j0}^{-1}] & \mbox{ if $n>i\geq j$,}\vspace{1.5mm}\\
        [\left(\mu_{n-1,n-1}\cdots\mu_{i+1,i+1}\mu_{ii}\right)\circ  \mu^{\diamondsuit}_j\circ \mu^{\diamondsuit}_{(i;j)}(\overline A_{n-1,n-1}) \cdot \overline A_{in}^{-1}\cdot\overline A_{j0}^{-1}] & \mbox{ if $n>j>i$}.
    \end{cases}
   \end{align}

\end{enumerate}
\end{theorem}

\begin{figure}[h]
    \centering
    \includegraphics[width=350pt]{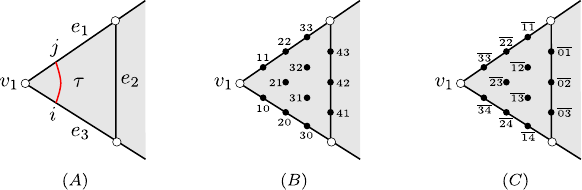}
    \caption{(A) The triangle $\tau$ with edges labeled $e_1,e_2,e_3$, and a distinguished vertex labeled $v_1$ (note that $e_1 \neq e_3$).  
(B) The labeling of the small vertices in $\overline V_\lambda \cap \tau$ for $n=4$.  
(C) An alternative labeling of the small vertices in $\overline V_\lambda \cap \tau$ for $n=4$.
}\label{Fig;tau-v1}
\end{figure}

\begin{figure}[h]
    \centering
    \includegraphics[width=220pt]{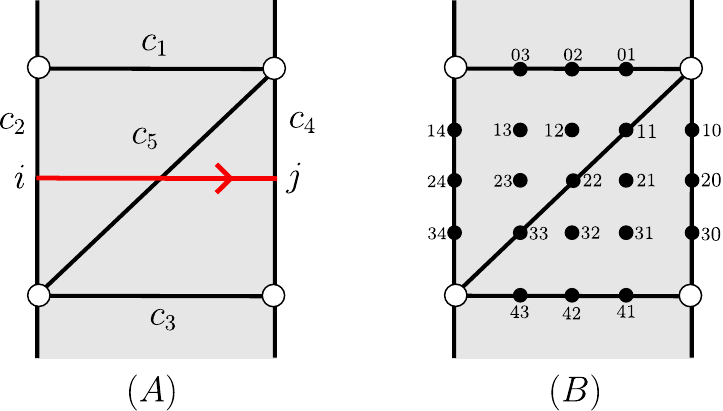}
    \caption{(A) The picture for an essential arc, in red. (B) The labeling for small vertices contained in the quadrilateral bounded by $c_1\cup c_2\cup c_3\cup c_4$ for $n=4$.}\label{Fig;essential-P4}
\end{figure}

\begin{remark}
Our $\OvS$, $\OvA$, and $\OvU$ correspond respectively to
$\mathsf{S}(\fS)$, $\mathscr{A}_{\omega}(\fS)$, and $\mathscr{U}_{\omega}(\fS)$ in \cite{huang2025quantum}.
\end{remark}

\begin{remark}
    In Figure~\ref{Fig;essential-P4}, we choose a different diagonal from the one used in \cite[Figure~2]{huang2025quantum}. 
    We omit the proof of \eqref{into-eq-Dij}, as the argument for \cite[Equation~(12)]{huang2025quantum} applies verbatim in our setting.
\end{remark}

We call the corner arc in the left (resp. right) picture in Figure~\ref{Fig;badarc} \emph{counterclockwise oriented} (resp. \emph{clockwise oriented}).

We state the following lemma, which will be used in \S\ref{subsec-inclusion-A}.

\begin{lemma}\label{lem-essential}
Let $\fS$ be a pb surface satisfying each its component contains at least two boundary components.
    The algebra $\SS$ is generated by finitely many stated essential arcs $C$, where each stated corner arc in $C$ is oriented counterclockwise.
\end{lemma}

\begin{proof}
    By \cite[Lemma~7.18]{huang2025quantum}, the algebra $\SS$ is generated by a finite set $C'$ of stated essential arcs.

    If $C'$ contains no clockwise oriented corner arcs, then we simply set $C = C'$. Otherwise, suppose that 
    $\alpha \in C'$ is a clockwise oriented corner arc, and let $c$ be its underlying oriented arc.
    For each $1 \le i,j \le n$, denote by $c_{ij}$ the stated oriented arc obtained from $c$ by assigning states $i$ and $j$ to its initial and terminal points, respectively.
    Let $\cev{c}_{ij}$ be the stated oriented arc obtained from $c_{ij}$ by reversing its orientation.

    Let $U$ be a properly embedded open tubular neighborhood of $\alpha$, which is diffeomorphic to $\mathbb P_2$.
    By \cite[Corollary~6.3]{LY23}, the embedding 
    $U \hookrightarrow \fS$ induces an algebra embedding
    \[
        \cS_\omega(\mathbb P_2) \longrightarrow \cS_\omega(\fS),
        \qquad b_{ij} \mapsto c_{ij},
    \]
    where $b_{ij} \in \cS_\omega(\mathbb P_2)$ is defined in \S\ref{sec:bigon}. 
    Then \cite[Lemma~6.5]{LS21} implies that 
    $\alpha$ lies in the subalgebra of $\SS$ generated by the $\cev{c}_{ij}$, $1 \le i,j \le n$, all of which are counterclockwise oriented stated corner arcs.

    Replacing each clockwise oriented corner arc in $C'$ by the corresponding counterclockwise generators above, we eventually obtain the desired generating set $C$.
\end{proof}

\section{Quantum cluster algebras associated with stated skein algebras}\label{sec-cluster-skein}

In this section, we formulate a quantum cluster  structure inside the skew-field of the stated $\SL$-skein algebra (see Proposition~\ref{prop-naturality} and Definition~\ref{def-sln-quantum-cluster-al}). 
We then show that each stated essential arc is a cluster variable (\eqref{pro-inclusion-3}, \eqref{pro-inclusion-4}, \eqref{eq-alpha-counterclock}), which implies that the stated $\SL$-skein algebra is contained in the associated quantum cluster algebra (Theorem~\ref{thm-skein-inclusion-A}), namely
\[
\SS \subset \mathscr{A}_{\omega}(\fS)
\]
(see Definition~\ref{def-sln-quantum-cluster-al}). 
Finally, we prove that $\SS=\mathscr{A}_{\omega}(\fS)=
\mathscr{U}_{\omega}(\fS)$ when $n=2$ (Theorem~\ref{thm-skein-eq-A-two}).

\subsection{Quantum cluster structure inside $\text{Frac}(\SS)$}\label{sec:quantuminside}
Let $\fS$ be a generalized triangulable pb surface without interior punctures, and let $\lambda$ be a generalized triangulation of $\fS$.

Let 
\[
\Ql \colon \VA \times \VA \to \bZ
\]
denote the restriction of 
\[
\Qast \colon \OVA \times \OVA \to \bZ
\]
(see \eqref{eq-def-Q-lambda-re}). 
Recall that $\RVlast \subset \VA$ denotes the set of vertices of $\OVA$ lying in the interior of $\fS^\ast$.
We then have the following.

\begin{lemma}\label{lem-compatible-pair}
\begin{enumerate}[label={\rm (\alph*)}]\itemsep0.3em
    \item  For $x\in \OVA$ and $y\in \RVlast$, we have
\[
\sum_{z\in \VA} C(z,x)\, Q_\lambda(z,y) \;=\; \Qast(x,y),
\]
where $C$ is defined in \eqref{def-matrix-C}.

\item For any $u\in \RVlast$ and $v\in \VA$, we have
\[
\sum_{v_1 \in \VA} Q_\lambda(v_1,u)\, P_\lambda(v_1,v)
    \;=\; 2n^2\, \delta_{u,v},
\]
where $P_\lambda$ is defined in \eqref{def-matrix-Pl}.
\end{enumerate}
\end{lemma}

\begin{proof}
(a) If $x\in \VA$, then $C(z,x)=\delta_{z,x}$, and the identity follows immediately.

Now suppose $x\in \OVA \setminus \VA$, and let $\tau$ be the attached triangle containing $x$.
Define
\[
V_x := \{\, x' \in \overline V_\tau \cap \VA \subset \OVA \mid p(x') = x \,\},
\]
where $p$ is the map defined in \eqref{eq-cov-pdef}.
Since $y\in \RVlast$, we have
\[
\Qast(x,y) + \sum_{z\in V_x} \Qast(z,y) = 0.
\]
By the definition of $C$, this yields
\[
\sum_{z\in \VA} C(z,x)\, Q_\lambda(z,y)
    = -\sum_{z\in V_x} \Qast(z,y)
    = \Qast(x,y).
\]

(b) We compute:
\begin{align*}
\sum_{v_1\in \VA} Q_\lambda(v_1,u)\, P_\lambda(v_1,v)
&= -\sum_{\substack{v_1\in \VA \\ v_2,v_3\in \OVA}}
     C(v,v_3)\, \Past(v_3,v_2)\, C^T(v_2,v_1)\, Q_\lambda(v_1,u) \\
&= -\sum_{v_2,v_3\in \OVA}
     C(v,v_3)\, \Past(v_3,v_2)\, \Qast(v_2,u)
     \qquad\text{by (a)} \\
&= \sum_{v_3\in \OVA}
     C(v,v_3)\, 2n^2\, \delta_{v_3,u}
     \qquad\text{by \cite[Equation~(214)]{LY23}} \\
&= 2n^2\, \delta_{v,u}
     \qquad\text{by \eqref{def-matrix-C}}.
\end{align*}
\end{proof}

Theorem~\ref{thm-stated-trace}(d) guarantees the existence of the skew-field $\Fr(\SS)$ of $\SS$.  
In what follows, we construct a quantum seed inside $\Fr(\SS)$.

Set  
\[
\mathcal V = \VA, \qquad 
\mathcal V_{\mathrm{mut}} = \mathring{\overline V}_{\lambda^\ast}.
\]
Define
\begin{align}\label{eq-def-pi}
 \Pi_\lambda := \tfrac{1}{n}\, P_\lambda .
\end{align}
By \eqref{def-PP3} and \eqref{eq-anti-matric-P-def}, the matrix $\Pi_\lambda$ is integral.

For each $v\in \mathcal V$, we have defined an element $\ga_v \in \SS$ (see \S\ref{sec;A_tori} and \S\ref{sub-extended}).  
Recall that $\xi = \omega^n$.  
Then Theorem~\ref{thm-stated-trace}(a) implies that
\begin{align}\label{eq-gaa-stated}
    \ga_v\, \ga_{v'} 
    = \xi^{\Pi_\lambda(v,v')}\, \ga_{v'}\, \ga_v
    \qquad\text{in } \SS .
\end{align}

For any $v_1,\ldots,v_r \in \mathcal V$ and $a_1,\ldots,a_r \in \mathbb{Z}$, define the Weyl-ordered product by
\begin{align}\label{Weyl_ordering-gaa-stated}
    \left[ \ga_{v_1}^{a_1} \ga_{v_2}^{a_2} \cdots \ga_{v_r}^{a_r} \right] 
    := 
    \xi^{-\frac{1}{2} \sum_{i<j} a_i a_j\, \Pi_\lambda(v_i,v_j)}\,
       \ga_{v_1}^{a_1} \ga_{v_2}^{a_2} \cdots \ga_{v_r}^{a_r}.
\end{align}
For ${\bf t} = (t_v)_{v \in V_\lambda} \in \mathbb{Z}^{\mathcal V}$, set
\begin{align}\label{def-gaa-monomial-stated}
    \ga^{\bf t} := 
    \left[\, \prod_{v\in V_\lambda} \ga_v^{t_v} \right].
\end{align}

Finally, define
\begin{align}\label{def-M-lambda-stated}
    M_\lambda \colon \mathbb Z^{\mathcal V} \longrightarrow \Fr(\SS),
    \qquad 
    {\bf t} \longmapsto \ga^{\bf t}.
\end{align}

\begin{lemma}\label{lem-seed-skein}
Let $\fS$ be a generalized triangulable pb surface without interior punctures, and let $\lambda$ be a generalized triangulation of $\fS$.  
Then $(Q_\lambda,\Pi_\lambda,M_\lambda)$ is a quantum seed (Definition~\ref{def-quantum-seed}) inside the skew-field $\Fr(\SS)$.
\end{lemma}

\begin{proof}
Using Theorem~\ref{thm-stated-trace} together with Lemma~\ref{lem-compatible-pair}, the argument of \cite[Lemma~4.1]{huang2025quantum} applies in this setting.
\end{proof}

Let $\mathsf{S}_{\fS,\lambda}$ denote the quantum mutation class (Definition~\ref{def-quantum-class}) containing the quantum seed ${\mathsf s}_\lambda=(Q_\lambda,\Pi_\lambda,M_\lambda)$. 
In the remainder of this subsection, we prove that 
\[
\mathsf{S}_{\fS,\lambda} = \mathsf{S}_{\fS,\lambda'}
\]
for any two generalized triangulations $\lambda$ and $\lambda'$ of $\fS$.

Let $e$ be a boundary edge of $\fS$. By convention, $e=e_1$ in the attached triangle.  
There is an embedding $\iota\colon \fS \to \fS^\ast$ such that $\iota(e)=e_2$ (see \cite[Figure~23]{LY23}).  
This embedding induces an algebra homomorphism
\begin{align}\label{def-iota-al}
    \iota_\ast \colon \SS \longrightarrow \SSa.
\end{align}
It is straightforward to check that
\begin{align}\label{eq-gv-surfaces}
    \iota_\ast(\ga_v) = \ga_v \in \SSa \qquad\text{for all } v\in\V.
\end{align}
We continue to use $\iota_\ast$ to denote the composition
\begin{align}\label{eq-def-iota-dS}
    \SS \xrightarrow{\iota_\ast} \SSa \twoheadrightarrow \dSS.
\end{align}
Then we have the following, which will be used later. 

\begin{lemma}\cite[Equation~(240)]{LY23}\label{lem-iota-gv}
    For each $v\in V_\lambda'$, we have 
    $$\iota_\ast (\ga_v)=\begin{cases}
    [\gaa_v \gaa_{p(v)}^{-1}] & \mbox{if $v\in \overline V_{\lambda^\ast}\setminus \overline V_\lambda$}\\
    \gaa_v & \mbox{otherwise},
   \end{cases}$$ 
    where $p$ is defined in \eqref{eq-cov-pdef}.
\end{lemma}

Equation~\eqref{def-matrix-Pl} yields an algebra embedding
\begin{align}\label{eq-def-L-lamda}
    L_\lambda\colon \As \longrightarrow \Aa,
   \qquad
   A^{\bf k} \longmapsto \overline A^{{\bf k}C}.
\end{align}

By \cite[Lemma~11.6, Equations~(226), (232), and (235)]{LY23}, we obtain the following.

\begin{lemma}\label{lem-com-reduced-stated}
The following diagram commutes:
\begin{equation*}
\begin{tikzcd}
\SS \arrow[r, "\iota_\ast"]
  \arrow[d, hookrightarrow, "\tr"']
& \dSS
  \arrow[d, hookrightarrow, "\overline{\mathrm{tr}}_{\lambda^\ast}"] \\
\As
  \arrow[r, hookrightarrow, "L_\lambda"]
& \Aa
\end{tikzcd}.
\end{equation*}
\end{lemma}

The lemma above immediately implies the following corollary.

\begin{corollary}\label{cor-injectivity}
The following two composition maps
\begin{align}\label{eq-embedding-stated-reduced}
    \SS \xrightarrow{\iota_\ast} \SSa \twoheadrightarrow \dSS,
\qquad
\SS \xrightarrow{\iota_\ast} \SSa \twoheadrightarrow \overline{\cS}_\omega(\fS^\ast)
\end{align}
are both injective.
\end{corollary}

Let $\fS$ be a triangulable pb surface equipped with a triangulation $\lambda$.  
Suppose $e \in \lambda$ is not a boundary edge. Then there exists a unique ideal arc $e'$ with  
$e' \neq e$ such that 
\[
\lambda' := (\lambda \setminus \{e\}) \cup \{e'\}
\]
is again a triangulation of $\fS$.  
We say that $\lambda$ and $\lambda'$ are obtained from each other by a {\bf flip}.  
In \cite{FG06,GS19}, the authors showed that the quivers $\Gamma_{\lambda'}$ and $\Gamma_\lambda$ are related by a specific sequence of mutations.

Let $\mathbb{P}_{4,e}$ denote the quadrilateral having $e$ as one of its diagonals.  
Note that $\mathbb{P}_{4,e}$ may not be embedded in $\fS$, although its interior is embedded.  
Let $V_{\lambda,e}$ be the set of all vertices of $\overline V_\lambda$ lying in the interior of $\mathbb{P}_{4,e}$; thus $V_{\lambda,e}$ consists of $(n-1)^2$ vertices.  
The vertices in $V_{\lambda,e}$ are labeled as in Figure~\ref{fig:P4-labeling} (the vertices of $\overline V_\lambda$ on the boundary of $\mathbb{P}_{4,e}$ are also labeled, and these labels will be used later).

For each $i = 0,1,\dots,n-2$ and each $0 \le t \le i$, define
\begin{align}
\label{def-Vite}
    V_{\lambda,e}^{i,t}
    &:= \{(i+1-t,\, t+1),\ (i+2-t,\, t+2),\ \ldots,\ (n-1-t,\, n-1-i+t)\}
    \subset V_{\lambda,e}, \\
    \label{def-Vie}
    V_{\lambda,e}^{i}
    &:= \bigsqcup_{0 \le s \le i} V_{\lambda,e}^{i,s} \subset V_{\lambda,e}.
\end{align}
Each $V_{\lambda,e}^{i}$ forms a grid of vertices inside a quadrilateral, containing $(i+1)(n-i-1)$ points.  
Note that these sets for different $i$ are not necessarily disjoint.

\begin{figure}
    \centering
    \includegraphics[width=0.28\linewidth]{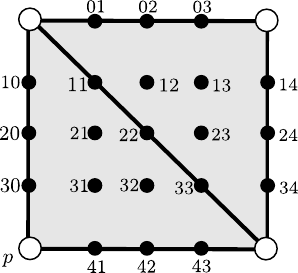}
    \caption{Labeling of the small vertices contained in $\mathbb{P}_{4,e}$ when $n=4$, where $e$ is the diagonal.}
    \label{fig:P4-labeling}
\end{figure}

The mutation sequence of Fock, Goncharov, and Shen proceeds as follows: first mutate at all vertices of $V_{\lambda,e}^{0}$ in any order; then at all vertices of $V_{\lambda,e}^{1}$ in any order; and so on, finishing with mutations at all vertices of $V_{\lambda,e}^{n-2}$ in any order.  
The total length of this mutation sequence is
\[
\sum_{i=0}^{n-2} (i+1)(n-1-i)
    = \sum_{j=1}^{n-1} j(n-j)
    = \frac{1}{6}(n^3 - n)
    =: r.
\]
Let us denote the resulting sequence of vertices by
\[
v_1, v_2, \ldots, v_r.
\]
Thus the first $n-1$ vertices in the sequence are precisely the elements of $V_{\lambda,e}^{0}$, and we have
\[
\Gamma_{\lambda'} = \mu_{v_r} \cdots \mu_{v_2} \mu_{v_1}(\Gamma_\lambda),
\qquad
\overline{Q}_{\lambda'} = \mu_{v_r} \cdots \mu_{v_2} \mu_{v_1}(\overline{Q}_\lambda).
\]

Define
\begin{align}\label{def-eq-mul}
    \mu_{\lambda'\lambda}
   := \mu_{v_r} \cdots \mu_{v_2} \mu_{v_1}.
\end{align}
The proof of \cite[Theorem~4.13]{huang2025quantum} implies the following result, which in turn shows that  
$\mu_{\lambda'\lambda}$ is a well-defined operation on the quantum seed
$(\overline Q_{\lambda}, \overline \Pi_{\lambda}, \overline M_{\lambda})$.  
In particular, $\mu_{\lambda'\lambda}$ is a composition of mutations that does not depend on the choice of ordering of the mutations at vertices in $V_{\lambda,e}^{i}$ for each $i$.

\begin{theorem}\cite[Theorem~4.13]{huang2025quantum}\label{lem-mutation-A-seeds-flip}
Suppose that $\fS$ has no interior punctures. Then
\[
    (\overline Q_{\lambda'}, \overline \Pi_{\lambda'}, \overline M_{\lambda'})
    = \mu_{v_r} \cdots \mu_{v_1}
      (\overline Q_{\lambda}, \overline \Pi_{\lambda}, \overline M_{\lambda}).
\]
\end{theorem}

Suppose that $\lambda$ and $\lambda'$ are any two triangulations of $\fS$.  
A {\bf triangulation sweep} connecting $\lambda$ and $\lambda'$ is a sequence of triangulations
\[
\Lambda = (\lambda_1, \ldots, \lambda_m)
\]
such that $\lambda_1 = \lambda$, $\lambda_m = \lambda'$, and each $\lambda_{i+1}$ is obtained from $\lambda_i$ by a flip for every $1 \le i \le m-1$.  
It is well known that for any pair of triangulations $\lambda$ and $\lambda'$, there exists a triangulation sweep connecting them \cite{Lab09}.  

For any such sweep $\Lambda$, define
\begin{equation}\label{eq-Theta-change2}
    \mu_{\lambda'\lambda}
    := \mu_{\lambda_m\lambda_{m-1}}
       \cdots
       \mu_{\lambda_2\lambda_{1}}.
\end{equation}

Applying Theorem~\ref{lem-mutation-A-seeds-flip} consecutively along the sweep, we obtain
\[
(\overline Q_{\lambda'}, \overline \Pi_{\lambda'}, \overline M_{\lambda'})
    = \mu_{\lambda'\lambda}
      (\overline Q_{\lambda}, \overline \Pi_{\lambda}, \overline M_{\lambda}).
\]
This shows that $\mu_{\lambda'\lambda}$ is a well-defined operation on the quantum seed
$(\overline Q_{\lambda}, \overline \Pi_{\lambda}, \overline M_{\lambda})$, independent of the choice of triangulation sweep $\Lambda$.

Recall that  $\mathsf{S}_{\fS,\lambda}$ is the quantum mutation class of  the quantum seed ${\mathsf s}_\lambda=(Q_\lambda,\Pi_\lambda,M_\lambda)$ in $\Fr(\SS)$. 
The next proposition establishes the naturality of this quantum mutation class. 

\begin{proposition}\label{prop-naturality}
Let $\fS$ be a generalized triangulable pb surface without interior punctures, and let $\lambda$ and $\lambda'$ be any two generalized triangulations of $\fS$.  
Then
\[
\mathsf{S}_{\fS,\lambda} = \mathsf{S}_{\fS,\lambda'}.
\]
\end{proposition}

\begin{proof}
We may assume that $\fS$ is connected, $\fS$ is not a bigon, and $\lambda'$ is obtained from $\lambda$ by flipping an edge $e\in\lambda$.  
Thus $\lambda' = (\lambda\setminus\{e\}) \cup \{e'\}$, where $e'\neq e$ is the unique ideal arc making $\lambda'$ a triangulation of $\fS$.  
Consequently,
\[
(\lambda')^\ast = (\lambda^\ast \setminus \{e\}) \cup \{e'\}.
\]

There exists a quadrilateral $\mathbb P_4$ in $\fS$ (possibly not embedded globally, but with embedded interior $\mathring{\mathbb P}_4$) containing $e$ and $e'$ as its two diagonals.  
Theorem~\ref{lem-mutation-A-seeds-flip} implies that there are vertices 
\[
v_1,\dots,v_r \in \mathring{\mathbb P}_4 \cap \OVA
\]
such that
\begin{align}\label{eq-mutation-flip}
    (\overline Q_{\lll}, \overline \Pi_{\lll}, \overline M_{\lll})
    = \mu_{v_r} \cdots \mu_{v_1}
      (\overline Q_{\last}, \overline \Pi_{\last}, \overline M_{\last}).
\end{align}

Observe that the mutation sequence $\mu_{v_1}\cdots\mu_{v_r}$ involves only vertices in $\V$.  
Thus, by \eqref{eq-gv-surfaces}, the following diagram commutes:
\begin{equation}\label{eq-com-coord}
\begin{tikzcd}[row sep=3em, column sep=5em]
 \Fr(\SS)
   \arrow[r, "\mu_{v_r}\circ\cdots \circ \mu_{v_1}"]
   \arrow[d, "\iota_\ast"]
 & \Fr(\SS)
   \arrow[d, "\iota_\ast"]  \\
 \Fr(\dSS)
   \arrow[r, "\mu_{v_r}\circ\cdots \circ \mu_{v_1}"]
 & \Fr(\dSS)
\end{tikzcd}.
\end{equation}

For each $v\in\OVA$, let ${\bf e}_v\in\mathbb Z^{\OVA}$ denote the standard basis vector, and when $v\in\VA$ we also regard ${\bf e}_v\in\mathbb Z^{\VA}$.

 {\bf Case 1}: \( v\in\V \)

\begin{align*}
 \iota_\ast\bigl(\mu_{v_r}\cdots\mu_{v_1}(M_\lambda({\bf e}_v))\bigr)
 &= \mu_{v_r}\cdots\mu_{v_1}\bigl(\iota_\ast(M_\lambda({\bf e}_v))\bigr)
   &&\text{by \eqref{eq-com-coord}} \\
 &= \mu_{v_r}\cdots\mu_{v_1}\bigl(\overline M_{\last}({\bf e}_v)\bigr)
   &&\text{by \eqref{eq-gv-surfaces}} \\
 &= \overline M_{\lll}({\bf e}_v)
   &&\text{by \eqref{eq-mutation-flip}} \\
 &= \iota_\ast(M_{\lambda'}({\bf e}_v))
   &&\text{by \eqref{eq-gv-surfaces}}.
\end{align*}
By Corollary~\ref{cor-injectivity},  
\[
\mu_{v_r}\cdots\mu_{v_1}(M_\lambda({\bf e}_v))
   = M_{\lambda'}({\bf e}_v).
\]

{\bf Case 2}: \( v\in \VA\setminus\V \)

\begin{align*}
 \iota_\ast\bigl(\mu_{v_r}\cdots\mu_{v_1}(M_\lambda({\bf e}_v))\bigr)
 &= \mu_{v_r}\cdots\mu_{v_1}\bigl(\iota_\ast(M_\lambda({\bf e}_v))\bigr)
   &&\text{by \eqref{eq-com-coord}} \\
 &= \mu_{v_r}\cdots\mu_{v_1}
    \Bigl(\bigl[\overline M_{\last}({\bf e}_v)\,
            \overline M_{\last}({\bf e}_{p(v)})^{-1}\bigr]\Bigr)
   &&\text{by Lemma~\ref{lem-iota-gv}} \\
 &= \bigl[\overline M_{\lll}({\bf e}_v)\,
          \overline M_{\lll}({\bf e}_{p(v)})^{-1}\bigr]
   &&\text{by \eqref{eq-mutation-flip}} \\
 &= \iota_\ast(M_{\lambda'}({\bf e}_v))
   &&\text{by Lemma~\ref{lem-iota-gv}}.
\end{align*}
By Corollary~\ref{cor-injectivity},  
\[
\mu_{v_r}\cdots\mu_{v_1}(M_\lambda({\bf e}_v))
   = M_{\lambda'}({\bf e}_v).
\]

This implies 
\begin{align*}
    ( Q_{\lambda'},  \Pi_{\lambda'}, M_{\lambda'})
    = \mu_{v_r} \cdots \mu_{v_1}
      ( Q_{\lambda},  \Pi_{\lambda},  M_{\lambda}).
\end{align*}
\end{proof}

\begin{remark}
Lemma~\ref{lem-com-reduced-stated} together with \cite[Corollary~4.18(c)]{huang2025quantum} shows that the isomorphism $\Psi_{\lambda'\lambda}^A$ in \cite[Equation~(244)]{LY23} is realized by a sequence of mutations.
\end{remark}

\begin{definition}\label{def-sln-quantum-cluster-al}

 Let $\fS$ be a generalized triangulable pb surface without interior punctures.
    Define   $$\mathsf{S}(\fS):=\mathsf{S}_{\fS,\lambda},\;
    \mathscr{A}_{\omega}(\fS):=
    \mathscr{A}_{\mathsf{S}(\fS)}^+,\;
     \mathscr{U}_{\omega}(\fS):=
    \mathscr{U}_{\mathsf{S}(\fS)}^+,\;
    \mathscr{A}_{\omega}^{\rm fr}(\fS):=
    \mathscr{A}_{\mathsf{S}(\fS)},\;
    \text{ and }
    \mathscr{U}_{\omega}^{\rm fr}({\fS}):=
    \mathscr{U}_{\mathsf{S}(\fS)}$$
    where $\mathscr A^+$, $\mathscr A$, $\mathscr U^+$, $\mathscr U$ are in Definition~\ref{def-quan-cluster-algebra}, and $\lambda$ is a triangulation of $\fS$.
   
\end{definition}

Proposition~\ref{prop-naturality} implies that  
$\mathsf{S}(\fS)$, $\mathscr{A}_{\omega}(\fS)$, $\mathscr{U}_{\omega}(\fS)$, $\mathscr{A}_{\omega}^{\rm fr}(\fS)$, and $\mathscr{U}_{\omega}^{\rm fr}(\fS)$
are independent of the choice of $\lambda$.

\subsection{A quasi-isomorphism}\label{sub-quas-sl2}

With the notation in \S\ref{sub-extended}, let $\fS$ be a generalized triangulable pb surface with a generalized triangulation $\lambda$. Let $Q^{{\rm qc}}_{\lambda^\ast}$ the quiver obtained from $\overline Q_{\lambda^\ast}$ by deleting all arrows incident to vertices in $\overline V_{\lambda^\ast}\setminus V'_\lambda$.

For any vertex $v\in \overline V_{\lambda^{\ast}}$, denote 
\begin{equation}\label{4eq:qc0}
    A^{\rm qc}_v=\begin{cases}
    [\overline A_{v} \overline A_{p(v)}^{-1}]\in \overline {\mathscr A}_{\omega}(\fS^\ast) & \mbox{if $v\in \overline V_{\lambda^\ast}\setminus \overline V_\lambda$},\\
    \overline A_v\in \overline {\mathscr A}_{\omega}(\fS^\ast) & \mbox{otherwise},
   \end{cases} 
\end{equation}
   where $p$ is defined in \eqref{eq-cov-pdef}.
   
\begin{definition}\label{def-qusi-com}
Recall that $\xi=\omega^n\in R$. Let $\mathcal A$ be a domain over $R$. Two elements $A_1,A_2\in \mathcal A$ are said to be \textbf{quasi-commutative} if there exists an integer $k$ such that
\[
A_1 A_2 = \xi^k A_2 A_1.
\]

The integer $k$ is uniquely determined; we denote it by $\Pi(A_1,A_2)$.
\end{definition}

It is straightforward to verify that for any $v,v'\in \overline V_{\lambda^\ast}$, $A^{\rm qc}_v$ and $A^{\rm qc}_{v'}$ are quasi-commutative, denote by $\Pi^{\rm qc}_{v,v'}$ the integer $\Pi(A^{\rm qc}_v,A^{\rm qc}_{v'})$ and $\Pi^{\rm qc}_{\lambda^\ast}=(\Pi^{\rm qc}_{v,v'})_{v,v'\in \overline V_{\lambda^\ast}}$.

Define 
\begin{equation*}
M^{\rm qc}_{\lambda^\ast}:\mathbb Z^{\overline V_{\lambda^\ast}}\to {\rm Frac}(\overline {\mathscr A}_{\omega}(\fS^\ast)),\qquad
{\bf k}=(k_v)_{v\in \overline{V}_{\lambda^*}}\longmapsto 
\left[\prod_{v\in \overline V_{\lambda^\ast}} (A^{\rm qc}_{v})^{k_v}\right].
\end{equation*}

The following is immediate.

\begin{equation}\label{4eq:qc}
   \left[\prod_{v\in \overline V_{\lambda^\ast}} (A^{\rm qc}_{v})^{Q^{\rm qc}_{\lambda^\ast} (v,u)}\right]=\left[\prod_{v\in \overline V_{\lambda^\ast}} (\overline A_{v})^{Q_{\lambda^\ast}(v,u)}\right].  
\end{equation}

\begin{lemma}
    The triple $(Q^{{\rm qc}}_{\lambda^\ast},\Pi^{\rm qc}_{\lambda^\ast}, M^{\rm qc}_{\lambda^\ast})$ is a quantum seed in ${\rm Frac}(\overline {\mathscr A}_{\omega}(\fS^\ast))$. Moreover, for any two generalized triangulations $\lambda$ and $\lambda'$ of $\fS$, we have $(Q^{{\rm qc}}_{\lambda^\ast},\Pi^{\rm qc}_{\lambda^\ast}, M^{\rm qc}_{\lambda^\ast})$ and $(Q^{{\rm qc}}_{\lambda'^\ast},\Pi^{\rm qc}_{\lambda'^\ast}, M^{\rm qc}_{\lambda'^\ast})$ are mutation equivalent.
\end{lemma}

\begin{proof}
    By Lemma~\ref{lem-iota-gv} and \eqref{4eq:qc0}, we see that the triple $(Q^{{\rm qc}}_{\lambda^\ast}, \Pi^{\rm qc}_{\lambda^\ast}, M^{\rm qc}_{\lambda^\ast})$ is obtained from the quantum seed ${\mathsf s}_\lambda = (Q_\lambda, \Pi_\lambda, M_\lambda)$ by adjoining isolated frozen variables $\overline A_v$ for $v \in \overline V_{\lambda^\ast} \setminus V'_\lambda$. Consequently, it forms a quantum seed. 
    Furthermore, Proposition~\ref{prop-naturality} implies that for any two generalized triangulations $\lambda$ and $\lambda'$ of $\fS$, the corresponding seeds $(Q^{{\rm qc}}_{\lambda^\ast}, \Pi^{\rm qc}_{\lambda^\ast}, M^{\rm qc}_{\lambda^\ast})$ and $(Q^{{\rm qc}}_{\lambda'^\ast}, \Pi^{\rm qc}_{\lambda'^\ast}, M^{\rm qc}_{\lambda'^\ast})$ are mutation equivalent.
\end{proof}

Let ${\mathscr A}_{\omega}^{\rm qc}(\fS^\ast)$ denote the quantum cluster algebra associated with the seed
${\mathbf s}^{\rm qc}_{\lambda^\ast}:=(Q^{{\rm qc}}_{\lambda^\ast},\Pi^{\rm qc}_{\lambda^\ast}, M^{\rm qc}_{\lambda^\ast})$. Observe that $\mathscr A_{\omega}^{\rm qc}(\fS^\ast)$ is isomorphic to the cluster algebra obtained from $\mathscr A_{\omega}(\fS)$ 
by first localizing at all frozen variables and then adjoining a corresponding set of new isolated invertible frozen variables.

\begin{proposition}\label{Prop-quasi-iso}
    The assignments $A^{\rm qc}_{v}\mapsto A^{\rm qc}_{v}$ with $v\in \overline V_{\lambda^{\ast}}$ induce a quasi-isomorphism (Definition \ref{def:QQH}) 
    \begin{equation*}
        f:{\mathscr A}_{\omega}^{\rm qc}(\fS^\ast)=\mathscr A_{{\mathbf s}^{\rm qc}_{\lambda^\ast}}\to \overline {\mathscr A}_{\omega}(\fS^\ast)=\mathscr A_{\overline {\mathbf s}_{\lambda^\ast}}.
    \end{equation*}
\end{proposition}

\begin{proof}
    In view of \eqref{4eq:qc0}, conditions (1) and (2) of Proposition~\ref{prop:quasi-iso} are satisfied. Condition (3) follows directly from \eqref{4eq:qc}, while condition (4) is immediate. Consequently, all the hypotheses of Proposition~\ref{prop:quasi-iso} are met, and the result follows.
\end{proof}

\subsection{The inclusion of $\SS$ into the quantum cluster algebra $\mathscr A_\omega (\fS)$}\label{subsec-inclusion-A}
In this subsection, we first state Proposition~\ref{prop-flips-L}, which constitutes the key step in the proof of the first main theorem of this section (Theorem~\ref{thm-skein-inclusion-A}). 
We then introduce the necessary definitions required for Proposition~\ref{prop-flips-L}.

Let $\fS$ be a triangulable pb surface equipped with a triangulation $\lambda$, and let $p$ be a puncture of $\fS$ lying on $\partial \fS$.  
Suppose the triangles incident to $p$ are \( \tau_1, \dots, \tau_k \), allowing possible repetitions.

For each \( 1 \le i \le n-1 \), we construct a multi-subset  
\( V(\lambda,p,i) \subset \overline{V}_\lambda \) as follows:
\begin{enumerate}
    \item Every vertex in \( V(\lambda,p,i) \) lies in one of the triangles  
    \( \tau_1, \dots, \tau_k \).
    \item For each \( 1 \le j \le k \), choose barycentric coordinates on \( \tau_j \) such that  
    \( p = (n,0,0) \).  
    Then \( V(\lambda,p,i) \cap \tau_j \) consists of the small vertices whose first barycentric
    coordinate is \( n - i \).
\end{enumerate}

For each \( v \in \overline{V}_\lambda \), we define a multi-subset  
\( V(\lambda,p,v) \subset \{1,2,\dots,n-1\} \) by
\begin{align*}
    V(\lambda,p,v)
    := \{\, 1 \le i \le n-1 \mid v \in V(\lambda,p,i) \,\},
\end{align*}
where the multiplicity of \( i \) in \( V(\lambda,p,v) \) is the multiplicity of \( v \) in  
\( V(\lambda,p,i) \).

\begin{figure}[h]
    \centering
    \includegraphics[width=0.25\linewidth]{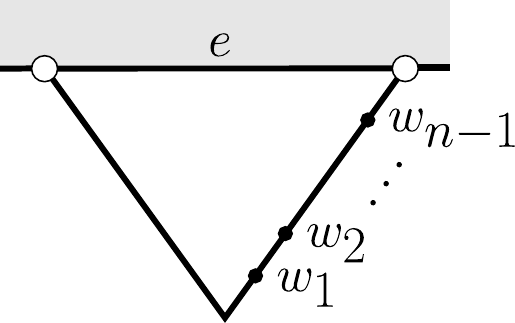}
    \caption{Labeling of small vertices contained in $e_2$ in an attached triangle.}
    \label{fig:labeling-wi}
\end{figure}

We use $\partial \lambda$ to denote the set of boundary edges of $\lambda$.  
Recall that for each $e \in \partial \lambda$, there is an attached triangle $\tau$ in $\fS^\ast$ (see Figure~\ref{Fig;attaching}(A)).  
Let $p_e$ denote the unique vertex of $\tau$ that does not lie in $\overline \fS$ (refer to \S\ref{intro-sec-reduced} for $\overline \fS$).  
The small vertices in $\overline V_\tau \cap e_2$ are labeled as in Figure~\ref{fig:labeling-wi} by 
$w_1, \ldots, w_{n-1}$, where $e_2$ is illustrated in Figure~\ref{Fig;attaching}(A).  
Define $\overline A_{e,i} :=\overline A_{w_i} \in \Aa$.
Then we have the following.

\begin{proposition}\label{prop-flips-L}
    Let $\fS$ be a triangulable pb surface without interior punctures. Let $\lambda$ be a triangulation of $\fS$ and $\lambda'$ a triangulation of $\fS^\ast$.  For $v\in \VA$, we have

    \[
        L_{\lambda}\!\left(\mu_{\lambda'\lambda^\ast}(A_v)\right)
        =
        \left[
        \mu_{\lambda'\lambda^\ast}(\overline A_v)
        \prod_{e \in \partial \lambda}\;
        \prod_{i \in V(\lambda',p_e,v)}
           \overline A_{e,i}^{-1}
        \right]
        \in \Aa,
    \]
    where $L_\lambda\colon \As \longrightarrow \Aa$ is defined in~\eqref{eq-def-L-lamda} and 
    $\mu_{\lambda'\lambda^\ast}$ is the mutation sequence defined in~\eqref{eq-Theta-change2}.
\end{proposition}

\def\bP{\mathbb{P}}
\def\Afour{\oa(\bP_4,\lambda_4)}
\def\Afourx{\oa(\bP_4,\lambda_4)\otimes_{R} R[x_1^{\pm 1},\cdots, x_{n-1}^{\pm 1}]}

Before proving Proposition~\ref{prop-flips-L}, we first state some useful lemmas.

Let $\lambda_4$ be the triangulation of $\mathbb{P}_4$ shown in Figure~\ref{fig:P4-labeling}, where the small vertices in 
$\overline V_{\lambda_4}$ are labeled as in that figure.
Let $\lambda_{4}' \neq \lambda_4$ be another triangulation of $\mathbb{P}_4$.
Let $p$ be the vertex of $\mathbb{P}_4$ adjacent to the small vertices $(n-1,0)$ and $(n,1)$ (see Figure~\ref{fig:P4-labeling}).

Define an algebra embedding
\[
    f\colon 
    \oa(\mathbb{P}_4,\lambda_4)
    \longrightarrow
    \oa(\mathbb{P}_4,\lambda_4)
    \otimes_R R[x_1^{\pm1},\ldots,x_{\,n-1}^{\pm1}],
    \qquad
    \overline A_v \longmapsto 
    \overline A_v \!\!\prod_{i \in V(\lambda_4,p,v)} \! x_i^{-1},
\]
for each $v \in \overline V_{\lambda_4}$.

\begin{lemma}\label{lem-P4-flip-iso}
    For each $v \in \overline V_{\lambda_4}$, we have
    \[
        f\!\left(\mu_{\lambda_4'\lambda_4}(\overline A_v)\right)
        =
        \mu_{\lambda_4'\lambda_4}(\overline A_v)
        \prod_{i \in V(\lambda_4',p,v)} x_i^{-1}\in \Afourx,
    \]
    where $\mu_{\lambda_4'\lambda_4}$ is defined in~\eqref{def-eq-mul}.
\end{lemma}

\begin{proof}
Let $e$ be the diagonal in Figure~\ref{fig:P4-labeling}. To simplify the notation, for each $0\leq i\leq n-2$ and $0\leq t\leq i$, we use $V_{i,t}$ and $V_i$ to denote $V_{\lambda_4,e}^{i,t}$ (see \eqref{def-Vite}) and $V_{\lambda_4,e}^{i}$ (see \eqref{def-Vie}) respectively. For each $0\leq i\leq n-2$, define
$$
\mu_i:=\prod_{v\in V_i}\mu_v,
$$
where the order of the product is arbitrary. We will use the induction on $i$ to prove the following statement: for each $0\leq i\leq n-2$, $0\leq t\leq i$, and $v\in V_{i,t}$, we have
\begin{align}\label{eq-f-mui}
f(\mu_i\cdots \mu_0 (\overline A_v))= \mu_i\cdots \mu_0 (\overline A_v) x_{n-i-1+t}^{-1}.
\end{align}

When $i=0$: Note that $\mu_0(\overline A_v) = \mu_v(\overline A_v)$. The full subquiver of $\Gamma_{\lambda_4}$, consisting of $v$ and vertices adjacent to $v$, is as shown in Figure~\ref{fig:quiver-v}, where
$$
f(\overline A_{v}) = \a_{v} f(\overline A_{v_1}) = \a_{v_1},\quad
f(\overline A_{v_2}) = \a_{v_2},\quad
f(\overline A_{v_3}) = \a_{v_3}x_{n-1}^{-1},\quad
f(\overline A_{v_4}) = \a_{v_4}x_{n-1}^{-1}.
$$
Then clearly, we have
\begin{align*}
f(\mu_0(\overline A_v)) = f(\mu_v(\overline A_v)) = \mu_v(\overline A_v) x_{n-1}^{-1} = \mu_0(\overline A_0) x_{n-1}^{-1}.
\end{align*}

\begin{figure}[ht]
    \centering
    \begin{tikzpicture}[>=stealth, node distance=2cm]
      \node (v)  at (0,0) {$v$};
      \node (v1) [right of=v] {$v_1$};
      \node (v3) [left of=v] {$v_3$};
      \node (v2) [below right of=v] {$v_2$};
      \node (v4) [above left of=v] {$v_4$};

      \draw[<-] (v2) -- (v);
      \draw[<-] (v4) -- (v);

      \draw[<-] (v) -- (v1);
      \draw[<-] (v) -- (v3);
    \end{tikzpicture}
    \caption{The full subquiver consisting of $v$ and vertices adjacent to $v$.}
    \label{fig:quiver-v}
\end{figure}
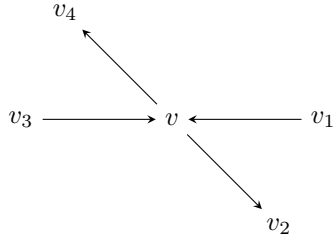

When $i=1$: Note that $\mu_1(\mu_0(\overline A_v)) = \mu_v(\mu_0(\overline A_v))$. The full subquiver of $\mu_0(\Gamma_{\lambda_4})$, consisting of $v$ and vertices adjacent to $v$, is as shown in Figure~\ref{fig:quiver-v} (see \cite[Figure~10.3]{FG06}).

When $t=0$, we have $v\in V(\lambda_4,p,n-1)$ (note that $V(\lambda_4,p,n-1)\cap V_0=\emptyset$), $v_1,v_2\in V_{0}$, and $v_3,v_4\in V(\lambda_4,p,n-2)$ (note that $V(\lambda_4,p,n-2)\cap V_0=\emptyset$). Then
\begin{align*}
f(\mu_0(\overline A_{v})) &= f(\a_v) = \a_v x_{n-1}^{-1} = \mu_0(\a_v) x_{n-1}^{-1}\\
f(\mu_0(\overline A_{v_1})) &= \mu_0(\overline A_{v_1}) x_{n-1}^{-1} \quad \text{(by the case when $i=0$)}\\
f(\mu_0(\overline A_{v_2})) &= \mu_0(\overline A_{v_2}) x_{n-1}^{-1} \quad \text{(by the case when $i=0$)}\\
f(\mu_0(\overline A_{v_3})) &= f(\a_{v_3}) = \a_{v_3} x_{n-2}^{-1} = \mu_0(\a_{v_3}) x_{n-2}^{-1}\\
f(\mu_0(\overline A_{v_4})) &= f(\a_{v_4}) = \a_{v_4} x_{n-2}^{-1} = \mu_0(\a_{v_4}) x_{n-2}^{-1}
\end{align*}
This implies that
\begin{align*}
f(\mu_1\mu_0(\overline A_v)) = f(\mu_v\mu_0(\overline A_v)) = \mu_v\mu_0(\overline A_v) x_{n-2}^{-1} = \mu_1\mu_0(\overline A_v) x_{n-2}^{-1}.
\end{align*}

When $t=1$, we have $v,v_1,v_2\notin \sqcup_{1\leq j\leq n-1} V(\lambda_4,p,j) \sqcup V_{0}$ and $v_3,v_4\in V_{0}$. Then
\begin{align*}
f(\mu_0(\overline A_{v})) &= f(\a_v) = \a_v = \mu_0(\a_v) \\
f(\mu_0(\overline A_{v_2})) &= f(\a_{v_1}) = \a_{v_1} = \mu_0(\a_{v_1}) \\
f(\mu_0(\overline A_{v_2})) &= f(\a_{v_2}) = \a_{v_2} = \mu_0(\a_{v_2}) \\
f(\mu_0(\overline A_{v_3})) &= \mu_0(\overline A_{v_3}) x_{n-1}^{-1} \quad \text{(by the case when $i=0$)}\\
f(\mu_0(\overline A_{v_4})) &= \mu_0(\overline A_{v_4}) x_{n-1}^{-1} \quad \text{(by the case when $i=0$)}.
\end{align*}
This implies that
\begin{align*}
f(\mu_1\mu_0(\overline A_v)) = f(\mu_v\mu_0(\overline A_v)) = \mu_v\mu_0(\overline A_v) x_{n-1}^{-1} = \mu_1\mu_0(\overline A_v) x_{n-1}^{-1}.
\end{align*}

Assume that \eqref{eq-f-mui} holds for $i\leq m-1$ ($m\geq 2$). We will prove that \eqref{eq-f-mui} also holds for $i=m$. Note that
\[
\mu_{m}\cdots\mu_0(\overline A_v) = \mu_v(\mu_{m-1}\cdots\mu_0(\overline A_v)).
\]
The full subquiver of $\mu_{m-1}\mu_0(\Gamma_{\lambda_4})$, consisting of $v$ and vertices adjacent to $v$, is as shown in Figure~\ref{fig:quiver-v} (see \cite[Figure~10.3]{FG06}).

When $t=0$, we have $v\in V(\lambda_4,p,n-m)$ (note that $V(\lambda_4,p,n-m)\cap (\cup_{1\leq l\leq m-1}V_l)=\emptyset$), $v_1,v_2\in V_{m-1,0}$, and $v_3,v_4\in V(\lambda_4,p,n-m-1)$ (note that $V(\lambda_4,p,n-m-1)\cap (\cup_{1\leq l\leq m-1}V_l)=\emptyset$). Then
\begin{align*}
f(\mu_{m-1}\cdots\mu_0(\overline A_v)) &= f(\a_v) = \a_v x_{n-m}^{-1} = \mu_{m-1}\cdots\mu_0(\overline A_v) x_{n-m}^{-1}\\
f(\mu_{m-1}\cdots\mu_0(\overline A_{v_1})) &= \mu_{m-1}\cdots\mu_0(\overline A_{v_1}) x_{n-m}^{-1} \quad \text{(by induction assumption)}\\
f(\mu_{m-1}\cdots\mu_0(\overline A_{v_2})) &= \mu_{m-1}\cdots\mu_0(\overline A_{v_2}) x_{n-m}^{-1} \quad \text{(by induction assumption)}\\
f(\mu_{m-1}\cdots\mu_0(\overline A_{v_3})) &= f(\a_{v_3}) = \a_{v_3} x_{n-m-1}^{-1} = \mu_{m-1}\cdots\mu_0(\overline A_v) x_{n-m-1}^{-1}\\
f(\mu_{m-1}\cdots\mu_0(\overline A_{v_4})) &= f(\a_{v_4}) = \a_{v_4} x_{n-m-1}^{-1} = \mu_{m-1}\cdots\mu_0(\overline A_{v_4}) x_{n-m-1}^{-1}.
\end{align*}
This implies that
\begin{align*}
f(\mu_{m}\cdots\mu_0(\overline A_v))
= f(\mu_v(\mu_{m-1}\cdots\mu_0(\overline A_v)))
= \mu_v(\mu_{m-1}\cdots\mu_0(\overline A_v)) x_{n-m-1}^{-1}
= \mu_{m}\cdots\mu_0(\overline A_v) x_{n-m-1}^{-1}.
\end{align*}

When $1\leq t\leq m-1$, we have $v\in V_{m-2,t-1}$, $v_1,v_2\in V_{m-1,t}$ and $v_3,v_4\in V_{m-1,t-1}$. Note that $V_{m-1}\cap V_{m-2}=\emptyset$. Then the induction assumption implies that
\begin{align*}
f(\mu_{m-1}\cdots\mu_0(\overline A_v))
= f(\mu_{m-2}\cdots\mu_0(\a_v))
&= \mu_{m-2}\cdots\mu_0(\a_v) x_{n-m+t}^{-1}
= \mu_{m-1}\cdots\mu_0(\overline A_v) x_{n-m+t}^{-1} \\
f(\mu_{m-1}\cdots\mu_0(\overline A_{v_1}))
&= \mu_{m-1}\cdots\mu_0(\overline A_{v_1}) x_{n-m+t}^{-1} \\
f(\mu_{m-1}\cdots\mu_0(\overline A_{v_2}))
&= \mu_{m-1}\cdots\mu_0(\overline A_{v_2}) x_{n-m+t}^{-1} \\
f(\mu_{m-1}\cdots\mu_0(\overline A_{v_3}))
&= \mu_{m-1}\cdots\mu_0(\overline A_v) x_{n-m+t-1}^{-1}\\
f(\mu_{m-1}\cdots\mu_0(\overline A_{v_4}))
&= \mu_{m-1}\cdots\mu_0(\overline A_{v_4}) x_{n-m+t-1}^{-1}.
\end{align*}
This implies that
\begin{align*}
f(\mu_{m}\mu_0(\overline A_v))
= f(\mu_v(\mu_{m-1}\mu_0(\overline A_v)))
= \mu_v(\mu_{m-1}\mu_0(\overline A_v)) x_{n-m+t-1}^{-1}
= \mu_{m}\mu_0(\overline A_v) x_{n-m+t-1}^{-1}.
\end{align*}

When $t=m$, we have $v,v_1,v_2\notin \sqcup_{1\leq j\leq n-1} V(\lambda_4,p,j) \cup (\cup_{0\leq l\leq m-1} V_{l})$ and $v_3,v_4\in V_{m-1,m-1}$. Then
\begin{align*}
f(\mu_{m-1}\cdots\mu_0(\overline A_{v}))
&= f(\a_v) = \a_v = \mu_{m-1}\cdots\mu_0(\a_v) \\
f(\mu_{m-1}\cdots\mu_0(\overline A_{v_2}))
&= f(\a_{v_1}) = \a_{v_1} = \mu_{m-1}\cdots\mu_0(\a_{v_1}) \\
f(\mu_{m-1}\cdots\mu_0(\overline A_{v_2}))
&= f(\a_{v_2}) = \a_{v_2} = \mu_{m-1}\cdots\mu_0(\a_{v_2}) \\
f(\mu_{m-1}\cdots\mu_0(\overline A_{v_3}))
&= \mu_{m-1}\cdots\mu_0(\overline A_{v_3}) x_{n-1}^{-1} \quad \text{(by the induction assumption)}\\
f(\mu_{m-1}\cdots\mu_0(\overline A_{v_4}))
&= \mu_{m-1}\cdots\mu_0(\overline A_{v_3}) x_{n-1}^{-1} \quad \text{(by the induction assumption)}.
\end{align*}
This implies that
\begin{align*}
f(\mu_m\cdots\mu_0(\overline A_v))
= f(\mu_v\mu_{m-1}\cdots\mu_0(\overline A_v))
= \mu_v\mu_{m-1}\cdots\mu_0(\overline A_v) x_{n-1}^{-1}
= \mu_m\cdots\mu_0(\overline A_v) x_{n-1}^{-1}.
\end{align*}

Thus the induction implies \eqref{eq-f-mui} for each $0\leq i\leq n-2$ and $0\leq t\leq i$.

For each $0\leq i\leq n-2$ and $0\leq t\leq i$, define
$$
V_{i,t}^\partial=\{(i+1-t, 1+t), (n-1-t,n-i+t)\}.
$$
Note that $V_{i,t}^\partial \cap V_{l}=\emptyset$ for $l\geq i+1$. Then, for each $v\in V_{i,t}^\partial$, we have
\begin{align}\label{eq-mu-p4-Ax}
f(\mu_{\lambda_4'\lambda_4}(\overline A_v))
= f(\mu_{i}\cdots\mu_0(\overline A_v))
= \mu_{i}\cdots\mu_0(\overline A_v) x_{n-i-1+t}^{-1}
=\mu_{\lambda_4'\lambda_4}(\overline A_v) x_{n-i-1+t}^{-1}.
\end{align}

It is straightforward to check that, for each $1\leq i\leq n-1$, we have
\begin{align}\label{eq-mu-p4-Ax-1}
V(\lambda_4',p,i)
=(\sqcup_{0\leq j\leq i-1} V_{n-i-1+j,j}^\partial)\sqcup \{(n-i,0), (n,i)\}.
\end{align}

Equations~\eqref{eq-mu-p4-Ax} and \eqref{eq-mu-p4-Ax-1} imply that, for each $v\in V(\lambda_4',p,i)$, we have
\begin{align}\label{eq-mu-p4-Ax-2}
f(\mu_{\lambda_4'\lambda_4}(\overline A_v)) = \overline A_v x_{i}^{-1}.
\end{align}

Note that
\begin{align}\label{eq-mu-p4-Ax-3}
\mathring{\overline{V}}_{\lambda_4}= \sqcup_{1\leq i\leq n-1} V(\lambda_4',p,i),
\end{align}
where $\mathring{\overline{V}}_{\lambda_4}\subset \overline V_{\lambda_4}$ is the set of all the mutable vertices. Then Equations~\eqref{eq-mu-p4-Ax-2} and \eqref{eq-mu-p4-Ax-3} complete the proof.
\end{proof}

\def\Afo{\overline{\mathcal{A}}(\bP_4,\lambda_4)}\
\def\Afw{\widetilde{\mathcal{A}}(\bP_4,\lambda_4)}
\def\Afoux{\overline{\mathcal{A}}(\bP_4,\lambda_4)\otimes \mathbb Z[x_1^{\pm 1},\cdots, x_{n-1}^{\pm 1}]}

When $\omega^{\frac{1}{2}}=1$ and $R=\mathbb Z$, we use $\overline{\mathcal A}$ and $\mathcal A$ to denote 
$\overline{\mathcal A}_\omega$ and $\mathcal A_\omega$, respectively.
Define
\[
\Afw := \mathbb Z[\overline A_v^{\pm 1} \mid v \in V_4 := \overline V_{\lambda_4} \setminus \{n1, n2, \ldots, (n, n-1)\} ].
\]
Let $\Gamma_4$ be the full subquiver of $\Gamma_{\lambda_4}$ consisting of all vertices in $V_4$.
Then $(\Gamma_4, (\overline A_v)_{v \in V_4})$ forms a classical seed.
Define the following projections:
\begin{align*}
    P_1 &\colon \Afo \longrightarrow \Afw, \qquad
    \overline A_v \longmapsto
    \begin{cases}
        \overline A_v, & v \in V_v,\\
        1, & \text{otherwise},
    \end{cases}
    \\[0.8em]
    P_2 &\colon \Afoux \longrightarrow \Afo, \qquad
    \begin{cases}
        \overline A_v \longmapsto \overline A_v, & v \in \overline V_4,\\[0.3em]
        x_i \longmapsto \overline A_{ni}, & 1 \le i \le n-1.
    \end{cases}
\end{align*}
Then we have the following lemma.

\begin{lemma}\label{lem-pro-mut}
    Let $v_1,\cdots,v_m\in \mathring{\overline{V}}_{\lambda_4}$, and $h(x_1,\cdots,x_{n-1})\in \mathbb Z[x_1^{\pm 1},\cdots, x_{n-1}^{\pm 1}]$. Then:

\begin{enumerate}[label={\rm (\alph*)}]\itemsep0.3em
    \item  $P_1(\mu_{v_1}\cdots \mu_{v_m}(\overline A_{v_1}))
=\mu_{v_1}\cdots \mu_{v_m}(\overline A_{v_1})\in \Afw.$

\item $P_2(\mu_{v_1}\cdots \mu_{v_m}(\overline A_{v_1})h(x_1,\cdots,x_{n-1}))
=\mu_{v_1}\cdots \mu_{v_m}(\overline A_{v_1})h(\overline A_{n1},\cdots,\overline A_{n,n-1})\in \Afo.$
\end{enumerate}
\end{lemma}
\begin{proof}
Part (a) is well known; see, for example, \cite[Proposition~2.29]{CZ} and \cite[Lemma~5.4]{HLY}.

    Part (b) is immediate from the definition of $P_2$.
\end{proof}


It is straightforward to check that $f$ induces the following $\mathbb Z$-algebra embedding
\[
    \bar f\colon 
    \Afw
    \longrightarrow
    \Afo,
    \qquad
    \overline A_v \longmapsto 
    \overline A_v \!\!\prod_{i \in V(\lambda_4,p,v)} \! \overline A_{ni}^{-1} \text{for each $v \in V_{4}$}.
\]

\begin{lemma}\label{lem-P4-flip-iso-1}
    For each $v \in  V_{4}$, we have
    \[
        \bar f\!\left(\mu_{\lambda_4'\lambda_4}(\overline A_v)\right)
        =
        \mu_{\lambda_4'\lambda_4}(\overline A_v)
        \prod_{i \in V(\lambda_4',p,v)} \overline A_{ni}^{-1}\in \Afo,
    \]
    where $\mu_{\lambda_4'\lambda_4}$ is defined in~\eqref{def-eq-mul}.
\end{lemma}
\begin{proof}
    We have 
    \begin{align*}
        \bar f\!\left(\mu_{\lambda_4'\lambda_4}(\overline A_v)\right) &= \bar f\, P_1\!\left(\mu_{\lambda_4'\lambda_4}(\overline A_v)\right) \qquad \text{(by Lemma~\ref{lem-pro-mut}(a))}\\
        & = P_2\; f \!\left(\mu_{\lambda_4'\lambda_4}(\overline A_v)\right) \qquad \text{(by the definition of $\bar f$)}\\
        & = P_2 \!\left(
        \mu_{\lambda_4'\lambda_4}(\overline A_v)
        \prod_{i \in V(\lambda_4',p,v)} x_i^{-1}\right) \qquad \text{(by Lemma~\ref{lem-P4-flip-iso})}\\
        & = 
        \mu_{\lambda_4'\lambda_4}(\overline A_v)
        \prod_{i \in V(\lambda_4',p,v)} A_{ni}^{-1} \qquad \text{(by Lemma~\ref{lem-pro-mut}(b))}.
    \end{align*}
\end{proof}

\def\Zq{\mathbb{Z}[\omega^{\pm \frac{1}{2}}]}

\begin{proof}[Proof of Proposition~\ref{prop-flips-L}]
It follows from Proposition~\ref{Prop-quasi-iso} that 
\begin{align}\label{eq-Lla-Zq}
    L_{\lambda}\!\left(\mu_{\lambda'\lambda^\ast}(A_v)\right)
    = [\mu_{\lambda'\lambda^\ast}(\overline A_v)\, \overline A^{{\bf k}_v}]
    \in \Aa 
    \qquad \text{for some ${\bf k}_v \in \mathbb Z^{\overline V_{\rm fro}}$},
\end{align}
where $\overline V_{\rm fro}$ is the set of small vertices in $\overline V_{\lambda^\ast}$ contained in the boundary of $\fS^\ast$.  
Note that $R$ is a $\Zq$-algebra.  
By applying the functor $R \otimes_{\Zq} (-)$, we see that the vector ${\bf k}_v$ is independent of the choice of the ground ring $R$ and the invertible element $\omega^{\frac{1}{2}}$.  
Therefore, we may assume $R = \mathbb Z$ and $\omega^{\frac{1}{2}} = 1$.  
Note that the triangulation $\lambda'$ is obtained from $\lambda^\ast$ by a sequence of flips.
Then the proposition follows directly from Lemmas~\ref{lem-P4-flip-iso} and \ref{lem-P4-flip-iso-1}.
\end{proof}

Before stating the first main theorem of this section, we record one additional useful lemma.
Recall that $\mathbb P_2$ has a unique generalized triangulation, which we continue to denote by $\mathbb P_2$. 
Let $\lambda$ be the induced triangulation of $\mathbb P_4=\mathbb P_2^\ast$ (see Figure~\ref{Fig;attaching}(B)).
We write
\[
L:=L_{\mathbb P_2}\colon \cA(\mathbb P_2) \longrightarrow \overline{\mathcal{A}}_\omega(\mathbb P_4,\lambda)
\]
for the algebra embedding defined in \eqref{eq-def-L-lamda}.
We label the small vertices of $\overline V_\lambda$ as in Figure~\ref{Fig;essential-P4}(B), so that the subset 
$V_{\mathbb P_2}\subset \overline V_\lambda$ inherits this labeling.

Define
\[
\mathcal A_4:=
\overline{\mathcal{A}}_\omega(\mathbb P_4,\lambda)\otimes _R 
R[x_1^{\pm 1},y_1^{\pm 1},\dots, x_{n-1}^{\pm 1}, y_{n-1}^{\pm 1}],
\]
and consider the algebra embedding
\[
F\colon \overline{\mathcal{A}}_\omega(\mathbb P_4,\lambda)\longrightarrow \mathcal A_4,
\qquad 
\overline A_{ij}\longmapsto
\begin{cases}
    \a_{ij}, & 1\le i=j\le n-1,\\[1mm]
    \a_{ij}\,x_{\,i-j}^{-1}, & 0\le j<i\le n,\\[1mm]
    \a_{ij}\,y_{\,n-j+i}^{-1}, & 0\le i<j\le n,
\end{cases}
\]
where we set $\a_{n0}=\a_{0n}=x_n=y_0=1$.

We can now state the lemma.

\begin{lemma}\label{lem-gL}
\begin{enumerate}[label={\rm (\alph*)}]\itemsep0.3em
\item  In $\mathcal A_4$, we have
\[
\begin{aligned}
& F(\overline A_{i-1,n-1})
    = \bigl[\,y_i^{-1}\,\overline A_{i-1,n-1}\,\bigr],
    &&  1\leq i\leq n, \\[1.2ex]
& F\!\left(\mu^{\diamondsuit}_{j}(\overline A_{n-1,n-2})\right)
    = \bigl[\,\mu^{\diamondsuit}_{j}(\overline A_{n-1,n-2})\,x_j^{-1}\,\bigr],
    && 1\leq j\leq n, \\[1.2ex]
& F\!\left(
(\mu_{n-1,n-1}\cdots \mu_{i+1,i+1}\mu_{ii})
\circ \mu^{\diamondsuit}_j \circ \mu^{\triangle}_{i}(\overline A_{n-1,n-1})
\right)  \\
&\qquad = \Bigl[
(\mu_{n-1,n-1}\cdots \mu_{i+1,i+1}\mu_{ii})
\circ \mu^{\diamondsuit}_j \circ \mu^{\triangle}_{i}(\overline A_{n-1,n-1})
\; y_i^{-1}\,x_j^{-1}
\Bigr],
    &&  n>i\ge j\geq 1, \\[1.2ex]
& F\!\left(
(\mu_{n-1,n-1}\cdots \mu_{i+1,i+1}\mu_{ii})
\circ \mu^{\diamondsuit}_j \circ \mu^{\diamondsuit}_{(i;j)}(\overline A_{n-1,n-1})
\right) \\
&\qquad = \Bigl[
(\mu_{n-1,n-1}\cdots \mu_{i+1,i+1}\mu_{ii})
\circ \mu^{\diamondsuit}_j \circ \mu^{\diamondsuit}_{(i;j)}(\overline A_{n-1,n-1})
\;y_i^{-1}\,x_j^{-1}
\Bigr],
    && n>j>i\geq 1,
\end{aligned}
\]
where all mutations are as in Theorem~\ref{intro-thm-skein-inclusion-A} (f).

\item In $\overline{\mathcal{A}}_\omega(\mathbb P_4)$, we have
\[
\begin{aligned}
& L(A_{i-1,n-1})
    = \bigl[\,\overline A_{in}^{-1}\,\overline A_{i-1,n-1}\,\bigr],
    && 1\leq i\leq n, \\[1.2ex]
& L\!\left(\mu^{\diamondsuit}_{j}(A_{n-1,n-2})\right)
    = \bigl[\,\mu^{\diamondsuit}_{j}(\overline A_{n-1,n-2})\,\overline A_{j0}^{-1}\,\bigr],
    && 1\leq j\leq n, \\[1.2ex]
& L\!\left(
(\mu_{n-1,n-1}\cdots \mu_{i+1,i+1}\mu_{ii})
\circ \mu^{\diamondsuit}_j \circ \mu^{\triangle}_{i}(A_{n-1,n-1})
\right)  \\
&\qquad = \Bigl[
(\mu_{n-1,n-1}\cdots \mu_{i+1,i+1}\mu_{ii})
\circ \mu^{\diamondsuit}_j \circ \mu^{\triangle}_{i}(\overline A_{n-1,n-1})
\;\overline A_{in}^{-1}\,\overline A_{j0}^{-1}
\Bigr],
    &&  n>i\ge j\geq 1, \\[1.2ex]
& L\!\left(
(\mu_{n-1,n-1}\cdots \mu_{i+1,i+1}\mu_{ii})
\circ \mu^{\diamondsuit}_j \circ \mu^{\diamondsuit}_{(i;j)}(A_{n-1,n-1})
\right) \\
&\qquad = \Bigl[
(\mu_{n-1,n-1}\cdots \mu_{i+1,i+1}\mu_{ii})
\circ \mu^{\diamondsuit}_j \circ \mu^{\diamondsuit}_{(i;j)}(\overline A_{n-1,n-1})
\;\overline A_{in}^{-1}\,\overline A_{j0}^{-1}
\Bigr],
    &&  n>j>i\geq 1.
\end{aligned}
\]
\end{enumerate}
\end{lemma}

\begin{proof}
    (a) The first equality follows directly from the definition of $F$.  
We will therefore prove only the second equality, since the same techniques apply to the third and fourth ones.

Recall that 
\[
\mu^{\diamondsuit}_j:= \mu^r_{(n-1;j-1)} \mu^r_{(n-2;j-1)} \cdots \mu^r_{(j;j-1)},
\]
where $\mu_{(s;t)}^r$ is defined in \eqref{def-must}.  
Set $x_0=1$.  
For each $j \le l \le n-1$ and $l-j+1 \le m \le l-1$, we will prove
\begin{equation}\label{eq-g-muuuu}
\begin{split}
    &F\!\left(\mu_{lm}\cdots \mu_{l,l-j+1} \mu_{(s-1;j-1)}^r \cdots \mu_{(j;j-1)}^r(\overline A_{lm})\right)\\
    =&\mu_{lm}\cdots \mu_{l,l-j+1}\mu_{(s-1;j-1)}^r\cdots \mu_{(j;j-1)}^r(\overline A_{lm}) 
      x_{l-m-1}^{-1} x_j^{-1},
\end{split}
\end{equation}
by induction on~$l$.

 Base case $l=j$:
When $l=j$, \eqref{eq-g-muuuu} reduces to
\begin{equation}\label{eq-g-muuuu1}
    F\bigl(\mu_{jm}\cdots \mu_{j1}(\overline A_{jm})\bigr)
    = \mu_{jm}\cdots \mu_{j1}(\overline A_{jm})\, x_{j-m-1}^{-1} x_j^{-1}.
\end{equation}

It is straightforward that
\[
F(\mu_{j1}(\a_{j1})) = \mu_{j1}(\a_{j1})\, x_{j-2}^{-1} x_j^{-1}.
\]
For each $2\leq m< j$, define
\[
Q_1 := \mu_{j,m-1}\cdots \mu_{j1}(\overline Q_\lambda).
\]
The following can be proved by induction on $m$ and mutation rules:
\begin{align}\label{eq-Q1}
Q_1(jm, v)=
\begin{cases}
    1 & v=(j,m+1), (j,m-1),\\
    -1 & v=(j-1,m), (j+1,m+1), (j,0),\\
    0 & \text{otherwise}.
\end{cases}
\end{align}
Using \eqref{eq-Q1}, one verifies \eqref{eq-g-muuuu1} by induction on~$m$.

Inductive step $l>j$:
Define
\[
Q_2 := \mu_{(l-1;j-1)}^r\cdots \mu_{(j;j-1)}^r(\overline Q_\lambda).
\]
By Appendix~\ref{app-proof}, we have
\begin{align}\label{eq-Q2}
Q_2((l,l-j+1), v)=
\begin{cases}
    1 & v=(l,l-j+2), (j-1,0), (l+1,l-j+1),\\
    -1 & v=(l-1,l-j), (l+1,l-j+2),\\
    0 & \text{otherwise}.
\end{cases}
\end{align}
Combining \eqref{eq-Q2} with the induction hypothesis in $l$, we compute
\[
F\!\left(\mu_{l,l-j+1}\mu_{(l-1;j-1)}^r\cdots \mu_{(j;j-1)}^r(\overline A_{l,l-j+1})\right)
=
\mu_{l,l-j+1}\mu_{(l-1;j-1)}^r\cdots \mu_{(j;j-1)}^r(\overline A_{l,l-j+1})\,
x_{j-2}^{-1} x_j^{-1}.
\]

Now let $l-j+1 < m \le l-1$, and define
\[
Q_3 := \mu_{l,m-1} \cdots \mu_{l,l-j+1} \mu_{(l-1;j-1)}^r\cdots \mu_{(j;j-1)}^r(\overline Q_\lambda).
\]
Appendix~\ref{app-proof} gives
\begin{align}\label{eq-Q3}
Q_3((l,m), v)=
\begin{cases}
    1 & v=(l,m-1), (l,m+1),\\
    -1 & v=(l-1,m-1), (l+1,m+1),\\
    0 & \text{otherwise}.
\end{cases}
\end{align}
Using \eqref{eq-Q3} together with the induction hypothesis on~$l$, one proves \eqref{eq-g-muuuu} by induction on~$m$.  
This completes the proof of \eqref{eq-g-muuuu} by double induction.

Finally, setting $l=n-1$ and $m=n-2$ in \eqref{eq-g-muuuu} yields
\[
F\!\left(\mu^{\diamondsuit}_{j}(\overline A_{n-1,n-2})\right)
    = \mu^{\diamondsuit}_{j}(\overline A_{n-1,n-2})\, x_j^{-1}.
\]

    (b) Due to Proposition~\ref{Prop-quasi-iso}, we may assume that $R=\mathbb Z$ and $\omega^{1/2}=1$. 
Note that 
\[
\mathcal A(\mathbb P_2)
    = \overline{\mathcal A}(\mathbb P_4)\big/(\,\a_{i n}=\a_{i0}=1,\; 1\le i\le n-1\,),
\qquad
\overline{\mathcal A}(\mathbb P_4)
    = \mathcal A_4\big/(\,\a_{in}=y_i,\; \a_{i0}=x_i,\; 1\le i\le n-1\,),
\]
and that $L$ is induced by $F$.  
Thus, statement (b) follows from statement (a) together with an argument analogous to that in Lemma~\ref{lem-pro-mut}(a).
\end{proof}

\begin{remark}
The proof of Lemma \ref{lem-gL} can alternatively be established via the computation of $g$-vectors, following the same approach as in the proof of Lemma \ref{lem:quasi2} below.
\end{remark}

The following theorem is the first main result of this section. It asserts that the stated ${\rm SL}_n$–skein algebra is contained in the quantum cluster algebra $\mathscr A_{\omega}(\fS)$. We further show that these two algebras coincide when $n=2$ or when the pb surface is a polygon.
Although there has been substantial work on realizing the reduced stated ${\rm SL}_n$–skein algebra as a quantum cluster algebra \cite{muller2016skein,ishibashi2023skein,LY22,huang2025quantum}, to the best of our knowledge, the quantum cluster algebra structure of the (non-reduced) stated ${\rm SL}_n$–skein algebra has not yet been addressed in the literature.

\begin{theorem}\label{thm-skein-inclusion-A}
Let $\fS$ be a generalized triangulable pb surface without interior punctures, such that every connected component contains at least two punctures. 
    Then we have $$\SS\subset \mathscr A_{\omega}(\fS).$$
Moreover, every stated essential arc is a cluster variable
in $\mathscr A_{\omega}(\fS)$.
    
This inclusion is compatible with the inclusion $\widetilde{\cS}_\omega(\fS^*) \subset \overline{\mathscr A}_{\omega}(\fS^*)$ established in Theorem~\ref{intro-thm-skein-inclusion-A}(c). Specifically, the following diagram commutes:
\begin{equation}\label{eq-com-skein-A}
    \begin{tikzcd}[
    row sep=0.7cm,       
    column sep=1cm,   ]
    \SS \arrow[d, " "', hook] \arrow[rrrr, "\iota_\ast",hook] &&&& \dSS \arrow[d, " ", hook] \\  
    {\mathscr A}_{\omega}(\fS) \arrow[rr, 
        "{\begin{minipage}{3cm}\centering \scriptsize adding isolated \\ frozen variables \end{minipage}}"',hook] 
        &&  {\mathscr A}_{\omega}^{\rm qc}(\fS^\ast) \arrow[rr, "\text{quasi-iso.}"'] &&  \overline {\mathscr A}_{\omega}(\fS^\ast),
\end{tikzcd}
\end{equation}
where $\iota_*$ is the algebra embedding defined in \eqref{eq-def-iota-dS}.
\end{theorem}
\begin{proof}
    Recall that a properly embedded oriented arc in $\fS$ is called an essential arc
    if its endpoints lie on two distinct components of $\partial \fS$.
    By Lemma~\ref{lem-essential}, the algebra $\SS$ is generated by finitely many stated essential arcs, none of which is a clockwise oriented corner arc.
    Let $\alpha$ be any stated essential arc in $\SS$ that is not a clockwise oriented corner arc.
    It suffices to show that $\alpha \in \mathscr A_{\omega}(\fS)$.

\medskip
\noindent
\textbf{Case 1: $\alpha$ is a counterclockwise oriented corner arc.}
    Write 
    \begin{align}\label{pro-inclusion-al}
        \alpha = 
        \begin{array}{c}\includegraphics[scale=0.8]{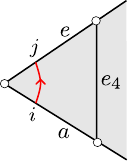}\end{array},
    \end{align}
    where the two boundary edges incident to $\alpha$ are labeled $e$ and $a$, and the ideal arc forming a triangle with $e$ and $a$ is labeled $e_4$.  
    Then
    \begin{align}\label{pro-inclusion-al1}
        \iota_\ast(\alpha) =
        \begin{array}{c}\includegraphics[scale=0.8]{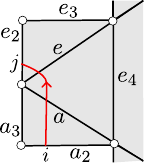}\end{array} 
        \in \dSS,
    \end{align}
    where $\iota_\ast$ is defined in \eqref{eq-def-iota-dS}, and where the two additional edges attached to $e$ (resp. $a$) are denoted $e_2,e_3$ (resp.\ $a_2,a_3$).

    Let $\lambda$ be a triangulation of $\fS$ containing $e_4$, and let $\lambda'$ be the triangulation of $\fS^\ast$ obtained from $\lambda^\ast$ by a flip at $e$, yielding
    \(
        \lambda' = (\lambda^\ast \setminus\{e\}) \cup \{e'\}.
    \)
    The arcs $e',e_2,a_3,a_2$ bound an embedded $\mathbb P_4$ in $\fS^\ast$, denoted $P_4$, with $a$ as its diagonal.  
    Identifying $a_3,a_2,e',e_2$ with $c_1,c_2,c_3,c_4$ in Figure~\ref{Fig;essential-P4}(A) induces, via Figure~\ref{Fig;essential-P4}(B), a labeling of the small vertices in $\overline V_{\lambda'}\cap P_4$.
    Set $\mu_e=\mu_{\lambda'\lambda^\ast}$.
    Then Theorem~\ref{intro-thm-skein-inclusion-A}(c) implies
    \begin{align}\label{pro-inclusion-1}
        \iota_\ast(\alpha) =
        \begin{cases}
            [\,\overline A_{in}^{-1} \cdot \mu_e(\overline A_{i-1,n-1})\,], & j=n, \\[1mm]
            [\,\mu^{\diamondsuit}_{j} \circ \mu_e(\overline A_{n-1,n-2}) \cdot \overline A_{j0}^{-1}\,], & i=n, \\[1mm]
            [\,(\mu_{n-1,n-1}\cdots \mu_{i+1,i+1}\mu_{ii}) \circ 
              \mu^{\diamondsuit}_{j} \circ \mu^{\triangle}_{i} \circ 
              \mu_e(\overline A_{n-1,n-1})
              \cdot \overline A_{in}^{-1} \overline A_{j0}^{-1}\,], & n>i\ge j,\\[1mm]
            [\,(\mu_{n-1,n-1}\cdots\mu_{i+1,i+1}\mu_{ii}) \circ 
              \mu^{\diamondsuit}_{j} \circ \mu^{\diamondsuit}_{(i;j)}\circ 
              \mu_e(\overline A_{n-1,n-1})
              \cdot \overline A_{in}^{-1}\overline A_{j0}^{-1}\,], & n>j>i.
         \end{cases}
    \end{align}

    Using Proposition~\ref{prop-flips-L} and Lemma~\ref{lem-gL}(b), we obtain
    \begin{align}\label{pro-inclusion-2}
        \begin{aligned}
            &L_\lambda(\mu_e(A_{i-1,n-1})) 
                = [\,\overline A_{in}^{-1}\mu_e(\overline A_{i-1,n-1})\,],
            && 1\le i\le n,\\[1.2ex]
            &L_\lambda(\mu^{\diamondsuit}_{j}\circ \mu_e(A_{n-1,n-2}))
                = [\,\mu^{\diamondsuit}_{j}\circ \mu_e(\overline A_{n-1,n-2})\overline A_{j0}^{-1}\,],
            && 1\le j\le n,\\[1.2ex]
            &L_\lambda\!\left(
                (\mu_{n-1,n-1}\cdots\mu_{i+1,i+1}\mu_{ii})
                \circ \mu^{\diamondsuit}_{j} \circ \mu^{\triangle}_{i} \circ 
                \mu_e(A_{n-1,n-1})
            \right) \\
            &\qquad= \Bigl[
                (\mu_{n-1,n-1}\cdots\mu_{i+1,i+1}\mu_{ii})
                \circ \mu^{\diamondsuit}_{j} \circ \mu^{\triangle}_{i}\circ 
                \mu_e(\overline A_{n-1,n-1})
                \overline A_{in}^{-1}\overline A_{j0}^{-1}
            \Bigr],
            && n>i\ge j\ge 1,\\[1.2ex]
            &L_\lambda\!\left(
                (\mu_{n-1,n-1}\cdots\mu_{i+1,i+1}\mu_{ii})
                \circ \mu^{\diamondsuit}_{j} \circ \mu^{\diamondsuit}_{(i;j)} \circ 
                \mu_e(A_{n-1,n-1})
            \right) \\
            &\qquad= \Bigl[
                (\mu_{n-1,n-1}\cdots\mu_{i+1,i+1}\mu_{ii})
                \circ \mu^{\diamondsuit}_{j}\circ \mu^{\diamondsuit}_{(i;j)}\circ 
                \mu_e(\overline A_{n-1,n-1})
                \overline A_{in}^{-1}\overline A_{j0}^{-1}
            \Bigr],
            && n>j>i\ge 1.
        \end{aligned}
    \end{align}

    Combining \eqref{pro-inclusion-1}, \eqref{pro-inclusion-2}, and Lemma~\ref{lem-com-reduced-stated}, we conclude that
    \begin{align}\label{pro-inclusion-3}
        \alpha =
        \begin{cases}
           \mu_e(A_{i-1,n-1}), & j=n,\\[1mm]
           \mu^{\diamondsuit}_{j}\circ\mu_e(A_{n-1,n-2})\,], & i=n,\\[1mm]
           (\mu_{n-1,n-1}\cdots\mu_{i+1,i+1}\mu_{ii})\circ \mu^{\diamondsuit}_j\circ\mu^{\triangle}_i\circ\mu_e(A_{n-1,n-1}), & n>i\ge j,\\[1mm]
            (\mu_{n-1,n-1}\cdots\mu_{i+1,i+1}\mu_{ii})\circ \mu^{\diamondsuit}_j\circ\mu^{\diamondsuit}_{(i;j)}\circ\mu_e(A_{n-1,n-1}), & n>j>i.
        \end{cases}
    \end{align}
    Hence $\alpha\in \mathscr A_{\omega}(\fS)$.

\medskip
\noindent
\textbf{Case 2: $\alpha$ is not a corner arc.}
    Suppose $\alpha$ is incident to boundary edges $c_2$ and $c_4$.  
    Since it is not a corner arc, there exist ideal arcs $c_1$ and $c_3$ such that 
    $c_1,c_2,c_3,c_4$ and $\alpha$ appear as in Figure~\ref{Fig;essential-P4}, with $\alpha$ shown in red.
    Then
    \begin{align}\label{pro-inclusion-5}
        \iota_\ast(\alpha) =
        \begin{array}{c}\includegraphics[scale=0.5]{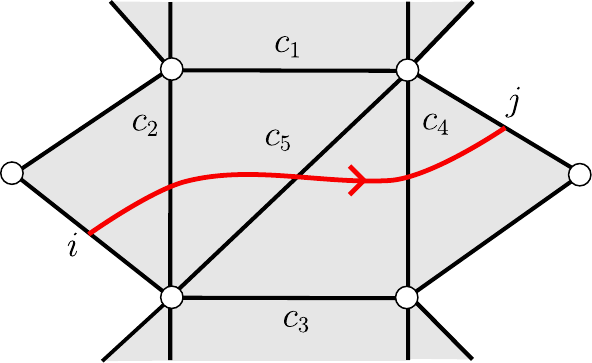}\end{array}
        \in \dSS.
    \end{align}

    Let $\lambda_1$ be a triangulation of $\fS$ containing $c_1,c_3$, and let $\lambda_2$ be obtained from $\lambda_1^\ast$ by flipping $c_2$ and $c_4$.  
    Let $c_2'$ (resp.\ $c_4'$) be the arc obtained from $c_2$ (resp.\ $c_4$) by the corresponding flip, and let $c_5,c_6$ be the boundary edges of $\fS^\ast$ adjacent to the red arc in \eqref{pro-inclusion-5}.  
    Then $c_2',c_4',c_5,c_6$ bound an embedded $\mathbb P_4$, denoted $P_4'$, containing $\iota_\ast(\alpha)$.
    Identifying $P_4'$ with Figure~\ref{Fig;essential-P4}(A) so that $\iota_\ast(\alpha)$ matches the red arc there, the labeling in Figure~\ref{Fig;essential-P4}(B) induces a labeling on $\overline V_{\lambda_2}\cap P_4'$.

    Set $\mu_{c_2c_4}=\mu_{\lambda_2\lambda_1^\ast}$.  
    By the same argument as in Case~1, using Theorem~\ref{intro-thm-skein-inclusion-A}(c), Proposition~\ref{prop-flips-L}, and Lemma~\ref{lem-gL}(a), we obtain
    \begin{align}\label{pro-inclusion-4}
        \alpha =
        \begin{cases}
            \mu_{c_2c_4}(A_{i-1,n-1}), & j=n,\\[1mm]
           \mu^{\diamondsuit}_{j}\circ\mu_{c_2c_4}(A_{n-1,n-2}), & i=n,\\[1mm]
           (\mu_{n-1,n-1}\cdots\mu_{i+1,i+1}\mu_{ii})\circ 
               \mu^{\diamondsuit}_j\circ\mu^{\triangle}_{i}\circ
               \mu_{c_2c_4}(A_{n-1,n-1}), & n>i\ge j,\\[1mm]
            (\mu_{n-1,n-1}\cdots\mu_{i+1,i+1}\mu_{ii})\circ 
               \mu^{\diamondsuit}_j\circ\mu^{\diamondsuit}_{(i;j)}\circ
               \mu_{c_2c_4}(A_{n-1,n-1}), & n>j>i.
        \end{cases}
    \end{align}
    Hence $\alpha\in \mathscr A_{\omega}(\fS)$ in this case as well.

\medskip
Combining both cases, every generator of $\SS$ lies in $\mathscr A_\omega(\fS)$, completing the proof of $\SS\subset \mathscr A_{\omega}(\fS)$.

The commutative diagram \eqref{eq-com-skein-A} follows immediately from Lemmas~ \ref{lem-iota-gv}, \ref{lem-com-reduced-stated}, \eqref{4eq:qc0} and Proposition \ref{Prop-quasi-iso}.
\end{proof}

Equations \eqref{pro-inclusion-3} and \eqref{pro-inclusion-3} give a formula for expressing any stated essential arc that is not a clockwise oriented corner arc as an exchangeable cluster variable.  
The following remark provides an explicit formula for expressing a clockwise oriented stated corner arc as an exchangeable cluster variable in $\mathscr A_{\omega}(\fS)$.

\begin{remark}
    Let $\beta$ be the stated corner arc obtained from the arc in \eqref{pro-inclusion-al} by reversing its orientation and swapping the states at its two endpoints.
    We keep the same labeling conventions for ideal arcs as in \eqref{pro-inclusion-al} and \eqref{pro-inclusion-al1}, and we use the notations $\lambda$, $\lambda'$, $\mu_e$, $P_4$, and $e'$ as in Case~1 of the proof of Theorem~\ref{thm-skein-inclusion-A}.

    Identifying $e', e_2, a_3, a_2$ with $c_1, c_2, c_3, c_4$ in Figure~\ref{Fig;essential-P41}(A) determines, via Figure~\ref{Fig;essential-P41}(B), a labeling of the small vertices in $\overline V_{\lambda'} \cap P_4$.

    For any $j>1$, define  
    \[
        \overline \mu^{\diamondsuit}_j
        =\bar \mu_{(1;n-j)}\cdots \bar \mu_{(j-2;n-j)} \bar \mu_{(j-1;n-j)},
    \]
    where $\overline{\mu}_{(k;t)}$ is defined in \eqref{intro-eq-bar-mukj}.  
    For any $i,j$ with $i\ge j>1$, define
    \[
        \overline \mu^{\diamondsuit}_{(i;j-1)}
        =\left(\mu_{(2;1)}\mu_{(3;2)}\cdots\mu_{(j-1;j-2)}\right)
        \circ
        \left(\mu_{(j;j-1)}\mu_{(j+1;j-1)}\cdots \mu_{(i-1;j-1)}\right),
    \]
    where $\mu_{(k;j)}$ is defined in \eqref{intro-eq-mukj}.

    Using an argument parallel to the proof of \eqref{pro-inclusion-3}, together with \cite[Equation~(12)]{huang2025quantum}, we obtain
    \begin{align}\label{eq-alpha-counterclock}
       \beta=
    \begin{cases}
        \mu_e(\overline A_{i1}),
        & \text{if } i\ge j=1, \\[1.5mm]
        \overline \mu^{\diamondsuit}_{j}\circ \mu_e(\overline A_{\overline{12}}),
        & \text{if } j>i=1, \\[1.5mm]
        \left(\mu_{j-1,j-1}\cdots \mu_{22}\mu_{11}\right)
        \circ \overline \mu^{\diamondsuit}_j
        \circ \overline \mu^{\diamondsuit}_{(i;j-1)}
        \circ \mu_e(\overline A_{j-1,j-1}),
        & \text{if } i\ge j>1, \\[1.5mm]
        \left(\mu_{i-1,i-1}\cdots \mu_{22}\mu_{11}\right)
        \circ \overline \mu^{\diamondsuit}_j
        \circ \overline \mu^{\diamondsuit}_{(i;i-1)}
        \circ \mu_e(\overline A_{i-1,i-1}),
        & \text{if } j>i>1.
    \end{cases}
   \end{align}
\end{remark}

\begin{figure}[h]
    \centering
    \includegraphics[width=220pt]{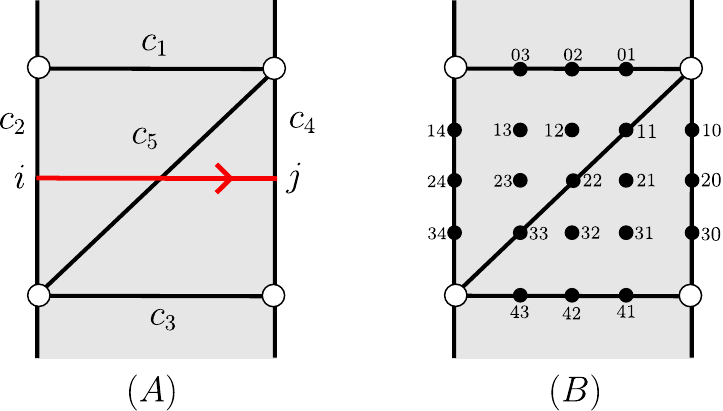}
    \caption{(A) The picture for an essential arc, in red. (B) The labeling for small vertices contained in the quadrilateral bounded by $c_1\cup c_2\cup c_3\cup c_4$ for $n=4$.}\label{Fig;essential-P41}
\end{figure}

\subsection{$\SS= \mathscr A_{\omega}(\fS)=\mathscr U_{\omega}(\fS)$ when $n=2$}

Let $\mathcal{A}$ be an $R$-algebra.
A \textbf{prime element} of $\mathcal{A}$ is a nonzero, nonunit element $a \in \mathcal{A}$ such that
\[
a\mathcal{A} = \mathcal{A}a
\quad \text{and} \quad
\mathcal{A}/(a) \text{ is a domain}.
\]
Here $(a) := \mathcal{A}a = a\mathcal{A}$ denotes the principal ideal generated by $a$.

Let $\fS$ be a triangulable pb surface without interior punctures.
Equation~\eqref{eq-inclusion-A-s} implies that all frozen variables in ${\mathscr A}_\omega(\fS)$ (see \S\ref{sec-mutation-quantum}) lie in ${\cS}_\omega(\fS)$. 
The following lemma, whose proof is deferred to Appendix~\ref{Appendix-B-frozen}, will be used to establish the main result of this subsection.

\begin{lemma}\label{lem-prime-ele}
    When $n=2$, each frozen variable in $\SS$ is a prime element. 
\end{lemma}

The following result establishes the equality between the stated ${\rm SL}_2$-skein algebra and the (upper) quantum cluster algebras $\mathscr A_{\omega}(\fS)$ and $\mathscr U_{\omega}(\fS)$. This is the second main theorem of this section.

\begin{theorem}\label{thm-skein-eq-A-two}
Let $\fS$ be a pb surface without interior punctures. We require that every component of $\fS$ contains at least two punctures.
    When $n=2$, we have 
$$\SS= \mathscr A_{\omega}(\fS)=\mathscr U_{\omega}(\fS).$$
\end{theorem}

\begin{proof}

 Theorems~\ref{thm-inclusion-quantum} and \ref{thm-skein-inclusion-A} show that 
\begin{align}\label{inclusions-S-A-U}
   \SS\subset \mathscr A_{\omega}(\fS)\subset \mathscr U_{\omega}(\fS)
\end{align}

We will first show $\SS= \mathscr A_{\omega}(\fS)$.
Due to \eqref{inclusions-S-A-U}, it therefore suffices to show that 
    $\mathscr A_{\omega}(\fS)\subset \SS$.
    
    Let $\lambda$ be a generalized triangulation of $\fS$.  
    We identify $\overline V_{\lambda^\ast}$ with $\lambda^\ast$, and under this identification we denote the subset $V_\lambda' \subset \lambda^\ast$ by $\lambda'$.
    Let $v\in\lambda'$, and let $v_1,\ldots,v_m$ be a sequence of edges in $\lambda^\ast$ contained in the interior of $\fS^\ast$.
    Then every cluster variable of $\mathscr A_\omega(\fS)$ is of the form 
    \[
        \mu_{v_m}\cdots \mu_{v_1}(A_v).
    \]
    Let $\bar\lambda$ be the triangulation of $\fS^\ast$ obtained from $\lambda^\ast$ by flips along $v_1,\ldots,v_m$.
    Denote by $\partial\lambda$ the set of boundary edges of $\lambda$.
    Recall that for each $e\in\partial\lambda$, there is an attached triangle $\tau$ in $\fS^\ast$ (see Figure~\ref{Fig;attaching}(A)).  
    Let $p_e$ be the unique vertex of $\tau$ not lying in $\fS$.
    Define the multi-set
    \[
        V(\bar\lambda,v):=
        \{\,e\in \partial\lambda \mid \text{one end of $v$ is connected to $p_e$}\,\}.
    \]
    There are four possible cases:
    \[
        V(\bar\lambda,v)=\emptyset,\qquad 
        V(\bar\lambda,v)=\{e_0\},\qquad
        V(\bar\lambda,v)=\{e_1,e_2\}\text{ with }e_1=e_2,\qquad
        V(\bar\lambda,v)=\{e_1,e_2\}\text{ with }e_1\neq e_2.
    \]
    Proposition~\ref{prop-flips-L} gives
    \[
        L_\lambda\bigl(\mu_{v_m}\cdots \mu_{v_1}(A_v)\bigr)
        = \left[
            \mu_{v_m}\cdots \mu_{v_1}(\overline A_v)\,
            \prod_{e\in V(\bar\lambda,v)} \a_{e,1}^{-1}
          \right]\in\Aa,
    \]
    where $L_\lambda$ is defined in~\eqref{eq-def-L-lamda}.

    Note that $v$ is an edge of $\bar\lambda$ (possibly a boundary edge). Theorem~\ref{lem-mutation-A-seeds-flip} implies
    \[
        \mu_{v_m}\cdots \mu_{v_1}(\overline A_v)
            = {\gaa}_{v,\bar\lambda}
            \in \overline{\mathcal A}_\omega(\fS^\ast),
    \]
    where ${\gaa}_{v,\bar\lambda}$ is the element ${\gaa}_v$ defined in \S\ref{sec;A_tori} with respect to the triangulation~$\bar\lambda$.

    \medskip
    \noindent\textbf{Case 1: $V(\bar\lambda,v)=\emptyset$.}
    In this case,
    \begin{align*}
        L_\lambda (\mu_{v_m}\cdots \mu_{v_1}(A_v))
        &=
        \left[\begin{array}{c}\includegraphics[scale=0.6]{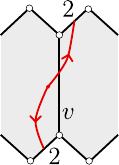}\end{array}\right]_{\rm norm}
        =\iota_\ast\!\left(
            \left[\begin{array}{c}\includegraphics[scale=0.6]{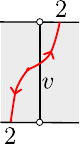}\end{array}\right]_{\rm norm}
        \right),
    \end{align*}
    where $[\;\;]_{\rm norm}$ is defined in \S\ref{sub-sec-invariant} and $\iota_\ast$ in \eqref{def-iota-al}.  
    Then Lemma~\ref{lem-com-reduced-stated} gives
    \[
        \mu_{v_m}\cdots \mu_{v_1}(A_v)
            =
            \left[\begin{array}{c}\includegraphics[scale=0.6]{gv-one.pdf}\end{array}\right]_{\rm norm}
            \in\SS.
    \]

    \medskip
    \noindent\textbf{Case 2: $V(\bar\lambda,v)=\{e_0\}$.}
    We have
    \begin{align}\label{eq-picture-1}
        L_\lambda (\mu_{v_m}\cdots \mu_{v_1}(A_v))
        =
        \left[
            \begin{array}{c}\includegraphics[scale=0.75]{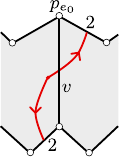}\end{array}
            \overline A_{e_0,1}^{-1}
        \right].
    \end{align}
    Relation~\eqref{wzh.seven} gives
    \begin{align}\label{eq-picture-2}
        \begin{array}{c}\includegraphics[scale=0.75]{gv-three.pdf}\end{array} 
        =
        \left[
            \begin{array}{c}\includegraphics[scale=0.75]{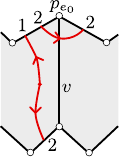}\end{array}
        \right]_{\rm norm}
        \in\overline{\cS}_\omega(\fS^\ast)
        \subset
        \overline{\mathcal A}_\omega(\fS^\ast).
    \end{align}
    It is straightforward that
    \begin{align}\label{eq-picture-3}
        \begin{array}{c}\includegraphics[scale=0.75]{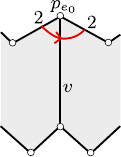}\end{array}
        =
        \overline A_{e_0,1}
        \in
        \overline{\mathcal A}_\omega(\fS^\ast).
    \end{align}

    Combining \eqref{eq-picture-1}, \eqref{eq-picture-2}, and \eqref{eq-picture-3} gives
    \[
        L_\lambda (\mu_{v_m}\cdots \mu_{v_1}(A_v))
        =
        \begin{array}{c}\includegraphics[scale=0.75]{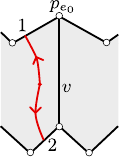}\end{array}
        =
        \iota_\ast\!\left(
            \begin{array}{c}\includegraphics[scale=0.75]{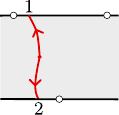}\end{array}
        \right).
    \]
    Lemma~\ref{lem-com-reduced-stated} therefore yields
    \[
        \mu_{v_m}\cdots \mu_{v_1}(A_v)
        =
        \begin{array}{c}\includegraphics[scale=0.75]{gv-six.pdf}\end{array}
        \in \SS.
    \]

    \medskip
 
    The same argument used in Case~2 applies to the remaining two cases as well.
Then we have
\begin{align}\label{eq-pro-sl221}
    \SS = \mathscr A_{\omega}(\fS).
\end{align}

\medskip

It is shown in \cite{muller2016skein} that
\[
\overline{\mathscr A}_\omega(\fS^*)
=
\overline{\mathscr U}_\omega(\fS^*).
\]
Together with \eqref{eq-pro-sl221}, the quasi-isomorphism (Definition~\ref{def:QQH}) in Proposition~\ref{Prop-quasi-iso} implies that
\begin{align}\label{eq-pro-sl22}
    \cS_\omega^{\rm fr}(\fS)
    =
    \mathscr{A}_{\omega}^{\rm fr}(\fS)
    =
    \mathscr{U}_{\omega}^{\rm fr}(\fS),
\end{align}
where $\cS_\omega^{\rm fr}(\fS)$ denotes the localization of $\cS_\omega(\fS)$ at the multiplicative subset generated by all frozen variables, and $\mathscr{A}_{\omega}^{\rm fr}(\fS)$ and $\mathscr{U}_{\omega}^{\rm fr}(\fS)$ are defined in Definition~\ref{def-sln-quantum-cluster-al}.

Using Lemma~\ref{lem-prime-ele} together with \eqref{eq-pro-sl22}, \cite[Theorem~C]{QY} implies that
\begin{align}
    \SS = \mathscr U_{\omega}(\fS).
\end{align}
Combining this with \eqref{inclusions-S-A-U}, we conclude that
\begin{align*}
   \SS = \mathscr A_{\omega}(\fS) = \mathscr U_{\omega}(\fS).
\end{align*}

    \end{proof}

\begin{remark}
Combined with the identity $\overline{\cS}_\omega(\fS^\ast)=\overline{\mathscr A}_{\omega}(\fS^\ast)$
from \cite{muller2016skein,LY22}, the commutative diagram \eqref{eq-com-skein-A} and Theorem~\ref{thm-skein-inclusion-A} imply, from the perspective of quantum cluster algebras, that the reduced stated ${\rm SL}_2$-skein algebra of the extended pb surface $\fS^\ast$ is obtained from the stated ${\rm SL}_2$-skein algebra of $\fS$ by first localizing at all frozen variables and then adjoining an equal number of new invertible frozen variables.
\end{remark}

\section{Polygon case}\label{sec-polygon}

Let $\mathcal{A}$ be an $\mathbb Z[\omega^{\pm \frac{1}{2}}]$-algebra with a basis $B$. We call $B$ a \emph{positive basis} if, for any $b, b' \in B$, the product $bb'$ can be expressed as a linear combination of elements in $B$ with coefficients in $\mathbb{Z}_{\ge 0}[\omega^{\pm \frac{1}{2}}]$. 

Consider the $(k+2)$-gon $\mathbb{P}_{k+2}$, whose boundary punctures are labeled $1, 2, \dots, k+2$ in clockwise order. 
Equation~\eqref{eq-inclusion-A-s} implies that all frozen variables in ${\mathscr A}_\omega(\mathbb P_{k+2})$ (see \S\ref{sec-mutation-quantum}) lie in ${\cS}_\omega(\mathbb{P}_{k+2})$. We denote by $\mathcal{F}$ the multiplicative subset of ${\cS}_\omega(\mathbb{P}_{k+2})$ generated by all frozen variables. Recall that ${\cS}_\omega^{\rm fr}(\mathbb{P}_{k+2})$ denotes the localization of ${\cS}_\omega(\mathbb{P}_{k+2})$ at the multiplicative subset $\mathcal{F}$.

The two main results of this section are summarized in the following theorems. They establish quantum cluster algebra structures on both the reduced stated ${\rm SL}_n$-skein algebra $\overline{\cS}_{\omega}(\mathbb{P}_{k+2})$ and the localized algebra ${\cS}_\omega^{\rm fr}(\mathbb{P}_{k+2})$. Furthermore, they show that the theta basis forms a rotation-invariant, positive basis for each of these algebras.

Note that Theorem~\ref{thm-trace-A}(c) implies 
$\widetilde\cS_{\omega}(\mathbb P_{k+2})=\overline\cS_{\omega}(\mathbb P_{k+2})$

\begin{theorem}\label{thm:poly1} For any $k\geq 1$,
    we have $\overline\cS_{\omega}(\mathbb P_{k+2})=\overline{\mathscr A}_\omega(\mathbb P_{k+2})=\overline{\mathscr U}_\omega(\mathbb P_{k+2})$ (see \eqref{eq-def-OA-cluster}). Moreover, the theta basis of $\overline{\mathscr U}_\omega(\mathbb P_{k+2})$, introduced in \cite{GHKK,davison2021strong}, is rotation-invariant and forms a positive basis for $\overline\cS_{\omega}(\mathbb P_{k+2})$.
\end{theorem}

\begin{theorem}\label{thm:poly2} For any $k\geq 0$, 
    we have $\cS_{\omega}^{\rm fr}(\mathbb P_{k+2})={\mathscr A}_\omega^{\rm fr}(\mathbb P_{k+2})={\mathscr U}_\omega^{\rm fr}(\mathbb P_{k+2})$ (Definition~\ref{def-sln-quantum-cluster-al}). Moreover, the theta basis of ${\mathscr U}_\omega^{\rm fr}(\mathbb P_{k+2})$, introduced in \cite{GHKK,davison2021strong}, is rotation-invariant and forms a positive basis for $\cS_{\omega}^{\rm fr}(\mathbb P_{k+2})$.
\end{theorem}

In the rest of this section, we prove Theorems 
\ref{thm:poly1} and \ref{thm:poly2}.

\vspace{1.5mm}

Let $\lambda$ be the star-like triangulation centered at vertex $1$, i.e., $\lambda=\{(1,i)\mid i=3,\cdots,k+1\}$, where $(1,i)$ is the obvious ideal arc in $\mathbb P_{k+2}$ connecting the vertices labeled by $1$ and $i$. Denote $\Delta_i$ the $i$-th triangle $(1,i+1,i+2)$ in $\lambda$. For each $i$, we label the vertex of $\overline{V}_\lambda$ inside $\Delta_i$ by $10^i,11^i,20^i,21^i,22^i,\cdots j0^i,j1^i,\cdots jj^i,\cdots$, as illustrated in Figure~\ref{Fig:lambdai} (where the top puncture, i.e. the one labeled by $i$, in Figure~\ref{Fig:lambdai} is the one labeled by $1$, and the superscripts in $js^i$ are omitted). In particular, we have $jj^i=j0^{i+1}$. We denote the cluster variable of the seed $\overline {\mathbf s}_\lambda$ associated with $\lambda$ of $\overline\cS_{\omega}(\mathbb P_{k+2})$ at $js^i$ by $\overline A_{js}{(i)}$. 


\subsection{A quasi-isomorphism}\label{sub-quasi-polygon}
For any integers $j, s$ satisfying $1 \leq j \leq n$ and $0 \leq s \leq j$, and for any integer $i$ with $1 \leq i \leq k$,
let
\begin{equation}\label{eq:qc}
    \overline A_{js}^{\rm qc}(i)=\begin{cases}
        \overline A_{js}(i) & \mbox{if $j=n$} \\
        \overline A_{js}(i) & \mbox{if $s=0$ and $i=1$} \\
        \bigl[\overline A_{js}(i)\cdot \overline A^{-1}_{ns} (i)\cdot \overline A^{-1}_{j0}(1)\cdot \prod_{t=1}^{i-1}(\overline A_{n,n-s-t+1}(i-t)\cdot \overline A^{-1}_{n,j-s-t+1}(i-t))\bigr] & \mbox{otherwise}
    \end{cases}
\end{equation}
with the convention that $\overline A_{n\ell}=1$ whenever $\ell\leq 0$ or $\ell\geq n$. Note that $\overline A_{js}^{\rm qc}(i)$ is well-defined (in particular, there is no ambiguity), as it satisfies the identity 
$$\overline A_{jj}^{\rm qc}(i)=\overline A_{j0}^{\rm qc}(i+1).$$

Let $\overline Q_\lambda^{\rm qc}$ be the weighted quiver obtained from $\overline Q_\lambda$ (see \eqref{eq-def-Q-lambda-re}) by removing all arrows incident to $js^i$ either when $j=n$, or when $s=0$ and $i=1$. Here we identify the weighted quiver with its signed adjacency matrix.

It is straightforward to verify that for any $js^i,j's'^{\;i'}\in \overline{V}_\lambda$, the elements $\overline A_{js}^{\rm qc}(i)$ and $\overline A_{j's'}^{\rm qc}(i')$ are quasi-commutative (Definition~\ref{def-qusi-com}), denote by 
$\overline \Pi^{\rm qc}({js^i},{j's'}^{\;i'})$ the integer satisfying 
\begin{equation*}
    \overline A_{js}^{\rm qc}(i)\overline A_{j's'}^{\rm qc}(i')=\xi^{\overline \Pi^{\rm qc}({js}^i,{j's'^{\; i'}})}\overline A_{j's'}^{\rm qc}(i')\overline A_{js}^{\rm qc}(i).
\end{equation*}

Recall that $\mathring{\overline V}_\lambda$ denotes the subset of $\overline{V}_\lambda$ consisting of vertices contained in the interior of $\mathbb P_{k+2}$.
For each $v=js^i\in\overline V_\lambda$, we use $\overline A^{\rm qc}_v$
(resp. $\overline A_v$) to denote  $\overline A^{\rm qc}_{js}(i)$ (resp. $\overline A_{js}(i)$).
The following lemma is immediate from \eqref{eq:qc}.

\begin{lemma}\label{lem:qc3}
    For any $u\in \mathring {\overline{V}}_\lambda$, we have
    \begin{equation}
        [\prod_{v\in \overline V_\lambda} (\overline A^{\rm qc}_{v})^{\overline Q^{\rm qc}_\lambda(v,u)}]=[\prod_{v\in \overline V_\lambda} (\overline A_{v})^{\overline Q_\lambda(v,u)}].
    \end{equation}
\end{lemma}

The following result is an immediate consequence of Lemma \ref{lem:qc3}.

\begin{lemma}\label{lem:compatible}
Let $u\in \mathring{\overline{V}}_\lambda,v\in \overline V_\lambda$. Then
   \begin{equation}\label{eq:qc1}
        \sum_{v_1\in\overline V_{\lambda}} \overline Q^{\rm qc}_\lambda(v_1,u) \Pi(\overline A^{\rm qc}_{v_1},\overline A_v)=\sum_{v_1\in\overline V_{\lambda}} \overline Q_\lambda(v_1,u) \Pi(\overline A_{v_1},\overline A_v)
        \quad (\text{see  Definition~\ref{def-qusi-com} for $\Pi(-,-)$}).
    \end{equation} 
    Moreover, we have
  \begin{equation}
        \sum_{v_1\in\overline V_{\lambda}} \overline Q^{\rm qc}_\lambda(v_1,u)\overline \Pi^{\rm qc}_{\lambda}(v_1,v)=\sum_{v_1\in\overline V_{\lambda}} \overline Q_\lambda(v_1,u)\overline \Pi_{\lambda}(v_1,v)=2n\delta_{u,v},
    \end{equation}   
  where $\overline \Pi_\lambda$ is defined in \eqref{eq-prod-Qpi}.
\end{lemma}

\begin{proof}
First, by direct calculation and the properties of the pairing $\Pi$, we obtain
\begin{equation}
 \begin{split}
  \sum_{v_1\in\overline V_{\lambda}} \overline Q^{\rm qc}_\lambda(v_1,u) \Pi(\overline A^{\rm qc}_{v_1},\overline A_v) 
  &= \Pi \left(\prod_{v_1\in\overline V_{\lambda}}(\overline A^{\rm qc}_{v_1})^{ \overline Q^{\rm qc}_\lambda(v_1,u)},\overline A_v\right) \\
  &=  \Pi \left(\prod_{v_1\in \overline V_{\lambda}}(\overline A_{v_1})^{ \overline Q_\lambda(v_1,u)},\overline A_v\right)
  \quad \text{(by Lemma~\ref{lem:qc3})}\\
  &= \sum_{v_1\in \overline V_{\lambda}} \overline Q_\lambda(v_1,u) \Pi(\overline A_{v_1},\overline A_v).
   \end{split}
\end{equation} 
This proofs \eqref{eq:qc1}.

It is shown in \cite{huang2025quantum} that
\begin{equation}\label{eq:qc2}
     \sum_{v_1\in \overline V_{\lambda}} \overline Q_\lambda(v_1,u)\overline \Pi_{\lambda}({v_1},v)=2n\delta_{u,v}.
\end{equation}
Combining \eqref{eq:qc1} and \eqref{eq:qc2} yields
\begin{equation}\label{eq:qc3}
    \sum_{v_1\in\overline V_{\lambda}} \overline Q^{\rm qc}_\lambda(v_1,u) \Pi(\overline A^{\rm qc}_{v_1},\overline A_v)=2n\delta_{u,v}.
\end{equation}

Finally, we compute
\begin{equation}
   \begin{split}
  \sum_{v_1\in\overline V_{\lambda}} \overline Q^{\rm qc}_\lambda(v_1,u)\overline \Pi^{\rm qc}_{\lambda}(v_1,v)
  &= \sum_{v_1\in \overline V_{\lambda}} \overline Q^{\rm qc}_\lambda(v_1,u)\overline \Pi_{\lambda}(\overline A^{\rm qc}_{v_1},\overline A^{\rm qc}_v)\\
  &= \sum_{v_1\in \overline V_{\lambda}} \overline Q^{\rm qc}_\lambda(v_1,u)\overline \Pi_{\lambda}(\overline A^{\rm qc}_{v_1},\overline A_v)
  \quad \text{(by \eqref{eq:qc} and \eqref{eq:qc3})}\\
  &= 2n\delta_{u,v} \quad \text{by \eqref{eq:qc3}}.
  \end{split}
\end{equation}
This completes the proof.
\end{proof}

Define 
\begin{equation*}
\overline M^{\rm qc}_\lambda:\mathbb Z^{\overline V_\lambda}\to {\rm Frac}(\overline\cS_{\omega}(\mathbb P_{k+2})), \qquad 
    {\bf k}=(k_v)_{v\in \overline{V}_{\lambda}}\longmapsto 
\left[\prod_{v\in \overline V_\lambda} (\overline A^{\rm qc}_{v})^{k_v}\right].    
\end{equation*}

\begin{corollary}
The triple $\overline{\mathsf s}^{\rm qc}_\lambda=(\overline Q^{\rm qc}_\lambda, \overline{\Pi}^{\rm qc}_\lambda,  \overline M^{\rm qc}_\lambda)$ is a quantum seed in ${\rm Frac}(\overline\cS_{\omega}(\mathbb P_{k+2}))$.
\end{corollary}

\begin{proof}
    It follows immediately by Lemma \ref{lem:compatible}.
\end{proof}




Let $\overline {\mathscr A}_{\omega}^{\rm qc}(\mathbb P_{k+2})$ denote the quantum cluster algebra associated with the seed
$\overline{\mathsf s}^{\rm qc}_\lambda=(\overline Q^{\rm qc}_\lambda, \overline{\Pi}^{\rm qc}_\lambda,  \overline M^{\rm qc}_\lambda)$.

\begin{proposition}\label{prop-quasi-polygon}
    The assignments $\overline A^{\rm qc}_{{js}}(i) \mapsto \overline A^{\rm qc}_{{js}}(i)$ induce a quasi-isomorphism (Definition \ref{def:QQH}) 
    \begin{equation}\label{eq:quasi-iso}
        f:\overline {\mathscr A}^{\rm qc}_{\omega}(\mathbb P_{k+2})=\mathscr A_{\overline{\mathsf s}^{\rm qc}_\lambda}\to \mathscr A_{\overline {\mathsf s}_\lambda}=\overline {\mathscr A}_{\omega}(\mathbb P_{k+2}).
    \end{equation}
    Moreover, we have $\overline {\mathscr A}^{\rm qc}_{\omega}(\mathbb P_{k+2})=\overline {\mathscr A}_{\omega}(\mathbb P_{k+2})$ as $R$-algebras.
\end{proposition}

\begin{proof}
Since $\overline A_{ns}(i)$, $\overline A_{j0}(1)$, $\overline A_{n,n-s-t+1}(i-t)$, and $\overline A_{n,j-s-t+1}(i-t)$ are  frozen variables, conditions (1) and (2) of Proposition \ref{prop:quasi-iso} follow directly from \eqref{eq:qc}. Furthermore, condition (3) is satisfied by Lemma \ref{lem:qc3}, while condition (4) holds trivially. This completes the proof.
\end{proof}

\begin{definition}\label{def:proportional}
Let $\mathscr A$ be a quantum cluster algebra. For two reflection invariant elements $x,x'\in \mathscr A$, we say that $x$ and $x'$ are \textbf{proportional}, denoted by $x\asymp x'$ if $x^{-1}x'\in R\langle A_v^{\pm 1}\mid v\in \mathcal V\setminus \mathcal{V}_{\rm mut}\rangle$.
\end{definition}

For each $2\leq i\leq k$, let $\lambda_i$ denote the triangulation obtained from $\lambda$ by flipping the arc $(1,i+1)\in \lambda$, and let $\overline {\mathsf{s}} _{\lambda_i}^{\rm qc}$ and ${\mathsf{s}} _{\lambda_i}$ be the associated quantum seeds in $\overline {\mathscr A}_{\omega}^{\rm qc}(\mathbb P_{k+2})$ and $\overline {\mathscr A}_{\omega}(\mathbb P_{k+2})$, respectively. 
Denote the vertices of the triangle $(i,i+1,i+2)$ by $10,11,20,21,22,\cdots$, as illustrated in Figure~\ref{Fig:lambdai}, and the quantum cluster variables of $\overline {\mathscr A}_{\omega}^{\rm qc}(\mathbb P_{k+2})$ corresponding to these vertices of $\overline {\mathsf{s}} _{\lambda_i}^{\rm qc}$ by 
$$\overline A^{\rm qc}_{10}\langle i\rangle, \overline A^{\rm qc}_{11}\langle i\rangle, \overline A^{\rm qc}_{10}\langle i\rangle, \overline A^{\rm qc}_{21}\langle i\rangle, \overline A^{\rm qc}_{22}\langle i\rangle,\cdots$$  respectively. For any $j,s$ with $1\leq s\leq j\leq n-1$, denote by $\overline A_{js}\langle i\rangle$ the quantum cluster variable of $\overline {\mathscr A}_{\omega}(\mathbb P_{k+2})$ corresponding to $\overline A_{js}^{\rm qc}\langle i\rangle$ via the quasi-isomorphism $f$ in \eqref{eq:quasi-iso}.

\begin{figure}[h]
    \centering
    \includegraphics[width=150pt]{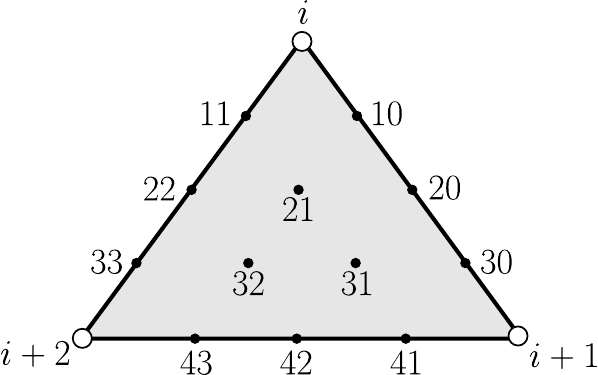}
    \caption{Labels of the vertices inside the triangle $(i,i+1,i+2)$ in $\lambda_i$ when $n=4$.}\label{Fig:lambdai}
\end{figure}

For any integers $j,s$ with $1 < s \leq j \leq n-1$, define
\begin{equation}\label{eq:mutationdiamond}
    \widetilde{\mu}_{(j,s)}^{\diamondsuit}
=    (\mu_{n-1,s-1}\cdots \mu_{j+1,s-1}\mu_{j,s-1})\cdots
    (\mu_{n-s+2,s-1}\cdots\mu_{j-s+4,2} \mu_{j-s+3,2})
    (\mu_{n-s+1,1}\cdots \mu_{j-s+3,1} \mu_{j-s+2,1}).
\end{equation}
See Figure~\ref{Fig:mu} for a pictorial illustration of $\widetilde{\mu}_{(j,s)}^{\diamondsuit}$.

\begin{figure}[h]
    \centering
    \includegraphics[width=120pt]{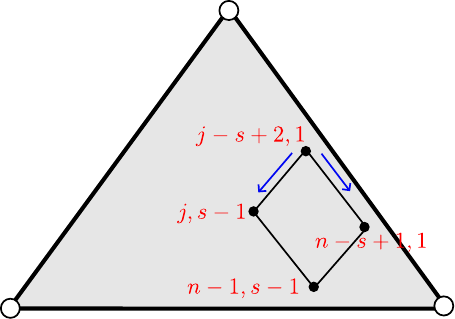}
    \caption{A pictorial illustration of $\widetilde{\mu}_{(j,s)}^{\diamondsuit}$}\label{Fig:mu}
\end{figure}

\begin{lemma}\label{lem:quasi2}
For any integers $j, s$ with $1 \leq s \leq j \leq n-1$ and any $i$ with $2 \leq i \leq k$, the following identities hold:

\begin{enumerate}[label={\rm (\alph*)}, itemsep=0.3em]

\item 
\begin{equation}
\overline{A}^{\mathrm{qc}}_{js}\langle i\rangle
=
\Bigl[
\overline{A}_{js}\langle i\rangle \cdot 
\overline{A}^{-1}_{ns}(i)\cdot \overline{A}^{-1}_{nj}(i-1) \cdot
\prod_{t=1}^{i-1} \bigl( \overline{A}_{n,n-s-t+1}(i-t) \cdot \overline{A}^{-1}_{n,j-s-t+1}(i-t) \bigr)
\cdot
\prod_{t=1}^{i-1} \bigl( \overline{A}^{-1}_{n,n-t+1}(i-t) \cdot \overline{A}_{n,j-t+1}(i-t) \bigr)
\Bigr].
\end{equation}

\item For $1<s$, the following equality holds for the seeds $\widetilde{\mu}_{(j,s)}^{\diamondsuit}(\overline {\mathsf{s}} _{\lambda}^{\rm qc})$ and $\widetilde{\mu}_{(j,s)}^{\diamondsuit}(\overline {\mathsf{s}} _{\lambda})$
\begin{equation}
\widetilde{\mu}_{(j,s)}^{\diamondsuit}\bigl(\overline{A}^{\mathrm{qc}}_{n-1,s-1}(1)\bigr)
=
\Bigl[
\widetilde{\mu}_{(j,s)}^{\diamondsuit}\bigl(\overline{A}_{n-1,s-1}(1)\bigr) \cdot 
\overline{A}^{-1}_{j-s+1,0}(1) \cdot 
\overline{A}^{-1}_{ns}(1)
\Bigr].
\end{equation}

\item For $1<s$, the following equality holds for the seeds $\widetilde{\mu}_{(j,s)}^{\diamondsuit}(\overline {\mathsf{s}} _{\lambda_i}^{\rm qc})$ and $\widetilde{\mu}_{(j,s)}^{\diamondsuit}(\overline {\mathsf{s}} _{\lambda_i})$
\begin{equation}
\begin{split}
\widetilde{\mu}_{(j,s)}^{\diamondsuit}\bigl(\overline{A}^{\mathrm{qc}}_{n-1,s-1}\langle i\rangle\bigr)
= \Bigl[ &
\widetilde{\mu}_{(j,s)}^{\diamondsuit}\bigl(\overline{A}_{n-1,s-1}\langle i\rangle\bigr) \cdot 
\overline{A}^{-1}_{n,j-s+1}(i) \cdot 
\overline{A}^{-1}_{ns}(i) \\
& \cdot 
\prod_{t=1}^{i-1} \bigl( \overline{A}_{n,n-s-t+1}(i-t) \cdot \overline{A}^{-1}_{n,j-s-t+1}(i-t) \bigr) \\
& \cdot 
\prod_{t=1}^{i-1} \bigl( \overline{A}_{n,n-s-t+2}(i-t) \cdot \overline{A}^{-1}_{n,j-s-t+2}(i-t) \bigr)
\Bigr].
\end{split}
\end{equation}

\end{enumerate}
\end{lemma}

We left the proof of Lemma \ref{lem:quasi2} to Appendix \ref{app:lemmas1234}.


\subsection{Standard cluster variables for $\overline {\mathscr A}_{\omega}^{\rm qc}(\mathbb P_{k+2})$ }\label{sec:standard var}

For the triangle $\Delta_i=(1,i+1,i+2)$, let 
\begin{equation}\label{eq:Ii}
    I(\Delta_i):=\mathring{\overline V}_{\Delta_i}\sqcup\{11^i,22^i,\ldots, (n-1,n-1)^i\}
\end{equation}
be the set of small vertices in $\Delta_i$ which excludes the frozen vertices on the bottom and right edges of $\Delta_i$. We define a linear order on $I(\Delta_i)$ as follows: \[\text{$js\prec j's'$ in $I(\Delta_i)$ iff either $s<s'$, or $s=s'$  and $j> j'$}.\] 
For example, in Figure \ref{Fig:lambdai}, we have $n=4$ and 
\[I(\Delta_i)=\{31\prec 21\prec 11\prec 32 \prec 22 \prec 33\}.\]
Notice that the sequence $(3,2,1,3,2,3)$ of row indices of $(31,21,11,32,22,33)$ yields a reduced word for the longest element $w_0$ of the Weyl group of ${\rm SL}_4$. 

For the star-like triangulation $\lambda=\cup_{i=1}^k \Delta_i$ of $\mathbb P_{k+2}$, let \[I(\lambda):=\sqcup_{i=1}^k I(\Delta_i)=\mathring{\overline V}_{\lambda}\cup \{11^k,22^k,\ldots,(n-1,n-1)^k\}.\]
Note that $I(\Delta_i)\cap I(\Delta_i') =\emptyset$ for $i\neq i'$.
Then we extend the linear order on each $I(\Delta_i)$ for $1\leq i\leq k$ to a linear order on $I(\lambda)$ as follows: for any $x\in I(i')$ and $y\in I(i)$ with $i'<i$, we set $x\prec y$.

\begin{remark}
Observe that the minimal element in $I(\lambda)$ is $(n-1,1)^1$
 and the maximal element in $I(\lambda)$ is $(n-1,n-1)^k$. 
\end{remark}

We observe that the vertices in $I(\lambda)$ are arranged into $n-1$ rows. The vertices in the $j$-th row are given by the sequence
\[j1^1,\ldots,jj^1, j1^2,\ldots,jj^2,\ldots,j1^k,\ldots,jj^k,\]
which form a full subquiver as follows:
\[
\xymatrixrowsep{5mm}
\xymatrixcolsep{5mm}
\xymatrix{
 \fbox{$jj^k$}&&(j,j-1)^k\ar[ll] && \;\ldots\ldots\;\ar[ll]&&j2^1\ar[ll]&&j1^1\ar[ll] 
}
\]

For a vertex $js^i\in I(\lambda)$ in the $j$-th row, we introduce the following notations:
\begin{itemize}
    \item $d(v_1,v_2)$: the length of the path between the vertices $v_1$ and $v_2$ in the above quiver.
    \item $(js^i)^*$: the unique vertex in the above (ice) quiver satisfying 
    $d(jj^k, (js^i)^*)=d(j1^1,js^i)+1,$ 
    e.g., $(j1^1)^*=(j,j-1)^k$. By convention, we set $(jj^k)^*=\emptyset$ for the frozen vertex $jj^k$.  
    \item $\overset{\leftarrow}{\mu}_{js^i}$: the sequence of mutations defined by mutating the above (ice) quiver from the vertex $j1^1$ to  $(js^i)^*$ from right to left, i.e.,
    $ \overset{\leftarrow}{\mu}_{js^i}:=\mu_{(js^i)^*}\cdots \mu_{j2^1}\mu_{j1^1}.$
    By convention, we set  $\overset{\leftarrow}{\mu}_{jj^k}=\mathrm{id}$ for the frozen vertex $jj^k$.
\end{itemize}

Thus, we obtain a collection $\{\overset{\leftarrow}{\mu}_{js^i} \mid js^i \in I(\lambda)\}$ of mutation sequences. Furthermore, certain compositions of these mutation sequences will play a crucial role in the subsequent analysis.

Now we fix a vertex $js^i$ in the linearly ordered set $I(\lambda)=(I(\lambda),\prec)$ and define
\begin{equation}\label{eq:mustandard}
    \overset{\leftarrow}{\mu}_{\prec js^{i}}:=\prod_{\ell t^{i'}\prec js^{i}}\overset{\leftarrow}{\mu}_{\ell t^{i'}},
\end{equation}
where the factors in the above composition are ordered according to the increasing order of the indices with respect to $\prec$ from right to left.

For any $1\leq i'<k$, denote 
\begin{equation}\label{eq:mustandard1}
    \overset{\leftarrow}\mu(\Delta_{i'}):=\prod_{v\in I(\Delta_{i'})}\overset{\leftarrow}\mu_v
\end{equation}
with the factors are ordered according to the increasing order of the indices with respect to $\prec$ from right to left.

Thus, we have
\begin{equation}\label{eq:muleftarrow}
   \overset{\leftarrow}{\mu}_{\prec js^{i}}=\bigl(\prod_{\Delta_{i}\ni \ell t^{i}\; \prec js^{i}}\overset{\leftarrow}{\mu}_{\ell t^{i}}\bigr) \overset{\leftarrow}\mu(\Delta_{i-1})\cdots\overset{\leftarrow}\mu(\Delta_{2})\overset{\leftarrow}\mu(\Delta_{1}) 
\end{equation}


Recall that the vertices of $\mathbb P_{k+2}$ are labeled $1, \dots, k+2$, with indices regarded modulo $k+2$. Let $a, b$ be integers such that $a \not\equiv b \pmod{k+2}$, and let $1 \le i, j \le n$. There exists a unique properly embedded oriented arc $D^{(a,b)}$ (or simply $D^{ab}$) in $\mathbb P_{k+2}$, up to isotopy, starting from a interior point of the boundary edge $(a, a+1)$ to a interior point of the boundary edge $(b, b+1)$. We denote by $D^{(a,b)}_{(p,q)}$ (or simply $D^{ab}_{pq}$) the stated essential arc obtained by assigning state $p$ to the endpoint on edge $(a, a+1)$ and state $q$ to the endpoint on edge $(b, b+1)$. In particular, if $b=a+1$, then $D^{(a,b)}_{(p,q)}$ is a corner arc; see Figure~\ref{Fig:Pn-Dij}.

\begin{figure}[h]
    \centering
    \includegraphics[width=80pt]{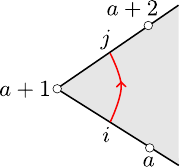}
    \caption{A pictorial illustration of $D_{ij}^{a,a+1}$.}\label{Fig:Pn-Dij}
\end{figure}

For a vertex $js^i\in I(\lambda)$, we introduce the following notations:
\begin{itemize}
    \item $\overline E^{\rm qc}(js^i)$: the cluster variable in the seed $\overset{\leftarrow}{\mu}_{\prec js^{i}}(\overline{\mathsf s}^{\rm qc}_\lambda)$ of $\overline {\mathscr A}^{\rm qc}_{\omega}(\mathbb P_{k+2})$ located at the vertex $j1^1\in I(\lambda)$.
    \item $\overline E(js^i)$: the cluster variable in the seed $\overset{\leftarrow}{\mu}_{\prec js^{i}}(\overline{\mathsf s}_\lambda)$ of $\overline {\mathscr A}_{\omega}(\mathbb P_{k+2})$ located at the vertex $j1^1\in I(\lambda)$.
    \item $D(js^i)$: the stated essential arc in $\overline{\mathscr{S}}_{\omega}(\mathbb{P}_{k+2})$ defined by $D(js^i):=D^{(a,b)}_{(p,q)}$, where the pairs $(a,b)$ and $(p,q)$ are determined by $js^i$ as follows:
\begin{equation}\label{eqn:D-ab-pq}
  (a,b):=\begin{cases}
    (i,i+2), & s=1,\\
    (i,i+1), & s>1,
  \end{cases} \quad \text{and} \quad (p,q):=\begin{cases}
      (j,1), & s=1,\\
      (j-s+1, \;n-s+1), & s>1.
  \end{cases}
\end{equation}
\end{itemize}

In particular, we have 
\begin{equation}\label{eq:Ej0}
\overline E^{\rm qc}(j1^1)=\overline A^{\rm qc}_{j1}(1), \quad \overline E^{\rm qc}(j1^i)=\overset{\leftarrow}\mu(\Delta_{i-1})\cdots\overset{\leftarrow}\mu(\Delta_{2})\overset{\leftarrow}\mu(\Delta_{1})(\overline A^{\rm qc}_{j1}(1)), 
\end{equation}
\begin{equation}\label{eq:Ej1}
  \overline E(j1^1)=\overline A_{j1}(1),\quad \overline E(j1^i)=\overset{\leftarrow}\mu(\Delta_{i-1})\cdots\overset{\leftarrow}\mu(\Delta_{2})\overset{\leftarrow}\mu(\Delta_{1})(\overline A_{j1}(1)).  
\end{equation}

The cluster variables in the sets $\{ \overline E^{\rm qc}(js^i)\mid js^i\in I(\lambda)\}$ and $\{ \overline E(js^i)\mid js^i\in I(\lambda)\}$ are referred to as the \textbf{standard (cluster) variables} of the cluster algebras $\overline {\mathscr A}^{\rm qc}_{\omega}(\mathbb P_{k+2})$ and $\overline {\mathscr A}_{\omega}(\mathbb P_{k+2})$, respectively.

\begin{remark}
By Proposition \ref{prop-quasi-polygon} and Theorem \ref{intro-thm-skein-inclusion-A}, one has 
$\overline{\mathscr{S}}_{\omega}(\mathbb{P}_{k+2})\subseteq \overline{\mathscr{A}}_{\omega}(\mathbb{P}_{k+2})
=\overline{\mathscr{A}}^{\rm qc}_{\omega}(\mathbb{P}_{k+2}).$
Furthermore, every cluster variable in $\overline{\mathscr{A}}^{\rm qc}_{\omega}(\mathbb{P}_{k+2})$ is proportional to a cluster variable of $\overline{\mathscr{A}}_{\omega}(\mathbb{P}_{k+2})$ in the sense of Definition \ref{def:proportional}. In addition, each stated essential arc representing an element of $\overline{\mathscr{S}}_{\omega}(\mathbb{P}_{k+2})$ is also proportional to a cluster variable in $\overline{\mathscr{A}}_{\omega}(\mathbb{P}_{k+2})$.
\end{remark}

\medskip

In the remainder of this section, we prove the equality
\[
\overline{\mathscr{S}}_{\omega}(\mathbb{P}_{k+2}) = \overline{\mathscr{A}}_{\omega}(\mathbb{P}_{k+2}).
\]
In light of the preceding remark, it suffices to establish the reverse inclusion $\overline{\mathscr{A}}_{\omega}(\mathbb{P}_{k+2}) \subseteq \overline{\mathscr{S}}_{\omega}(\mathbb{P}_{k+2})$. The key steps of the proof are outlined below.

\textbf{Step (i):} We first demonstrate that the cluster algebra $\overline{\mathscr{A}}_{\omega}(\mathbb{P}_{k+2})$ is generated by the standard elements $\{\overline{E}(js^i) \mid js^i\in I(\lambda)\}$, along with the frozen variables and their inverses. This follows directly from the analogous result for $\overline{\mathscr{A}}^{\rm qc}_{\omega}(\mathbb P_{k+2})$ (see Proposition \ref{pro:standard}).

\textbf{Step (ii):} We then verify that $\overline{E}(js^i)\asymp D(js^i) \in \overline{\mathscr{S}}_{\omega}(\mathbb{P}_{k+2})$ (or equivalently, $\overline{E}^{\rm qc}(js^i) \asymp D(js^i)$) for all $js^i\in I(\lambda)$; see Theorem \ref{thm:standardgenerators}. 

To execute Step (ii), recall that for $2\leq i\leq k$, 
$\lambda_i$ denotes the triangulation obtained from $\lambda$ by flipping the arc $(1,i+1)\in \lambda$, and  $\overline {\mathsf{s}} _{\lambda_i}^{\rm qc}$ denotes the quantum seed of $\overline {\mathscr A}_{\omega}^{\rm qc}(\mathbb P_{k+2})$ associated with the triangulation $\lambda_i$. We will also consider the quantum seed  
    \begin{equation}\label{eq:ti}
        \overline {\mathsf{t}} _{i}^{\rm qc}:=\overset{\leftarrow}\mu(\Delta_{i-1})\cdots\overset{\leftarrow}\mu(\Delta_{2})\overset{\leftarrow}\mu(\Delta_{1})(\overline {\mathsf{s}} _{\lambda}^{\rm qc})
    \end{equation}
    in $\overline {\mathscr A}_{\omega}^{\rm qc}(\mathbb P_{k+2})$. By convention, for $i=1$, we set $\overline {\mathsf{s}} _{\lambda_1}^{\rm qc}=\overline {\mathsf{s}} _{\lambda}^{\rm qc}=\overline {\mathsf{t}} _{1}^{\rm qc}$. 

Let $\Delta_i'$ denote the triangle in the triangulation $\lambda_i$ with vertices $(i,i+1,i+2)$, and define
$$I(\Delta_i'):=\mathring{\overline V}_{\Delta_i'}\sqcup\{11,22,\ldots, (n-1,n-1)\}\quad \text{(see Figure~\ref{Fig:lambdai}).}$$ 

The argument relies on the following key observations:

\begin{enumerate}
    \item The seeds $\overline {\mathsf{s}} _{\lambda_i}^{\rm qc}$ and $\overline {\mathsf{t}} _{i}^{\rm qc}$ share a common sub-seed (see Definition \ref{def:subseed}), which we denote by $\overline {\mathsf{t}} _{i}^{\rm sub}$ (see Lemma \ref{lem:B-d}). This sub-seed is supported on the vertex sets $I(\Delta_i')$ and $I(\Delta_1)$, respectively, with the vertex sets $\{11,22,\cdots, (n-1,n-1)\}$ and $\{11^1,22^1,\cdots, (n-1,n-1)^1\}$, respectively, frozen.
    \item Both the standard variable $\overline E^{\rm qc}(js^i)$ and the stated essential arc $D(js^i)$ can be derived from the sub-seed $\overline {\mathsf{t}} _{i}^{\rm sub}$ via specific, explicit mutation sequences (see Lemmas \ref{lem:B-d}, \ref{lem:B4}, and \cite[Theorem~1.2]{huang2025quantum}).
    \item Lemmas \ref{lem:B-d} and \ref{lem:B4} are established through the direct computation of $g$-vectors, a process made feasible by the explicit nature of the mutation sequences involved.
\end{enumerate}

\medskip

\medskip

It is shown in \cite[Theorem 8.10(1),(3)]{Q2024} that the $R$-algebra
$\overline{\mathscr A}^{\mathrm{qc}}_{\omega}(\mathbb P_{k+2})
=
\overline{\mathscr U}^{\mathrm{qc}}_{\omega}(\mathbb P_{k+2})$
is generated by the standard cluster variables,  frozen variables and the inverse of frozen variables  in 
$\overline{\mathscr A}_{\omega}(\mathbb P_{k+2})$. 
Proposition~\ref{prop-quasi-polygon} then implies the following.

\begin{proposition}\label{pro:standard}
For any integer $k\geq 1$, the quantum cluster algebra $\overline{\mathscr{A}}_{\omega}(\mathbb{P}_{k+2})$ coincides with the quantum upper cluster algebra $\overline{\mathscr{U}}_{\omega}(\mathbb{P}_{k+2})$.
Furthermore, as an $R$-algebra, it is generated by
\begin{itemize}
    \item the standard cluster variables $\overline{E}(js^i)$ for all $js^i\in I(\lambda)$, and
    \item the frozen variables and their inverses: $\overline{A}^{\pm1}_{j0}(1)$, $\overline{A}^{\pm1}_{ns}(i)$, and $\overline{A}^{\pm1}_{jj}(k)$,
          where $j,s\in \{1,\dots,n-1\}$ and $i\in \{1,\dots,k\}$.
\end{itemize}
\end{proposition}

\begin{lemma}\label{lem:B-d}
For each integer $i$ with $2 \le i \le k$, the following assertions hold:

\begin{enumerate}[label={\rm (\alph*)}]\itemsep0.3em
\item In the quiver of $\overline{\mathsf{s}}_{\lambda_i}^{\rm qc}$, there are no arrows connecting vertices in $\mathring{\overline{V}}_{\Delta_i'}$ to those outside \[
I(\Delta_i'):= \mathring{\overline{V}}_{\Delta_i'} \sqcup \big\{11, \dots, (n-1,n-1)\big\}.
\] Consequently, we can define a sub-seed $\overline{\mathsf{s}}_{i}^{\rm sub}$ of $\overline{\mathsf{s}}_{\lambda_i}^{\rm qc}$ supported on
$I(\Delta_i')$
with the vertex set $\{11,22,\dots,(n-1,n-1)\}$ frozen.

\item In the quiver of $\overline{\mathsf{t}}_{i}^{\rm qc}$, there are no arrows connecting vertices in $\mathring{\overline{V}}_{\Delta_1}$ to those outside $I(\Delta_1)$. 
Consequently, we can define a sub-seed $\overline{\mathsf{t}}_{i}^{\rm sub}$ of $\overline{\mathsf{t}}_{i}^{\rm qc}$ supported on $I(\Delta_1)$, with the vertex set $\{11^1,\dots, (n-1,n-1)^1\}$ frozen.

\item Under the vertex identification $js^1 \mapsto js$ from $I(\Delta_1)$ to $I(\Delta_i')$, we have $\overline{\mathsf{t}}_{i}^{\rm sub} = \overline{\mathsf{s}}_{i}^{\rm sub}$. Equivalently, these two sub-seeds share identical cluster variables and quivers. In particular, the cluster variables of $\overline{\mathsf{t}}_{i}^{\rm sub}$ are precisely
\[
\big\{\,\overline{A}^{\rm qc}_{js}\langle i\rangle \,\big|\, js \in I(\Delta_i')\,\big\},
\]
and by \eqref{eq:Ej0}, we have 
\[
\overline{E}^{\rm qc}(j1^i) = \overline{A}^{\rm qc}_{j1}\langle i\rangle.
\]
\end{enumerate}
\end{lemma}

The proof of Lemma \ref{lem:B-d} is deferred to Appendix \ref{app:lemmas1234}.


Note that the $j$-th row of the quiver of the sub-seed $\overline{\mathsf{t}}_{i}^{\rm sub}$, which is supported on the vertex set $\{j1^1,j2^1,\ldots,jj^1\}$, is given below:
\[
\xymatrixrowsep{5mm}
\xymatrixcolsep{5mm}
\xymatrix{
jj^1 &&(j,j-1)^1\ar[ll] && \;\ldots\ldots\;\ar[ll]&&j2^1\ar[ll]&&j1^1\ar[ll]
}
\]

We now introduce two families of mutation sequences associated to the sub-seed $\overline{\mathsf{t}}_{i}^{\rm sub}$:
\begin{itemize}
    \item For all integers $d,j$ satisfying $1 \le d < j\leq n-1$, define
    $\mu^1_{(j,d)}:=\mu_{jd^1}\cdots \mu_{j2^1}\mu_{j1^1}.$
    \item For every integer $\ell$ with $1\leq \ell\leq n-2$, set
    $\mu_{\ell}^{\rm up}:=\mu^1_{(\ell+1,1)}\cdots \mu^1_{(n-2,n-2-\ell)} \mu^1_{(n-1,n-1-\ell)}.$
\end{itemize}

\begin{lemma}\label{lem:B4}
For any $js^i\in I(\lambda)$ with $s\geq 2$, the standard variable $\overline{E}^{\rm qc}(js^i)$ in $\overline{\mathscr{A}}^{\rm qc}_{\omega}(\mathbb{P}_{k+2})$ can be obtained from the sub-seed $\overline{\mathsf{t}}_{i}^{\rm sub}$ in the following two ways:
\begin{enumerate}[label={\rm (\alph*)}]\itemsep0.3em
\item The standard variable $\overline{E}^{\rm qc}(js^i)$ coincides with the cluster variable at vertex $j1^1\in I(\Delta_1)$ in the seed
$\mu_{s-1}^{\rm up}\cdots\mu_2^{\rm up}\mu_{1}^{\rm up}(\overline{\mathsf{t}}_{i}^{\rm sub})$.
In particular, with respect to the initial seed $\overline{\mathsf{s}}_{\lambda}^{\rm qc}$, we have
\begin{equation}
\overline{E}^{\rm qc}(js^i) = \mu_{s-1}^{\rm up}\cdots\mu_2^{\rm up}\mu_{1}^{\rm up}\circ \overleftarrow{\mu}(\Delta_{i-1})\cdots\overleftarrow{\mu}(\Delta_{2})\overleftarrow{\mu}(\Delta_{1})(\overline{A}^{\rm qc}_{j1}(1)).
\end{equation}

\item The standard variable $\overline{E}^{\rm qc}(js^i)$ coincides with the cluster variable at vertex $(n-1,s-1)^1\in I(\Delta_1)$ in the seed
$\widetilde{\mu}_{(j,s)}^{\diamondsuit}(\overline{\mathsf{t}}_{i}^{\rm sub})$,
where $\widetilde{\mu}_{(j,s)}^{\diamondsuit}$ is defined in \eqref{eq:mutationdiamond}.
In particular, with respect to the seed $\overline{\mathsf{s}}_{\lambda_i}^{\rm qc}$ containing $\overline{\mathsf{t}}_{i}^{\rm sub}$ as a sub-seed, we have
\[
\overline{E}^{\rm qc}(js^i) = \widetilde{\mu}_{(j,s)}^{\diamondsuit}(\overline{A}^{\rm qc}_{n-1,s-1}\langle i\rangle).
\]
\end{enumerate}
\end{lemma}

\begin{remark}
In the lemma above, we focus on the case $s\geq 2$. The case $s=1$ has been addressed in Lemma \ref{lem:B-d}.
\end{remark}

We left the proof of Lemma \ref{lem:B4} to Appendix \ref{app:lemmas1234}.

\smallskip

The following theorem is the third main result of this section. It gives a skein-theoretic realization of the standard cluster variables. More precisely, every standard cluster variable is proportional to a stated essential arc in $\overline{\cS}_\omega(\mathbb{P}_{k+2})$. This theorem also serves as a critical final step in the proofs of Theorems~\ref{thm:poly1} and~\ref{thm:poly2}.

\begin{theorem}\label{thm:standardgenerators}
For any $js^i\in I(\lambda)$, we have $\overline E^{\rm qc}(js^i)\asymp D(js^i)\in \overline{\cS}_\omega(\mathbb P_{k+2})$. More precisely, for any integer $i$ with $1\leq i\leq k$ and any integers $s,j$ with $1\leq s\leq j\leq n-1$, the following statements hold.

\begin{enumerate}[label={\rm (\alph*)}]\itemsep0.3em
    \item For $i=1$, we have  \begin{equation*}
   \overline E^{\rm qc}(js^1)
  =\begin{cases}
      \Bigl[D(js^1)\cdot  \overline A_{n,n-1}(2)\cdot \overline A^{-1}_{n1}(1)\Bigr] & \mbox{ if $s=1$},\vspace{1pt}\\
      \Bigl[D(js^1)\cdot \overline A_{n,s-1}(1)\cdot \overline A_{ns}^{-1}(1) \Bigr]   & \mbox{ if $s\geq 2$.} 
  \end{cases}
    \end{equation*}


\item For $2\leq i\leq k$, we have two sub-cases: 

$\bullet$ If $s=1$, then we have 
\begin{equation*}
   \begin{aligned}
 \overline E^{\rm qc}(j1^i)=
\begin{cases}
\Bigl[
D(j1^i)\cdot \overline A^{-1}_{n1}(i)\cdot \overline A_{n,n-1}(i+1)
\cdot \prod_{t=1}^{i-1} \bigl(
\overline{A}_{n,n-t}(i-t)\cdot \overline{A}^{-1}_{n,j-t}(i-t) \\
\qquad\qquad\qquad\qquad\cdot \overline{A}^{-1}_{n,n-t+1}(i-t)\cdot \overline{A}_{n,j-t+1}(i-t)
\bigr)
\Bigr], & i<k, \\[6pt]
\Bigl[ D(j1^i)\cdot \overline A^{-1}_{n1}(k)\cdot \overline A_{11}(k)
\cdot \prod_{t=1}^{k-1} \bigl(
\overline{A}_{n,n-t}(k-t)\cdot \overline{A}^{-1}_{n,j-t}(k-t) \\
\qquad\qquad\qquad\qquad\cdot \overline{A}^{-1}_{n,n-t+1}(k-t)\cdot \overline{A}_{n,j-t+1}(k-t)
\bigr) \Bigr], & i=k.
\end{cases}
\end{aligned}
\end{equation*}

$\bullet$ If $s\geq 2$, then we have
\begin{equation*}
\begin{split}
  \overline E^{\rm qc}(js^i)
= \Bigl[ &
 D(js^i) \cdot \overline A_{n,j-s+1}(i-1)\cdot \overline A_{n,s-1}(i)\cdot 
\overline{A}^{-1}_{n,j-s+1}(i) \cdot 
\overline{A}^{-1}_{ns}(i) \\
& \cdot 
\prod_{t=1}^{i-1} \bigl( \overline{A}_{n,n-s-t+1}(i-t) \cdot \overline{A}^{-1}_{n,j-s-t+1}(i-t) \bigr) \\
& \cdot 
\prod_{t=1}^{i-1} \bigl( \overline{A}_{n,n-s-t+2}(i-t) \cdot \overline{A}^{-1}_{n,j-s-t+2}(i-t) \bigr)
\Bigr].
\end{split}
\end{equation*}
\end{enumerate}
\end{theorem}

\begin{proof}
(a) If $s=1$, by \eqref{eq:Ej0}, we have 
\begin{equation*}
     \overline E^{\rm qc}(j1^1)=\overline A^{\rm qc}_{j1}(1)=[\overline A_{j1}(1)\cdot \overline A^{-1}_{n1}(1)\cdot \overline A^{-1}_{j0}(1)].
\end{equation*}
Using \cite[Theorem 1.2 (c)]{huang2025quantum}, we obtain
\begin{equation}
     \overline E^{\rm qc}(j1^1)=\Bigl[\left(D^{1,3}_{j1}\cdot \overline A_{j0}(1)\cdot \overline A_{n,n-1}(2)\right)\cdot \overline A^{-1}_{n1}(1)\cdot \overline A^{-1}_{j0}(1) \Bigr]=\Bigl[D^{1,3}_{j1}\cdot  \overline A_{n,n-1}(2)\cdot \overline A^{-1}_{n1}(1)\Bigr].
\end{equation}
Thus the result is true for $s=1$ by \eqref{eqn:D-ab-pq}.

If $s\geq 2$, by Lemma \ref{lem:B4} (b), we have 
\begin{equation}
     \overline E^{\rm qc}(js^1)=\widetilde \mu_{(j,s)}^{\diamondsuit}(\overline A^{\rm qc}_{n-1,s-1}\langle 1\rangle)=\widetilde \mu_{(j,s)}^{\diamondsuit}(\overline A^{\rm qc}_{n-1,s-1}(1)),
\end{equation}
where the mutation sequence $\widetilde \mu_{(j,s)}^{\diamondsuit}$
is taken for the seed   $\overline {\mathsf{s}} _{\lambda_1}^{\rm qc}=\overline {\mathsf{s}} _{\lambda}^{\rm qc}$.
Thus, by Lemma \ref{lem:quasi2}(b) we obtain
\begin{equation}\label{eq:ccc}
     \overline E^{\rm qc}(js^1)=[\widetilde \mu_{(j,s)}^{\diamondsuit}(\overline A_{n-1,s-1}(1))\cdot \overline A_{j-s+1,0}^{-1}(1)\cdot \overline A_{ns}^{-1}(1)].
\end{equation}
Using Theorem \ref{intro-thm-skein-inclusion-A}(d) (please note the different labels of the vertices), we obtain
\begin{equation}\label{eq:cc}
  D_{n-1-j+s,s}^{1,2}= [\widetilde \mu_{(j,s)}^{\diamondsuit}(\overline A_{n-1,s-1}(1))\cdot \overline A^{-1}_{j-s+1,0}(1)\cdot \overline A^{-1}_{n,s-1}(1)]. 
\end{equation}

By \eqref{eq:ccc} and \eqref{eq:cc}, we obtain
\begin{equation}
   \overline E^{\rm qc}(js^1)=   
   [D_{n-1-j+s,s}^{1,2}\cdot \overline A_{n,s-1}(1)\cdot \overline A_{ns}^{-1}(1)].
\end{equation}
Thus the result is true for $s\geq 2$ by \eqref{eqn:D-ab-pq}.


(b) Now we suppose $2\leq i\leq k$.  If $s=1$, by 
Lemmas~\ref{lem:B-d} (c) and \ref{lem:quasi2}, we have
\begin{equation}\label{eq:asy1}
    \overline E^{\rm qc}(j1^i)
    =\overline A_{j1}^{\rm qc}\langle i\rangle=
    \Bigl[
\overline{A}_{j1}\langle i\rangle \cdot
p 
\Bigr],
\end{equation}
where $p=\overline{A}^{-1}_{n1}(i) \cdot 
\overline{A}^{-1}_{nj}(i-1) \cdot
\prod_{t=1}^{i-1} \bigl( \overline{A}_{n,n-t}(i-t) \cdot \overline{A}^{-1}_{n,j-t}(i-t) \bigr)
\cdot
\prod_{t=1}^{i-1} \bigl( \overline{A}^{-1}_{n,n-t+1}(i-t) \cdot \overline{A}_{n,j-t+1}(i-t)\bigr)$.

Recall that, for each $2\leq i\leq k$, the triangulation $\lambda_i$ is obtained from $\lambda$ by flipping the arc $(1,i+1)\in \lambda$.
The vertices inside the triangle $(i,i+1,i+2)$ in $\lambda_i$ are labeled as in Figure~\ref{Fig:lambdai}.

Let $\lambda'$ be the triangulation obtained from $\lambda_i$ by flipping at $(1,i+2)$ in the case $i<k$, and let $\overline {\mathsf{s}} _{\lambda'}^{\rm qc}$ be the corresponding seed of $\overline {\mathscr A}_{\omega}^{\rm qc}(\mathbb P_{k+2})$. The cluster variable associated with the vertex $js$ in the triangle $(i,i+1,i+2)$ is identical for both seeds $\overline {\mathsf{s}} _{\lambda_i}^{\rm qc}$ and $\overline {\mathsf{s}} _{\lambda'}^{\rm qc}$. By \cite[Theorem~1.2(c)]{huang2025quantum} (please note the different labels of the vertices), we have
\begin{equation}\label{eq:asy21}
D^{i,i+2}_{j1}
=
\begin{cases}
\Bigl[
\overline A_{j1}\langle i\rangle \cdot (\overline A_{j0}\langle i\rangle)^{-1} \cdot \overline A^{-1}_{n,n-1}(i+1)
\Bigr]
=
\Bigl[
\overline A_{j1}\langle i\rangle \cdot \overline A^{-1}_{nj}(i-1) \cdot \overline A^{-1}_{n,n-1}(i+1)
\Bigr], & i<k, \\[6pt]
\Bigl[
\overline A_{j1}\langle k\rangle \cdot (\overline A_{j0}\langle k\rangle)^{-1} \cdot \overline A^{-1}_{1,1}(k)
\Bigr]
=
\Bigl[
\overline A_{j1}\langle k\rangle \cdot \overline A^{-1}_{nj}(k-1) \cdot \overline A^{-1}_{1,1}(k)
\Bigr], & i=k.
\end{cases}
\end{equation}

Combine \eqref{eq:asy1} and \eqref{eq:asy21}, we obtain
\begin{equation*}
\begin{aligned}
\overline E^{\rm qc}(j1^i)=
\begin{cases}
\Bigl[
D^{i,i+2}_{j1}\cdot \overline A^{-1}_{n1}(i)\cdot \overline A_{n,n-1}(i+1)
\cdot \prod_{t=1}^{i-1} \bigl(
\overline{A}_{n,n-t}(i-t)\cdot \overline{A}^{-1}_{n,j-t}(i-t) \\
\qquad\qquad\qquad\qquad\cdot \overline{A}^{-1}_{n,n-t+1}(i-t)\cdot \overline{A}_{n,j-t+1}(i-t)
\bigr)
\Bigr], & i<k, \\[6pt]
\Bigl[ D^{k,k+2}_{j1}\cdot \overline A^{-1}_{n1}(k)\cdot \overline A_{11}(k)
\cdot \prod_{t=1}^{k-1} \bigl(
\overline{A}_{n,n-t}(k-t)\cdot \overline{A}^{-1}_{n,j-t}(k-t) \\
\qquad\qquad\qquad\qquad\cdot \overline{A}^{-1}_{n,n-t+1}(k-t)\cdot \overline{A}_{n,j-t+1}(k-t)
\bigr) \Bigr], & i=k.
\end{cases}
\end{aligned}
\end{equation*}
Thus the result is true for $s=1$ by \eqref{eqn:D-ab-pq}.

If $s\geq 2$,  by Lemma~\ref{lem:B4} (b), we have 
\begin{equation}
 \begin{split}
      \overline E^{\rm qc}(js^i)=
       \widetilde \mu_{(j,s)}^{\diamondsuit}(\overline A^{\rm qc}_{n-1,s-1}\langle i\rangle),
   \end{split}
\end{equation} 
where the mutation sequence  $\widetilde \mu_{(j,s)}^{\diamondsuit}$ is taken for the seed $\overline {\mathsf{s}} _{\lambda_i}^{\rm qc}$.
Thus, by Lemma \ref{lem:quasi2}(c) we obtain
\begin{equation}\label{eq:cc2}
\begin{split}
 \overline E^{\rm qc}(js^i)
= \Bigl[ &
\widetilde{\mu}_{(j,s)}^{\diamondsuit}\bigl(\overline{A}_{n-1,s-1}\langle i\rangle\bigr) \cdot 
\overline{A}^{-1}_{n,j-s+1}(i) \cdot 
\overline{A}^{-1}_{ns}(i) \\
& \cdot 
\prod_{t=1}^{i-1} \bigl( \overline{A}_{n,n-s-t+1}(i-t) \cdot \overline{A}^{-1}_{n,j-s-t+1}(i-t) \bigr) \\
& \cdot 
\prod_{t=1}^{i-1} \bigl( \overline{A}_{n,n-s-t+2}(i-t) \cdot \overline{A}^{-1}_{n,j-s-t+2}(i-t) \bigr)
\Bigr].
\end{split}
\end{equation}

Using Theorem \ref{intro-thm-skein-inclusion-A}(d) (please note the different of the labels of the vertices), we obtain
\begin{equation}\label{eq:cc1}
   D^{i,i+1}_{j-s+1,n+1-s}=\Bigl[\widetilde{\mu}_{(j,s)}^{\diamondsuit}\bigl(\overline{A}_{n-1,s-1}\langle i\rangle\bigr)\cdot \overline A_{n,j-s+1}^{-1}(i-1)\cdot \overline A^{-1}_{n,s-1}(i) \Bigr].
\end{equation}

Therefore, by \eqref{eq:cc2} and \eqref{eq:cc1}, we obtain

\begin{equation*}
\begin{split}
\overline E^{\rm qc}(js^i)
= \Bigl[ &
 D^{i,i+1}_{j-s+1,n+1-s} \cdot \overline A_{n,j-s+1}(i-1)\cdot \overline A_{n,s-1}(i)\cdot 
\overline{A}^{-1}_{n,j-s+1}(i) \cdot 
\overline{A}^{-1}_{ns}(i) \\
& \cdot 
\prod_{t=1}^{i-1} \bigl( \overline{A}_{n,n-s-t+1}(i-t) \cdot \overline{A}^{-1}_{n,j-s-t+1}(i-t) \bigr) \\
& \cdot 
\prod_{t=1}^{i-1} \bigl( \overline{A}_{n,n-s-t+2}(i-t) \cdot \overline{A}^{-1}_{n,j-s-t+2}(i-t) \bigr)
\Bigr].
\end{split}
\end{equation*}
Thus the result is true for $s\geq 2$ by \eqref{eqn:D-ab-pq}. The proof is complete.
\end{proof}

\begin{proof}[Proof of Theorems \ref{thm:poly1} and \ref{thm:poly2}]
First, we establish the equality of the algebras in Theorem \ref{thm:poly1}. By Theorem~\ref{intro-thm-skein-inclusion-A}, we have the inclusion
\begin{equation}\label{eq:oneside}
 \overline{\mathscr{S}}_{\omega}(\mathbb{P}_{k+2}) = \widetilde{\mathscr{S}}_{\omega}(\mathbb{P}_{k+2}) \subseteq \overline{\mathscr{A}}_\omega(\mathbb{P}_{k+2}).   
\end{equation}
From \eqref{intro-eq-Cij}, it follows that all frozen variables of $\overline{\mathscr{A}}_\omega(\mathbb{P}_{k+2})$ and their inverses are contained in $\overline{\mathscr{S}}_{\omega}(\mathbb{P}_{k+2})$. Furthermore, by Theorem~\ref{thm:standardgenerators}, all standard cluster variables belong to $\overline{\mathscr{S}}_{\omega}(\mathbb{P}_{k+2})$. Consequently, applying Proposition~\ref{pro:standard}, we obtain the reverse inclusion:
\begin{equation}\label{eq:anotherside}
    \overline{\mathscr{S}}_{\omega}(\mathbb{P}_{k+2}) \supseteq \overline{\mathscr{A}}_\omega(\mathbb{P}_{k+2}) = \overline{\mathscr{U}}_\omega(\mathbb{P}_{k+2}).
\end{equation}
Thus we have $\overline{\mathscr{S}}_{\omega}(\mathbb{P}_{k+2}) = \overline{\mathscr{A}}_\omega(\mathbb{P}_{k+2}) = \overline{\mathscr{U}}_\omega(\mathbb{P}_{k+2})$.

Next, we prove that the theta basis is invariant under rotation. Let $\Theta$ denote the theta basis of $\overline{\mathscr{A}}_\omega(\mathbb{P}_{k+2}) = \overline{\mathscr{U}}_\omega(\mathbb{P}_{k+2})$. For any triangulation $\lambda$ of $\mathbb{P}_{k+2}$, as $R[\overline A^{\pm 1}_v\mid v\in \overline V_\lambda\setminus \mathring {\overline{V}}_\lambda]$ basis, it is parametrized by
\begin{equation}
    \Theta = \left\{ \vartheta_g^{\overline{\mathsf{s}}_\lambda} \mid g \in \mathbb{Z}^{\mathring{\overline{V}}_\lambda} \right\},
\end{equation}
where $\vartheta_g^{\overline{\mathsf{s}}_\lambda}$ is the theta function associated with $g$ relative to the initial seed $\overline{\mathsf{s}}_\lambda$. It is known that the theta basis is independent of the choice of initial seed and thus is independent of the choice of the triangulation $\lambda$. 

Let $\sigma$ be a rotation of the polygon $\mathbb{P}_{k+2}$. This rotation induces a cluster automorphism on the upper cluster algebra $\overline{\mathscr{U}}_\omega(\mathbb{P}_{k+2})$ and bijection $\sigma \colon \mathbb{Z}^{\overline{V}_\lambda} \to \mathbb{Z}^{\overline{V}_{\sigma(\lambda)}}, \mathbb{Z}^{\mathring{\overline{V}}_\lambda} \to \mathbb{Z}^{\mathring{\overline{V}}_{\sigma(\lambda)}}$. For simplicity, we denote both maps by $\sigma$. Since $\sigma(\lambda)$ is also a valid triangulation of $\mathbb{P}_{k+2}$, we have
\begin{equation}
    \sigma\left(\vartheta_{g}^{\overline{\mathsf{s}}_\lambda}\right) = \vartheta_{\sigma(g)}^{\overline{\mathsf{s}}_{\sigma(\lambda)}}.
\end{equation}
Since $\Theta$ is defined intrinsically regardless of the triangulation, the image of the basis under $\sigma$ satisfies
\begin{equation*}
    \sigma(\Theta) = \left\{ \vartheta_{\sigma(g)}^{\overline{\mathsf{s}}_{\sigma(\lambda)}} \mid g \in \mathbb{Z}^{\mathring{\overline{V}}_\lambda} \right\} = \left\{ \vartheta_{h}^{\overline{\mathsf{s}}_{\sigma(\lambda)}} \mid h \in \mathbb{Z}^{\mathring{\overline{V}}_{\sigma(\lambda)}} \right\} = \Theta.
\end{equation*}
Thus, the theta basis $\Theta$ is invariant under the rotation $\sigma$.

This completes the proof of Theorem~\ref{thm:poly1}. 

Theorem~\ref{thm:poly2} follows by an analogous argument.
\end{proof}

\section{An example: web interpretation of the dual canonical basis of $\mathcal{O}_q({\rm SL}_3)$}

In this section, we assume $n=3$ and refer to a ‘$3$-web’ simply as a ‘web’.
We investigate the cluster structures of $\overline{\mathscr{A}}_\omega(\mathbb{P}_{4})=\overline{\cS}_{\omega}(\mathbb{P}_{4})$ and ${\mathscr{A}}_\omega(\mathbb P_2)$. As will be verified later, $\cS_{\omega}(\mathbb{P}_{2})={\mathscr{A}}_\omega(\mathbb P_2)=\mathscr U_{\omega}(\mathbb{P}_{2})$, see Lemma \ref{lem:A=U=skein}. Specifically, we first provide a web interpretation of the cluster variables and subsequently describe the clusters. Since the cluster types of both $\overline{\mathscr{A}}_\omega(\mathbb{P}_{4})$ and ${\mathscr{A}}_\omega(\mathbb P_2)$ are of type $D_4$, each contains 16 exchangeable cluster variables and 50 clusters.

In accordance with \cite{FST}, the set of exchangeable cluster variables is in one-to-one correspondence with the set of non-boundary tagged arcs of the once-punctured quadrilateral, denoted by $\mathbb{P}_{4,1}$. Furthermore, the set of clusters is in one-to-one correspondence with the set of tagged triangulations of~$\mathbb{P}_{4,1}$, and the flip operation on tagged triangulations is compatible with seed mutation. We label the boundary punctures clockwise as $1,2,3,4$ and the interior puncture as $0$. For distinct $i, j \in \{1,2,3,4\}$, let $\gamma_{ij}$ denote the arc from $i$ to $j$ such that $0$ lies to the left of $\gamma_{ij}$. For any $i$, we denote by $\gamma_{i}$ the tagged arc connecting $0$ and $i$ which is tagged plain at $0$, and by $\gamma^{\rm tag}_i$ the tagged arc connecting $0$ and $i$ which is tagged notched at $0$. For further details, we refer the reader to \cite{FST}. Consequently, the set of non-boundary tagged arcs is given by
\[
\{\gamma_{i},\gamma^{\rm tag}_{i}\mid i=1,2,3,4\}\cup \{\gamma_{13},\gamma_{24},\gamma_{31},\gamma_{42}\}\cup \{\gamma_{12},\gamma_{23},\gamma_{34},\gamma_{41}\}.
\] 

According to \cite[Definition 7.4]{FST}, two tagged arcs $\gamma$ and $\gamma'$ are \textbf{compatible} if they do not cross, up to isotopy with respect to $\mathbb P_{4,1}\setminus \{0,1,2,3,4\}$, in the interior, with the exception of the case where both $\gamma$ and $\gamma'$ connect the puncture $0$ with different taggings and have distinct other endpoints. A \textbf{tagged triangulation} is a maximal collection of pairwise compatible tagged arcs.

\vspace{2mm}

By applying Theorem~\ref{intro-thm-skein-inclusion-A}, we determine all cluster variables of $\overline{\mathscr{A}}_\omega(\mathbb{P}_{4})$ and, in conjunction with the geometric description provided in \cite{FST}, determine all of its clusters (see \cite{ishibashi2023skein} for a related work).
Furthermore, utilizing the quasi-isomorphism established in \S\ref{sub-quasi-polygon}, we obtain the complete description of the cluster variables (see \eqref{eq:initialcluqc}-\eqref{eq:cv4qc}) and clusters (see Theorem~\ref{thm:webint}) for ${\mathscr{A}}_\omega(\mathbb P_2)$.
Using these results, we provide a web interpretation
of the dual canonical basis of ${\mathscr{A}}_\omega(\mathbb P_2)=\mathcal O_q({\rm SL}_3)$ (see Theorem~\ref{thm-basis-bigon}).
We refer to \cite{le_sikora2025} for another construction of a web interpretation of this dual canonical basis.

\subsection{Cluster structure of $\overline{\mathscr{A}}_\omega(\mathbb{P}_{4})=\overline{\cS}_{\omega}(\mathbb{P}_{4})$ when $n=3$}\label{sec:projskein}

Recall that for each puncture $p \in \{1,2,3,4\}$ and indices $i,j \in \{1,2,3\}$, the corner arcs $C(p)_{ij}$ and $\overline{C}(p)_{ij}$ are stated arcs, as depicted in Figure \ref{Fig;badarc}.

By Theorem~\ref{intro-thm-skein-inclusion-A}(d) and~(e), we determine the $8$ frozen variables as follows:

\begin{equation}\label{eq:frozen}
\begin{split}
   \overline A_{01}&=\overline C(1)_{33}=(C(1)_{11})^{-1}, \qquad\overline A_{02}= {C(1)_{33}}=(\overline{C}(1)_{11})^{-1},\\
   \overline A_{10}&= {C(2)_{33}}=(\overline{C}(2)_{11})^{-1},\qquad \overline A_{20}=\overline {C}(2)_{33}=(C(2)_{11})^{-1},\\
   \overline A_{31}&= {C(3)_{33}}=(\overline{C}(3)_{11})^{-1},\qquad \overline A_{32}=\overline {C}(3)_{33}=(C(3)_{11})^{-1},\\
   \overline A_{13}&=\overline {C}(4)_{33}=(C(4)_{11})^{-1},\qquad \overline A_{23}= {C(4)_{33}}=(\overline{C}(4)_{11})^{-1}.
\end{split}
\end{equation}

\begin{figure}
    \centering
    \includegraphics[width=0.2\linewidth]{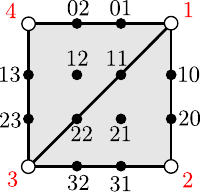}
    \caption{The labeling of $\overline V_\lambda$ when $n=3$.}
    \label{fig:P4-3}
\end{figure}

For the triangulation $\lambda$ of $\mathbb{P}_4$ with the labels of vertices depicted in Figure \ref{fig:P4-3}, we associate the seed $\overline{\mathsf{s}}_\lambda$ with the triangulation $\{\gamma_1,\gamma_2,\gamma_3,\gamma_4\}$ of $\mathbb{P}_{4,1}$. Specifically, we assign $\overline A_{\gamma_1}:=\overline A_{11}$, $\overline A_{\gamma_2}:=\overline A_{21}$, $\overline A_{\gamma_3}:=\overline A_{22}$, and $\overline A_{\gamma_4}:=\overline A_{12}$. Applying Theorem~\ref{intro-thm-skein-inclusion-A}(d) and~(e), we obtain:
\begin{equation}\label{eq:initialclu}
\begin{split}
   \overline A_{\gamma_1}&=[\overline {C}(4)_{31}\cdot \overline A_{01}]=[\overline {C}(4)_{31}\cdot \overline {C}(1)_{33}], \quad
   \overline A_{\gamma_2}=[{C(2)_{21}}\cdot \overline A_{31} \cdot \overline A_{10}]=[{C(2)_{21}}\cdot C(3)_{33} \cdot {C(2)_{33}}],\\
   \overline A_{\gamma_3}&=[\overline {C}(2)_{31}\cdot \overline A_{32}]=[\overline {C}(2)_{31}\cdot \overline {C}(3)_{33}], \quad
   \overline A_{\gamma_4}=[{C(4)_{21}}\cdot \overline A_{23} \cdot \overline A_{02}]=[{C(4)_{21}}\cdot C(4)_{33} \cdot {C(1)_{33}}].
\end{split}
\end{equation}

For any tagged arc $\gamma$, we denote the associated exchangeable cluster variable by $\overline A_\gamma$.

Consider the corner arc $C(4)_{32}$. By Theorem~\ref{intro-thm-skein-inclusion-A}(d), we have 
\[C(4)_{32}=[\mu_{\overline A_{12}}(\overline A_{12})\cdot \overline A^{-1}_{01}]=[\mu_{\overline A_{\gamma_4}}(\overline A_{\gamma_4})\cdot \overline A^{-1}_{01}].\] 
Consequently,
\begin{equation}\label{eq:oneflip}
    \overline A_{\gamma_{13}}=\mu_{\overline A_{\gamma_4}}(\overline A_{\gamma_4})=[C(4)_{32}\cdot \overline A_{01}]=[C(4)_{32}\cdot \overline {C}(1)_{33}].
\end{equation}

By symmetry, we obtain
\begin{equation}\label{eq:oneflip1}
    \overline A_{\gamma_{31}}=[C(2)_{32}\cdot \overline {C}(3)_{33}].
\end{equation}

Next, consider the essential arc $D^{3,1}_{22}$, where $D^{3,1}_{22}$ is defined right above Theorem~\ref{thm:standardgenerators}. By Theorem~\ref{intro-thm-skein-inclusion-A}(f), we have 
\[D^{3,1}_{22}=[\mu_{\overline A_{22}}\mu_{\overline A_{21}}(\overline A_{22})\cdot \overline A^{-1}_{20}\cdot \overline A^{-1}_{23}]=[\mu_{\overline A_{\gamma_3}}\mu_{\overline A_{\gamma_3}}(\overline A_{\gamma_3})\cdot \overline A^{-1}_{20}\cdot \overline A^{-1}_{23}].\] 
Thus,
\begin{equation}\label{eq:oneflip2}
    \overline A_{\gamma_{41}}=\mu_{\overline A_{\gamma_3}}\mu_{\overline A_{\gamma_2}}(\overline A_{\gamma_3})=[D^{3,1}_{22}\cdot \overline A_{20}\cdot \overline A_{23}]=[D^{3,1}_{22}\cdot \overline{C}(2)_{33}\cdot {C(4)_{33}}].
\end{equation}

By symmetry, we have
\begin{equation}\label{eq:oneflip3}
    \overline A_{\gamma_{23}}=\mu_{\overline A_{\gamma_1}}\mu_{\overline A_{\gamma_4}}(\overline A_{\gamma_1})=[D^{1,3}_{22}\cdot \overline A_{10}\cdot \overline A_{13}]=[D^{1,3}_{22}\cdot \overline{C}(4)_{33}\cdot {C(2)_{33}}].
\end{equation}

Let $\lambda'$ be another triangulation of $\mathbb{P}_4$. By Theorem~\ref{lem-mutation-A-seeds-flip}, $\overline{\mathsf{s}}_{\lambda'}=\mu_{\overline A_{21}}\mu_{\overline A_{12}}\mu_{\overline A_{22}}\mu_{\overline A_{11}}(\overline{\mathsf{s}}_\lambda)$. This corresponds to the triangulation $\mu_{\gamma_2}\mu_{\gamma_4}\mu_{\gamma_3}\mu_{\gamma_1}\{\gamma_1,\gamma_2,\gamma_3,\gamma_4\}=\{\gamma_{24},\gamma_{42},\gamma^{\rm tag}_2,\gamma^{\rm tag}_4\}$ of $\mathbb{P}_{4,1}$. Analogous to \eqref{eq:initialclu}, we derive:
\begin{equation}\label{eq:cv1}
\begin{split}
   \overline A_{\gamma_{24}}&=\mu_{\gamma_2}\mu_{\gamma_4}\mu_{\gamma_3}\mu_{\gamma_1}(\overline A_{\gamma_{1}})=[{C(1)_{21}}\cdot \overline A_{10}\cdot \overline A_{02}]=[{C(1)_{21}}\cdot {C(2)_{33}}\cdot C(1)_{33}], \\
   \overline A_{\gamma_4^{\rm tag}}&=\mu_{\gamma_2}\mu_{\gamma_4}\mu_{\gamma_3}\mu_{\gamma_1}(\overline A_{\gamma_2})=[\overline {C}(1)_{31}\cdot \overline A_{20}]=[\overline {C}(1)_{31}\cdot \overline{C}(2)_{33}],\\
   \overline A_{\gamma_{42}}&=\mu_{\gamma_2}\mu_{\gamma_4}\mu_{\gamma_3}\mu_{\gamma_1}(\overline A_{\gamma_{3}})=[{C(3)_{21}}\cdot \overline A_{23}\cdot \overline A_{31}]=[{C(3)_{21}}\cdot {C(4)_{33}}\cdot C(3)_{33}], \\
   \overline A_{\gamma_2^{\rm tag}}&=\mu_{\gamma_2}\mu_{\gamma_4}\mu_{\gamma_3}\mu_{\gamma_1}(\overline A_{\gamma_4})=[\overline {C}(3)_{31}\cdot \overline A_{13}]=[\overline {C}(3)_{31}\cdot \overline{C}(4)_{33}].
   \end{split}
\end{equation}

Similar to \eqref{eq:oneflip} and \eqref{eq:oneflip1}, we obtain
\begin{equation}\label{eq:cv2}
\begin{split}
      \overline A_{\gamma_1^{\rm tag}}&=\mu_{\gamma_1}\bigl(\mu_{\gamma_2}\mu_{\gamma_4}\mu_{\gamma_3}\mu_{\gamma_1}(\overline A_{\gamma_1})\bigr)=[C(1)_{32}\cdot \overline A_{20}]=[C(1)_{32}\cdot \overline{C}(2)_{33}],\\
        \overline A_{\gamma_3^{\rm tag}}&=\mu_{\gamma_3}\bigl(\mu_{\gamma_2}\mu_{\gamma_4}\mu_{\gamma_3}\mu_{\gamma_1}(\overline A_{\gamma_3})\bigr)=[C(3)_{32}\cdot \overline A_{13}]=[C(3)_{32}\cdot \overline{C}(4)_{33}].
\end{split}
\end{equation}

Similarly, following \eqref{eq:oneflip2} and \eqref{eq:oneflip3}, we have
\begin{equation}\label{eq:cv3}
 \overline A_{\gamma_{12}}=[D_{22}^{2,4}\cdot \overline A_{01}\cdot \overline A_{31}]=[D_{22}^{2,4}\cdot \overline {C}(1)_{33}\cdot C(3)_{33}],\quad \overline A_{\gamma_{34}}=[D_{22}^{4,2}\cdot \overline {C}(3)_{33}\cdot C(1)_{33}].
\end{equation}

In summary, combining \eqref{eq:initialclu}-\eqref{eq:cv3}, we obtain all cluster variables of $\overline{\mathscr{A}}_\omega(\mathbb{P}_{4})=\overline{\cS}_{\omega}(\mathbb{P}_{4})$. Furthermore, the $50$ distinct clusters of $\overline{\mathscr{A}}_\omega(\mathbb{P}_{4})=\overline{\cS}_{\omega}(\mathbb{P}_{4})$ are in one-to-one correspondence with the set of tagged triangulations of $\mathbb{P}_{4,1}$.

\subsection{Cluster structure of ${\mathscr{A}}_\omega(\mathbb P_2)
$ when $n=3$}\label{sec:clustervar}

Recall that we use $\mathbb P_2$ to denote the unique generalized triangulation of $\mathbb P_2$, and that the triangulation $\mathbb P_2^*$ of $\mathbb P_4$ is shown in Figure~\ref{Fig;attaching}(B) (see also Figure~\ref{fig:P4-3}). 
From \eqref{eq-com-skein-A}, there is an algebra embedding
\[
\iota_* \colon ({\cS}_{\omega}(\mathbb{P}_{2})\subseteq ) {\mathscr{A}}_\omega(\mathbb P_2) \longrightarrow \overline{\mathscr{A}}_\omega(\mathbb P_4) = \overline{\cS}_{\omega}(\mathbb{P}_{4}).
\]
Via this embedding, we identify ${\mathscr{A}}_\omega(\mathbb P_2)$ with a subalgebra of $\overline{\mathscr{A}}_\omega(\mathbb P_4) = \overline{\cS}_{\omega}(\mathbb{P}_{4})$.

In this subsection, we use the calculations in \S\ref{sec:projskein} to give a web interpretation of all the $20$ cluster variables in ${\mathscr{A}}_\omega(\mathbb P_2)$ and prove ${\mathscr{A}}_\omega(\mathbb P_2) = {\cS}_{\omega}(\mathbb{P}_{2})$.

As in \S\ref{sec:projskein}, for the triangulation $\lambda=\mathbb P_2^*$ of $\mathbb{P}_4$ with the labels of vertices depicted in Figure~\ref{fig:P4-3}, we have 
$$V_{\mathbb P_2} = \{01,02,32,31,11,12,21,22\}.$$

Recall that for each $i,j \in \{1,2,3\}$, we defined a stated arc $b_{ij} \in {\cS}_{\omega}(\mathbb{P}_{2})$ in \S\ref{sec:bigon}. We use $\cev{b}_{ij}$ to denote the stated arc obtained from 
$b_{ij}$ by reversing its orientation.

By Lemma \ref{lem-iota-gv} and \eqref{eq:frozen}, we obtain the $4$ frozen cluster variables of ${\mathscr{A}}_\omega(\mathbb P_2)$ as follows:
\begin{equation}\label{eq:frozen1}
\begin{split}
   A_{01}&= [\overline A_{01}\cdot  \overline A^{-1}_{23}]=[\overline {C}(1)_{33}\cdot \overline {C}(4)_{11}]=\cev{b}_{13}, \qquad A_{02}= [\overline A_{02}\cdot  \overline A^{-1}_{13}]=[{C(1)_{33}}\cdot {C(4)_{11}}]={b}_{13},\\
   A_{31}&= [\overline A_{31}\cdot  \overline A^{-1}_{20}]=[{C(3)_{33}}\cdot {C(2)_{11}}]=\cev{b}_{31}, \qquad A_{32}= [\overline A_{32}\cdot  \overline A^{-1}_{10}]=[\overline {C}(3)_{33}\cdot \overline {C}(2)_{11}]={b}_{31}.
\end{split}
\end{equation}

By Lemma \ref{lem-iota-gv} and \eqref{eq:frozen}, the $4$ initial exchangeable cluster variables of ${\mathscr{A}}_\omega(\mathbb P_2)$ are \begin{equation}\label{eq:initialcluqc}
\begin{split}
   A_{11}=A_{\gamma_1}&=\overline A_{\gamma_1}=[\overline {C}(4)_{31}\cdot \overline {C}(1)_{33}]=\cev{b}_{33},\\
  A_{21}= A_{\gamma_2}&=[\overline A_{\gamma_2}\cdot \overline A_{10}^{-1}]=[{C(2)_{21}}\cdot C(3)_{33}]= \cev{b}_{32},\\
 A_{33}= A_{\gamma_3}&=\overline A_{\gamma_3}=[\overline {C}(2)_{31}\cdot \overline {C}(3)_{33}]=b_{33}, \\
  A_{12}=A_{\gamma_4}&=[\overline A_{\gamma_4} \cdot \overline A_{23}^{-1}]=[{C(4)_{21}}\cdot {C(1)_{33}}]=b_{23}.
\end{split}
\end{equation}

For the remaining cluster variables, we compute by direct calculation:
\begin{align}
 A_{\gamma_{31}}& =\mu_{A_{\gamma_{2}}}(A_{\gamma_2}) =[\mu_{\overline A_{\gamma_{2}}}(\overline A_{\gamma_2})\cdot \overline A^{-1}_{20}]=[\overline A_{\gamma_{31}}\cdot \overline A^{-1}_{20}]\\&=[C(2)_{32}\cdot \overline{C}(3)_{33}\cdot {C(2)_{11}}]=b_{32},\nonumber\\[5pt]
A_{\gamma_{13}}& =[C(4)_{32}\cdot \overline{C}(1)_{33}\cdot {C(4)_{11}}]=\cev{b}_{23},\\[5pt]
A_{\gamma_{24}}& =\mu_{A_{\gamma_{1}}}(A_{\gamma_1}) =[\mu_{\overline A_{\gamma_{1}}}(\overline A_{\gamma_1})\cdot \overline A^{-1}_{10}\cdot  \overline A^{-1}_{23}]=[\overline A_{\gamma_{24}}\cdot \overline A^{-1}_{10}\cdot  \overline A^{-1}_{23}]\\&=[C(1)_{21}\cdot C(1)_{33}\cdot \overline{C}(4)_{11}]=\cev{b}_{12},\nonumber\\[5pt]
A_{\gamma_{42}}& =[C(3)_{21}\cdot C(3)_{33}\cdot \overline{C}(2)_{11}]=b_{21},\\[5pt]
A_{\gamma^{\rm tag}_{1}} &= \mu_{A_{\gamma_3}}\mu_{A_{\gamma_2}}\mu_{A_{\gamma_{4}}}(A_{\gamma_3}) 
= [\mu_{\overline A_{\gamma_3}}\mu_{\overline A_{\gamma_2}}\mu_{\overline A_{\gamma_{4}}}(\overline A_{\gamma_3})\cdot \overline A^{-1}_{13}\cdot \overline A^{-1}_{20}] \\
&= [\overline A_{\gamma^{\rm tag}_{1}}\cdot \overline A^{-1}_{13}\cdot \overline A^{-1}_{20}] = [C(1)_{32}\cdot {C(4)_{11}}]=b_{12}, \nonumber\\[5pt]
A_{\gamma^{\rm tag}_{3}}& =[C(3)_{32}\cdot {C(2)_{11}}]=\cev{b}_{21},\\[5pt]
A_{\gamma_{23}} &= \mu_{A_{\gamma_1}}\mu_{A_{\gamma_{4}}}(A_{\gamma_1}) 
= [\mu_{\overline A_{\gamma_1}}\mu_{\overline A_{\gamma_{4}}}(A_{\overline \gamma_1})\cdot \overline A^{-1}_{13}\cdot \overline A^{-1}_{10}] \\
&= [\overline A_{\gamma_{23}}\cdot \overline A^{-1}_{13}\cdot \overline A^{-1}_{10}] = [D^{1,3}_{22}\cdot \overline{C}(4)_{33}\cdot C(2)_{33}\cdot \overline{C}(2)_{11}\cdot C(4)_{11}]=D^{1,3}_{22}=\cev{b}_{22},\nonumber \\[6pt]
A_{\gamma_{41}} &= D^{3,1}_{22}=b_{22},\\[5pt]
A_{\gamma_2^{\rm tag}} &= \mu_{A_{\gamma_4}}\mu_{A_{\gamma_3}}\mu_{A_{\gamma_{1}}}(A_{\gamma_4}) = [\mu_{\overline A_{\gamma_4}}\mu_{\overline A_{\gamma_3}}\mu_{\overline A_{\gamma_{1}}}(\overline A_{\gamma_4})\cdot \overline A^{-1}_{13}\cdot \overline A^{-1}_{10}] 
\\
&= [\overline A_{\gamma_2^{\rm tag}}\cdot \overline A^{-1}_{13}\cdot \overline A^{-1}_{10}] = [\overline{C}(3)_{31}\cdot \overline{C}(2)_{11}]={b}_{11}, \nonumber\\[5pt]
A_{\gamma_4^{\rm tag}} &= [\overline{C}(1)_{31}\cdot \overline{C}(4)_{11}]=\cev b_{11},\\[5pt]
A_{\gamma_{12}} &= \mu_{A_{\gamma_3}}\mu_{A_{\gamma_{4}}}(A_{\gamma_3}) = [\mu_{\overline A_{\gamma_3}}\mu_{\overline A_{\gamma_{4}}}(A_{\overline \gamma_3})\cdot \overline A^{-1}_{13}\cdot \overline A^{-1}_{10}] \\
&= [\overline A_{\gamma_{12}}\cdot \overline A^{-1}_{13}\cdot \overline A^{-1}_{10}]= [D^{2,4}_{22}\cdot \overline{C}(1)_{33}\cdot C(3)_{33}\cdot \overline{C}(2)_{11}\cdot C(4)_{11}]
, \nonumber\\[5pt]
A_{\gamma_{34}} &= [D^{4,2}_{22}\cdot \overline{C}(3)_{33}\cdot C(1)_{33}\cdot \overline{C}(4)_{11}\cdot C(2)_{11}].
\end{align}

Let 
\begin{align*}
    \alpha:= \raisebox{-.25in}{

\begin{tikzpicture}
\tikzset{->-/.style=

{decoration={markings,mark=at position #1 with

{\arrow{latex}}},postaction={decorate}}}

\filldraw[draw=white,fill=gray!20] (0,0) rectangle (1.5, 1.5);
\draw [line width =1pt,decoration={markings, mark=at position 0.2 with {\arrow{<}}},postaction={decorate}] (0,0)--(0,1.5);
\draw [line width =1pt,decoration={markings, mark=at position 0.2 with {\arrow{<}}},postaction={decorate}] (1.5,0)--(1.5,1.5);
\draw [line width =0.8pt,decoration={markings, mark=at position 0.5 with {\arrow{<}}},postaction={decorate}](0,0.5)--(0.75,0.5);
\draw [line width =0.8pt,decoration={markings, mark=at position 0.5 with {\arrow{<}}},postaction={decorate}](1.5,0.5)--(0.75,0.5);
\draw [line width =0.8pt,decoration={markings, mark=at position 0.5 with {\arrow{>}}},postaction={decorate}](0,1)--(0.75,1);
\draw [line width =0.8pt,decoration={markings, mark=at position 0.5 with {\arrow{>}}},postaction={decorate}](1.5,1)--(0.75,1);
\draw [line width =0.8pt,decoration={markings, mark=at position 0.5 with {\arrow{>}}},postaction={decorate}](0.75,0.5)--(0.75,1);
\node [left] at(0,0.5) {$3$};
\node [right] at(1.5,0.5) {$1$};
\node [left] at(0,1) {$1$};
\node [right] at(1.5,1) {$3$};
\end{tikzpicture}
} ,\qquad \cev{\alpha}:= \raisebox{-.25in}{

\begin{tikzpicture}
\tikzset{->-/.style=

{decoration={markings,mark=at position #1 with

{\arrow{latex}}},postaction={decorate}}}

\filldraw[draw=white,fill=gray!20] (0,0) rectangle (1.5, 1.5);
\draw [line width =1pt,decoration={markings, mark=at position 0.2 with {\arrow{<}}},postaction={decorate}] (0,0)--(0,1.5);
\draw [line width =1pt,decoration={markings, mark=at position 0.2 with {\arrow{<}}},postaction={decorate}] (1.5,0)--(1.5,1.5);
\draw [line width =0.8pt,decoration={markings, mark=at position 0.5 with {\arrow{>}}},postaction={decorate}](0,0.5)--(0.75,0.5);
\draw [line width =0.8pt,decoration={markings, mark=at position 0.5 with {\arrow{>}}},postaction={decorate}](1.5,0.5)--(0.75,0.5);
\draw [line width =0.8pt,decoration={markings, mark=at position 0.5 with {\arrow{<}}},postaction={decorate}](0,1)--(0.75,1);
\draw [line width =0.8pt,decoration={markings, mark=at position 0.5 with {\arrow{<}}},postaction={decorate}](1.5,1)--(0.75,1);
\draw [line width =0.8pt,decoration={markings, mark=at position 0.5 with {\arrow{<}}},postaction={decorate}](0.75,0.5)--(0.75,1);
\node [left] at(0,0.5) {$3$};
\node [right] at(1.5,0.5) {$1$};
\node [left] at(0,1) {$1$};
\node [right] at(1.5,1) {$3$};
\end{tikzpicture}
}.
\end{align*}

Relation~\eqref{wzh.five} implies
\begin{align*}
    D^{2,4}_{22}\cdot \overline{C}(1)_{33}\cdot C(3)_{33}\cdot \overline{C}(2)_{11}\cdot C(4)_{11}
    &\overset{\omega}{=}\alpha,\\
    D^{4,2}_{22}\cdot \overline{C}(3)_{33}\cdot C(1)_{33}\cdot \overline{C}(4)_{11}\cdot C(2)_{11}
    &\overset{\omega}{=}\cev{\alpha},
\end{align*}
where $X \overset{\omega}{=} Y$ means that
$X = \omega^{\frac{k}{2}}\, Y$
for some integer $k$.
Since the cluster variables $A_{\gamma_{12}}$ and $A_{\gamma_{34}}$ are reflection invariant (see \S\ref{sub-sec-invariant}), it follows that both $\alpha$ and $\cev{\alpha}$ are reflection-normalizable. Therefore,
\begin{equation}\label{eq:cv4qc}
A_{\gamma_{12}} = [\alpha]_{\mathrm{norm}}, \qquad 
A_{\gamma_{34}} = [\cev{\alpha}]_{\mathrm{norm}}.
\end{equation}

In summary, by synthesizing the results from \eqref{eq:initialcluqc}-\eqref{eq:cv4qc}, we determine the complete set of cluster variables for ${\mathscr{A}}_\omega(\mathbb{P}_{2})$. Furthermore, the $50$ distinct clusters of ${\mathscr{A}}_\omega(\mathbb{P}_{2})$ are in one-to-one correspondence with the set of tagged triangulations of the once-punctured quadrilateral $\mathbb{P}_{4,1}$.

\begin{lemma}\label{lem:A=U=skein}
    We have $\cS_{\omega}(\mathbb{P}_{2})={\mathscr{A}}_\omega(\mathbb P_2)=\mathscr U_{\omega}(\mathbb{P}_{2})$. 
\end{lemma}

\begin{proof}
    The inclusion $\cS_{\omega}(\mathbb{P}_{2})\supseteq {\mathscr{A}}_\omega(\mathbb P_2)$ holds because all cluster variables of ${\mathscr{A}}_\omega(\mathbb P_2)$ lie in $\cS_{\omega}(\mathbb{P}_{2})$. 
    Coupled with Theorem~\ref{thm-skein-inclusion-A}, we obtain $\cS_{\omega}(\mathbb{P}_{2})={\mathscr{A}}_\omega(\mathbb P_2)$. 
    Furthermore, as the cluster type of ${\mathscr{A}}_\omega(\mathbb P_2)$ is $D_4$ (which is acyclic), it follows from \cite[Theorem~3.7]{BMS}---a result that extends to the quantum setting---that ${\mathscr{A}}_\omega(\mathbb P_2)={\mathscr{U}}_\omega(\mathbb P_2)$. 
The lemma is thus established.
\end{proof}

\subsection{Dual canonical basis of $\mathcal O_q(\rm SL_3)$ and web-interpretation}\label{sub-dual-basis-web}

Based on the discussion in  \S\ref{sec:clustervar}, the cluster variables of the algebra ${\mathscr{A}}_\omega(\mathbb P_2)={\cS}_{\omega}(\mathbb{P}_{2})= \mathcal{O}_q({\rm SL}_3)$ consist of $18$ stated corner arcs, denoted by the set $\{b_{ij},\cev{b}_{ij}\mid i,j=1,2,3\}$ (which includes $4$ frozen variables: $b_{11},b_{33},\cev{b}_{11},\cev{b}_{33}$), and $2$ normalized stated webs, $[\alpha]_{\rm norm}$ and $[\cev{\alpha}]_{\rm norm}$. For each cluster variable $A$, we associate a set of corner arcs $C(A)$ defined as follows:
\begin{equation}
    C(A)=\begin{cases}
        \{A\} & \mbox{if $A$ is a corner arc,}\\
       \{\cev{b}_{33}, {b}_{11}\}  & \mbox{if $A=[\alpha]_{\rm norm}$,}\\
      \{b_{33}, \cev{b}_{11}\} & \mbox{if $A=[\cev{\alpha}]_{\rm norm}$.}
    \end{cases}
\end{equation}

To determine whether two cluster variables belong to the same cluster, we fix three points on each boundary component of $\mathbb{P}_2$, labeled $1,2,3$ in counterclockwise order, as illustrated in Figure~\ref{fig:P2-labeled-points}.
A \textbf{labeled arc} $C$ is a properly embedded oriented curve in $\mathbb{P}_2$ connecting the left point $i$ to the right point $j$ for some $i,j \in \{1,2,3\}$. We define $l(C):=i$ and $r(C):=j$.

For each stated corner arc $b_{ij}$ (resp.\ $\cev{b}_{ij}$), we denote by $\langle b_{ij} \rangle$ (resp.\ $\langle \cev{b}_{ij} \rangle$) the labeled arc oriented from the left point $i$ to the right point $j$ (resp.\ from the right point $j$ to the left point $i$).

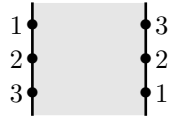
\begin{figure}[htbp]
    \centering
    \raisebox{-0.25in}{
    \begin{tikzpicture}
    \tikzset{->-/.style={
        decoration={markings,mark=at position #1 with {\arrow{latex}}},
        postaction={decorate}
    }}

    \filldraw[draw=white,fill=gray!20] (0,0) rectangle (1.5, 1.5);
    \draw[line width=1pt] (0,0)--(0,1.5);
    \draw[line width=1pt] (1.5,0)--(1.5,1.5);

    \node at(0,0.3) {$\bullet$};
    \node[left] at(0,0.3) {$3$};
    \node at(0,0.75) {$\bullet$};
    \node[left] at(0,0.75) {$2$};
    \node at(0,1.2) {$\bullet$};
    \node[left] at(0,1.2) {$1$};

    \node at(1.5,0.3) {$\bullet$};
    \node[right] at(1.5,0.3) {$1$};
    \node at(1.5,0.75) {$\bullet$};
    \node[right] at(1.5,0.75) {$2$};
    \node at(1.5,1.2) {$\bullet$};
    \node[right] at(1.5,1.2) {$3$};

    \end{tikzpicture}
    }
    \caption{Three labeled points on each boundary component of $\mathbb P_2$.}
    \label{fig:P2-labeled-points}
\end{figure}

\begin{definition}\label{def:comp}
Two labeled arcs $C$ and $C'$ are called \textbf{ compatible} if one of the following three conditions, $(1)$, $(2)$, or $(3)$, holds. They are called \textbf{strongly compatible} if one of the conditions $(1)$, $(2)$, or $(3')$ holds.
    \begin{enumerate}
        \item[$(1)$] $C$ and $C'$ do not intersect, which corresponds to the condition $$(l(C)-l(C'))(r(C)-r(C'))<0.$$
        \item[$(2)$] $C$ and $C'$ intersect on the boundary of $\mathbb{P}_2$, i.e., $(l(C)-l(C'))(r(C)-r(C'))=0$, with the following exception: the intersection point has state $2$ and the orientations of $C$ and $C'$ at this intersection point are different.
        \item[$(3)$] The curves $C$ and $C'$ intersect in the interior of $\mathbb{P}_2$, i.e.,
$(l(C)-l(C'))(r(C)-r(C'))>0,$
and $C$ and $C'$ have opposite orientations (that is, one is oriented from left to right, and the other from right to left).
\item[$(3')$] The curves $C$ and $C'$ intersect in the interior of $\mathbb{P}_2$, i.e., $(l(C)-l(C'))(r(C)-r(C'))>0,$
and $C$ and $C'$ have opposite orientations; moreover, at least one of $C$ or $C'$ has state $2$ on one boundary.
    \end{enumerate}


Two cluster variables $A$ and $A'$ are called \textbf{compatible} if
\begin{itemize}
    \item $\langle A \rangle$ and $\langle A' \rangle$ are strongly compatible when both $A$ and $A'$ are corner arcs; or
    \item for every $C \in C(A)$ and $C' \in C(A')$, the corresponding labeled arcs $\langle C \rangle$ and $\langle C' \rangle$ are compatible whenever at least one of $A$ or $A'$ is not a stated corner arc.
\end{itemize}
\end{definition}

The following theorem is the first main result of this section. It provides a geometric characterization of all clusters of 
${\cS}_{\omega}(\mathbb{P}_{2})=\mathcal O_q({\rm SL}_3)$.
We will later use this result to give a web interpretation of the dual canonical basis of $\mathcal{O}_q(\mathrm{SL}_3)$.

\begin{theorem}\label{thm:webint}
    Two cluster variables $A$ and $A'$ belong to the same cluster if and only if they are compatible. Furthermore, a cluster constitutes a maximal set of pairwise compatible cluster variables.
\end{theorem}

\begin{proof}
    According to \cite[Definition 7.4]{FST}, two tagged arcs $\gamma$ and $\gamma'$ in $\mathcal P_{4,1}$ are compatible if they do not cross, up to isotopy with respect to $\mathbb P_{4,1}\setminus \{0,1,2,3,4\}$, in the interior, with the exception of the case where both $\gamma$ and $\gamma'$ connect the puncture $0$ with different taggings and have distinct other endpoints. The result then follows directly from the bijection between the set of cluster variables and tagged arcs \eqref{eq:initialcluqc}--\eqref{eq:cv4qc}, in conjunction with Definition \ref{def:comp}.
\end{proof}

For a complete list of clusters, we refer the reader to Appendix \ref{app-order}.

\medskip

To give a web interpretation of the dual canonical basis of 
${\cS}_{\omega}(\mathbb{P}_{2})=\mathcal O_q({\rm SL}_3)$, 
we first present several relations in ${\cS}_{\omega}(\mathbb{P}_{2})$ 
in Lemmas~\ref{lem-basis1}--\ref{lem-basis4}.

\begin{lemma}\label{lem-basis1}
When $n=3$, we have the following relations:

\begin{enumerate}[label={\rm (\alph*)}]\itemsep0.3em

\item $\begin{array}{c}\includegraphics[scale=0.75]{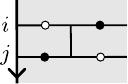}\end{array} \overset{\omega}{=}
\begin{array}{c}\includegraphics[scale=0.75]{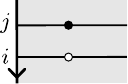}\end{array},\qquad
\begin{array}{c}\includegraphics[scale=0.75]{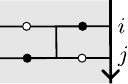}\end{array} \overset{\omega}{=}
\begin{array}{c}\includegraphics[scale=0.75]{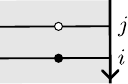}\end{array}$ when $2\in \{i,j\}$ and $i\neq j$.

\item
 \begin{align*}
     \begin{array}{c}\includegraphics[scale=0.75]{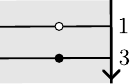}\end{array} &= q
\begin{array}{c}\includegraphics[scale=0.75]{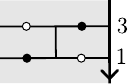}\end{array} + q^{-\frac{1}{3}} \begin{array}{c}\includegraphics[scale=0.75]{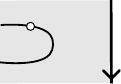}\end{array},\\
 \begin{array}{c}\includegraphics[scale=0.75]{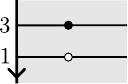}\end{array} &= q^{-1}
\begin{array}{c}\includegraphics[scale=0.75]{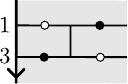}\end{array} + q^{\frac{1}{3}} \begin{array}{c}\includegraphics[scale=0.75]{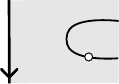}\end{array}.
\end{align*}
\end{enumerate}
    
\end{lemma}
\begin{proof}
    (a) follows from relation~\eqref{wzh.five}.
    (b) follows from relations~\eqref{wzh.five} and \eqref{wzh.seven}.
\end{proof}

\begin{lemma}\cite{LS21,frohman20223}\label{lem-basis2}
    When $n=3$, we have the following relation:
    \begin{align*}
         \raisebox{-.15in}{
	\begin{tikzpicture}
		\tikzset{->-/.style=
			
			{decoration={markings,mark=at position #1 with
					
					{\arrow{latex}}},postaction={decorate}}}
		\filldraw[draw=white,fill=gray!20] (0,0.2) rectangle (1.2, 1.2);
		\draw [line width =0.8pt,decoration={markings, mark=at position 0.7 with {\arrow{>}}},postaction={decorate}](0.4,0.4)--(0,0.4);
		\draw [line width =0.8pt,decoration={markings, mark=at position 0.7 with {\arrow{>}}},postaction={decorate}](0,1)--(0.4,1);
		\draw [line width =0.8pt,decoration={markings, mark=at position 0.6 with {\arrow{>}}},postaction={decorate}](1.2,0.4)--(0.8,0.4);
        \draw [line width =0.8pt,decoration={markings, mark=at position 0.6 with {\arrow{>}}},postaction={decorate}](0.8,1)--(1.2,1);
        \draw[line width =0.8pt] (0.4,0.4)rectangle (0.8, 1);
	\end{tikzpicture}
}
    = \raisebox{-.15in}{
	\begin{tikzpicture}
		\tikzset{->-/.style=
			
			{decoration={markings,mark=at position #1 with
					
					{\arrow{latex}}},postaction={decorate}}}
		\filldraw[draw=white,fill=gray!20] (0,0.2) rectangle (1.2, 1.2);
		\draw [line width =0.8pt,decoration={markings, mark=at position 0.5 with {\arrow{>}}},postaction={decorate}](1.2,0.4)--(0,0.4);
		\draw [line width =0.8pt,decoration={markings, mark=at position 0.5 with {\arrow{>}}},postaction={decorate}](0,1)--(1.2,1);
	\end{tikzpicture}
} + \raisebox{-.15in}{
	\begin{tikzpicture}
		\tikzset{->-/.style=
			
			{decoration={markings,mark=at position #1 with
					
					{\arrow{latex}}},postaction={decorate}}}
		\filldraw[draw=white,fill=gray!20] (0,0.2) rectangle (1.2, 1.2);
         \draw[line width =0.8pt,decoration={markings, mark=at position 0.7 with {\arrow{<}}},postaction={decorate}]
  plot[smooth] coordinates {(0,0.4) (0.5,0.7) (0,1)};
  \draw[line width =0.8pt,decoration={markings, mark=at position 0.7 with {\arrow{>}}},postaction={decorate}]
  plot[smooth] coordinates {(1.2,0.4) (0.7,0.7) (1.2,1)};
	\end{tikzpicture}
}.
    \end{align*}
\end{lemma}

\begin{lemma}\cite[Lemma~4.9]{LY23}\label{lem-basis3}
     When $n=3$, we have the following relation:
    \begin{align}\label{eq-Yilr}
    \begin{array}{c}\includegraphics[scale=1.8]{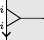}\end{array}=
        \begin{array}{c}\includegraphics[scale=1.8]{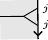}\end{array}=0,
    \end{align}
    where $i,j \in \{1,2,3\}$, and the orientations of the webs are arbitrary.
\end{lemma}

For a positive integer $k$ and $0\leq t\leq k-1$, let $J_{k,t}$ denote the set of stated web diagrams in $\mathbb{P}_2$ that contain the local configuration shown in Figure~\ref{fig:JK}.

\begin{figure}[htbp]
    \centering
    \includegraphics[width=0.4\linewidth]{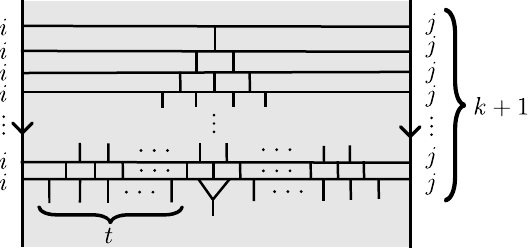}
    \caption{The local configuration appearing in diagrams in $J_{k,t}$, where $i,j\in\{1,2,3\}$ and the orientation of the web diagram is arbitrary.}
    \label{fig:JK}
\end{figure}

\begin{lemma}\label{lem-basis4}
Let $k$ be a positive integer and $0 \leq t \leq k-1$. For any $W \in J_{k,t}$, we have 
\[
W = 0 \in \cS_\omega(\mathbb{P}_2).
\]
\end{lemma}

\begin{proof}
We prove the statement by induction on $k$.

For the base case $k=1$, the result follows from Lemmas~\ref{lem-basis2} and \ref{lem-basis3}.

For the induction step, observe that the local configuration in Figure~\ref{fig:JK} contains a square. 
Given any $W \in J_{k,t}$, applying Lemma~\ref{lem-basis2} to this square yields
$W = W_1 + W_2,$
where
\begin{align*}
    \begin{cases}
        W_1 \text{ contains the left local diagram in \eqref{eq-Yilr}}, \quad W_2 \in J_{k-1,0}, & t=0,\\[4pt]
        W_1 \in J_{k-1,t-1}, \quad W_2 \in J_{k-1,t}, & 1 \leq t \leq k-2,\\[4pt]
        W_1 \text{ contains the right local diagram in \eqref{eq-Yilr}}, \quad W_2 \in J_{k-1,k-2}, & t = k-1.
    \end{cases}
\end{align*}
The conclusion then follows from the induction hypothesis together with Lemma~\ref{lem-basis3}.
\end{proof}

We introduce the following definition, which will be used to define stated web diagrams in $\mathbb P_2$ corresponding to cluster monomials in ${\mathscr{A}}_\omega(\mathbb P_2) = {\cS}_{\omega}(\mathbb{P}_{2})$ (Definition~\ref{def-quan-cluster-algebra}(d)).

\begin{definition}
A \textbf{compatible system} is a collection of pairwise  compatible labeled arcs. 
A \textbf{system} is a maximal compatible system.
A \textbf{weighted (compatible) system} is a (compatible) system in which each labeled arc is assigned a nonnegative integer, called its weight.
\end{definition}

Let $S$ be a weighted compatible system. 
We place $S$ in general position so that any two labeled arcs in $S$ realize their minimal intersection numbers. 
If $S$ contains $\langle b_{ij}\rangle$ and $\langle \cev{b}_{ij}\rangle$ for $i,j\in \{1,3\}$, then their general position in $S$ is defined so that the oriented circle in $\mathbb P_2$ formed by $\langle b_{ij}\rangle$ and $\langle \cev{b}_{ij}\rangle$ is counterclockwise; see Figure~\ref{fig:moves}(A).
We define a stated web diagram $W(S)$ in $\mathbb P_2$ associated to $S$ via the following steps:
\begin{itemize}
    \item Replace each labeled arc in $S$ with $m$ parallel copies of that arc, where $m$ is its weight, to obtain a collection $S'$ of labeled arcs. We require that any two labeled arcs in $S'$ realize their minimal intersection numbers.

    \item Obtain $S''$ from $S'$ by replacing each crossing $\raisebox{-.10in}{
	
	\begin{tikzpicture}[scale=0.6, rotate=90]
		\tikzset{->-/.style=
			
			{decoration={markings,mark=at position #1 with
					
					{\arrow{latex}}},postaction={decorate}}}
		\filldraw[draw=white,fill=gray!20] (-0,-0.2) rectangle (1, 1.2);
		\draw [line width =0.6pt,decoration={markings, mark=at position 0.5 with {\arrow{>}}},postaction={decorate}](0.6,0.6)--(1,1);
		\draw [line width =0.6pt,decoration={markings, mark=at position 0.5 with {\arrow{>}}},postaction={decorate}](0.6,0.4)--(1,0);
		\draw[line width =0.6pt] (0,0)--(0.6,0.6);
		\draw[line width =0.6pt] (0,1)--(0.4,0.6);
		\draw[line width =0.6pt] (0.4,0.6)--(0.6,0.4);
	\end{tikzpicture}
}$ with $\raisebox{-.12in}{
	\begin{tikzpicture}[scale=0.6, rotate=90]
		\tikzset{->-/.style=
			
			{decoration={markings,mark=at position #1 with
					
					{\arrow{latex}}},postaction={decorate}}}
		\filldraw[draw=white,fill=gray!20] (0,-0.2) rectangle (1.2, 1.2);
		\draw [line width =0.6pt,decoration={markings, mark=at position 0.7 with {\arrow{>}}},postaction={decorate}](0,0)--(0.4,0.5);
		\draw [line width =0.6pt,decoration={markings, mark=at position 0.7 with {\arrow{>}}},postaction={decorate}](0,1)--(0.4,0.5);
		\draw[line width =0.6pt] (0.4,0.5)--(0.8,0.5);
		\draw [line width =0.6pt,decoration={markings, mark=at position 0.6 with {\arrow{>}}},postaction={decorate}](0.8,0.5)--(1.2,0);
        \draw [line width =0.6pt,decoration={markings, mark=at position 0.6 with {\arrow{>}}},postaction={decorate}](0.8,0.5)--(1.2,1);
	\end{tikzpicture}
}$.
    
    \item Construct a crossingless web diagram $\widetilde S$ in $\mathbb P_2$ from $S''$ by applying the procedure shown in Figure~\ref{fig:moves}(B) for each $i=1,2,3$. Denote by $\partial_i$ the set of endpoints of $\widetilde S$ obtained from this procedure.
    
    \item Finally, obtain the stated web $W(S)$ from $\widetilde S$ by assigning state $i$ to each endpoint in $\partial_i$, for each $i=1,2,3$.
    The height ordering of $W(S)$ along each component of $\partial \mathbb{P}_2$ is indicated by the two arrows in
     $
\raisebox{-.08in}{

\begin{tikzpicture}
\tikzset{->-/.style=

{decoration={markings,mark=at position #1 with

{\arrow{latex}}},postaction={decorate}}}

\filldraw[draw=white,fill=gray!20] (0,0) rectangle (0.5, 0.5);
\draw [line width =0.6pt,decoration={markings, mark=at position 0.5 with {\arrow{<}}},postaction={decorate}](0,0)--(0,0.5);
\draw [line width =0.6pt,decoration={markings, mark=at position 0.5 with {\arrow{<}}},postaction={decorate}] (0.5,0)--(0.5,0.5);
\end{tikzpicture}
}
$.
\end{itemize}

\begin{figure}[htbp]
    \centering
    
    \begin{minipage}{0.48\textwidth}
        \centering
        \raisebox{-0.25in}{
        \begin{tikzpicture}
        \tikzset{->-/.style={
            decoration={markings,mark=at position #1 with {\arrow{latex}}},
            postaction={decorate}
        }}

        \filldraw[draw=white,fill=gray!20] (0,0) rectangle (1.5, 1.5);
        \draw[line width=1pt] (0,0)--(0,1.5);
        \draw[line width=1pt] (1.5,0)--(1.5,1.5);
        \draw [line width =0.8pt,decoration={markings, mark=at position 0.5 with {\arrow{<}}},postaction={decorate}] 
            (0,1) .. controls (0.5,1.2) and (1,1.1) .. (1.5,0.5);
        \draw [line width =0.8pt,decoration={markings, mark=at position 0.5 with {\arrow{>}}},postaction={decorate}] 
            (0,1) .. controls (0.5,0.3) and (1,0.4) .. (1.5,0.5);
        \end{tikzpicture}
        }
        
        \vspace{4pt}
        (A)
    \end{minipage}
    \hspace{-0.1\textwidth} 
    \begin{minipage}{0.48\textwidth}
        \centering
        \raisebox{-0.25in}{
        \begin{tikzpicture}
            \filldraw[draw=none, fill=gray!20] (0,0) rectangle (1.5,1.5);
            
            \draw[line width=1pt] (0,0) -- (0,1.5);
            \draw[line width=1pt] (1.5,0) -- (1.5,1.5);
            
            \draw[line width=0.8pt] (0,0.75) -- (0.5,1.2);
            \draw[line width=0.8pt] (0,0.75) -- (0.5,1);
            \draw[line width=0.8pt] (0,0.75) -- (0.5,0.3);
            
            \draw[line width=0.8pt] (1.5,0.75) -- (1,1.2);
            \draw[line width=0.8pt] (1.5,0.75) -- (1,1);
            \draw[line width=0.8pt] (1.5,0.75) -- (1,0.3);
            
            \node at (0,0.75) {$\bullet$};
            \node[left] at (0,0.75) {$i$};
            \node at (0.4,0.8) {$\vdots$};
            
            \node at (1.5,0.75) {$\bullet$};
            \node[right] at (1.5,0.75) {$i$};
            \node at (1.1,0.8) {$\vdots$};
        \end{tikzpicture}
        }
        $\!\longrightarrow\!$
        \raisebox{-0.25in}{
        \begin{tikzpicture}
            \filldraw[draw=none, fill=gray!20] (0,0) rectangle (1.5,1.5);
            
            \draw[line width=1pt] (0,0) -- (0,1.5);
            \draw[line width=1pt] (1.5,0) -- (1.5,1.5);
            
            \draw[line width=0.8pt] (0,1.2) -- (0.5,1.2);
            \draw[line width=0.8pt] (0,1) -- (0.5,1);
            \draw[line width=0.8pt] (0,0.3) -- (0.5,0.3);
            
            \draw[line width=0.8pt] (1.5,1.2) -- (1,1.2);
            \draw[line width=0.8pt] (1.5,1) -- (1,1);
            \draw[line width=0.8pt] (1.5,0.3) -- (1,0.3);
            
            \node at (0,0.75) {$\bullet$};
            \node[left] at (0,0.75) {$i$};
            \node at (0.4,0.8) {$\vdots$};
            
            \node at (1.5,0.75) {$\bullet$};
            \node[right] at (1.5,0.75) {$i$};
            \node at (1.1,0.8) {$\vdots$};
        \end{tikzpicture}
        }
        
        \vspace{4pt}
        (B)
    \end{minipage}
    
    \caption{(A) The general position of $\langle b_{ij}\rangle, \langle \cev{b}_{ij}\rangle$ in $S$.  (B) The procedure to obtain $\widetilde{S}$ from $S''$.}
    \label{fig:moves}
\end{figure}
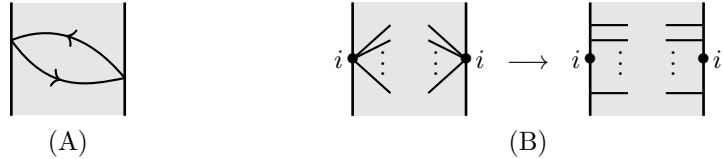

See \eqref{eq-intro-exam} for an example.

As we will see in Proposition~\ref{prop-MS-cluster-monomial}, each $W(S)$ is reflection-normalizable (see \S\ref{sub-sec-invariant}), and its normalization is a cluster monomial  in 
${\mathscr{A}}_\omega(\mathbb P_2) = {\cS}_{\omega}(\mathbb{P}_{2})$. 
To prove Proposition~\ref{prop-MS-cluster-monomial}, we require Lemmas~\ref{lem-basis5}--\ref{lem-basis8}, which establish relations in ${\cS}_{\omega}(\mathbb{P}_{2})$ involving $W(S)$.

\vspace{2mm}

For each nonnegative integer $k$, define
$$
W_k:= W\left(
\raisebox{-0.25in}{
    \begin{tikzpicture}
    \tikzset{->-/.style={
        decoration={markings,mark=at position #1 with {\arrow{latex}}},
        postaction={decorate}
    }}

    \filldraw[draw=white,fill=gray!20] (0,0) rectangle (1.5, 1.5);
    \draw[line width=1pt] (0,0)--(0,1.5);
    \draw[line width=1pt] (1.5,0)--(1.5,1.5);
    \draw [line width =0.8pt,decoration={markings, mark=at position 0.8 with {\arrow{>}}},postaction={decorate}] (0,1.2)--(1.5,0.3);
    \draw [line width =0.8pt,decoration={markings, mark=at position 0.8 with {\arrow{>}}},postaction={decorate}] (1.5,1.2)--(0,0.3);

    \node at(0,0.3) {$\bullet$};
    \node[left] at(0,0.3) {$3$};
    \node at(0,0.75) {$\bullet$};
    \node[left] at(0,0.75) {$2$};
    \node at(0,1.2) {$\bullet$};
    \node[left] at(0,1.2) {$1$};

    \node at(1.5,0.3) {$\bullet$};
    \node[right] at(1.5,0.3) {$1$};
    \node at(1.5,0.75) {$\bullet$};
    \node[right] at(1.5,0.75) {$2$};
    \node at(1.5,1.2) {$\bullet$};
    \node[right] at(1.5,1.2) {$3$};

     \node at(0.4,1.2) {$k$};
    \node at(1.1,1.2) {$k$};

    \end{tikzpicture}
    } \right),
$$
where $k$ indicates the weight of the corresponding labeled arc. 

Assume that $k\geq 1$. For any $1\leq t\leq k+1$, we define 
$W_{k,t}\in \cS_\omega (\mathbb P_2)$ as following
\begin{align*}
    W_{k,t}:=
    \begin{cases}
        W_1 W_k & t=1,\\
        \text{the web diagram in Figure~\ref{fig:WK}}
        & 2\leq t\leq k,\\
        W_{k+1} & t=k+1.
    \end{cases}
\end{align*}
Then we have the following.

\begin{figure}[htbp]
    \centering
    \includegraphics[width=0.4\linewidth]{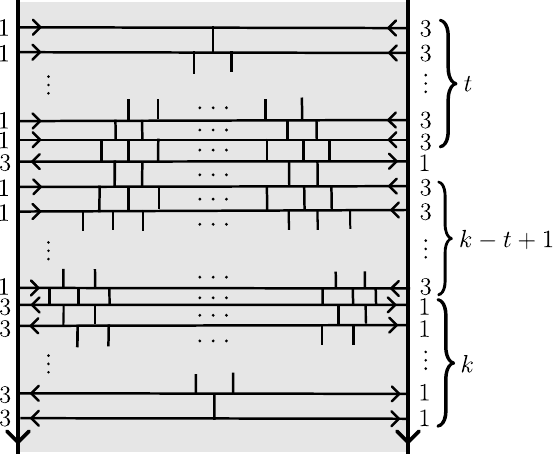}
    \caption{$W_{k,t}$ when $2\leq t\leq k-1$.}
    \label{fig:WK}
\end{figure}

\begin{lemma}\label{lem-basis5}
For each $1 \leq t \leq k$, we have 
\[
W_{k,t} = W_{k,t+1} \in \cS_\omega(\mathbb{P}_2).
\]
\end{lemma}

\begin{proof}
A direct calculation using Lemmas~\ref{lem-basis1}-\ref{lem-basis3}  gives $W_{k,1} = W_{k,2}$.

Lemma~\ref{lem-basis1}(b) implies
\begin{align}
\begin{array}{c}\includegraphics[scale=0.4]{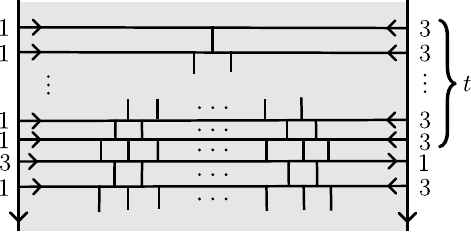}\end{array}
&=
\begin{array}{c}\includegraphics[scale=0.4]{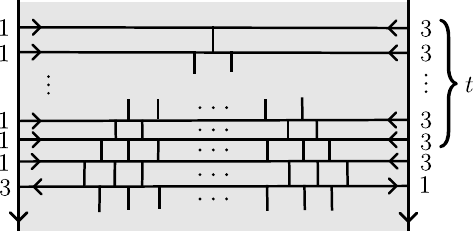}\end{array}
+ q^{\frac{4}{3}} 
\begin{array}{c}\includegraphics[scale=0.4]{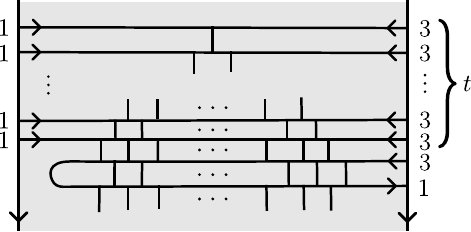}\end{array} \label{eq-basis1}\\
&=
q^{-\frac{4}{3}}
\begin{array}{c}\includegraphics[scale=0.4]{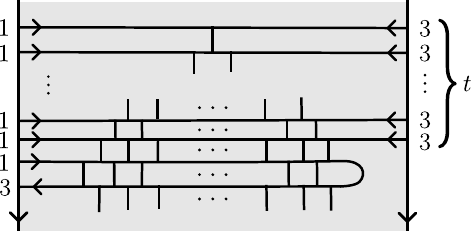}\end{array}
+ 
\begin{array}{c}\includegraphics[scale=0.4]{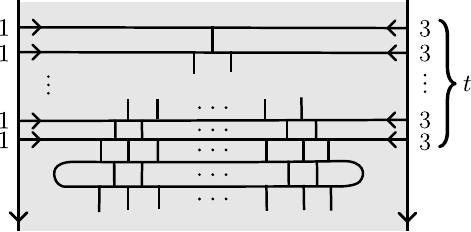}\end{array}. \nonumber
\end{align}

Next, we simplify each term. We have
\begin{align}\label{eq-basis2}
\begin{array}{c}\includegraphics[scale=0.4]{JK3.pdf}\end{array}
&=
\begin{array}{c}\includegraphics[scale=0.4]{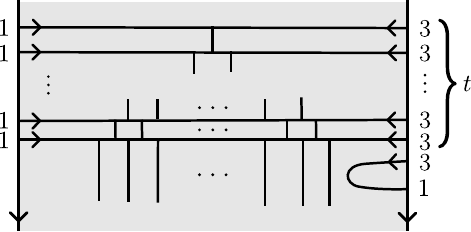}\end{array}
= q^{-\frac{1}{3}}
\begin{array}{c}\includegraphics[scale=0.4]{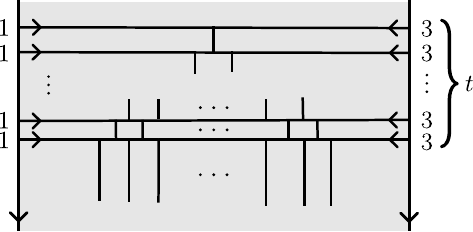}\end{array},
\end{align}
where the first equality follows from repeated applications of Lemmas~\ref{lem-basis2} and \ref{lem-basis4}, and the second equality follows from \eqref{wzh.six}.

Similarly,
\begin{align}\label{eq-basis3}
\begin{array}{c}\includegraphics[scale=0.4]{JK4.pdf}\end{array}
&=
\begin{array}{c}\includegraphics[scale=0.4]{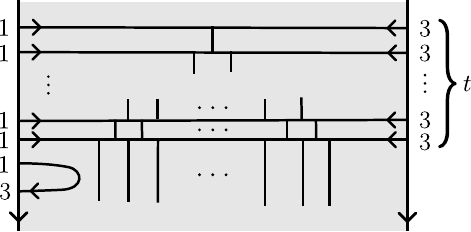}\end{array}
= q^{\frac{1}{3}}
\begin{array}{c}\includegraphics[scale=0.4]{JK8.pdf}\end{array},
\end{align}
where the first equality again follows from repeated applications of Lemmas~\ref{lem-basis2} and \ref{lem-basis4}, and the second equality follows from \cite[Equation~(56)]{LS21}.

Finally, let 
$s =
\begin{cases}
\frac{t}{2}, & \text{if } t \text{ is even},\\
\frac{t-1}{2}, & \text{if } t \text{ is odd}.
\end{cases}$
Then
\begin{align}\label{eq-basis4}
\begin{array}{c}\includegraphics[scale=0.4]{JK5.pdf}\end{array}
&=
\begin{array}{c}\includegraphics[scale=0.4]{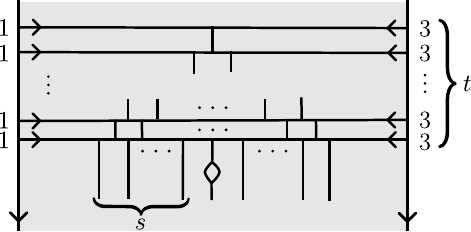}\end{array}
= -(q + q^{-1})
\begin{array}{c}\includegraphics[scale=0.4]{JK8.pdf}\end{array},
\end{align}
where the first equality follows from repeated applications of Lemmas~\ref{lem-basis2} and \ref{lem-basis4}, and the second equality follows from \cite[Equation~(1)]{frohman20223}.

Substituting \eqref{eq-basis2}--\eqref{eq-basis4} into \eqref{eq-basis1}, we obtain
\[
\begin{array}{c}\includegraphics[scale=0.4]{JK11.pdf}\end{array}
=
\begin{array}{c}\includegraphics[scale=0.4]{JK22.pdf}\end{array}.
\]
This proves $W_{k,t} = W_{k,t+1}$  for all $2 \leq t \leq k$
This completes the proof.

\end{proof}

\begin{lemma}\label{lem-basis6}
    For any nonnegative integer $k$, we have 
    $W_1^k=W_k\in\cS_\omega(\mathbb P_2)$.
\end{lemma}

\begin{proof}
We prove the statement by induction on $k$.

For the base cases $k=0$ and $k=1$, the result is immediate.

Assume that the statement holds for some $k \geq 1$. Then
\[
W_1^{k+1} = W_1  W_1^k = W_1  W_k = W_{k+1},
\]
where the second equality follows from the induction hypothesis and the third equality follows from Lemma~\ref{lem-basis5}. This completes the proof.
\end{proof}

\begin{lemma}\label{lem-basis7}
    For any positive integer $t$, we have 
    $$\begin{array}{c}\includegraphics[scale=0.4]{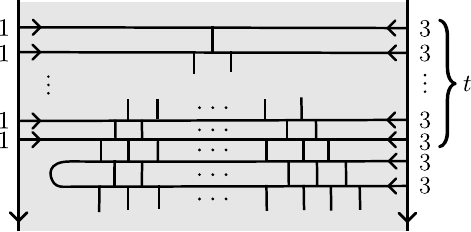}\end{array}=0\in\cS_\omega(\mathbb P_2).$$
\end{lemma}
\begin{proof}
    We have 
    $$\begin{array}{c}\includegraphics[scale=0.4]{W1.pdf}\end{array}=\begin{array}{c}\includegraphics[scale=0.4]{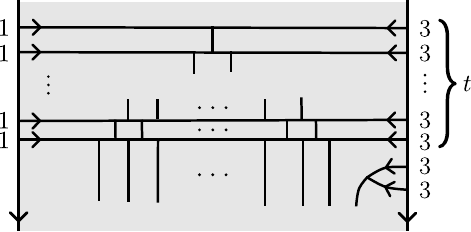}\end{array}=0,$$
    where the first equality follows from repeated applications of Lemmas~\ref{lem-basis2} and \ref{lem-basis4}, and the second equality follows from Lemma~\ref{lem-basis3}.
\end{proof}

\begin{lemma}\label{lem-basis8}
For any nonnegative integers $k,t$, we have 

\begin{enumerate}[label={\rm (\alph*)}]\itemsep0.3em

\item $
    W\left(
\raisebox{-0.25in}{
    \begin{tikzpicture}
    \tikzset{->-/.style={
        decoration={markings,mark=at position #1 with {\arrow{latex}}},
        postaction={decorate}
    }}

    \filldraw[draw=white,fill=gray!20] (0,0) rectangle (1.5, 1.5);
    \draw[line width=1pt] (0,0)--(0,1.5);
    \draw[line width=1pt] (1.5,0)--(1.5,1.5);
    \draw [line width =0.8pt,decoration={markings, mark=at position 0.8 with {\arrow{>}}},postaction={decorate}] (0,1.2)--(1.5,0.3);
    \draw [line width =0.8pt,decoration={markings, mark=at position 0.8 with {\arrow{>}}},postaction={decorate}] (1.5,1.2)--(0,0.3);

    \node at(0,0.3) {$\bullet$};
    \node[left] at(0,0.3) {$3$};
    \node at(0,0.75) {$\bullet$};
    \node[left] at(0,0.75) {$2$};
    \node at(0,1.2) {$\bullet$};
    \node[left] at(0,1.2) {$1$};

    \node at(1.5,0.3) {$\bullet$};
    \node[right] at(1.5,0.3) {$1$};
    \node at(1.5,0.75) {$\bullet$};
    \node[right] at(1.5,0.75) {$2$};
    \node at(1.5,1.2) {$\bullet$};
    \node[right] at(1.5,1.2) {$3$};
    \node at(1.1,1.2) {$k$};
     \node at(0.4,1.2) {$k$};

    \end{tikzpicture}
    } \right) \cdot
    W\left(
\raisebox{-0.25in}{
    \begin{tikzpicture}
    \tikzset{->-/.style={
        decoration={markings,mark=at position #1 with {\arrow{latex}}},
        postaction={decorate}
    }}

    \filldraw[draw=white,fill=gray!20] (0,0) rectangle (1.5, 1.5);
    \draw[line width=1pt] (0,0)--(0,1.5);
    \draw[line width=1pt] (1.5,0)--(1.5,1.5);
    \draw [line width =0.8pt,decoration={markings, mark=at position 0.8 with {\arrow{>}}},postaction={decorate}] (1.5,1.2)--(0,0.3);

    \node at(0,0.3) {$\bullet$};
    \node[left] at(0,0.3) {$3$};
    \node at(0,0.75) {$\bullet$};
    \node[left] at(0,0.75) {$2$};
    \node at(0,1.2) {$\bullet$};
    \node[left] at(0,1.2) {$1$};

    \node at(1.5,0.3) {$\bullet$};
    \node[right] at(1.5,0.3) {$1$};
    \node at(1.5,0.75) {$\bullet$};
    \node[right] at(1.5,0.75) {$2$};
    \node at(1.5,1.2) {$\bullet$};
    \node[right] at(1.5,1.2) {$3$};
    \node at(1.1,1.2) {$t$};

    \end{tikzpicture}
    } \right)
    \overset{\omega}{=}
     W\left(
\raisebox{-0.25in}{
    \begin{tikzpicture}
    \tikzset{->-/.style={
        decoration={markings,mark=at position #1 with {\arrow{latex}}},
        postaction={decorate}
    }}

    \filldraw[draw=white,fill=gray!20] (0,0) rectangle (1.5, 1.5);
    \draw[line width=1pt] (0,0)--(0,1.5);
    \draw[line width=1pt] (1.5,0)--(1.5,1.5);
    \draw [line width =0.8pt,decoration={markings, mark=at position 0.8 with {\arrow{>}}},postaction={decorate}] (0,1.2)--(1.5,0.3);
    \draw [line width =0.8pt,decoration={markings, mark=at position 0.8 with {\arrow{>}}},postaction={decorate}] (1.5,1.2)--(0,0.3);

    \node at(0,0.3) {$\bullet$};
    \node[left] at(0,0.3) {$3$};
    \node at(0,0.75) {$\bullet$};
    \node[left] at(0,0.75) {$2$};
    \node at(0,1.2) {$\bullet$};
    \node[left] at(0,1.2) {$1$};

    \node at(1.5,0.3) {$\bullet$};
    \node[right] at(1.5,0.3) {$1$};
    \node at(1.5,0.75) {$\bullet$};
    \node[right] at(1.5,0.75) {$2$};
    \node at(1.5,1.2) {$\bullet$};
    \node[right] at(1.5,1.2) {$3$};
    \node at (0.55,0.3)  {$k+t$};   
     \node at(0.4,1.2) {$k$};

    \end{tikzpicture}
    } \right)$

\item $
    W\left(
\raisebox{-0.25in}{
    \begin{tikzpicture}
    \tikzset{->-/.style={
        decoration={markings,mark=at position #1 with {\arrow{latex}}},
        postaction={decorate}
    }}

    \filldraw[draw=white,fill=gray!20] (0,0) rectangle (1.5, 1.5);
    \draw[line width=1pt] (0,0)--(0,1.5);
    \draw[line width=1pt] (1.5,0)--(1.5,1.5);
    \draw [line width =0.8pt,decoration={markings, mark=at position 0.8 with {\arrow{>}}},postaction={decorate}] (0,1.2)--(1.5,0.3);
    \draw [line width =0.8pt,decoration={markings, mark=at position 0.8 with {\arrow{>}}},postaction={decorate}] (1.5,1.2)--(0,0.3);

    \node at(0,0.3) {$\bullet$};
    \node[left] at(0,0.3) {$3$};
    \node at(0,0.75) {$\bullet$};
    \node[left] at(0,0.75) {$2$};
    \node at(0,1.2) {$\bullet$};
    \node[left] at(0,1.2) {$1$};

    \node at(1.5,0.3) {$\bullet$};
    \node[right] at(1.5,0.3) {$1$};
    \node at(1.5,0.75) {$\bullet$};
    \node[right] at(1.5,0.75) {$2$};
    \node at(1.5,1.2) {$\bullet$};
    \node[right] at(1.5,1.2) {$3$};
    \node at(1.1,1.2) {$k$};
     \node at(0.4,1.2) {$k$};

    \end{tikzpicture}
    } \right)\cdot 
    W\left(
\raisebox{-0.25in}{
    \begin{tikzpicture}
    \tikzset{->-/.style={
        decoration={markings,mark=at position #1 with {\arrow{latex}}},
        postaction={decorate}
    }}

    \filldraw[draw=white,fill=gray!20] (0,0) rectangle (1.5, 1.5);
    \draw[line width=1pt] (0,0)--(0,1.5);
    \draw[line width=1pt] (1.5,0)--(1.5,1.5);
      \draw [line width =0.8pt,decoration={markings, mark=at position 0.8 with {\arrow{>}}},postaction={decorate}] (0,1.2)--(1.5,0.3);

    \node at(0,0.3) {$\bullet$};
    \node[left] at(0,0.3) {$3$};
    \node at(0,0.75) {$\bullet$};
    \node[left] at(0,0.75) {$2$};
    \node at(0,1.2) {$\bullet$};
    \node[left] at(0,1.2) {$1$};

    \node at(1.5,0.3) {$\bullet$};
    \node[right] at(1.5,0.3) {$1$};
    \node at(1.5,0.75) {$\bullet$};
    \node[right] at(1.5,0.75) {$2$};
    \node at(1.5,1.2) {$\bullet$};
    \node[right] at(1.5,1.2) {$3$};
    \node at(0.4,1.2) {$t$};

    \end{tikzpicture}
    } \right)
    \overset{\omega}{=}
     W\left(
\raisebox{-0.25in}{
    \begin{tikzpicture}
    \tikzset{->-/.style={
        decoration={markings,mark=at position #1 with {\arrow{latex}}},
        postaction={decorate}
    }}

    \filldraw[draw=white,fill=gray!20] (0,0) rectangle (1.5, 1.5);
    \draw[line width=1pt] (0,0)--(0,1.5);
    \draw[line width=1pt] (1.5,0)--(1.5,1.5);
    \draw [line width =0.8pt,decoration={markings, mark=at position 0.8 with {\arrow{>}}},postaction={decorate}] (0,1.2)--(1.5,0.3);
    \draw [line width =0.8pt,decoration={markings, mark=at position 0.8 with {\arrow{>}}},postaction={decorate}] (1.5,1.2)--(0,0.3);

    \node at(0,0.3) {$\bullet$};
    \node[left] at(0,0.3) {$3$};
    \node at(0,0.75) {$\bullet$};
    \node[left] at(0,0.75) {$2$};
    \node at(0,1.2) {$\bullet$};
    \node[left] at(0,1.2) {$1$};

    \node at(1.5,0.3) {$\bullet$};
    \node[right] at(1.5,0.3) {$1$};
    \node at(1.5,0.75) {$\bullet$};
    \node[right] at(1.5,0.75) {$2$};
    \node at(1.5,1.2) {$\bullet$};
    \node[right] at(1.5,1.2) {$3$};
     \node at(1.1,1.2) {$k$};
    \node at (0.95,0.3)  {$k+t$};   
    \end{tikzpicture}
    } \right)$
\end{enumerate}
\end{lemma}

\begin{proof}
    (a) follows from repeated applications of Lemma~\ref{lem-basis1}(b) and Lemma~\ref{lem-basis7}.

The same argument used in (a) also applies to (b).
\end{proof}

\begin{remark}\label{rem-basis}
Let $W$ be a stated web diagram in $\mathbb{P}_2$. We denote by $\cev{W}$ the stated web diagram obtained from $W$ by reversing all its orientations. It is shown in \cite[Corollary~4.8]{LS21} that there exists an algebra isomorphism from $\cS_\omega(\mathbb{P}_2)$ to itself sending each stated web diagram $W$ to $\cev{W}$. 

This isomorphism implies that Lemmas~\ref{lem-basis6} and \ref{lem-basis8} remain valid after reversing the orientations of the stated web diagrams.
\end{remark}

\def\WS{\mathcal {WS}}

For each labeled arc $C$, we denote by $\widetilde{C}$ the stated corner in $\mathbb P_2$ such that 
$\langle \widetilde{C} \rangle = C$.

For each weighted system $S$, we define an associated cluster monomial  as follows.
\begin{itemize}
    \item If every pair of elements in $S$ is strongly compatible, define 
    \[
    A(S):= \{\widetilde{C}\mid C\in S \}.
    \]
    Then define 
    \[
    f_S\colon A(S)\rightarrow \mathbb N,\qquad A \mapsto w(\langle A \rangle),
    \]
    where $w(\langle A \rangle)$ denotes the weight of $\langle A \rangle$ in $S$.

    \item If $S$ contains $\langle b_{11} \rangle$ and $\langle \cev{b}_{33} \rangle$, define 
    \[
    A(S):=
    \begin{cases}
        \{\widetilde{C}\mid C\in S\setminus\{\langle b_{11} \rangle, \langle \cev{b}_{33} \rangle\}\}
        \sqcup \{[\alpha]_{\rm norm}\} \sqcup \{b_{11}\}, 
        & w(\langle b_{11} \rangle)\geq w(\langle \cev{b}_{33} \rangle),\\[6pt]
        \{\widetilde{C}\mid C\in S\setminus\{\langle b_{11} \rangle, \langle \cev{b}_{33} \rangle\}\}
        \sqcup \{[\alpha]_{\rm norm}\} \sqcup \{\cev{b}_{33}\}, 
        & w(\langle b_{11} \rangle)< w(\langle \cev{b}_{33} \rangle).
    \end{cases}
    \]
    Then define 
    \[
    f_S\colon A(S)\rightarrow \mathbb N
    \]
    by
    \[
    f_S(A)=
    \begin{cases}
        w(\langle A\rangle), 
        & \langle A\rangle \in S\setminus\{\langle b_{11} \rangle, \langle \cev{b}_{33} \rangle\},\\[4pt]
        \min\{w(\langle b_{11} \rangle),\, w(\langle \cev{b}_{33} \rangle)\}, 
        & A=[\alpha]_{\rm norm},\\[4pt]
        w(\langle b_{11} \rangle)- w(\langle \cev{b}_{33} \rangle), 
        & A=b_{11},\\[4pt]
        w(\langle \cev{b}_{33} \rangle)- w(\langle b_{11} \rangle), 
        & A=\cev{b}_{33}.
    \end{cases}
    \]

    \item If $S$ contains $\langle \cev{b}_{11} \rangle$ and $\langle b_{33} \rangle$, define 
    \[
    A(S):=
    \begin{cases}
        \{\widetilde{C}\mid C\in S\setminus\{\langle \cev{b}_{11} \rangle, \langle b_{33} \rangle\}\}
        \sqcup \{[\cev{\alpha}]_{\rm norm}\} \sqcup \{\cev{b}_{11}\}, 
        & w(\langle \cev{b}_{11} \rangle)\geq w(\langle b_{33} \rangle),\\[6pt]
        \{\widetilde{C}\mid C\in S\setminus\{\langle \cev{b}_{11} \rangle, \langle b_{33} \rangle\}\}
        \sqcup \{[\cev{\alpha}]_{\rm norm}\} \sqcup \{b_{33}\}, 
        & w(\langle \cev{b}_{11} \rangle)< w(\langle b_{33} \rangle).
    \end{cases}
    \]
    Then define 
    \[
    f_S\colon A(S)\rightarrow \mathbb N
    \]
    by
    \[
    f_S(A)=
    \begin{cases}
        w(\langle A\rangle), 
        & \langle A\rangle \in S\setminus\{\langle \cev{b}_{11} \rangle, \langle b_{33} \rangle\},\\[4pt]
        \min\{w(\langle \cev{b}_{11} \rangle),\, w(\langle b_{33} \rangle)\}, 
        & A=[\cev{\alpha}]_{\rm norm},\\[4pt]
        w(\langle \cev{b}_{11} \rangle)- w(\langle b_{33} \rangle), 
        & A=\cev{b}_{11},\\[4pt]
        w(\langle b_{33} \rangle)- w(\langle \cev{b}_{11} \rangle), 
        & A=b_{33}.
    \end{cases}
    \]
\end{itemize}

By Theorem~\ref{thm:webint}, the set $A(S)$ forms a cluster in $\cS_\omega (\mathbb P_2)$. 
We define the cluster monomial associated to $S$ by
\[
M(S):= \left[\prod_{A\in A(S)}  A^{f_S(A)}\right].
\]

We use $\mathcal{WS}$ to denote the set of all weighted systems in $\mathbb{P}_2$.
 Theorem~\ref{thm:webint} immediately implies the following.

 \begin{lemma}\label{lem-bij-basis}
     The map $M\colon \WS \rightarrow \{\text{cluster monomials in ${\mathscr{A}}_\omega(\mathbb P_2)=\cS_\omega(\mathbb P_2)$}\}$ is a bijection.
 \end{lemma}

The following proposition shows that $W(S)$ is reflection-normalizable (see \S\ref{sub-sec-invariant}) and that $[W(S)]_{\rm norm} = M(S)$. This constitutes the final and crucial step in providing a web interpretation of the dual canonical basis of ${\cS}_{\omega}(\mathbb{P}_{2})=\mathcal O_q({\rm SL}_3)$.

\begin{proposition}\label{prop-MS-cluster-monomial}
For each $S\in \WS$, the following statements hold:
\begin{enumerate}[label={\rm (\alph*)}, itemsep=0.3em]
    \item $M(S) \overset{\omega}{=} W(S)$.
    \item $W(S)$ is reflection-normalizable, and $[W(S)]_{\rm norm} = M(S)$.
\end{enumerate}
\end{proposition}

\begin{proof}
    (a) Note that $\{b_{31}, \cev{b}_{31}, b_{13}, \cev{b}_{13}\}\subset A(S)$. 
We introduce a linear order on $A(S)\setminus \{b_{31}, \cev{b}_{31}, b_{13}, \cev{b}_{13}\}$ in Appendix~\ref{app-order}. 
For any $A\in A(S)\setminus \{b_{31}, \cev{b}_{31}, b_{13}, \cev{b}_{13}\}$, we impose
\[
b_{31}> \cev{b}_{31}> A > b_{13} > \cev{b}_{13}.
\]
This defines a linear order on $A(S)$. Accordingly, the product $\prod_{A\in A(S)} A^{f_S(A)}$ is taken with respect to this order from left to right, that is,
\[
\prod_{A\in A(S)} A^{f_S(A)} 
= A_1^{f_S(A_1)} A_2^{f_S(A_2)} \cdots A_8^{f_S(A_8)},
\]
where $A_1 > A_2 > \cdots > A_8$.

By the definition of this linear order, together with Lemmas~\ref{lem-basis1}(a), \ref{lem-basis6}, \ref{lem-basis8}, and Remark~\ref{rem-basis}, we obtain
\[
\prod_{A\in A(S)} A^{f_S(A)} \overset{\omega}{=} W(S).
\]
Hence, $M(S) \overset{\omega}{=} W(S)$.

    (b) It is well-known that the cluster monomial $M(S)$ is reflection invariant (see \S\ref{sub-sec-invariant}).
Then  (a) implies that $ W(S)$ 
is reflection-normalizable and $[W(S)]_{\rm norm} = M(S)$.
\end{proof}

\begin{remark}
Although we fixed a linear order on $A(S)$ in the proof of Proposition~\ref{prop-MS-cluster-monomial}(a), such an order is not unique; other linear orders on $A(S)$ can also be used to prove Proposition~\ref{prop-MS-cluster-monomial}(a).
\end{remark}

\begin{remark}
Although the definition of $W(S)$ depends on the choice of orientation of the circle in Figure~\ref{fig:moves}, Proposition~\ref{prop-MS-cluster-monomial}(b) implies that $[W(S)]_{\rm norm}$ is independent of this choice.
\end{remark}

The following theorem is the second main result of this section. It shows that 
$$\{[W(S)]_{\rm norm} \mid S \in \mathcal{WS}\}$$ forms the dual canonical basis of ${\cS}_{\omega}(\mathbb{P}_{2})=\mathcal O_q({\rm SL}_3)$, thereby providing a web interpretation of this basis.

\begin{theorem}\label{thm-basis-bigon}
The set $\{[W(S)]_{\rm norm} \mid S \in \mathcal{WS}\}$ is the dual canonical basis of ${\cS}_{\omega}(\mathbb{P}_{2})=\mathcal O_q({\rm SL}_3)$.
\end{theorem}
\begin{proof}
   It is well known that the cluster monomials in  ${\mathscr{A}}_\omega(\mathbb P_2)=\cS_\omega(\mathbb P_2)$
  form the dual canonical basis of $\mathcal O_q({\rm SL}_3)$ (cf. \cite[Theorem 4.26]{BR}, noting that the definition of (quantum) cluster monomials may differ slightly). Then the statement follows from Lemma~\ref{lem-bij-basis} and Proposition~\ref{prop-MS-cluster-monomial}.
\end{proof}

\begin{remark}\label{rem-Fro}
When the ground ring $R$ is the complex field $\mathbb{C}$ and $\omega$ is a nonzero complex number, we write 
$\cS_\omega(\fS;\mathbb{C})$ for $\cS_\omega(\fS)$.
For general $n$, suppose that $\omega$ is a root of unity such that the order of $\omega^{2n^2}$ is $N$, and set $\eta=\omega^{N^2}$. 
Let $\fS$ be a pb surface such that each connected component of $\fS$ has nonempty boundary.
There is an algebra embedding \cite{kim2025frobenius}, called the {\bf Frobenius map},
\[
\Phi\colon \cS_\eta(\fS;\mathbb{C})\longrightarrow \cS_\omega(\fS;\mathbb{C}),
\]
which sends each stated crossingless arc diagram $\alpha$ to $\omega^{\frac{k}{2}} \alpha^N$, where $k$ is an integer depending on $\alpha$. In particular, $k=0$ when $\alpha$ is a stated essential arc, i.e., when $\alpha$ connects two distinct boundary components of $\fS$.

In \cite{kim2025frobenius}, the authors compute $\Phi(W)$ when the stated $n$-web $W$ consists only of arcs and loops.
It is a difficult and important problem to compute $\Phi(W)$ for a general stated $n$-web.
Very few examples are known for $\Phi(W)$ when $W$ contains sinks or sources, even in the case $n=3$.

Theorem~\ref{thm-basis-bigon} determines $\Phi(W(S))$ for each weighted system $S\in\WS$.
For $S\in\WS$, let $NS$ denote the weighted system obtained from $S$ by multiplying the weight of each labeled arc by $N$ (while keeping the underlying system unchanged).
It is well known that $\Phi$ sends each cluster monomial to its $N$-th power. Hence, Theorem~\ref{thm-basis-bigon} implies
\[
\Phi(W(S))=W(NS).
\]
\end{remark}

\newpage

\begin{appendices}
\section{Proofs of equalities \eqref{eq-Q2} and \eqref{eq-Q3}}\label{app-proof}

\begin{proof}[Proof of \eqref{eq-Q2} and \eqref{eq-Q3}]
For any $l > j$ and $1 \leq k < j$, the following can be proved by double induction on $l$ and $k$, using \eqref{eq-Q1} and the mutation rules (it can also be verified using Keller's quiver mutation applet): In the quiver
    \[
Q_2 := \mu_{(l-1;j-1)}^r\cdots \mu_{(j;j-1)}^r(\overline Q_\lambda),
\]
we have:

\begin{itemize}
    \item If $k=j-1$, then
    \begin{align}
Q_2((l,l-j+1), v)=
\begin{cases}
    1 & \text{if } v=(l,l-j+2), (j-1,0), (l+1,l-j+1),\\
    -1 & \text{if } v=(l-1,l-j), (l+1,l-j+2),\\
    0 & \text{otherwise}.
\end{cases}
\end{align}
    \item If $1\leq k<j-1$, then
    \begin{align}\label{eq:induc}
Q_2((l,l-k), v)=
\begin{cases}
    1 & \text{if } v=(l,l-k+1), (l-1,l-k-2), (l+1,l-k),\\
    -1 & \text{if } v=(l,l-k-1),(l-1,l-k-1),(l+1,l-k+1),\\
    0 & \text{otherwise}.
\end{cases}
\end{align}
\end{itemize}
In particular, \eqref{eq-Q2} holds.

Furthermore, using \eqref{eq:induc}, for any $l>j$, we proceed by induction on $m$ with $l-j+1<m\leq l-1$. In the quiver
\[
Q_3 := \mu_{l,m-1} \cdots \mu_{l,l-j+1} \mu_{(l-1;j-1)}^r\cdots \mu_{(j;j-1)}^r(\overline Q_\lambda),
\]
we have
\begin{align}
Q_3((l,m), v)=
\begin{cases}
    1 & \text{if } v=(l,m-1), (l,m+1),\\
    -1 & \text{if } v=(l-1,m-1), (l+1,m+1),\\
    0 & \text{otherwise}.
\end{cases}
\end{align}
This establishes \eqref{eq-Q3}.
\end{proof}

\section{Proof of Lemma~\ref{lem-prime-ele}}\label{Appendix-B-frozen}

It was proved in \cite[Theorem~11.1]{LS21} that the stated ${\rm SL}_2$-skein algebra is isomorphic to the stated skein algebra introduced in \cite{le2018triangular}. 
The stated skein algebra of a pb surface $\fS$ is defined using stated tangles in $\fS \times (-1,1)$.

\subsection{Stated skein algebras}\label{sub-S-iso}
A tangle in $\fS \times (-1,1)$ is a properly embedded one-dimensional submanifold $\alpha \subset \fS \times (-1,1)$ equipped with a framing such that, at each point of $\partial \alpha$, the framing is given by the positive direction of $(-1,1)$. 
If $\alpha$ is further equipped with a map $s \colon \partial \alpha \to \{-,+\}$, then we call $\alpha$ a {\bf stated tangle}.

The \textbf{stated skein algebra} $\mathcal S_{\omega}(\fS)$ of $\fS$ is defined as the quotient of the $R$-module freely generated by isotopy classes of stated tangles in $\fS \times (-1,1)$, subject to the relations \eqref{cross}--\eqref{hight}. 
As in the stated ${\rm SL}_2$-skein algebra, the algebra structure on $\mathcal S_{\omega}(\fS)$ is given by stacking stated tangles.

\begin{equation}\label{cross}
\raisebox{-.20in}{
\begin{tikzpicture}
\filldraw[draw=white,fill=gray!20] (-0,-0.2) rectangle (1, 1.2);
\draw [line width =1pt](0.6,0.6)--(1,1);
\draw [line width =1pt](0.6,0.4)--(1,0);
\draw[line width =1pt] (0,0)--(0.4,0.4);
\draw[line width =1pt] (0,1)--(0.4,0.6);
\draw[line width =1pt] (0.6,0.6)--(0.4,0.4);
\end{tikzpicture}
}=
q^{-\frac{1}{2}}
\raisebox{-.20in}{
\begin{tikzpicture}
\filldraw[draw=white,fill=gray!20] (-0,-0.2) rectangle (1, 1.2);
 \draw[line width =1pt] plot[smooth] coordinates {(0,1) (0.4,0.5) (0,0)};
\draw[line width =1pt] plot[smooth] coordinates {(1,1) (0.6,0.5) (1,0)};
\end{tikzpicture}
}
+
 q^{\frac{1}{2}}
\raisebox{-.20in}{
\begin{tikzpicture}
\filldraw[draw=white,fill=gray!20] (-0,-0.2) rectangle (1, 1.2);
\draw[line width =1pt] plot[smooth] coordinates {(0,1) (0.5,0.6) (1,1)};
\draw[line width =1pt] plot[smooth] coordinates {(0,0) (0.5,0.4) (1,0)};
\end{tikzpicture}
}, 
\end{equation}
\begin{equation}\label{unknot}
\raisebox{-.15in}{
\begin{tikzpicture}
\filldraw[draw=white,fill=gray!20] (-0,-0) rectangle (1, 1);
\draw [line width =1pt] (0.5,0.5) circle (0.3);
\end{tikzpicture}
}=-(q+q^{-1})
\raisebox{-.15in}{
\begin{tikzpicture}
\filldraw[draw=white,fill=gray!20] (-0,-0) rectangle (1, 1);
\end{tikzpicture}
},
\end{equation}
\begin{equation}\label{arc}
\raisebox{-.20in}{
\begin{tikzpicture}
\filldraw[draw=white,fill=gray!20] (0,0) rectangle (1.2, 0.6);
\draw [line width =1.5pt,decoration={markings, mark=at position 0.9 with {\arrow{>}}},postaction={decorate}](0,-0)--(1.2,0);
\draw [line width =1pt]  (0.8 ,0) arc (0:180:0.3 and 0.35);
\node[below] at (0.8 ,0) {\small $+$};
\node[below] at (0.2 ,0) {\small $-$};
\end{tikzpicture}
}=q^{\frac{1}{4}}
\raisebox{-.10in}{
\begin{tikzpicture}
\filldraw[draw=white,fill=gray!20] (0,0) rectangle (1.2, 0.6);
\draw [line width =1.5pt](0,-0)--(1.2,0);
\end{tikzpicture}
},\;\;
\raisebox{-.20in}{
\begin{tikzpicture}
\filldraw[draw=white,fill=gray!20] (0,0) rectangle (1.2, 0.6);
\draw [line width =1.5pt,decoration={markings, mark=at position 0.9 with {\arrow{>}}},postaction={decorate}](0,-0)--(1.2,0);
\draw [line width =1pt]  (0.8 ,0) arc (0:180:0.3 and 0.35);
\node[below] at (0.8 ,0) {\small $+$};
\node[below] at (0.2 ,0) {\small $+$};
\end{tikzpicture}
}=
\raisebox{-.20in}{
\begin{tikzpicture}
\filldraw[draw=white,fill=gray!20] (0,0) rectangle (1.2, 0.6);
\draw [line width =1.5pt,decoration={markings, mark=at position 0.9 with {\arrow{>}}},postaction={decorate}](0,-0)--(1.2,0);
\draw [line width =1pt]  (0.8 ,0) arc (0:180:0.3 and 0.35);
\node[below] at (0.8 ,0) {\small $-$};
\node[below] at (0.2 ,0) {\small $-$};
\end{tikzpicture}
} =0,
\end{equation}
\begin{equation}\label{hight}
\raisebox{-.25in}{
\begin{tikzpicture}
\filldraw[draw=white,fill=gray!20] (-0.2,-0) rectangle (1.2, 0.8);
\draw [line width =1.5pt,decoration={markings, mark=at position 0.93 with {\arrow{>}}},postaction={decorate}](-0.2,-0)--(1.2,0);
\draw [line width =1pt](0.2,0)--(0.2,0.8);
\draw [line width =1pt](0.8,0)--(0.8,0.8);
\node[below] at (0.2,0) {\small $-$};
\node[below] at (0.8,0) {\small $+$};
\end{tikzpicture}
}=q
\raisebox{-.25in}{
\begin{tikzpicture}
\filldraw[draw=white,fill=gray!20] (-0.2,-0) rectangle (1.2, 0.8);
\draw [line width =1.5pt,decoration={markings, mark=at position 0.93 with {\arrow{>}}},postaction={decorate}](-0.2,-0)--(1.2,0);
\draw [line width =1pt](0.2,0)--(0.2,0.8);
\draw [line width =1pt](0.8,0)--(0.8,0.8);
\node[below] at (0.2,0) {\small $+$};
\node[below] at (0.8,0) {\small $-$};
\end{tikzpicture}
} - q^{\frac{5}{4}}
\raisebox{-.11in}{
\begin{tikzpicture}
\filldraw[draw=white,fill=gray!20] (-0.2,-0) rectangle (1.2, 0.8); 
\draw [line width =1.5pt](-0.2,-0)--(1.2,0);
 \draw[line width =1pt] plot[smooth] coordinates {(0,0.8) (0.5, 0.3) (1,0.8)};
\end{tikzpicture}
}.
\end{equation}

The isomorphism constructed in \cite[Theorem~11.1]{LS21} between $\mathcal S_{\omega}(\fS)$ and $\cS_\omega(\fS)$ for $n=2$ sends each stated tangle $\alpha$ in $\fS \times (-1,1)$ to a stated $2$-web in $\fS \times (-1,1)$ by assigning an arbitrary orientation to $\alpha$ and replacing each minus state by $1$ and each plus state by $2$.

Note that our $q^{\frac{1}{4}}$ corresponds to $q^{-\frac{1}{2}}$ in \cite{le2018triangular,LY22}.
We have the following lemma.

\begin{lemma}[Height exchange relations, Lemma 2.4 of \cite{le2018triangular}]\label{B-height-ex}
For $v\in \{\pm 1\}$, we have
    $$
    \raisebox{-.20in}{
\begin{tikzpicture}
\filldraw[draw=white,fill=gray!20] (-0.2,-0) rectangle (1.2, 0.8);
\draw [line width =1.5pt,decoration={markings, mark=at position 0.19 with {\arrow{<}}},postaction={decorate}](-0.2,-0)--(1.2,0);
\draw [line width =1pt](0.2,0)--(0.2,0.8);
\draw [line width =1pt](0.8,0)--(0.8,0.8);
\node[below] at (0.2,0) {\small $v$};
\node[below] at (0.8,0) {\small $+$};
\end{tikzpicture}
} = q^{\frac{v}{2}}
\raisebox{-.20in}{
\begin{tikzpicture}
\filldraw[draw=white,fill=gray!20] (-0.2,-0) rectangle (1.2, 0.8);
\draw [line width =1.5pt,decoration={markings, mark=at position 0.93 with {\arrow{>}}},postaction={decorate}](-0.2,-0)--(1.2,0);
\draw [line width =1pt](0.2,0)--(0.2,0.8);
\draw [line width =1pt](0.8,0)--(0.8,0.8);
\node[below] at (0.2,0) {\small $v$};
\node[below] at (0.8,0) {\small $+$};
\end{tikzpicture}
},\qquad
    \raisebox{-.20in}{
\begin{tikzpicture}
\filldraw[draw=white,fill=gray!20] (-0.2,-0) rectangle (1.2, 0.8);
\draw [line width =1.5pt,decoration={markings, mark=at position 0.19 with {\arrow{<}}},postaction={decorate}](-0.2,-0)--(1.2,0);
\draw [line width =1pt](0.2,0)--(0.2,0.8);
\draw [line width =1pt](0.8,0)--(0.8,0.8);
\node[below] at (0.2,0) {\small $-$};
\node[below] at (0.8,0) {\small $v$};
\end{tikzpicture}
} = q^{-\frac{v}{2}}
\raisebox{-.20in}{
\begin{tikzpicture}
\filldraw[draw=white,fill=gray!20] (-0.2,-0) rectangle (1.2, 0.8);
\draw [line width =1.5pt,decoration={markings, mark=at position 0.93 with {\arrow{>}}},postaction={decorate}](-0.2,-0)--(1.2,0);
\draw [line width =1pt](0.2,0)--(0.2,0.8);
\draw [line width =1pt](0.8,0)--(0.8,0.8);
\node[below] at (0.2,0) {\small $-$};
\node[below] at (0.8,0) {\small $v$};
\end{tikzpicture}
},
    $$
    where we identify \(+\) with \(+1\) and \(-\) with \(-1\).
\end{lemma}

A {\bf simple arc} in $\fS$ is a proper embedding 
$\alpha \colon [0,1] \to \fS$.
We call $\alpha$ a {\bf trivial arc} if it is relatively homotopic to an embedded interval in $\partial \fS$.

A {\bf simple loop} in $\fS$ is an embedded circle in the interior of $\fS$. 
A simple loop is called {\bf trivial} if it bounds an embedded disk in $\fS$.

A {\bf multicurve} $\beta$ in $\fS$ is a finite collection of pairwise disjoint simple arcs and simple loops containing no trivial arc or loop.

For each boundary component $c$ of $\fS$, let $\partial_c \beta$ denote $(\partial \beta)\cap c$. 
We call $\beta$ an $\fS$-tangle diagram if, for every boundary component $c$ of $\fS$, a linear order $\le_c$ on $\partial_c \beta$ is specified. 
A state $s$ of $\beta$ is called {\bf increasingly stated} if 
$s(E_1) \le s(E_2)$ whenever $E_1 \le_c E_2$ for $E_1, E_2 \in \partial_c \beta$.

Recall that the orientation of $\partial \fS$ induced by the orientation of $\fS$ is the positive orientation of $\partial \fS$. 
We say that $\beta$ is {\bf positively ordered} if, for each boundary component $c$ of $\fS$, the linear order on $\partial_c \beta$ is induced by the positive orientation of $c$; that is, the order of points in $\partial_c \beta$ increases when moving in the positive orientation of $c$.

Let $B(\fS)$ denote the set of increasingly stated, positively ordered multicurves in $\fS$.  
By \cite[Theorem~2.8]{le2018triangular}, $\cS_\omega(\fS)$ is a free $R$-module with basis $B(\fS)$.

\subsection{Partially reduced stated skein algebras}

For each puncture $p$ contained in $\partial \overline{\fS}$, there exists a unique bad arc $C_p$ as shown in Figure~\ref{fig:badarc-sl2}(A). 
Let $\alpha \in B(\fS)$. Isotope $C_p$ so that it lies in a sufficiently small neighborhood of $p$ and satisfies 
$C_p \cap \alpha = \emptyset$. 
We then define $\alpha \sqcup C_p$ to be the disjoint union of $\alpha$ and $C_p$. 
Note that $\alpha \sqcup C_p$ still belongs to $B(\fS)$. 
Then we have the following lemma.

\begin{figure}[htbp]
\centering
\raisebox{-.15in}{
\begin{tikzpicture}
\filldraw[draw=white,fill=gray!20] (-1.5,0) rectangle (1.5,1);
\draw[line width=1.5pt] (-1.5,0)--(1.5,0);
\draw[line width=1pt] (1,0) arc (0:180:1 and 0.7);
\filldraw[white, draw=black] (0,0) circle (0.1);
\node at (-1,-0.2) {\small $+$};
\node at (0,-0.3) {\small $p$};
\node at (1,-0.2) {\small $-$};
\node at (0,-0.8) {\small $(A)$};
\end{tikzpicture}
}\qquad
\raisebox{-.15in}{
\begin{tikzpicture}
\filldraw[draw=white,fill=gray!20] (-1.5,0) rectangle (1.5,1);
\draw[line width=1.5pt] (-1.5,0)--(1.5,0);
\draw[line width=1pt] (1,0) arc (0:180:1 and 0.7);
\filldraw[white, draw=black] (0,0) circle (0.1);
\node at (-1,-0.2) {\small $+$};
\node at (0,-0.3) {\small $p$};
\node at (1,-0.2) {\small $+$};
\node at (0,-0.8) {\small $(B)$};
\end{tikzpicture}
}\qquad
\raisebox{-.15in}{
\begin{tikzpicture}
\filldraw[draw=white,fill=gray!20] (-1.5,0) rectangle (1.5,1);
\draw[line width=1.5pt] (-1.5,0)--(1.5,0);
\draw[line width=1pt] (1,0) arc (0:180:1 and 0.7);
\filldraw[white, draw=black] (0,0) circle (0.1);
\node at (-1,-0.2) {\small $-$};
\node at (0,-0.3) {\small $p$};
\node at (1,-0.2) {\small $-$};
\node at (0,-0.8) {\small $(C)$};
\end{tikzpicture}
}
\caption{The bad arc $C_p$ in (A), and the stated arcs $C_p(+)$ and $C_p(-)$ in (B) and (C), respectively.}
\label{fig:badarc-sl2}
\end{figure}
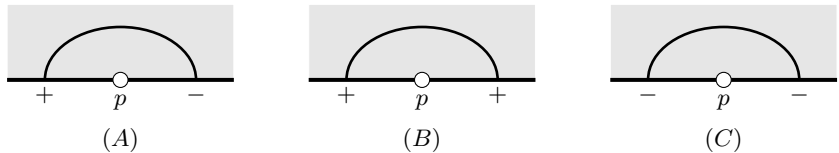

\begin{lemma}\cite[Lemma 4.4]{LY22}\label{lem-Cp}
    ($a$) $C_p \,\mathcal S_\omega (\fS) =  \mathcal S_\omega (\fS)\, C_p$.

    ($b$) Let $\alpha \in B(\fS)$. In $\mathcal S_\omega (\fS)$, we have $\alpha\, C_p \overset{\omega}{=} C_p\,\alpha \overset{\omega}{=} \alpha \sqcup C_p.$
\end{lemma}

We have the following relation in $\mathcal S_\omega (\fS)/(C_p)$, where $(C_p)=C_p \,\mathcal S_\omega (\fS) =  \mathcal S_\omega (\fS)\, C_p$.

\begin{lemma}[\cite{CLL}]\label{B-lem-prod-1}
   In $\mathcal S_\omega (\fS)/(C_p)$, we have 
   $C_p(+)C_p(-)=C_p(-)C_p(+)=1$, where $C_p(+)$ and $C_p(-)$ are stated arcs illustrated in Figure~\ref{fig:badarc-sl2}.
\end{lemma}

Let $I$ be a subset of punctures  contained in $\partial \overline{\fS}$. 
Define
\[
\mathcal S_\omega^I(\fS) := \mathcal S_\omega(\fS)/(C_p,p\in I),
\]
where $(C_p,p\in I)$ is the two sided idea of $\mathcal S_\omega(\fS)$ generated by $C_p,p\in I$.
Note that $\mathcal S_\omega^I(\fS)=\mathcal S_\omega(\fS)$ when $I=\emptyset$.
For an element $\alpha \in \mathcal S_\omega(\fS)$, we use the same notation to denote its image under the projection $\mathcal S_\omega(\fS)\to \mathcal S_\omega^I(\fS)$.

Let $B_I(\fS)$ be the subset of $B(\fS)$ consisting of increasingly stated, positively ordered multicurves in $\fS$ that contain no bad arcs $C_p$ for $p\in I$. 
Then Lemma~\ref{lem-Cp} implies the following.

\begin{lemma}\label{B-lem-basis}
    $B_I(\fS)$ is a basis of $\mathcal S_\omega^I(\fS)$.
\end{lemma}

\def\cut{\mathsf{Cut}}
\def\pr{{\bf pr}}

Let $e$ be an ideal arc of $\fS$ such that it is contained in the interior of $\fS$. After cutting $\fS$ along $e$, we get a new pb surface $\fS'$, which has two copies $e_1,e_2$ for $c$ such that 
${\fS}= \fS'/(e_1=e_2)$. We use $\pr_e$ to denote the projection from $\fS'$ to $\fS$.  Suppose that $\alpha$ is a stated tangle diagram in $\fS$, which is transverse to $e$.
Let $s$ be a map from $e\cap\alpha$ to $\{\pm 1\}$, and let $h$ be a linear order on $e\cap\alpha$. Then there is a lift stated tangle diagram $\alpha(h,s)$ in $\fS'$. 
 For $i=1,2$, the heights of the endpoints of $\alpha(h,s)$ on $e_i$ are induced by $h$ (via $\pr_e$), and the states of the endpoints of $\alpha(h,s)$ on $e_i$ are induced by $s$ (via $\pr_e$).
Then the splitting map 
$$
\Theta_e : \mathcal{S}_\omega(\fS) \to \mathcal{S}_\omega(\fS')
$$
is defined by 
\begin{align}\label{eq-def-splitting}
    \Theta_e(\alpha) =\sum_{s\colon \alpha \cap e \to \{\pm 1\}} \alpha(h, s). 
\end{align}
Furthermore $\Theta_e$ is an algebra embedding \cite{le2018triangular}.

It is well-known that \cite{le2018triangular}
\begin{align}\label{com-splitting-skein}
    \Theta_{e_1}\circ \Theta_{e_2}
    = \Theta_{e_2}\circ \Theta_{e_1}
\end{align}
for any two disjoint interior ideal arcs $e_1,e_2$ of $\fS$. 

Let $I$ be a subset of punctures  contained in $\partial \overline{\fS}$. 
Let $I'$ denote the preimage of $I$ under the projection $\pr_e$.
Then the splitting map
\[
\Theta_e \colon \mathcal S_\omega (\fS) \to \mathcal S_\omega (\fS')
\]
induces an algebra homomorphism
\[
\Theta_e \colon \mathcal S_\omega^I (\fS) \to \mathcal S_\omega^{I'} (\fS').
\]


We obtain the following result.
Since the exact same technique as in \cite[Theorem~7.6]{CLL} applies here, we omit the detailed proof; see Remark~\ref{proof1} for a brief outline.

\begin{lemma}\label{lem-domain4}
The algebra homomorphism
\[
\Theta_e\colon \mathcal S_\omega^I(\fS)\rightarrow \mathcal S_\omega^{I'}(\fS')
\]
is an embedding.
\end{lemma}
For each \(\alpha \in B_I(\fS)\), let 
\begin{align}\label{ap-B-d}
    \text{\(d(\alpha)\):= the number of endpoints of \(\alpha\) lying on \(\partial \fS\).}
\end{align}
For each integer \(k\), let
\begin{align}\label{ap-B-Dk}
    D_k(\mathcal S_\omega^I(\fS)):=
    \text{\(R\)-submodule of \(\mathcal S_\omega^I(\fS)\) spanned by all
\(\alpha \in B_I(\fS)\) satisfying \(d(\alpha)\le k\).}
\end{align}
Define
\[
G_k(\mathcal S_\omega^I(\fS))
:=
D_k(\mathcal S_\omega^I(\fS))
/
D_{k-1}(\mathcal S_\omega^I(\fS)).
\]

\begin{remark}\label{proof1}

Let \(0 \neq x \in \mathcal S_\omega^I(\fS)\). We want to show that
\(\Theta_e(x) \neq 0\).
Since \(B_I(\fS)\) is an \(R\)-basis, there exists a nonempty finite subset
\(S \subset B_I(\fS)\) such that
\begin{equation*}
x=\sum_{\alpha \in S} c_\alpha \alpha,
\qquad 0 \neq c_\alpha \in R .
\end{equation*}
Let $t=\max_{\alpha \in S} i(\alpha,e),$
where \(i(\alpha,e)\) denotes the algebraic intersection number between \(\alpha\) and \(e\).
Then the subset
$S' := \{\alpha \in S \mid i(\alpha,e)=t\}$
is nonempty.

For \(\alpha \in S'\), define
$k=d(\alpha)+t.$
Let
$P \colon
D_k(\mathcal S_\omega^I(\fS'))
\to
G_k(\mathcal S_\omega^I(\fS'))$
denote the natural projection.
The proof of \cite[Theorem~7.6]{CLL} then divides into two cases:

\begin{itemize}
    \item[(1)]
    There exists \(\alpha \in S'\) such that
    $P\bigl(\Theta_e(\alpha)\bigr)\neq 0.$

    \item[(2)]
    For all \(\alpha \in S'\), we have
    $P\bigl(\Theta_e(\alpha)\bigr)=0.$
\end{itemize}

Suppose that \(e\) connects two punctures \(p_1,p_2\) (possibly \(p_1=p_2\)).

If \(p_1,p_2\in I\), then the exact proof of
\cite[Theorem~7.6]{CLL}
applies to Lemma~\ref{lem-domain4}.

If exactly one of \(p_1,p_2\) lies in \(I\), then
\cite[Lemma~7.7]{CLL} implies that, for every
\(\alpha \in S'\),
$P\bigl(\Theta_e(\alpha)\bigr)\neq 0.$
Then Case~(1) applies, and the proof of
\cite[Theorem~7.6]{CLL}
works without modification.

Finally, if neither \(p_1\) nor \(p_2\) lies in \(I\), then it is immediate that, for every
\(\alpha \in S'\),
$P\bigl(\Theta_e(\alpha)\bigr)\neq 0.$
Then Case~(1) again applies, and the proof of
\cite[Theorem~7.6]{CLL}
carries over verbatim.

\end{remark}

\subsection{Quantum trace maps}

Let $\fS$ be a triangulable pb surface with a triangulation $\lambda$. 
Suppose that $\fS$ contains no interior punctures. 
Recall that we use $\partial\lambda$ denote the set of all the boundary edges in $\lambda$.
Let $\partial\widehat{\lambda} = \{\hat{e} \mid e \in \partial\lambda\}$ be another copy of the set $\partial\lambda$, and set $\widehat{\lambda} = \lambda \sqcup {\partial\widehat\lambda}$. 

An ideal {\bf multiarc} in $\fS$ is a finite collection of disjoint ideal arcs. 
For each ideal multiarc $\alpha$, let $D(\alpha)$ be the simple tangle diagram obtained from $\alpha$ by slightly shifting all strands incident to each puncture $p$ so that they end on the boundary edge immediately to the left of $p$ (see \cite[Figure~2]{LY22}), and then imposing the positive order on each boundary edge.
Note when $e \in \partial \lambda$, $D(e)$ is a corner arc (see Figure~\ref{fig:Dee}).


For each $e \in \widehat{\lambda}$, we review $X_e \in \mathcal S_\omega(\fS)$ introduced in \cite{LY22}. 
For $e \in \lambda$, set
\[
X_e =
\begin{cases}
D(e)(+,+), & \text{if the endpoints of $e$ are distinct},\\
q^{\frac{1}{4}} D(e)(+,+), & \text{if the endpoints of $e$ coincide},
\end{cases}
\]
where $D(e)(+,+)$ denotes $D(e)$ equipped with positive states at both endpoints.

For $\hat e \in\partial \widehat{\lambda}$, define $X_{\hat{e}} \in \mathcal S_\omega(\fS)$ by
\[
X_{\hat{e}} =
\begin{cases}
D(e)(+,-), & \text{if the endpoints of $e$ are distinct},\\
q^{-\frac{1}{4}} D(e)(+,-), & \text{if the endpoints of $e$ coincide},
\end{cases}
\]
where $D(e)(+,-)$ denotes $D(e)$ with states assigned as in Figure~\ref{fig:Dee}. 
Then we have the following.

\begin{figure}
    \centering
    \includegraphics[width=0.7\linewidth]{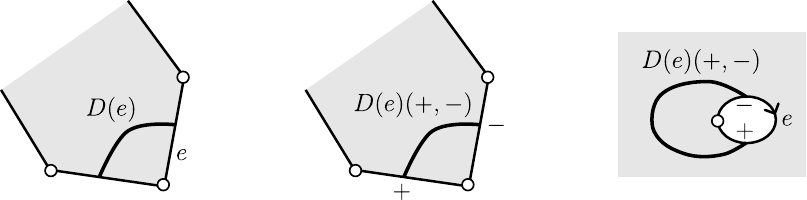}
    \caption{Left: the boundary edge $e$ and the corner arc $D(e)$. 
Middle: the bad arc $D(e)(+,-)$. 
Right: the bad arc $D(e)(+,-)$ when the two endpoints of $e$ coincide.}
    \label{fig:Dee}
\end{figure}

\begin{lemma}[\cite{LY22}]\label{lem-B-XXX}
    For each pair $a,b \in \widehat{\lambda}$, there exists an integer $\bar P(a,b)$ such that 
    $$X_a X_b = q^{\frac{1}{2}\bar P(a,b)} X_b X_a\in \mathcal S_\omega(\fS).$$
\end{lemma}

Note that our $\bar P(a,b)$ equals $-\bar P(a,b)$ in \cite{LY22}.

For each puncture $p$, there exists
a unique boundary edge $e(p)\in\lambda$ such that $D(e(p))$ is the corner arc surrounding the puncture $p$. 
Let $I$ be a subset of punctures of $\fS$.
Define the anti-symmetric matrix
$$P\colon \lambda_I:=\widehat\lambda\setminus \{\widehat{e(p)}\mid p\in I\}\; \times \; \lambda_I
\rightarrow\mathbb Z$$
to be the restriction of $\bar P$. 
Define the quantum torus
\begin{equation*}
\mathbb T_\omega(\fS,\lambda;I) = R \langle 
x_a^{\pm 1}, a \in \lambda_I \rangle \Big/ (
x_a 
x_b= q^{\frac{1}{2} P(a,b)} 
x_b x_a \text{ for } a,b\in \lambda_I ).
\end{equation*}

Lemma~\ref{lem-B-XXX} yields a well-defined algebra homomorphism
\[
\iota : \mathbb T_\omega^+(\fS,\lambda;I) \to \mathcal{S}_\omega^I(\fS), \qquad \iota(x_e)=X_e.
\]

\begin{lemma}\label{B-lem-inj-iota}
The algebra homomorphism
    $\iota : \mathbb T_\omega^+(\fS,\lambda;I) \to \mathcal{S}_\omega^I(\fS)$ is injective.
\end{lemma}

\begin{proof}
Here we use the technique in the proof of \cite[Theorem 4.1]{LY22}.

It suffices to prove that \( \iota \) maps the \( R\)-basis \( \{x^{\mathbf{k}} \mid \mathbf{k} \in \mathbb{N}^{\lambda_I} \} \) of \(\mathbb T_\omega^+(\fS,\lambda;I) \) injectively into a subset of an \( R\)-basis of \( \mathcal{S}_\omega^I(S) \).

Recall that we use $\partial\lambda$ to denote the set of all the boundary edges in $\lambda$. Define $\mathring{\lambda}=\lambda\setminus\partial\lambda$.
For \( \mathbf{k} \in \mathbb{N}^{\lambda_I} \), define \( \mathbf{k}' \in \mathbb{N}^\lambda \) by
\[
\mathbf{k}'(e) =
\begin{cases}
\mathbf{k}(a), & a \in \mathring{\lambda}\sqcup\{e(p)\mid p\in I\},\\
\mathbf{k}(a) + \mathbf{k}(\hat{a}), & a \in \partial\lambda\setminus \{e(p)\mid p\in I\}.
\end{cases}
\]
Let \( C_{\mathbf{k}} \) be the ideal multiarc consisting of \( \mathbf{k}'(a) \) parallel copies of \( a \) for each \( a \in \lambda \). Let \( D(C_{\bf k};\mathbf{k}) \) be the tangle diagram \( D(C_{\bf k}) \) equipped with increasing states such that, on each boundary edge \( a \in \partial\lambda\setminus\{e(p)\mid p\in I\} \), exactly \( \mathbf{k}(\hat{a}) \) endpoints carry negative states.

By the height exchange relation in Lemma~\ref{B-height-ex}, 
\begin{equation}\label{B-eq-iota-1}
     \iota(x^{\mathbf{k}})\overset{\omega}{=}  D(C_{\bf k};\mathbf{k}).
\end{equation}
Note that \( D(C_{\bf k};\mathbf{k}) \in B_I(\fS) \), and that \( D(C_{\mathbf{k}_1};\mathbf{k}_1) \neq D(C_{\mathbf{k}_2};\mathbf{k}_2) \) whenever \( \mathbf{k}_1 \neq \mathbf{k}_2 \). It follows that \( \iota \) is injective.
    
\end{proof}

Hence, we may identify \( \mathbb T_\omega^+(\fS,\lambda; I) \) with its image in \( \mathcal{S}_\omega^I(S) \) via the embedding \( \iota \).

By the same argument as in the proof of \cite[Lemma~4.6]{LY22}, we obtain the following.

\begin{lemma}\label{lem-domain1}
For any \( \alpha \in \mathcal{S}_\omega^I(S) \), there exists \( {\bf k} \in \mathbb{N}^{\lambda_I} \) such that both $x^{\bf k} \alpha$ and \(\alpha x^{\bf k} \) are contained in \(\mathbb T_\omega^+(\fS,\lambda; I) \).
Moreover, one can choose \( {\bf k} \) so that \( {\bf k}(\hat{a})=0 \) for all \( a \in \partial\lambda\setminus\{e(p)\mid p\in I\} \).
\end{lemma}

Let $p\in I$. Recall that $C_p(+)$ and $C_p(-)$ are stated arcs illustrated in Figure~\ref{fig:badarc-sl2}.

\begin{lemma}\label{lem-B-Xab}
 Let $p\in I$, and let $\alpha\in B_I(\fS)$.
 There exists a nonnegative integer $k$ such that 
 $X_{e(p)}^k\alpha\overset{\omega}{=} \beta\in \mathcal S_\omega^I(\fS)$, where $\beta\in B_I(\fS)$ contains no copy of $C_p(-)$.
\end{lemma}

\begin{proof}

Suppose that $\alpha$ contains $k$ copies of $C_p(-)$. 
Lemma~\ref{B-height-ex} implies
$\alpha\overset{\omega}{=} C_p(-)^k\beta$, where 
$\beta\in B_I(\fS)$ contains no copy of $C_p(-)$.
The definition of $X_{e(p)} $ implies $X_{e(p)} \overset{\omega}{=}
C_p(+)$. Then
Lemma~\ref{B-lem-prod-1} shows 
$X_{e(p)}^k\alpha\overset{\omega}{=} \beta\in \mathcal S_\omega^I(\fS)$.
\end{proof}

\begin{lemma}\label{lem-B-divisor}
For each puncture $p$ in $\partial\overline{\fS}$, the element 
$X_{e(p)} \in \mathcal S_\omega^I(\fS)$ is not a zero divisor.
\end{lemma}

\begin{proof}
If $p \in I$, then Lemma~\ref{B-lem-prod-1} implies that 
$X_{e(p)}$ is invertible. In particular, $X_{e(p)}$ is not a zero divisor.

Henceforth, assume that $p \notin I$. 
We only prove that $X_{e(p)}$ is not a left zero divisor, since the same argument also proves that it is not a right zero divisor.

Let $p'$ be the puncture, distinct from $p$, connected to $p$ by the boundary edge $e(p)$. 
For any $\alpha \in B_I(\fS)$, there is a unique way to resolve all crossings in the product 
$X_{e(p)}\alpha$ such that the resulting crossingless stated tangle diagram contains no trivial arc. 
Denote this stated tangle diagram by $D(\alpha,p)$. Then
\begin{equation}\label{B-eq-Xe-leading}
    X_{e(p)}\alpha \overset{\omega}{=} D(\alpha,p)
    + D_{d-1}(\mathcal S_\omega^I(\fS))\qquad\text{(see \eqref{ap-B-Dk})},
\end{equation}
where $d=d(\alpha)+2$ (see \eqref{ap-B-d}).

Since $p \notin I$, we have
$D(\alpha,p)\in B_{I\setminus\{p'\}}(\fS)$. Note that $D(\alpha,p)$ uniquely determines $\alpha$. Hence
\begin{equation}\label{B-eq-Xe-leading1}
    D(\alpha,p)\neq D(\alpha',p)
    \quad \Longrightarrow \quad
    \alpha\neq \alpha'
\end{equation}
for any $\alpha,\alpha'\in B_I(\fS)$.

We now distinguish two cases.

\noindent
\textbf{Case 1:} $p'\notin I$.
In this case, $D(\alpha,p)\in B_I(\fS)$ for every $\alpha\in B_I(\fS)$. 
Together with \eqref{B-eq-Xe-leading} and \eqref{B-eq-Xe-leading1}, this implies that $X_{e(p)}$ is not a left zero divisor.

\noindent
\textbf{Case 2:} $p'\in I$.
Suppose
$X_{e(p)}\sum_{\alpha\in S} c_\alpha \alpha =0,$
where $S$ is a finite subset of $B_I(\fS)$ and $c_\alpha\in R$ for each $\alpha\in S$.

By Lemma~\ref{lem-B-Xab}, there exists a nonnegative integer $k$ such that
\[
X_{e(p')}^k \sum_{\alpha\in S} c_\alpha \alpha
=
\sum_{\alpha\in S} c_\alpha' \alpha',
\]
where each $c_\alpha' \overset{\omega}{=} c_\alpha$ and each $\alpha'$ contains no copy of $C_{p'}(-)$.

Since $\alpha'$ contains no $C_{p'}(-)$, we have
$D(\alpha',p)\in B_I(\fS).$
Moreover, since
$X_{e(p)}X_{e(p')}
\overset{\omega}{=}
X_{e(p')}X_{e(p)},$
it follows that
$X_{e(p)}
\sum_{\alpha\in S} c_\alpha' \alpha'
=0.$

Applying \eqref{B-eq-Xe-leading} and \eqref{B-eq-Xe-leading1}, we obtain
$c_\alpha'=0
\text{ for all } \alpha\in S.$
Hence $c_\alpha=0$ for all $\alpha\in S$. Therefore $X_{e(p)}$ is not a left zero divisor.

This completes the proof.
\end{proof}

We have the following.

\begin{lemma}\label{lem-domain2}
Let $I$ be a subset of punctures of $\mathbb P_3$. Then
    $\mathcal{S}_\omega^I(\mathbb P_3)$ is a domain.
\end{lemma}

\begin{proof}
Suppose that $\alpha\beta=0$ for some 
$\alpha,\beta \in \mathcal{S}_\omega^I(\mathbb P_3)$.

Let $e_1$, $e_2$, and $e_3$ denote the three boundary edges of $\mathbb P_3$.
By Lemma~\ref{lem-domain1}, there exist nonnegative integers $k_1$, $k_2$, and $k_3$ such that
\[
X_{e_1}^{k_1} X_{e_2}^{k_2} X_{e_3}^{k_3}\alpha,
\quad
\beta X_{e_1}^{k_1} X_{e_2}^{k_2} X_{e_3}^{k_3}
\in \mathbb T_\omega^+(\mathbb P_3; I).
\]
Hence
\[
X_{e_1}^{k_1} X_{e_2}^{k_2} X_{e_3}^{k_3}
\alpha\beta
X_{e_1}^{k_1} X_{e_2}^{k_2} X_{e_3}^{k_3}
=0
\in \mathbb T_\omega^+(\mathbb P_3; I).
\]

Since $\mathbb T_\omega^+(\mathbb P_3; I)$ is a domain, it follows that
\[
X_{e_1}^{k_1} X_{e_2}^{k_2} X_{e_3}^{k_3}\alpha = 0
\qquad \text{or} \qquad
\beta X_{e_1}^{k_1} X_{e_2}^{k_2} X_{e_3}^{k_3}=0.
\]
By Lemma~\ref{lem-B-divisor}, we conclude that
\[
\alpha=0
\qquad \text{or} \qquad
\beta=0.
\]

Therefore, $\mathcal{S}_\omega^I(\mathbb P_3)$ is a domain.
\end{proof}

Since the triangulation of $\mathbb P_3$ is unique, we write $\mathbb T_\omega(\mathbb P_3; I)$ for $\mathbb T_\omega(\mathbb P_3, \lambda; I)$.

\begin{lemma}\label{lem-domain33}
Let $I$ be a set of punctures of $\mathbb P_3$.
    The algebra embedding
    $\iota : \mathbb T_\omega^+(\mathbb P_3;I) \to \mathcal{S}_\omega^I(\mathbb P_3)$  has a unique extension
    $\iota : \mathcal{S}_\omega^I(\mathbb P_3)  \to \mathbb T_\omega(\mathbb P_3;I)$ which is an algebra embedding.
\end{lemma}

\begin{proof}
    It follows from \cite[Proposition~2.2]{LY22}, Lemmas~\ref{lem-domain1}, and \ref{lem-domain2}.
\end{proof}

The following result shows that $\mathcal S_\omega^I(\fS)$ is a domain.

\begin{proposition}\label{prop-domain1}
Let $\fS$ be any pb surface, and let $I$ be a subset of the punctures contained in $\partial \overline\fS$. Then $\mathcal S_\omega^I(\fS)$ is a domain.
\end{proposition}

\begin{proof}

We first treat the case where $\fS$ is triangulable with a triangulation $\lambda$. Cutting $\fS$ along $\mathring{\lambda}$ yields a disjoint union of triangles $\bigsqcup_{\tau \in \mathbb{F}_\lambda} \tau$. Let $I'$ be the preimage of $I$, and define $I(\tau) := I' \cap \tau$. Then
\[
\mathcal S_\omega^{I'}\!\left(\bigsqcup_{\tau \in \mathbb{F}_\lambda} \tau\right)
=
\bigotimes_{\tau \in \mathbb{F}_\lambda}
\mathcal S_\omega^{I(\tau)}(\tau).
\]

Since both $\mathcal S_\omega^{I(\tau)}(\tau)$ and $\mathbb T_\omega(\tau; I(\tau))$ are flat, Lemma~\ref{lem-domain33} implies that
\[
\bigotimes_{\tau \in \mathbb{F}_\lambda}
\mathcal S_\omega^{I(\tau)}(\tau)
\xrightarrow{\ \iota\ }
\bigotimes_{\tau \in \mathbb{F}_\lambda}
\mathbb T_\omega(\tau; I(\tau))
\]
is an algebra embedding. By Lemma~\ref{lem-domain4}, the composition
\[
\mathcal S_\omega^I(\fS)
\xrightarrow{\ \Theta\ }
\bigotimes_{\tau \in \mathbb{F}_\lambda}
\mathcal S_\omega^{I(\tau)}(\tau)
\xrightarrow{\ \iota\ }
\bigotimes_{\tau \in \mathbb{F}_\lambda}
\mathbb T_\omega(\tau; I(\tau))
\]
is also an algebra embedding.

Since $\bigotimes_{\tau \in \mathbb{F}_\lambda} \mathbb T_\omega(\tau; I(\tau))$ is a quantum torus, it is a domain. Hence $\mathcal S_\omega^I(\fS)$ is a domain.

Now suppose that $\fS$ is connected and non-triangulable. Then $\fS$ is either a monogon, a bigon, a closed surface, or a sphere with one or two punctures. Results in \cite{detcherry2025embedding,le2018triangular,CLL,LY22} show that $\mathcal S_\omega^I(\fS)$ embeds into a quantum torus. Therefore, the same argument applies and $\mathcal S_\omega^I(\fS)$ is again a domain.
   
\end{proof}

\cite[Proposition~2.2]{LY22}, together with Lemma~\ref{lem-domain1} and Proposition~\ref{prop-domain1}, implies the following, which establishes the quantum trace map for $\mathcal{S}_\omega^I(\fS)$. Although this result is not needed in the present paper, it is of independent interest to record it here (see Remark~\ref{B-rem-skein-A}).

 \begin{corollary}\label{Cor-domain1}
Suppose that $\fS$ has no interior punctures and admits a triangulation $\lambda$. Let $I$ be a subset of the punctures of $\fS$. Then the algebra embedding
\[
\iota : \mathbb T_\omega^+(\fS,\lambda;I) \to \mathcal{S}_\omega^I(\fS)
\]
extends uniquely to an algebra embedding
\[
\iota : \mathcal{S}_\omega^I(\fS) \to \mathbb T_\omega(\fS,\lambda;I).
\]
\end{corollary}

\begin{remark}\label{B-rem-skein-A}
Using the quantum trace map in Corollary~\ref{Cor-domain1} together with the proof technique of Theorem~\ref{thm-skein-eq-A-two}, we expect that $\mathcal{S}_\omega^I(\fS)$ is equivalent to an (upper) quantum cluster algebra with $|I|$ invertible frozen variables and $|\partial\lambda\setminus I|$ non-invertible frozen variables.
\end{remark}

\begin{remark}
Note that the quantum trace map constructed in \cite[Section~2]{LY22} applies to all triangulable  pb surfaces via the notion of a \emph{quasitriangulation}. Accordingly,  Corollary~\ref{Cor-domain1} extends readily to all triangulable pb surfaces by replacing ``triangulation'' with ``quasitriangulation'' in the statement.
\end{remark}

\begin{proof}[Proof of Lemma~\ref{lem-prime-ele}]
The definition of $\ga_v$ for $v \in V_\lambda' \setminus \overline{V}_\lambda$ in \S\ref{sub-extended}, together with the isomorphism in \cite[Theorem~11.1]{LS21} (see \S\ref{sub-S-iso}), shows that each frozen variable in $\cS_\omega(\fS)$ for $n=2$ corresponds to a bad arc in $\mathcal S_\omega(\fS)$ (see Figure~\ref{fig:badarc-sl2}). The result then follows from Lemma~\ref{lem-Cp}(a) and Proposition~\ref{prop-domain1}.
\end{proof}

\section{Proofs of Lemmas \ref{lem:quasi2}, \ref{lem:B-d} 
and \ref{lem:B4}}\label{app:lemmas1234}

\subsection{Some lemmas on quiver mutations}


For any integer $n \geq 0$, we define three families of quivers, denoted by $Q^1(n)$, $Q^2(n)$, and $Q^3(n)$ as follows.

The vertex set of $Q^1(n)$ is given by
\[
V(Q^1(n)) = \{1_1, 1_2, \dots, 1_n\} \cup \{2_0, 2_1, \dots, 2_n\}.
\]
We refer to the vertices $\{1_1, \dots, 1_n\}$ as \emph{level 1} vertices and $\{2_0, \dots, 2_n\}$ as \emph{level 2} vertices. In the graphical representation, these vertices are arranged horizontally from right to left.

The arrow set of $Q^1(n)$ consists of the following four families:
\begin{align*}
    &\{\, 1_{i+1} \leftarrow 1_i \mid i = 1, \dots, n-1 \,\}, \\
    &\{\, 2_{i+1} \leftarrow 2_i \mid i = 0, 1, \dots, n-1 \,\}, \\
    &\{\, 1_{i+1} \rightarrow 2_{i} \mid i =0, 1, \dots, n-1 \,\}, \\
    &\{\, 1_i \leftarrow 2_i \mid i = 1, \dots, n-1 \,\}.
\end{align*}

The quiver $Q^2(n)$ is obtained from $Q^1(n)$ by removing the vertex $2_0$ (and all incident arrows).  
The quiver $Q^3(n)$ is obtained from $Q^1(n)$ by adding an additional arrow $1_n \leftarrow 2_n$.

Given a sequence of quivers $Q^{i_1}(n_1), Q^{i_2}(n_2), \dots, Q^{i_k}(n_k)$ with $i_1, i_2, \dots, i_k \in \{1,2,3\}$, we define their concatenation
\[
Q^{i_k}(n_k) \# \cdots \# Q^{i_2}(n_2) \# Q^{i_1}(n_1)
\]
to be the quiver obtained as follows. Take the disjoint union of $Q^{i_1}(n_1), \dots, Q^{i_k}(n_k)$, arranged from right to left in this order, and for each $s = 1, \dots, k-1$, add the following arrows:

\begin{itemize}
    \item an arrow from the \emph{last} vertex of level 1 in $Q^{i_s}(n_s)$  to the \emph{first} vertex of level 1 in $Q^{i_{s+1}}(n_{s+1})$;
    \item an arrow from the \emph{last} vertex of level 2 in $Q^{i_s}(n_s)$ to the \emph{first} vertex of level 2 in $Q^{i_{s+1}}(n_{s+1})$;
    \item an arrow from the \emph{first} vertex of level 2 in $Q^{i_{s+1}}(n_{s+1})$ to the \emph{last} vertex of level 1 in $Q^{i_s}(n_s)$.
\end{itemize}

To make this precise, we label the vertices of each component $Q^{i_s}(n_s)$ by attaching a bracketed index $[s]$ to distinguish the components. Thus, a vertex originally denoted $1_j$ or $2_j$ in $Q^{i_s}(n_s)$ is written as $1_j[s]$ or $2_j[s]$. Then, for each $s = 1, \dots, k-1$, we add the arrows
\[
1_1[s+1] \leftarrow 1_{n_s}[s], \quad
2_{\delta(s)}[s+1] \leftarrow  2_{n_s}[s], \quad
2_{\delta(s)}[s+1] \rightarrow  1_{n_s}[s],
\]
where
\[
\delta(s) =
\begin{cases}
0 & \text{if } i_{s+1} \in \{1,3\}, \\
1 & \text{if } i_{s+1} = 2.
\end{cases}
\]


We refer to the level~$1$ and level~$2$ vertices of each component $Q^{i_s}(n_s)$ ($s = 1, \dots, k$) as the level~$1$ and level~$2$ vertices, respectively, of the concatenated quiver  
$Q^{i_k}(n_k) \# \cdots \# Q^{i_2}(n_2)\# Q^{i_1}(n_1).$

For any two quivers $Q$ and $Q'$ of the form  
\[
Q^{i_k}(n_k) \# \cdots \# Q^{i_2}(n_2)\# Q^{i_1}(n_1),
\]  
if the number of level~$2$ vertices in $Q$ equals the number of level~$1$ vertices in $Q'$, we define the quiver  
\[
\begin{bmatrix}
Q \\
Q'
\end{bmatrix}
\]  
to be the quiver obtained from the disjoint union of $Q$ and $Q'$ by identifying each level~$2$ vertex of $Q$ with the corresponding level~$1$ vertex of $Q'$,  where the identification is performed from left to right in order.

This construction extends naturally to longer sequences: for quivers $Q_1, \dots, Q_n$ of the same form, if  
\[
\#\{\text{level~$2$ vertices of } Q_s\} = \#\{\text{level~$1$ vertices of } Q_{s+1}\}
\quad \text{for all } s = 1, \dots, n-1,
\]  
we define the stacked quiver  
\[
\begin{bmatrix}
Q_1 \\
Q_2 \\
\vdots \\
Q_n
\end{bmatrix}
\]  
by successively identifying the level~$2$ vertices of $Q_s$ with the level~$1$ vertices of $Q_{s+1}$ for $s = 1, \dots, n-1$ (each identification is again performed from left to right in order).

Under this notation, let $\lambda$ be the star-like triangulation of $\mathbb P_{k+2}$ centered at vertex $1$. Recall the vertex set $I(\Delta_i)$ given by \eqref{eq:Ii}.
Let $Q_{\rm sub}$ be the full subquiver of $\overline Q^{\rm qc}_{\lambda}(\mathbb{P}_{k+2})$ induced by the vertex set 
$I(\Delta_1)\cup I(\Delta_2)\cup \cdots\cup I(\Delta_k)$
with the half-edges between the vertices $\{jj^k \mid j=1,\dots,n-1\}$ removed.

Then
\begin{equation}\label{eq:Qlambda}
Q_{\rm sub}\cong \begin{bmatrix}
\underbrace{Q^1(1) \# \cdots \# Q^1(1)}_{k\ \text{times}} \\
\underbrace{Q^1(2) \# \cdots \# Q^1(2)}_{k\ \text{times}} \\
\vdots \\
\underbrace{Q^1(n-2) \# \cdots \# Q^1(n-2)}_{k\ \text{times}}
\end{bmatrix}.
\end{equation}

The following three lemmas are direct consequences of quiver mutation rules; they can also be verified using Keller's quiver mutation applet.

\begin{lemma}\label{lem:mutation-row1}
Let $k \geq 0$, $\ell \geq 1$, and $s \leq \ell$.  
Consider the quiver 
\[
Q = \underbrace{Q^1(\ell) \# \cdots \# Q^1(\ell)}_{k \text{ times}}\# Q^1(s) .
\]
Then the first and second rows of $Q$ contains $s+k\ell$ and $s + 1 + k(\ell + 1)$ vertices, respectively.
We index these vertices from right to left as $v^1_1, v^1_2, \dots, v^1_{s + k\ell}$ and $v^2_1, v^2_2, \dots, v^2_{\,s + 1 + k(\ell + 1)}$.

\begin{enumerate}[label={\rm (\alph*)}, itemsep=0.3em]

\item The vertices in $\{v^1_r, r\leq s-1\}\cup\{v^2_r, r\leq s\}$ do not adjacent to the vertices in $\{v^1_r, r>s\}\cup\{v^2_r, r> s+1\}$ in the quiver $\mu_{v^2_{s+1}} \cdots \mu_{v^2_2} \mu_{v^2_1}(Q)$.

\item Let $Q'$ be the quiver obtained from $\mu_{v^2_{s + k(\ell + 1)}} \cdots \mu_{v^2_2} \mu_{v^2_1}(Q)$ by removing the vertex $v^2_{s + 1 + k(\ell + 1)}$.  
Then 
\[
Q' \cong Q^3(\ell)\# \underbrace{Q^1(\ell) \# \cdots \# Q^1(\ell)}_{k-1 \text{ times}}\# Q^2(s).
\]
\end{enumerate}
\end{lemma}

\begin{lemma}\label{lem:mutation-row2}
Let $k \geq 0$, $\ell \geq 1$, and $s \leq \ell$.  
Consider the quiver 
\[
Q = 
\begin{bmatrix}
\underbrace{Q^1(\ell) \# \cdots \# Q^1(\ell)}_{k}\# Q^1(s) \\
 Q^3(\ell+1) \# \underbrace{Q^1(\ell+1) \# \cdots \# Q^1(\ell+1)}_{k-1}\# Q^2(s+1)
\end{bmatrix}.
\]
Then the three rows of $Q$ contains $s+k\ell$, $s + 1 + k(\ell + 1)$ and $s + 1 + k(\ell + 2)$ vertices, respectively.
We index these vertices from right to left as $v^1_1, v^1_2, \dots, v^1_{s + + k\ell}$;
$v^2_1, v^2_2, \dots, v^2_{\,s + 1 + k(\ell + 1)}$ and $v^3_1, v^3_2, \dots, v^3_{\,s + 1 + k(\ell + 2)}$.

\begin{enumerate}[label={\rm (\alph*)}, itemsep=0.3em]

\item The vertices in $\{v^1_r, r\leq s-1\}\cup\{v^2_r, r\leq s\}\cup\{v^3_r, r\leq s\}$ do not adjacent to the vertices in $\{v^1_r, r>s\}\cup\{v^2_r, r> s+1\}\cup\{v^3_r, r> s+1\}$ in the quiver $\mu_{v^2_{s+1}} \cdots \mu_{v^2_2} \mu_{v^2_1}(Q)$.

\item Let $Q'$ be the quiver obtained from $\mu_{v^2_{s + k(\ell + 1)}} \cdots \mu_{v^2_2} \mu_{v^2_1}(Q)$ by removing the vertex $v^2_{s + 1 + k(\ell + 1)}$.  
Then 
\[
Q' =
\begin{bmatrix}
Q^3(\ell) \# \underbrace{Q^1(\ell) \# \cdots \# Q^1(\ell)}_{k-1}\# Q^2(s) \\
\underbrace{Q^1(\ell+1) \# \cdots \# Q^1(\ell+1)}_{k} \# Q^1(s)
\end{bmatrix}.
\]
In particular, if $s=0$ then 
\[
Q' =
\begin{bmatrix}
Q^3(\ell) \# \underbrace{Q^1(\ell) \# \cdots \# Q^1(\ell)}_{k-1} \\
\underbrace{Q^1(\ell+1) \# \cdots \# Q^1(\ell+1)}_{k} \# Q^1(0)
\end{bmatrix}.
\]
\end{enumerate}
\end{lemma}

\begin{lemma}\label{lem:mutation-row3}
Let $k \geq 0$, $\ell \geq 1$.  
Consider the quiver 
\[
Q = Q^3(\ell) \# \underbrace{Q^1(\ell) \# \cdots \# Q^1(\ell)}_{k-1 \text{ times}} \# Q^2(1),
\]
where for $k = 0$ the middle block is empty and the expression is interpreted as $Q = Q^2(1)$.  
The first and second rows of $Q$ contains $1 + k\ell$ and $2 + k(\ell + 1)$ vertices, respectively. We index these vertices from right to left as
$v^1_1, v^1_2, \dots, v^1_{1 + k\ell}$ and $v^2_1, v^2_2, \dots, v^2_{1 + k(\ell+1)}$.

\begin{enumerate}[label={\rm (\alph*)}, itemsep=0.3em]

\item The vertex $v^2_1$ does not adjacent to the vertices in $\{v^1_r, r>2\}\cup\{v^2_r, r> 2\}$ in the quiver $\mu_{v^1_1}(Q)$.

\item Let $Q'$ be the quiver obtained from $\mu_{v^1_{k\ell}} \cdots \mu_{v^1_2} \mu_{v^1_1}(Q)$ by removing the vertex $v_{\,1 + k(\ell + 1)}$.  
Then 
\[
Q' \cong  \underbrace{Q^1(\ell) \# \cdots \# Q^1(\ell)}_{k \text{ times}}\# Q^1(0).
\]
\end{enumerate}
\end{lemma}

Recall the seed $\overline{\mathsf{t}}_{i}^{\rm qc}$ given by \eqref{eq:ti} and vertex set $I(\Delta_i)$ by \eqref{eq:Ii}. For any $\ell\leq k$, denote by $Q_{\rm sub}^\ell$ the subquiver of $Q_{\rm sub}$ induced by the vertex set $I(\Delta_1)\cup \cdots \cup I(\Delta_\ell)$.

\begin{lemma}\label{lem:subseed}
For any $2\leq i\leq k$ and $1\leq \ell\leq k+1-i$, the following statements hold in the quiver of $\overline{\mathsf{t}}_{i}^{\rm qc}$:
\begin{enumerate}
    \item[(a)] There are no arrows connecting vertices in $\bigcup_{j=1}^\ell I(\Delta_j) \setminus \{11^\ell,\dots, (n-1,n-1)^\ell\}$ to those outside $\bigcup_{j=1}^\ell I(\Delta_j)$. Consequently, we can define a sub-seed $\overline{\mathsf{t}}_{i}^{{\rm sub},\ell}$ of $\overline{\mathsf{t}}_{i}^{\rm qc}$ supported on $\bigcup_{j=1}^\ell I(\Delta_j)$, with the vertex set $\{11^\ell,\dots, (n-1,n-1)^\ell\}$ frozen.
    \item[(b)] The subquiver $Q'$ induced by the vertex set $\bigcup_{j=1}^\ell I(\Delta_j)$ is isomorphic to $Q_{\rm sub}^\ell$.
\end{enumerate}

\end{lemma}

\begin{proof}
By iteratively applying Lemmas \ref{lem:mutation-row1}, \ref{lem:mutation-row2}, and \ref{lem:mutation-row3} to \eqref{eq:Qlambda}, we observe that assertion (a) holds. Moreover, the subquiver $Q'$ takes the following form:
\begin{equation}\label{eq:Q'}
Q' = 
\begin{bmatrix}
\underbrace{Q^1(1) \# \cdots \# Q^1(1)}_{\ell} \\
\underbrace{Q^1(2) \# \cdots \# Q^1(2)}_{\ell}\\
\vdots\\
\underbrace{Q^1(n-2) \# \cdots \# Q^1(n-2)}_{\ell}
\end{bmatrix}.
\end{equation}
This establishes assertion (b).
\end{proof}

\subsection{$g$-vector}

Let $Q$ be an (ice) quiver with vertex set $\mathcal{V}$. The \textbf{framed quiver} $\widetilde{Q}$ associated with $Q$ is constructed by first converting all frozen vertices of $Q$ into mutable vertices and deleting all arrows between them, and then adjoining, for each vertex $v \in \mathcal{V}$, a new frozen vertex $v'$ together with a single arrow $v' \to v$. Consequently, the mutable vertex set of $\widetilde{Q}$ is $\mathcal V$ and the frozen set is $\mathcal V'$, where $\mathcal{V}' = \{v' \mid v \in \mathcal{V}\}$.

For any quiver $\widetilde{Q}'$ that is mutation-equivalent to $\widetilde{Q}$, a vertex $v \in \mathcal{V}$ is said to be \textbf{green} in $\widetilde{Q}'$ if $\widetilde{Q}'(w', v) \geq 0$ for all $w' \in \mathcal{V}'$, and \textbf{red} otherwise. By the sign-coherence of $c$-vectors (see, e.g.,~\cite{GHKK}), every vertex is either green or red.

A sequence of mutations $\mu_{v_m} \cdots \mu_{v_1}$ is called \textbf{green} if $v_{i+1}$ is green in the quiver $\mu_{v_i} \cdots \mu_{v_1}(\widetilde{Q})$ for all $i = 0, \dots, m-1$. Such a sequence is called \textbf{maximal green} if it is green and, additionally, every vertex in the final quiver $\mu_{v_m} \cdots \mu_{v_1}(\widetilde{Q})$ is red.

The commutative cluster algebra $\mathcal{A}^{\mathrm{prin}}_{Q}$ associated with the extended exchange matrix $\widetilde{Q}$ is referred to as the \textbf{cluster algebra with principal coefficients} (with initial exchange matrix $Q$). We denote the initial mutable cluster variables by $A_v$ ($v \in \mathcal{V}$) and the frozen variables by $A_{v'}$ ($v' \in \mathcal{V}'$).

By assigning degrees via
\[
\deg(A_v) = \mathbf{e}_v \quad \text{and} \quad \deg(A_{v'}) = -Q(-, v),
\]
the algebra $\mathcal{A}^{\mathrm{prin}}_{Q}$ becomes a graded algebra in which every cluster variable is homogeneous. Here, $\mathbf{e}_v \in \mathbb{Z}^{\mathcal{V}}$ denotes the standard basis vector indexed by $v$, and $Q(-,v)$ denotes the column of $Q$ indexed by $v$. The degree of a cluster variable $A$ is called its \textbf{$g$-vector}, denoted by $g(A)$.

Given a quantum seed $\mathsf{s}$ with quiver $Q$, let $A$ be a cluster variable of the quantum cluster algebra $\mathcal{A}_{\mathsf{s}}$. Let $\widetilde{A}$ be the cluster variable of $\mathcal{A}^{\mathrm{prin}}_{Q}$ corresponding to $A$. The \textbf{(extended) $g$-vector (with respect to the seed $\mathsf s$)} of $A$ is then defined as $g(\widetilde{A})$, and is denoted by $g(A)$.

\begin{lemma}\cite[Proposition 6.6]{FZ4}\label{lem:grec}
Let $\mathsf s=((\widetilde A_{v},v\in \mathcal V\sqcup \mathcal V'),\widetilde Q')$ be a seed of $\mathcal{A}^{\mathrm{prin}}_{Q}$. For any $v\in \mathcal V$, we have  
    $$g(\mu_v(\widetilde A_v))=\sum_{\substack{w \in \mathcal V \sqcup \mathcal V' \\ \widetilde Q'(w,v) > 0}}\widetilde Q'(w,v)g(\widetilde A_w)-g(\widetilde A_v)=\sum_{\substack{w \in \mathcal V \sqcup \mathcal V' \\ \widetilde Q'(w,v) < 0}}(-\widetilde Q'(w,v))g(\widetilde A_w)-g(\widetilde A_v).$$
\end{lemma}

\subsection{$g$-vectors of $\overline A_{js}\langle i\rangle$ and $\overline A^{\rm qc}_{js}\langle i\rangle$} 

For any $k \geq 2$, let $\lambda$ denote the star-like triangulation at vertex $1$ of $\mathbb P_{k+2}$. 
For any $2 \leq i \leq k$ and $j,s$ with $1 \leq s \leq j \leq n-1$, we recall that the quantum cluster variables $\overline A^{\rm qc}_{js}\langle i\rangle$ of $\overline {\mathscr A}_{\omega}^{\rm qc}(\mathbb P_{k+2})$ and $\overline A_{js}\langle i\rangle$ of $\overline {\mathscr A}_{\omega}(\mathbb P_{k+2})$ are defined in Section \ref{sub-quasi-polygon}. In this section, we compute the $g$-vectors of $\overline A_{js}\langle i\rangle$ and $\overline A^{\rm qc}_{js}\langle i\rangle$ with respect to the seed $\overline s_{\lambda}$ and $\overline s^{\rm qc}_{\lambda}$, respectively. We shall only focus on the quadrilateral formed by the vertices $1,i,i+1,i+2$ inside $\mathbb P_{k+2}$.

Let $\mathbb P_4$ be a quadrilateral with triangulation $\lambda_{\rm sub}$, as in Figure \ref{fig:P4-labeling}. Let $\lambda'_{\rm sub}$ be another triangulation of $\mathbb P_4$. We label the vertices of $\overline V_\lambda$ as in Figure \ref{fig:P4-labeling}. For each $i=0,1,\cdots, n-2$ and each $0\leq t\leq i$, define $V_\lambda^{i,t}$ and $V_\lambda^i$ as in \eqref{def-Vite} and \eqref{def-Vie}, respectively. Let $\mathcal A^{\rm prin}_{\overline Q_{\lambda_{\rm sub}}}$ be the commutative cluster algebra with principal coefficients and initial exchange matrix $\overline Q_{\lambda_{\rm sub}}$. Let $\mathsf s_{\lambda_{\rm sub}}^{\rm prin}$ and $\mathsf s_{\lambda'_{\rm sub}}^{\rm prin}$ be the seeds of $\mathcal A^{\rm prin}_{\overline Q_{\lambda_{\rm sub}}}$ associated with the triangulations $\lambda_{\rm sub}$ and $\lambda'_{\rm sub}$, respectively. 
For each $i=0,1,\cdots,n-2$, denote by $\widetilde \mu_{i}$ the mutation sequence taken at all $v\in V_\lambda^i$. Then $\mathsf s_{\lambda'_{\rm sub}}^{\rm prin}=\widetilde\mu_{n-2}\cdots \widetilde \mu_{1}\widetilde \mu_{0}(\mathsf s_{\lambda_{\rm sub}}^{\rm prin})$ \cite{FG06,GS19}. For any $v\in \overline V_\lambda$ denote by $A_v$ and $A'_v$ the cluster variables corresponding to the vertex $v$ in the seeds $\mathsf s_{\lambda_{\rm sub}}^{\rm prin}$ and $\mathsf s_{\lambda'_{\rm sub}}^{\rm prin}$, respectively. 

For each $i=0, 1,\cdots,n-2$ and $v\in \overline V_\lambda$, denote by $g^{(i)}_v$ the $g$-vector of the cluster variable at $v$ in the seed $\widetilde\mu_{i}\cdots \widetilde \mu_{1}\widetilde \mu_{0}(\mathsf s_{\lambda_{\rm sub}}^{\rm prin})$. 

The following lemma follows immediately from the mutation rules.

\begin{lemma}\label{lem:Q1}
 For any $i\geq 0$, let $Q'=\widetilde\mu_{i}\cdots \widetilde \mu_{1}\widetilde \mu_{0}({\overline Q}_\lambda)$. Then for every vertex $v=(s,t)\in V_\lambda^i$, we have $Q'(w,v)\neq 0$ if and only if $w\in\{ (s-1,t),(s+1,t),(s,t-1),(s,t+1)\}$. Moreover, we have
 $$Q'((s,t\pm 1),(s,t))=-Q'((s\pm 1,t),(s,t))=1.$$
\end{lemma}

\begin{lemma}\label{lem:Q2}
For any integer $i \geq 0$, let $\widetilde{Q}' = \widetilde{\mu}_i \cdots \widetilde{\mu}_1 \widetilde{\mu}_0(\widetilde{\overline{Q}}_\lambda).$  
Then for every vertex $v \in V_\lambda^{i+1}$, i.e., $v = (j+i+1-s,\, j+s)$ or $v = (j+s,\, j+i+1-s)$ for some integers $j,s$ satisfying  
$1 \leq j + s \leq j + i + 1 - s \leq n - 1,$
the following hold:
\begin{enumerate}[label=\textup{(\alph*)}]
    \item If $v = (j+i+1-s,\, j+s)$, then
    \begin{equation}\label{eq:mup41}
    \begin{aligned}
        \widetilde{Q}'((j+i+1-s,\, j)',\, v) &= \widetilde{Q}'((j+i+1-s,\, j+i+1)',\, v) = 1, \\
        \widetilde{Q}'((j+i+1-s,\, \ell)',\, v) &= 0 \quad \text{for all } \ell < j \text{ or } \ell > j+i+1.
    \end{aligned}
    \end{equation}
    
    \item If $v = (j+s,\, j+i+1-s)$, then
    \begin{equation}\label{eq:mup42}
    \begin{aligned}
        \widetilde{Q}'((j+s,\, j)',\, v) &= \widetilde{Q}'((j+s,\, j+i+1)',\, v) = 1, \\
        \widetilde{Q}'((j+s,\, \ell)',\, v) &= 0 \quad \text{for all } \ell < j \text{ or } \ell > j+i+1.
    \end{aligned}
    \end{equation}
\end{enumerate}
Consequently, the mutation sequence $\widetilde{\mu}_{i} \cdots \widetilde{\mu}_1 \widetilde{\mu}_0$ is a green sequence.
\end{lemma}

\begin{proof}
We proceed by induction on $i$. By symmetry, it suffices to prove part (a); part (b) follows dually.
 
When $i=0$, we have $v = (j+1,j)$, one checks directly from the definition of $\widetilde{\overline{Q}}_\lambda$ that \eqref{eq:mup41} holds. Thus the base case is verified.
 
Assume that the statement holds for all indices less than $i$. Set  
$\widetilde{Q}'' = \widetilde{\mu}_{i-1} \cdots \widetilde{\mu}_1 \widetilde{\mu}_0(\widetilde{\overline{Q}}_\lambda).$ 

We consider two cases.

\smallskip\noindent\textbf{Case 1: $s = 0$.}  
Then $v = (j+i+1,\, j)$. Note that $(j+i+1,\, j+1) \in V_\lambda^i$, while $v \notin V_\lambda^\ell$ for any $\ell \leq i$. By the induction hypothesis applied to $(j+i+1,\, j+1)$, we have
\begin{equation}\label{eq:mup43}
\begin{aligned}
  &  \widetilde{Q}''((j+i+1,\, j+i+1)',\, (j+i+1,\, j+1)) = 1, \\
  &  \widetilde{Q}''((j+i+1,\, \ell)',\, (j+i+1,\, j+1)) = 0 \quad \text{for all } \ell < j+1 \text{ or } \ell > j+i+1.
\end{aligned}
\end{equation}

By the induction hypothesis, the mutation sequence $\widetilde{\mu}_{i-1} \cdots \widetilde{\mu}_0$ is green. From Lemma~\ref{lem:Q1}, we see that the mutation sequence $\widetilde{\mu}_{i-1} \cdots \widetilde{\mu}_0$ does not affect the frozen vertices incident to $v$. In particular,
\[
\widetilde{Q}''((j+i+1,\, j)',\, v) = 1, \quad \text{and} \quad \widetilde{Q}''(w',\, v) = 0 \text{ for all other frozen vertices } w'.
\]
Now, applying the mutation $\widetilde{\mu}_i$ at the vertices in $V_\lambda^i$ (which includes $(j+i+1,\, j+1)$ but not $v$), the standard mutation rule yields precisely the arrow pattern described in \eqref{eq:mup41}. Hence the claim holds in this case.

\smallskip\noindent\textbf{Case 2: $s > 0$.}  
Then both $(j+i+1-s,\, j+s-1)$ and $(j+i+1-s,\, j+s+1)$ belong to $V_\lambda^i$. Moreover, observe that
\[
v = (j+i+1-s,\, j+s) = ((j+1) + (i-1) - (s-1),\, (j+1) + (s-1)),
\]
so $v \in V_\lambda^{i-1}$ as well. Applying the induction hypothesis to the relevant vertices at step $i-1$, we obtain
\begin{equation}\label{eq:mup44}
\begin{aligned}
    &\widetilde{Q}''((j+i+1-s,\, j)',\, (j+i+1-s,\, j+s-1)) = 1, \\
    &\widetilde{Q}''((j+i+1-s,\, j+i)',\, (j+i+1-s,\, j+s+1)) = 1, \\
    &\widetilde{Q}''((j+i+1-s,\, \ell)',\, v) = 0 \quad \text{for all } \ell < j+1 \text{ or } \ell > j+i, \\
    &\widetilde{Q}''((j+i+1-s,\, \ell)',\, (j+i+1-s,\, j+s\pm1)) = 0 \quad  \text{for all } \ell < j \text{ or } \ell > j+i+1.
\end{aligned}
\end{equation}
Furthermore, by Lemma~\ref{lem:Q1}, the quiver $\widetilde{Q}''$ contains arrows
\begin{equation}\label{eq:mup45}
    \widetilde{Q}''((j+i+1-s,\, j+s-1),\, v) = \widetilde{Q}''((j+i+1-s,\, j+s+1),\, v) = 1.
\end{equation}
Applying the mutation $\widetilde{\mu}_i$ at the vertices of $V_\lambda^i$ (which include $(j+i+1-s,\, j+s\pm1)$), the resulting quiver $\widetilde{Q}'$ acquires exactly the arrows stated in \eqref{eq:mup41}. 

This completes the proof.
\end{proof}








The following follows immediate by Lemmas \ref{lem:grec}, \ref{lem:Q1} and \ref{lem:Q2}.

\begin{lemma}
For every $ i = 1, \dots, n-2$ and integers $ j, s $ with $1 \leq j + s \leq j + i - s \leq n - 1,$
the following recurrence relations hold:
\begin{equation}\label{eq:recg}
\begin{aligned}
    g^{(i)}_{(j+i-s,\;j+s)} &= g^{(i-1)}_{(j+i-s,\,j+s+1)} + g^{(i-1)}_{(j+i-s,\,j+s-1)} - g^{(i-1)}_{(j+i-s,\;j+s)}, \\
    g^{(i)}_{(j+s,\;j+i-s)} &= g^{(i-1)}_{(j+s,\;j+i-s+1)} + g^{(i-1)}_{(j+s,\;j+i-s-1)} - g^{(i-1)}_{(j+s,\;j+i-s)}.
\end{aligned}
\end{equation}
\end{lemma}

\begin{lemma}\label{lem:gp4}
For any $v\in \overline V_{\lambda_{\rm sub}}$, denote by $\mathbf{e}_v\in \mathbb Z^{\overline V_{\lambda_{\rm sub}}}$ the standard vector associated with $v$. For each integer $i = 0, 1, \dots, n-2$ and integers $j, s$ satisfying $1 \leq j + s \leq j + i - s \leq n - 1$, the following identities hold:

\begin{enumerate}[label=\textup{(\alph*)}]
    \item 
    $g^{(i)}_{(j+i-s,\;j+s)} = \mathbf{e}_{(j+i-s,\,j+i+1)} + \mathbf{e}_{(j+s,\,j-1)} - \mathbf{e}_{(j+i-s,\;j+i-s)}$.
    
    \item 
    $g^{(i)}_{\,j+s,\;j+i-s} = \mathbf{e}_{(j+s,\,j+i+1)} + \mathbf{e}_{(j+s,\,j-1)} - \mathbf{e}_{(j+s,\;j+s)}$.

    \item In particular, for all $1 \leq s, t \leq n-1$, we have
    \begin{equation}\label{eq:g_n-1}
        g^{(n-2)}_{(s,t)} =
        \begin{cases}
            \mathbf{e}_{(s,n)} + \mathbf{e}_{(s,s+t-n)} - \mathbf{e}_{(s,s)} & \text{if } s + t \geq n, \\
            \mathbf{e}_{(s,0)} + \mathbf{e}_{(s,s+t)} - \mathbf{e}_{(s,s)}   & \text{if } s + t < n.
        \end{cases}
    \end{equation}
\end{enumerate}
\end{lemma}

\begin{proof}
We prove parts (a) and (b) simultaneously by induction on $i$.

For the base case $i = 0$, the condition $1 \leq j + s \leq j - s \leq n - 1$ implies that $s = 0$ and $1 \leq j \leq n-1$. In this case, Lemma~\ref{lem:grec} yields
\[
g^{(0)}_{(j,j)} = g(\mu_{jj}(A_{(j,j)})) = g(A_{(j,j+1)}) + g(A_{(j,j-1)}) - g(A_{(j,j)})
= \mathbf{e}_{(j,j+1)} + \mathbf{e}_{(j,j-1)} - \mathbf{e}_{(j,j)},
\]
which coincides with both (a) and (b) when $s = 0$.

Now, assume that (a) and (b) hold for all indices at level $i-1$ (where $i > 0$); that is, for all $j', s'$ such that
\[
1 \leq j' + s' \leq j' + (i-1) - s' \leq n - 1.
\]
Fix integers $j, s$ satisfying $1 \leq j + s \leq j + i - s \leq n - 1$. 
Observe that the indices $(j + i - s, j + s-1)$ and $(j+i-s, j+s+1)$ can be written as $\bigl((j+1) + (i-1) - s,\; (j+1) + s\bigr)$, which belong to $V_\lambda^{i-1}$. Thus, they satisfy the inductive hypothesis. Applying (a) at level $i-1$, we obtain
\begin{equation}\label{eq:inductive_a}
    g^{(i-1)}_{(j+i-s,\;j+s-1)}
    = \mathbf{e}_{(j+i-s,\,j+i)} + \mathbf{e}_{(j+i-s,\,j-1)} - \mathbf{e}_{(j+i-s,\;j+i-s)},
\end{equation}
and
\begin{equation}\label{eq:inductive_a1}
    g^{(i-1)}_{(j+i-s,\;j+s+1)}
    = \mathbf{e}_{(j+i-s,\,j+i+1)} + \mathbf{e}_{(j+i-s,\,j)} - \mathbf{e}_{(j+i-s,\;j+i-s)}.
\end{equation}
Furthermore, note that $(j+i-s, j+s) = \bigl((j+1) + (i-2) - (s-1),\; (j+1) + (s-1)\bigr)$ lies in $V_\lambda^{i-2} \setminus V_\lambda^{i-1}$. Consequently, by applying (a) at level $i-2$, we have
\begin{equation}\label{eq:inductive_b}
    g^{(i-1)}_{(j+i-s,\;j+s)} = g^{(i-2)}_{(j+i-s,\;j+s)}
    = \mathbf{e}_{(j+i-s,\,j+i)} + \mathbf{e}_{(j+i-s,\,j)} - \mathbf{e}_{(j+i-s,\;j+i-s)}.
\end{equation}

Using the recurrence relation \eqref{eq:recg} together with \eqref{eq:inductive_a}, \eqref{eq:inductive_a1}, and \eqref{eq:inductive_b}, we compute:
\begin{align*}
   g^{(i)}_{(j+i-s,\;j+s)}
&= g^{(i-1)}_{(j+i-s,\;j+s+1)} + g^{(i-1)}_{(j+i-s,\;j+s-1)} - g^{(i-1)}_{(j+i-s,\;j+s)} \\
&= \mathbf{e}_{(j+i-s,\;j+i+1)} + \mathbf{e}_{(j+i-s,\;j-1)} - \mathbf{e}_{(j+i-s,\;j+i-s)}.
\end{align*}
This establishes identity (a). The proof of (b) follows analogously.

Finally, we prove part (c). We consider only the case $s+t \geq n$, as the case $s+t < n$ can be handled similarly. 
Assume $s \geq t$. Then $(s,t) \in V_\lambda^{2(n-1)-(s+t)}$ and $(s,t) \notin V_\lambda^{i}$ for all $i > 2(n-1)-(s+t)$. Therefore, by applying (a), we have
\[
g_{(s,t)}^{(n-2)} = g_{(s,t)}^{2(n-1)-(s+t)} = g_{\bigl((t+s-n+1)+(2n-2-s-t)-(n-1-s),\,(t+s-n+1)+(n-1-s)\bigr)}^{2(n-1)-(s+t)} = \mathbf{e}_{(s,n)} + \mathbf{e}_{(s,s+t-n)} - \mathbf{e}_{(s,s)}.
\]
The case $s < t$ can be proved similarly using (b). This completes the proof.
\end{proof}

For any $k \geq 2$, let $\lambda_1$ denote the star-like triangulation at vertex $1$ of $\mathbb P_{k+2}$. 
For any $2 \leq i \leq k$ and $j,s$ with $1 \leq s \leq j \leq n-1$, we recall that the quantum cluster variables $\overline A^{\rm qc}_{js}\langle i\rangle$ of $\overline {\mathscr A}_{\omega}^{\rm qc}(\mathbb P_{k+2})$ and $\overline A_{js}\langle i\rangle$ of $\overline {\mathscr A}_{\omega}(\mathbb P_{k+2})$ are defined in Section \ref{sub-quasi-polygon}.

Return $\mathbb P_{k+2}$, consider the quadrilateral with vertices $1, i, i+1, i+2$ inside $\mathbb P_{k+2}$. We observe the following correspondences between the labelings:

\begin{itemize}
    \item The vertex labeled $js$ in the triangle $(i,i+1,i+2)$, as shown in Figure \ref{Fig:lambdai}, corresponds to the label $(j, n-s)$ in Figure \ref{fig:P4-labeling}.
    
    \item The vertex labeled $(s,t)$ in Figure \ref{fig:P4-labeling} corresponds to the label $ts^{i-1}$ (if $s < t$) or $s,s-t^{i}$ (if $s \geq t$) in $\mathcal V_{\mathbb P_{k+2}}$, consistent with the conventions introduced at the beginning of Section \ref{sec-polygon}.
\end{itemize}

The following corollary follows immediately from Lemma \ref{lem:gp4}.

\begin{corollary}\label{cor:g2}
For any $2\leq i\leq k$ and $j,s$ with $1\leq s\leq j\leq n-1$, with respect to the quantum seeds $\overline {\mathsf s}_{\lambda}$ and $\overline {\mathsf s}^{\rm qc}_{\lambda}$, respectively, we have
\begin{equation}
    g(\overline A_{js}\langle i\rangle)=\textbf{e}_{nj^{\,i-1}}+\textbf{e}_{js^i}-\textbf{e}_{j0^i}, \qquad   g(\overline A^{\rm qc}_{js}\langle i\rangle)=\textbf{e}_{js^i}-\textbf{e}_{j0^i},
\end{equation}
where $\textbf{e}_{js^i}\in \mathbb Z^{\overline V_{\lambda}}$ denotes the standard vector associated with the vertex labeled $js^i$.
\end{corollary}

\subsection{$g$-vectors for $\mathbb{P}_{k+2}$} Let $\lambda$ be the star-like triangulation at vertex $1$ of $\mathbb P_{k+2}$. Let $Q$ be the quiver obtained from $\overline Q_\lambda$ by removing the vertices $ns^i$ for $s\in \{1,2,\cdots, n-1\}, i\in \{1,2,\cdots,k\}$ and $j0^1$ for $j\in\{1,2,\cdots,n-1\}$. Let $\widetilde Q$ be the framed quiver of $Q$. Let $\mathcal{A}^{\mathrm{prin}}_{Q}$ denote the commutative cluster algebra with principal coefficients and initial exchange matrix given by the quiver $Q$.  
For each vertex ${js^i}$ of $Q$, let $\widetilde{A}_{js^i}$ be the corresponding initial cluster variable in $\mathcal{A}^{\mathrm{prin}}_{Q}$.

Recall the mutation sequences $\overset{\leftarrow}{\mu}_{js^{i}}$, $\overset{\leftarrow}{\mu}_{\prec js^{i}}$ and $\overset{\leftarrow}{\mu}(\Delta_i)$ defined in Section \ref{sec:standard var}. 

For integers $0 \leq \widetilde{s} \leq n-2$ and $1 \leq \widetilde{i} \leq k$, for any vertex $js^i$ of $\overline{Q}_\lambda$, define
\begin{equation*}
\widetilde{A}_{js^i}^{\widetilde{s}}[\widetilde{i}]
   \;=\;
\overset{\leftarrow}{\mu}_{j\widetilde s^{\;\widetilde i}}\overset{\leftarrow}{\mu}_{\prec j\widetilde s^{\;\widetilde i}}
   \bigl( \widetilde{A}_{js^i} \bigr).
\end{equation*}


\begin{remark}\label{rem:counterpart}
   Observe that for all integers $s, j$ satisfying $1 \leq s \leq j \leq n-1$ and all $i$ with $1 \leq i \leq k$, 
   \begin{itemize}
       \item the cluster variable $\widetilde{A}_{j1^1}^{\,s-1}[i]$ corresponds to the standard cluster variable $\overline E^{\rm qc}(js^i)$ as its commutative counterpart;
       \item the cluster variable $\widetilde{A}_{js^1}^{j}[i-1]$ serves as the commutative counterpart of the cluster variable in $\overline{\mathsf{t}}_{i}^{\rm sub}$ associated with the vertex $js^1$.
   \end{itemize}
\end{remark}


In this subsection, we compute the $g$-vector of $\widetilde{A}_{js^i}^{\widetilde{s}}[\widetilde{i}]$, which we denote by $g_{js^i}^{\widetilde{s}}[\widetilde{i}]$. We start with the following observation: if $j\geq \widetilde s$ and $i+\widetilde i\leq k+1$, then 
\begin{equation*}
\widetilde{A}_{js^i}^{\widetilde{s}}[\widetilde{i}]
   \;=\; 
\bigl( \mu_{j s^{i}} \cdots \mu_{j 1^{i}} \bigr)
    \bigl( \mu_{j j^{i-1}} \cdots \mu_{j 1^{i-1}} \bigr)
    \cdots
    \bigl( \mu_{j j^{2}} \cdots \mu_{j 1^{2}} \bigr)
    \bigl( \mu_{j j^{1}} \cdots \mu_{j 1^{1}} \bigr)  \overset{\leftarrow}{\mu}_{\prec j\widetilde s^{\;\widetilde i}}
   \bigl( \widetilde{A}_{js^i} \bigr).
\end{equation*}

\begin{lemma}
 For integers $1\leq \widetilde{s} \leq n-2$ and $1 \leq \widetilde{i} \leq k$, 
 
\begin{enumerate}[label=\textup{(\alph*)}]
    \item  we have the sequence of mutations 
   $\overset{\leftarrow}{\mu}_{j\widetilde s^{\;\widetilde i}}\overset{\leftarrow}{\mu}_{\prec j\widetilde s^{\;\widetilde i}}$
   is a sequence of green mutation for $\widetilde Q$. 
  \item Let $ j s^{i} $ be a vertex with $ j \geq \widetilde{s} $ and $ i +\widetilde i \leq k + 1 $. Consider the quiver
$$
\begin{aligned}
    Q' = {} & 
    \bigl( \mu_{j s^{i}} \cdots \mu_{j 1^{i}} \bigr)
    \bigl( \mu_{j j^{i-1}} \cdots \mu_{j 1^{i-1}} \bigr)
    \cdots
    \bigl( \mu_{j j^{2}} \cdots \mu_{j 1^{2}} \bigr)
    \bigl( \mu_{j j^{1}} \cdots \mu_{j 1^{1}} \bigr)  \overset{\leftarrow}{\mu}_{\prec j,\widetilde s^{\;\widetilde i}}(\widetilde{Q}).
\end{aligned}
$$
Then $ Q'_{v,\, j s^{i}} < 0 $ if and only if  
$$
    v = j (s-1)^{i} \quad (\text{provided } s > 1 \text{ or } i > 1) \quad \text{or} \quad v = j (s+1)^{i}.
$$
In these cases, we have $ Q'_{v,\, j s^{i}} = -1 $.

\item Consequently, for any vertex $ j s^{i} $ with $ j \geq \widetilde{s} $ and $ i \leq k + 1 - \widetilde{i} $, the following recurrence relation holds:
\begin{equation}\label{eq:recg1}
    g_{j s^{i}}^{\widetilde{s}}[\widetilde{i}] =
    \begin{cases}
        g_{j (s+1)^{i}}^{\widetilde{s}-1}[\widetilde{i}] - g_{j s^{i}}^{\widetilde{s}-1}[\widetilde{i}]
        & \text{if } s = 1 \text{ and } i = 1, \\[6pt]
        g_{j (s+1)^{i}}^{\widetilde{s}-1}[\widetilde{i}] + g_{j (s-1)^{i}}^{\widetilde{s}}[\widetilde{i}] - g_{j s^{i}}^{\widetilde{s}-1}[\widetilde{i}]
        & \text{otherwise.}
    \end{cases}
\end{equation}
\end{enumerate}
\end{lemma}

\begin{proof}
Part~(a) follows from \cite[Theorem~4.1]{SW21}.  
Part~(b) follows from Lemmas~\ref{lem:mutation-row1}, \ref{lem:mutation-row2}, and~\ref{lem:mutation-row3}.  
Part~(c) is an immediate consequence of parts~(a) and~(b) together with Lemma~\ref{lem:grec}.
\end{proof}

\begin{lemma}\label{lem:g-vector-formula1}
    Let $0 \leq \widetilde{s} \leq n-1$ and $1 \leq \widetilde{i} \leq k$. For any vertex $js^i$ with $j \geq \widetilde{s}$ and $i+ \widetilde{i} \leq k+1$, we have
    \begin{equation}\label{eq:gpk}
        g_{js^i}^{\widetilde{s}}[\widetilde{i}] = {\bf e}_{j,s+\widetilde{s}^{\,i+\widetilde{i}-1}} - {\bf e}_{j\widetilde{s}^{\widetilde{i}}},
    \end{equation}
    where ${\bf e}_v$ denotes the standard vector associated with the vertex $v$ of $Q$. By convention, ${\bf e}_{j0^1} = {\bf 0}$, and if $s \geq j$, we set ${\bf e}_{js^i} = {\bf e}_{j,s-j^{\,i+1}}$.
\end{lemma}

\begin{proof}
We proceed by induction on $\widetilde{i}$, $\widetilde{s}$, and $j(i-1)+s$.

We first consider the case $\widetilde{i} = 1$.

When $\widetilde{s} = 0$, we have $\widetilde{A}_{js^i}^{\widetilde{s}}[\widetilde{i}] = \widetilde{A}_{js^i}$ for any $s$, and thus $g_{js^i}^{0}[1] = {\bf e}_{js^{i}}={\bf e}_{js^{i}}-{\bf e}_{j0^{1}}$.

Now suppose $\widetilde{s}>0$, and assume that equation \eqref{eq:gpk} holds for all triples with $\widetilde{i} = 1$, $\widetilde{s}-1$, and arbitrary $j(i-1)+s$.

If $j(i-1)+s=1$, then $i=1$ and $s=1$. Thus by \eqref{eq:recg1}
\begin{equation}
\begin{aligned}
    g_{js^i}^{\widetilde{s}}[1]
    & =g_{j1^1}^{\widetilde{s}}[1]  
    = g_{j2^1}^{\widetilde{s}-1}[1] - g_{j1^1}^{\widetilde{s}-1}[1] \\
    &= \bigl({\bf e}_{j,1+\widetilde{s}^{1}} - {\bf e}_{j,\widetilde{s}-1^{1}}\bigr) 
       - \bigl({\bf e}_{j\widetilde{s}^{1}} - {\bf e}_{j,\widetilde{s}-1^{1}}\bigr) \\
    &= {\bf e}_{j,1+\widetilde{s}^{1}} - {\bf e}_{j,\widetilde{s}^{1}}.
\end{aligned}
\end{equation}
This implies  \eqref{eq:gpk} holds for $\widetilde{i}=1$, $\widetilde{s}$, and $j(i-1)+s=1$.

If $j(i-1)+s>1$, we further assume that \eqref{eq:gpk} holds for $\widetilde{i} = 1$, $\widetilde{s}$, and $j(i-1)+s-1$. Then by \eqref{eq:recg1}
\begin{equation}
\begin{aligned}
    g_{js^i}^{\widetilde{s}}[1] 
    &= g_{j,(s-1)^i}^{\widetilde{s}}[1] + g_{j,(s+1)^i}^{\widetilde{s}-1}[1] - g_{j,s^i}^{\widetilde{s}-1}[1] \\
    &= \bigl({\bf e}_{j,s+(\widetilde{s}-1)^{i}} - {\bf e}_{j\widetilde{s}^{1}}\bigr) 
       + \bigl({\bf e}_{j,s+\widetilde{s}^{i}} - {\bf e}_{j,\widetilde{s}-1^{1}}\bigr)- \bigl({\bf e}_{j,s+(\widetilde{s}-1)^{i}} - {\bf e}_{j,\widetilde{s}-1^{1}}\bigr) \\
    &= {\bf e}_{j,s+\widetilde{s}^{i}} - {\bf e}_{j,\widetilde{s}^{1}}.
\end{aligned}
\end{equation}

This establishes \eqref{eq:gpk} for all $\widetilde{i}=1$.

We now consider the case $\widetilde{i}>1$. Assume inductively that \eqref{eq:gpk} holds for $\widetilde{i}-1$, arbitrary $\widetilde{s}$, and all $j(i-1)+s-1$. Since $\widetilde{A}_{js^i}^{0}[\widetilde{i}]= \widetilde{A}_{js^i}^{n-1}[\widetilde{i}-1]$, the formula also holds for $\widetilde{s}=0$, any $js^i$, and the given $\widetilde{i}$.

Now let $\widetilde{s} > 0$, and assume that \eqref{eq:gpk} holds for $\widetilde{i}$, $\widetilde{s}-1$, and all $j(i-1)+s$.

If $j(i-1)+s=1$, then $i=s=1$. Thus by \eqref{eq:recg1}
\begin{equation}
\begin{aligned}
    g_{js^i}^{\widetilde{s}}[\widetilde{i}]  &=g_{j1^1}^{\widetilde{s}}[\widetilde{i}] 
    = g_{j2^i}^{\widetilde{s}-1}[\widetilde{i}] - g_{j1^i}^{\widetilde{s}-1}[\widetilde{i}] \\
    &= \bigl({\bf e}_{j,1+\widetilde{s}^{\widetilde{i}}} - {\bf e}_{j,(\widetilde{s}-1)^{\widetilde{i}}}\bigr)- \bigl({\bf e}_{j,1+(\widetilde{s}-1)^{\widetilde{i}}} - {\bf e}_{j,(\widetilde{s}-1)^{\widetilde{i}}}\bigr) \\
    &= {\bf e}_{j,1+\widetilde{s}^{\widetilde{i}}} - {\bf e}_{j,\widetilde{s}^{\widetilde{i}}}.
\end{aligned}
\end{equation}
This implies  \eqref{eq:gpk} holds for $\widetilde{i}$, $\widetilde{s}$, and $j(i-1)+s=1$.

If $s > 0$, we further assume that \eqref{eq:gpk} holds for $\widetilde{i}$, $\widetilde{s}$, and $j(i-1)+s-1$. Then by \eqref{eq:recg1}
\begin{equation}
\begin{aligned}
    &\quad\;\; g_{js^i}^{\widetilde{s}}[\widetilde{i}] 
    = g_{j,(s-1)^i}^{\widetilde{s}}[\widetilde{i}] + g_{j,(s+1)^i}^{\widetilde{s}-1}[\widetilde{i}] - g_{j,s^i}^{\widetilde{s}-1}[\widetilde{i}] \\
    &= \bigl({\bf e}_{j,s+(\widetilde{s}-1)^{i+\widetilde{i}-1}} - {\bf e}_{j\widetilde{s}^{\;\widetilde{i}}}\bigr)+ \bigl({\bf e}_{j,s+\widetilde{s}^{\;i+\widetilde{i}-1}} - {\bf e}_{j,\widetilde{s}-1^{\widetilde{i}}}\bigr)- \bigl({\bf e}_{j,s+(\widetilde{s}-1)^{i+\widetilde{i}-1}} - {\bf e}_{j,\widetilde{s}-1^{\widetilde{i}}}\bigr) \\
    &= {\bf e}_{j,s+\widetilde{s}^{\,i+\widetilde{i}-1}} - {\bf e}_{j,\widetilde{s}^{\,\widetilde{i}}}.
\end{aligned}
\end{equation}

This completes the induction and the proof.
\end{proof}

The following is immediate from Lemma \ref{lem:g-vector-formula1} and Remark \ref{rem:counterpart}.

\begin{corollary}\label{cor:g1}
  For all integers $s,j$ satisfying $1\leq s\leq j\leq n-1$ and all $i$ with $1\leq i\leq k$, we have 
  \begin{enumerate}
      \item[(a)] the $g$-vector of $\overline E^{\rm qc}(js^i)$ with respect to the seed $\overline {\mathsf s}^{\rm qc}_{\lambda}$ is $\mathbf{e}_{js^{i}}-\mathbf{e}_{j,s-1^{i}}$,
      \item[(b)] the $g$-vector of cluster variable in $\overline{\mathsf{t}}_{i}^{\rm sub}$ associated with the vertex $js^1$ with respect to the seed $\overline {\mathsf s}^{\rm qc}_{\lambda}$ is $\mathbf{e}_{js^{i}}-\mathbf{e}_{j0^{i}}$,
  \end{enumerate}
  where $\textbf{e}_{js^i}\in \mathbb Z^{\overline V_{\lambda}}$ denotes the standard vector associated with the vertex labeled $js^i$.
\end{corollary}

\subsection{$g$-vectors for $\mathbb P_3$} In this subsection, let $Q$ denote $\overline Q_{\mathbb{P}_3}$, and let $\widetilde{Q}$ denote its framed quiver.  We label the vertices of $\overline V_\lambda$ as in Figure \ref{Fig:lambdai}. Let $\mathcal{A}^{\mathrm{prin}}_{Q}$ denote the commutative cluster algebra with principal coefficients and initial exchange matrix given by the quiver $Q$.  
For each vertex $v$ of $Q$, let $\widetilde{A}_{v}$ be the corresponding initial cluster variable.

Let $Q^{\rm qc}$ be the quiver obtained from $Q$ by removing all vertices $j0$ for $j=1,\cdots,n-1$ and $ns$ for $s=1,\cdots,n-1$ and $\mathcal A_{Q^{\rm qc}}^{\rm prin}$ denote the commutative cluster algebra with principal coefficients and initial exchange matrix given by $Q^{\rm qc}$. For each vertex $v$ of $Q^{\rm qc}$, let $\widetilde{A}^{\rm qc}_{v}$ be the corresponding initial cluster variable.

Recall the mutation sequence $\widetilde \mu_{(j,s)}^\diamondsuit$ given in \eqref{eq:mutationdiamond}. In this subsection, we compute the $g$-vector of $\widetilde \mu^\diamondsuit_{(j,s)}(\widetilde A_{n-1,s-1})$ and $\widetilde \mu^\diamondsuit_{(j,s)}(\widetilde A^{\rm qc}_{n-1,s-1})$.





It is easy to see that the sequence of mutations $\widetilde{\mu}_{(j,s)}^{\diamondsuit}$ is a green sequence for $\widetilde{Q}$. This helps us compute the $g$-vector of $\widetilde{\mu}^{\diamondsuit}_{(j,s)}(\widetilde{A}_{n-1,s-1})$.

\begin{lemma}\label{lem:g-vector111}
For any integers $2 \leq s \leq j \leq n-1$, the $g$-vector of $\widetilde{\mu}^\diamondsuit_{(j,s)}(\widetilde{A}_{n-1,s-1})$ is
\[
{\bf e}_{j-s+1,0} + {\bf e}_{n,s-1} + {\bf e}_{j,s} - {\bf e}_{j,s-1},
\]
where ${\bf e}_v$ denotes the unit vector associated with the vertex $v$ of the quiver $Q$. 

Consequently, the $g$-vector of $\widetilde{\mu}^\diamondsuit_{(j,s)}(\widetilde{A}^{\mathrm{qc}}_{n-1,s-1})$ is
\[
{\bf e}_{j,s} - {\bf e}_{j,s-1},
\]
where now ${\bf e}_v$ denotes the unit vector associated with the vertex $v$ of the quiver $Q^{\mathrm{qc}}$.
\end{lemma}

\begin{proof}

We first consider the case $s = 2$. We have
\[\widetilde{\mu}^\diamondsuit_{(j,s)}=\mu_{n-1,1}\cdots\mu_{j+1,1}\mu_{j1}.\]

If $j=n-1$, then 
\[g\bigl(\widetilde{\mu}^\diamondsuit_{(j,s)}(\widetilde{A}_{n-1,s-1})\bigr)=g\bigl(\mu_{n-1,s-1}(\widetilde{A}_{n-1,s-1})\bigr)={\bf e}_{n-2,0}+{\bf e}_{n1}+ {\bf e}_{n-1,2} - {\bf e}_{n-1,1}.\]

For $j < n-1$, the following property can be established by induction:

\begin{itemize}
    \item For any integer $j'$ satisfying $j < j' \leq n-1$, let $Q' = \mu_{(j'-1,1)} \circ \cdots \circ \mu_{(j+1,1)} \circ \mu_{(j,1)}(Q).$
    In the quiver $Q'$, there exists an arrow from $(j',1)$ to $v$ (i.e., $Q'((j',1), v) > 0$) if and only if $v = (j'-1, 1)$ or $v = (j'+1, 1)$. Moreover, in either case, the multiplicity of the arrow is exactly one, that is, $Q'((j',1), v) = 1$.
\end{itemize}

Thus, we can inductively prove that
\begin{equation}
   g\bigl(\mu_{j',1} \circ \cdots \circ \mu_{j+1,1} \circ \mu_{j1}(\widetilde{A}_{j',1})\bigr)
= {\bf e}_{j-1,0}+{\bf e}_{j'+1,1}+ {\bf e}_{j2} - {\bf e}_{j1}.
\end{equation}
In particular, we have
\begin{equation}
   g\bigl(\widetilde{\mu}^\diamondsuit_{(j,s)}(\widetilde{A}_{n-1,s-1})\bigr)
   = {\bf e}_{j-1,0}+{\bf e}_{n1}+ {\bf e}_{j2} - {\bf e}_{j1}.
\end{equation}

We then consider the case $s > 2$. We have 
\[\widetilde{\mu}^\diamondsuit_{(j,s)}=(\mu_{n-1,s-1}\cdots \mu_{j+1,s-1}\mu_{j,s-1})\cdots
    (\mu_{n-s+2,2}\cdots\mu_{j-s+4,2} \mu_{j-s+3,2})
    (\mu_{n-s+1,1}\cdots \mu_{j-s+3,1} \mu_{j-s+2,1}).\]

If $j=n-1$, then $\widetilde{\mu}^\diamondsuit_{(j,s)}=
    \mu_{n-1,s-1}\cdots \mu_{n-s+2} \mu_{n-s+1,1}$. We can prove similarly to the case $s=2$ that
  \[g\bigl(\widetilde{\mu}^\diamondsuit_{(j,s)}(\widetilde{A}_{n-1,s-1})\bigr)=g\bigl(\mu_{n-1,s-1}(\widetilde{A}_{n-1,s-1})\bigr)={\bf e}_{n-s,0}+{\bf e}_{n,s-1}+ {\bf e}_{n-1,s} - {\bf e}_{n-1,s-1}.\]  

For $j<n-1$, the following property can be established by induction
\begin{itemize}
    \item For any integer $j'$ satisfying $j < j' \leq n$, let 
    \begin{align*}
      \widetilde \mu_{(j,s;j')}^\diamondsuit & = (\mu_{j'-1,s-1}\cdots \mu_{j+1,s-1}\mu_{j,s-1})(\mu_{n-2,s-2}\cdots \mu_{j,s-2}\mu_{j-1,s-2})\\
      & \quad 
      \cdots(\mu_{n-s+2,2}\cdots\mu_{j-s+4,2} \mu_{j-s+3,2})
    (\mu_{n-s+1,1}\cdots \mu_{j-s+3,1} \mu_{j-s+2,1})
    \end{align*}
    and $Q'=\widetilde \mu_{(j,s;j')}^\diamondsuit(Q)$.
    In the quiver $Q'$, there exists an arrow from $(j',1)$ to $v$ (i.e., $Q'((j',s-1), v) > 0$) if and only if $v = (j'-1, s-1)$ or $v = (j'+1, s-1)$. Moreover, in either case, the multiplicity of the arrow is exactly one, that is, $Q'((j',s-1), v) = 1$.
\end{itemize}

Thus, we can inductively prove that for any $j<j'\leq n$

\begin{equation}
   g\bigl(\widetilde{\mu}^\diamondsuit_{(j,s;j'+1)}(\widetilde{A}_{j',s-1})\bigr)= {\bf e}_{j+s-1,0}+{\bf e}_{j'+1,s-1}+ {\bf e}_{js} - {\bf e}_{j,s-1}.
\end{equation}

In particular, 
\[
g\bigl(\widetilde{\mu}^\diamondsuit_{(j,s)}(\widetilde{A}_{n-1,s-1})\bigr) =g\bigl(\widetilde{\mu}^\diamondsuit_{(j,s;n)}(\widetilde{A}_{n-1,s-1})\bigr)={\bf e}_{j-s+1,0}+{\bf e}_{n,s-1}+ {\bf e}_{js} - {\bf e}_{j,s-1}.
\]


The proof is complete.
\end{proof}

\medskip

\begin{proof}[Proof of Lemma \ref{lem:B-d}]
Part (a) follows directly from the definition of $\overline {\mathsf s}_{\lambda_i}^{\rm qc}$.

Part (b) corresponds to the case $\ell=1$ of Lemma \ref{lem:subseed}(a). Furthermore, by Lemma \ref{lem:subseed}(b), the quiver of $\overline{\mathsf{t}}_{i}^{\rm sub}$ is isomorphic to that of $\overline{\mathsf{s}}_{i}^{\rm sub}$ under the vertex identification $js^1\mapsto js$.

Finally, by Corollaries \ref{cor:g2} and \ref{cor:g1}(b), the $g$-vector of $\overline A^{\rm qc}_{js}\langle i\rangle$ coincides with that of the cluster variable in $\overline{\mathsf{t}}_{i}^{\rm sub}$ associated with the vertex $js^1$ (with respect to the seed $\overline {\mathsf s}^{\rm qc}_{\lambda}$). Since distinct cluster variables possess distinct $g$-vectors (see, e.g., \cite{FK} for the commutative case; by \cite[Theorem 6.1]{BZ}, this statement extends to the quantum case), it follows that $\overline{\mathsf{t}}_{i}^{\rm sub} = \overline{\mathsf{s}}_{i}^{\rm sub}$.

The proof is complete.
\end{proof}

\begin{proof}[Proof of Lemma \ref{lem:B4}]
We first consider the case that $i=1$. Let $\overline E_{k=1}^{\rm qc}(js^1)$ denote the standard cluster variable of $\overline {\mathscr A}^{\rm qc}_{\omega}(\mathbb P_3)$ (for $k=1$) associated with the vertex $js^1$. Then 
\[
\overline E_{k=1}^{\rm qc}(js^1)=\mu_{s-1}^{\rm up}\cdots\mu_2^{\rm up}\mu_{1}^{\rm up}(\overline A^{\rm qc}_{js}(1)).
\]

By Corollary \ref{cor:g1}(a) (for both $k=1$ and general $k$) and Lemma \ref{lem:g-vector111}, the $g$-vector of $\mu_{s-1}^{\rm up}\cdots\mu_{1}^{\rm up}(\overline A^{\rm qc}_{js^1})$ with respect to the seed $\overline{\mathsf s}^{\rm qc}_\lambda$ is $\mathbf{e}_{js^1}-\mathbf{e}_{js^0}$, which coincides with those of $\overline E^{\rm qc}(js^1)$ and $\widetilde{\mu}_{(j,s)}^{\diamondsuit}(\overline{A}^{\rm qc}_{n-1,s-1}\langle 1\rangle)$. As in the proof of Lemma \ref{lem:B-d}, we have 
\[
\overline E^{\rm qc}(js^1)=\mu_{s-1}^{\rm up}\cdots\mu_2^{\rm up}\mu_{1}^{\rm up}(\overline A^{\rm qc}_{js^1})= \widetilde{\mu}_{(j,s)}^{\diamondsuit}(\overline{A}^{\rm qc}_{n-1,s-1}\langle 1\rangle).
\]
This establishes the case $i=1$.

We now consider the general $i$ case. By \eqref{eq:muleftarrow}, we have
\[ \overline E^{\rm qc}(js^i)=\prod_{\Delta_{i}\ni \ell t^{i}\; \prec js^{i}}\overset{\leftarrow}{\mu}_{\ell t^{i}}(\overline E^{\rm qc}(j1^i)),\]
where the mutation sequence is applied to the seed $\overline{\mathsf t}^{\rm qc}_i$, and the factors in the composition are ordered according to the increasing order of the indices with respect to $\prec$ from right to left. 

Note that the mutation sequence $\prod_{\Delta_{i}\ni \ell t^{i}\; \prec js^{i}}\overset{\leftarrow}{\mu}_{\ell t^{i}}$ involves only vertices in 
\[
\bigcup_{j=1}^{k+1-i}I(\Delta_j)\setminus \{11^{k+1-i},\dots,(n-1,n-1)^{k+1-i}\}.
\] 

By applying Lemma \ref{lem:subseed}(a) to the case $\ell=k+1-i$, we can write
\[ \overline E^{\rm qc}(js^i)=\prod_{\Delta_{i}\ni \ell t^{i}\; \prec js^{i}}\overset{\leftarrow}{\mu}_{\ell t^{i}}(\overline E^{\rm qc}(j1^i)),\]
where the mutation sequence is now applied to the sub-seed $\overline{\mathsf t}^{{\rm sub},k+1-i}_i$ of $\overline{\mathsf t}^{\rm qc}_i$. Moreover, by Lemma \ref{lem:subseed}(b), we see that the quiver of $\overline{\mathsf t}^{{\rm sub},k+1-i}_i$ is isomorphic to the quiver of the star-like triangulation of the polygon $\mathbb P_{k+3-i}$. Thus, we can apply the case $i=1$ to the polygon $\mathbb P_{k+3-i}$ and obtain
\[ \overline E^{\rm qc}(js^i)=\mu_{s-1}^{\rm up}\cdots\mu_2^{\rm up}\mu_{1}^{\rm up}(\overline{E}^{\rm qc}_{j1}(1))=\mu_{s-1}^{\rm up}\cdots\mu_2^{\rm up}\mu_{1}^{\rm up}\circ \overleftarrow{\mu}(\Delta_{i-1})\cdots\overleftarrow{\mu}(\Delta_{2})\overleftarrow{\mu}(\Delta_{1})(\overline{A}^{\rm qc}_{j1}(1)),\]
where the first mutation sequence is applied to the seed $\overline{\mathsf t}^{\rm qc}_i$ and the second is applied to the seed $\overline{\mathsf s}^{\rm qc}_{\lambda}$. This proves part (a). 

By Lemma \ref{lem:B-d}(c) and the case $i=1$, we have
\[\mu_{s-1}^{\rm up}\cdots\mu_2^{\rm up}\mu_{1}^{\rm up}(\overline{E}^{\rm qc}_{j1}(1))=\widetilde{\mu}_{(j,s)}^{\diamondsuit}(\overline{E}^{\rm qc}_{n-1,s-1}\langle i\rangle)=\widetilde{\mu}_{(j,s)}^{\diamondsuit}(\overline{A}^{\rm qc}_{n-1,s-1}\langle i\rangle),\]
where the first two mutation sequences are applied to the seed $\overline{\mathsf t}^{\rm qc}_i$ and the last is applied to the seed $\overline{\mathsf s}^{\rm qc}_{\lambda_i}$. This proves part (b). 

The proof is complete.
\end{proof}

\medskip

\begin{proof}[Proof of Lemma \ref{lem:quasi2}]
(a) By Corollary \ref{cor:g2}, the $g$-vector of $\overline{A}_{js}\langle i\rangle$ in $\mathcal{A}$ with respect to the seed $\overline{\mathsf s}_\lambda$ is
\[
\mathbf{e}_{n j^{\,i-1}} + \mathbf{e}_{j s^{\,i}} - \mathbf{e}_{j 0^{\,i}}.
\]
Thus, the Laurent expansion of $\overline{A}_{js}\langle i\rangle$ with respect to the seed $\overline{\mathsf{s}}_\lambda$ contains a unique term
\[
[A_{n j}(i-1) \cdot A_{j s}(i) \cdot A^{-1}_{j 0}(i)]
\]
with coefficient $1$.

Similarly, by Corollary \ref{cor:g2} the Laurent expansion of $\overline{A}^{\mathrm{qc}}_{js}\langle i\rangle$ with respect to the seed $\overline{\mathsf{s}}^{\mathrm{qc}}_\lambda$ contains a unique term
\[
[A^{\mathrm{qc}}_{j s}(i) \cdot (A^{\mathrm{qc}}_{j 0}(i))^{-1}]
\]
with coefficient $1$.

Moreover, by \cite[Proposition 2.7]{CHL}, we have $\overline{A}_{js}\langle i\rangle \asymp \overline{A}^{\mathrm{qc}}_{js}\langle i\rangle$ (see Definition \ref{def:proportional}). The desired equality then follows from \eqref{eq:qc}.

(b) As in part (a), Lemma \ref{lem:g-vector111} implies that the Laurent expansion of $\widetilde{\mu}_{(j,s)}^{\diamondsuit}(\overline{A}_{j-s+2,1}(1))$ with respect to the seed $\overline{\mathsf{s}}_\lambda$ contains a unique term
\[
[\overline{A}_{j-s+1,0}(1) \cdot \overline{A}_{j,s}(1) \cdot \overline{A}_{n,s-1}(1) \cdot \overline{A}_{j,s-1}^{-1}(1)].
\]

On the other hand, the Laurent expansion of $\widetilde{\mu}_{(j,s)}^{\diamondsuit}(\overline{A}^{\mathrm{qc}}_{j-s+2,1}(1))$ with respect to $\overline{\mathsf{s}}^{\rm qc}_\lambda$ contains a unique term
\[
[\overline{A}^{\mathrm{qc}}_{j,s}(1) \cdot (\overline{A}^{\mathrm{qc}}_{j,s-1}(1))^{-1}]
= [\overline{A}_{j,s}(1) \cdot \overline{A}_{n s}^{-1}(1) \cdot \overline{A}_{j 0}^{-1}(1) \cdot \overline{A}_{j,s-1}^{-1}(1)  \cdot \overline{A}_{n,s-1}(1)\cdot \overline{A}_{j0}(1)].
\]

By \cite[Proposition 2.7]{CHL}, we have
\[
\widetilde{\mu}_{(j,s)}^{\diamondsuit}(\overline{A}^{\mathrm{qc}}_{j-s+2,1}(1)) \asymp \widetilde{\mu}_{(j,s)}^{\diamondsuit}(\overline{A}_{j-s+2,1}(1)). \qquad{\text{(see Definition \ref{def:proportional})}}
\]
Therefore,
\begin{equation*}
\widetilde{\mu}_{(j,s)}^{\diamondsuit}(\overline{A}^{\mathrm{qc}}_{j-s+2,1}(1))
= [\widetilde{\mu}_{(j,s)}^{\diamondsuit}(\overline{A}_{j-s+2,1}(1)) \cdot \overline{A}_{j-s+1,0}^{-1}(1) \cdot \overline{A}_{n s}^{-1}(1)].
\end{equation*}

(c) The statement can be proved similarly to part (b).
\end{proof}

\section{A linear order on exchangeable cluster variables in $A(S)$}\label{app-order}

When an exchangeable cluster variable in $A(S)$ is neither $\alpha$ nor $\cev{\alpha}$, it corresponds to a labeled arc in $S$. 
We assign the red labels $\{1,2,3,4\}$ to these arcs; this induces a corresponding labeling of the exchangeable cluster variables. 
This labeling determines a linear order, with the convention that
\begin{align*}
    \text{variable labeled }1 \;>\; \text{variable labeled }2 \;>\; \text{variable labeled }3 \;>\; \text{variable labeled }4.
\end{align*}

We distinguish two main cases: either $A(S)$ contains neither $\alpha$ nor $\cev{\alpha}$, or it contains one of them.

\medskip
\noindent
\textbf{Case 1:} $A(S)$ contains neither $\alpha$ nor $\cev{\alpha}$.

In this case, we further divide into two subcases depending on whether the labeled arcs in $S$ intersect in the interior of $\mathbb P_2$.

\smallskip
\noindent
\textit{Subcase 1:} The labeled arcs in $S$ intersect in the interior of $\mathbb P_2$.

\medskip

$
\raisebox{-0.25in}{
    \begin{tikzpicture}
    \tikzset{->-/.style={
        decoration={markings,mark=at position #1 with {\arrow{latex}}},
        postaction={decorate}
    }}

    \filldraw[draw=white,fill=gray!20] (0,0) rectangle (1.5, 1.5);
    \draw[line width=1pt] (0,0)--(0,1.5);
    \draw[line width=1pt] (1.5,0)--(1.5,1.5);
    \draw [line width =0.8pt] (0,1.2)--(1.5,0.3);
    \draw [line width =0.8pt] (0,1.2)--(1.5,0.75);
    \draw [line width =0.8pt] (0,0.75)--(1.5,0.3);
    \draw [line width =0.8pt] (0,0.75)--(1.5,1.2);
       \node at(0,0.3) {$\bullet$};
    \node[left] at(0,0.3) {$3$};
    \node at(0,0.75) {$\bullet$};
    \node[left] at(0,0.75) {$2$};
    \node at(0,1.2) {$\bullet$};
    \node[left] at(0,1.2) {$1$};

    \node at(1.5,0.3) {$\bullet$};
    \node[right] at(1.5,0.3) {$1$};
    \node at(1.5,0.75) {$\bullet$};
    \node[right] at(1.5,0.75) {$2$};
    \node at(1.5,1.2) {$\bullet$};
    \node[right] at(1.5,1.2) {$3$};

    \node [text=red] at(0.3,0.65) {$1$};
    \node [text=red] at(0.75,0.75) {$2$};
    \node [text=red] at(1.2,0.85) {$3$};
    \node [text=red] at(1.2,1.15) {$4$};
    \end{tikzpicture}
    }
$,\quad
$
\raisebox{-0.25in}{
    \begin{tikzpicture}
    \tikzset{->-/.style={
        decoration={markings,mark=at position #1 with {\arrow{latex}}},
        postaction={decorate}
    }}

    \filldraw[draw=white,fill=gray!20] (0,0) rectangle (1.5, 1.5);
    \draw[line width=1pt] (0,0)--(0,1.5);
    \draw[line width=1pt] (1.5,0)--(1.5,1.5);
    
    \draw [line width =0.8pt] (0,1.2)--(1.5,0.3);
    
    
    
    \draw [line width =0.8pt] (0,0.75)--(1.5,0.75);
     \draw [line width =0.8pt] 
            (0,1.2) .. controls (0.5,1.35) and (1,1.2) .. (1.5,0.75);

    \draw [line width =0.8pt] (0,0.75)--(1.5,0.75);
     \draw [line width =0.8pt] 
            (0,0.75) .. controls (0.5,0.3) and (1,0.25) .. (1.5,0.3);

       \node at(0,0.3) {$\bullet$};
    \node[left] at(0,0.3) {$3$};
    \node at(0,0.75) {$\bullet$};
    \node[left] at(0,0.75) {$2$};
    \node at(0,1.2) {$\bullet$};
    \node[left] at(0,1.2) {$1$};

    \node at(1.5,0.3) {$\bullet$};
    \node[right] at(1.5,0.3) {$1$};
    \node at(1.5,0.75) {$\bullet$};
    \node[right] at(1.5,0.75) {$2$};
    \node at(1.5,1.2) {$\bullet$};
    \node[right] at(1.5,1.2) {$3$};

    \node [text=red] at(0.75,0.35) {$1$};
    \node [text=red] at(1.2,0.75) {$2$};
    \node [text=red] at(0.35,0.96) {$3$};
    \node [text=red] at(0.75,1.15) {$4$};
    \end{tikzpicture}
    }
$,\quad
$
\raisebox{-0.25in}{
    \begin{tikzpicture}
    \tikzset{->-/.style={
        decoration={markings,mark=at position #1 with {\arrow{latex}}},
        postaction={decorate}
    }}

    \filldraw[draw=white,fill=gray!20] (0,0) rectangle (1.5, 1.5);
    \draw[line width=1pt] (0,0)--(0,1.5);
    \draw[line width=1pt] (1.5,0)--(1.5,1.5);
    
    \draw [line width =0.8pt] (0,0.3)--(1.5,1.2);
    
    
    
    \draw [line width =0.8pt] (0,0.75)--(1.5,0.75);
     \draw [line width =0.8pt] 
            (1.5,1.2) .. controls (1,1.35) and (0.5,1.2) .. (0,0.75);

    \draw [line width =0.8pt] (0,0.75)--(1.5,0.75);
     \draw [line width =0.8pt] 
            (1.5,0.75) .. controls (1,0.3) and (0.5,0.25) .. (0,0.3);

       \node at(0,0.3) {$\bullet$};
    \node[left] at(0,0.3) {$3$};
    \node at(0,0.75) {$\bullet$};
    \node[left] at(0,0.75) {$2$};
    \node at(0,1.2) {$\bullet$};
    \node[left] at(0,1.2) {$1$};

    \node at(1.5,0.3) {$\bullet$};
    \node[right] at(1.5,0.3) {$1$};
    \node at(1.5,0.75) {$\bullet$};
    \node[right] at(1.5,0.75) {$2$};
    \node at(1.5,1.2) {$\bullet$};
    \node[right] at(1.5,1.2) {$3$};

    \node [text=red] at(0.75,0.35) {$1$};
    \node [text=red] at(1.2,0.75) {$3$};
    \node [text=red] at(0.3,0.54) {$2$};
    \node [text=red] at(0.75,1.2) {$4$};
    \end{tikzpicture}
    }
$,\quad
$
\raisebox{-0.25in}{
    \begin{tikzpicture}
    \tikzset{->-/.style={
        decoration={markings,mark=at position #1 with {\arrow{latex}}},
        postaction={decorate}
    }}

    \filldraw[draw=white,fill=gray!20] (0,0) rectangle (1.5, 1.5);
    \draw[line width=1pt] (0,0)--(0,1.5);
    \draw[line width=1pt] (1.5,0)--(1.5,1.5);
    \draw [line width =0.8pt] (0,1.2)--(1.5,0.3);
    \draw [line width =0.8pt] (0,1.2)--(1.5,0.75);
    \draw [line width =0.8pt] (0,0.75)--(1.5,0.3);
    \draw [line width =0.8pt] (0,0.3)--(1.5,0.75);
       \node at(0,0.3) {$\bullet$};
    \node[left] at(0,0.3) {$3$};
    \node at(0,0.75) {$\bullet$};
    \node[left] at(0,0.75) {$2$};
    \node at(0,1.2) {$\bullet$};
    \node[left] at(0,1.2) {$1$};

    \node at(1.5,0.3) {$\bullet$};
    \node[right] at(1.5,0.3) {$1$};
    \node at(1.5,0.75) {$\bullet$};
    \node[right] at(1.5,0.75) {$2$};
    \node at(1.5,1.2) {$\bullet$};
    \node[right] at(1.5,1.2) {$3$};

    \node [text=red] at(0.3,0.65) {$2$};
    \node [text=red] at(0.75,0.75) {$3$};
    \node [text=red] at(1.2,0.85) {$4$};
    \node [text=red] at(0.2,0.35) {$1$};
    \end{tikzpicture}
    }
$,\quad
$
\raisebox{-0.25in}{
    \begin{tikzpicture}
    \tikzset{->-/.style={
        decoration={markings,mark=at position #1 with {\arrow{latex}}},
        postaction={decorate}
    }}

    \filldraw[draw=white,fill=gray!20] (0,0) rectangle (1.5, 1.5);
    \draw[line width=1pt] (0,0)--(0,1.5);
    \draw[line width=1pt] (1.5,0)--(1.5,1.5);
    
    \draw [line width =0.8pt] (0,0.3)--(1.5,1.2);

    \draw [line width =0.8pt] (0,1.2)--(1.5,0.75);
     \draw [line width =0.8pt] 
            (1.5,1.2) -- (0,0.75);
     \draw [line width =0.8pt] 
            (1.5,0.75) -- (0,0.3);

       \node at(0,0.3) {$\bullet$};
    \node[left] at(0,0.3) {$3$};
    \node at(0,0.75) {$\bullet$};
    \node[left] at(0,0.75) {$2$};
    \node at(0,1.2) {$\bullet$};
    \node[left] at(0,1.2) {$1$};

    \node at(1.5,0.3) {$\bullet$};
    \node[right] at(1.5,0.3) {$1$};
    \node at(1.5,0.75) {$\bullet$};
    \node[right] at(1.5,0.75) {$2$};
    \node at(1.5,1.2) {$\bullet$};
    \node[right] at(1.5,1.2) {$3$};

    \node [text=red] at(0.5,0.4) {$1$};
    \node [text=red] at(0.2,0.8) {$3$};
    \node [text=red] at(0.75,0.75) {$2$};
    \node [text=red] at (0.2,1.15) {$4$};
    \end{tikzpicture}
    }
$,

\medskip

$
\raisebox{-0.25in}{
    \begin{tikzpicture}
    \tikzset{->-/.style={
        decoration={markings,mark=at position #1 with {\arrow{latex}}},
        postaction={decorate}
    }}

    \filldraw[draw=white,fill=gray!20] (0,0) rectangle (1.5, 1.5);
    \draw[line width=1pt] (0,0)--(0,1.5);
    \draw[line width=1pt] (1.5,0)--(1.5,1.5);
    
    \draw [line width =0.8pt] (0,0.3)--(1.5,1.2);

    \draw [line width =0.8pt] (0,0.75)--(1.5,0.3);
     \draw [line width =0.8pt] 
            (1.5,1.2) -- (0,0.75);
     \draw [line width =0.8pt] 
            (1.5,0.75) -- (0,0.3);

       \node at(0,0.3) {$\bullet$};
    \node[left] at(0,0.3) {$3$};
    \node at(0,0.75) {$\bullet$};
    \node[left] at(0,0.75) {$2$};
    \node at(0,1.2) {$\bullet$};
    \node[left] at(0,1.2) {$1$};

    \node at(1.5,0.3) {$\bullet$};
    \node[right] at(1.5,0.3) {$1$};
    \node at(1.5,0.75) {$\bullet$};
    \node[right] at(1.5,0.75) {$2$};
    \node at(1.5,1.2) {$\bullet$};
    \node[right] at(1.5,1.2) {$3$};

    \node [text=red] at(1.3,0.35) {$1$};
    \node [text=red] at(0.8,0.8) {$3$};
    \node [text=red] at(1.2,0.65) {$2$};
    \node [text=red] at (1,1.12) {$4$};
    \end{tikzpicture}
    }
$,\quad
$
\raisebox{-0.25in}{
    \begin{tikzpicture}
    \tikzset{->-/.style={
        decoration={markings,mark=at position #1 with {\arrow{latex}}},
        postaction={decorate}
    }}

    \filldraw[draw=white,fill=gray!20] (0,0) rectangle (1.5, 1.5);
    \draw[line width=1pt] (0,0)--(0,1.5);
    \draw[line width=1pt] (1.5,0)--(1.5,1.5);
    
    \draw [line width =0.8pt] (0,0.75)--(1.5,0.75);

    \draw [line width =0.8pt] (0,1.2)--(1.5,0.3);
     \draw [line width =0.8pt] 
            (1.5,1.2) -- (0,0.75);
     \draw [line width =0.8pt] 
            (1.5,0.75) -- (0,0.3);

       \node at(0,0.3) {$\bullet$};
    \node[left] at(0,0.3) {$3$};
    \node at(0,0.75) {$\bullet$};
    \node[left] at(0,0.75) {$2$};
    \node at(0,1.2) {$\bullet$};
    \node[left] at(0,1.2) {$1$};

    \node at(1.5,0.3) {$\bullet$};
    \node[right] at(1.5,0.3) {$1$};
    \node at(1.5,0.75) {$\bullet$};
    \node[right] at(1.5,0.75) {$2$};
    \node at(1.5,1.2) {$\bullet$};
    \node[right] at(1.5,1.2) {$3$};

    \node [text=red] at(0.2,0.35) {$1$};
    \node [text=red] at(0.3,1.05) {$3$};
    \node [text=red] at(1.05,0.8) {$2$};
    \node [text=red] at (1.3,1.15) {$4$};
    \end{tikzpicture}
    }
$,\quad
$
\raisebox{-0.25in}{
    \begin{tikzpicture}
    \tikzset{->-/.style={
        decoration={markings,mark=at position #1 with {\arrow{latex}}},
        postaction={decorate}
    }}

    \filldraw[draw=white,fill=gray!20] (0,0) rectangle (1.5, 1.5);
    \draw[line width=1pt] (0,0)--(0,1.5);
    \draw[line width=1pt] (1.5,0)--(1.5,1.5);
    
    \draw [line width =0.8pt] (0,0.3)--(1.5,1.2);
    
    
    
    \draw [line width =0.8pt] (0,0.75)--(1.5,0.75);
     \draw [line width =0.8pt] 
            (0,1.2) -- (1.5,0.75);

    \draw [line width =0.8pt] (0,0.75)--(1.5,0.75);
     \draw [line width =0.8pt] 
            (0,0.75) -- (1.5,0.3);

       \node at(0,0.3) {$\bullet$};
    \node[left] at(0,0.3) {$3$};
    \node at(0,0.75) {$\bullet$};
    \node[left] at(0,0.75) {$2$};
    \node at(0,1.2) {$\bullet$};
    \node[left] at(0,1.2) {$1$};

    \node at(1.5,0.3) {$\bullet$};
    \node[right] at(1.5,0.3) {$1$};
    \node at(1.5,0.75) {$\bullet$};
    \node[right] at(1.5,0.75) {$2$};
    \node at(1.5,1.2) {$\bullet$};
    \node[right] at(1.5,1.2) {$3$};

    \node [text=red] at(1.3,0.35) {$1$};
    \node [text=red] at(1.05,0.69) {$2$};
    \node [text=red] at(1.2,1.05) {$3$};
    \node [text=red] at(0.2,1.15) {$4$};
    \end{tikzpicture}
    }
$,\quad
$
\raisebox{-0.25in}{
    \begin{tikzpicture}
    \tikzset{->-/.style={
        decoration={markings,mark=at position #1 with {\arrow{latex}}},
        postaction={decorate}
    }}

    \filldraw[draw=white,fill=gray!20] (0,0) rectangle (1.5, 1.5);
    \draw[line width=1pt] (0,0)--(0,1.5);
    \draw[line width=1pt] (1.5,0)--(1.5,1.5);
    
    \draw [line width =0.8pt] (0,0.3)--(1.5,1.2);
    
    
    
    \draw [line width =0.8pt] (0,0.75)--(1.5,0.75);
     \draw [line width =0.8pt] 
            (0,1.2) -- (1.5,0.75);

    \draw [line width =0.8pt] (0,0.75)--(1.5,0.75);
     \draw [line width =0.8pt] (0,0.3) .. controls (0.5,0.2) and (1,0.3) .. (1.5,0.75);

       \node at(0,0.3) {$\bullet$};
    \node[left] at(0,0.3) {$3$};
    \node at(0,0.75) {$\bullet$};
    \node[left] at(0,0.75) {$2$};
    \node at(0,1.2) {$\bullet$};
    \node[left] at(0,1.2) {$1$};

    \node at(1.5,0.3) {$\bullet$};
    \node[right] at(1.5,0.3) {$1$};
    \node at(1.5,0.75) {$\bullet$};
    \node[right] at(1.5,0.75) {$2$};
    \node at(1.5,1.2) {$\bullet$};
    \node[right] at(1.5,1.2) {$3$};

    \node [text=red] at(0.9,0.35) {$1$};
    \node [text=red] at(0.4,0.5) {$2$};
    \node [text=red] at(0.2,0.75) {$3$};
    \node [text=red] at(0.2,1.15) {$4$};
    \end{tikzpicture}
    }
$,\quad
$
\raisebox{-0.25in}{
    \begin{tikzpicture}
    \tikzset{->-/.style={
        decoration={markings,mark=at position #1 with {\arrow{latex}}},
        postaction={decorate}
    }}

    \filldraw[draw=white,fill=gray!20] (0,0) rectangle (1.5, 1.5);
    \draw[line width=1pt] (0,0)--(0,1.5);
    \draw[line width=1pt] (1.5,0)--(1.5,1.5);
    
    \draw [line width =0.8pt] (0,1.2)--(1.5,0.3);
    
    
    
    \draw [line width =0.8pt] (0,0.75)--(1.5,0.75);
     \draw [line width =0.8pt] 
            (0,1.2) .. controls (0.5,1.3) and (1,1.2) .. (1.5,0.75);

    \draw [line width =0.8pt] (0,0.75)--(1.5,0.75);
     \draw [line width =0.8pt] (0,0.3) -- (1.5,0.75);

       \node at(0,0.3) {$\bullet$};
    \node[left] at(0,0.3) {$3$};
    \node at(0,0.75) {$\bullet$};
    \node[left] at(0,0.75) {$2$};
    \node at(0,1.2) {$\bullet$};
    \node[left] at(0,1.2) {$1$};

    \node at(1.5,0.3) {$\bullet$};
    \node[right] at(1.5,0.3) {$1$};
    \node at(1.5,0.75) {$\bullet$};
    \node[right] at(1.5,0.75) {$2$};
    \node at(1.5,1.2) {$\bullet$};
    \node[right] at(1.5,1.2) {$3$};

    \node [text=red] at(0.2,0.35) {$1$};
    \node [text=red] at(0.3,0.75) {$2$};
    \node [text=red] at(0.2,1.04) {$3$};
    \node [text=red] at(0.7,1.2) {$4$};
    \end{tikzpicture}
    }
$

\medskip

$
\raisebox{-0.25in}{
    \begin{tikzpicture}
    \tikzset{->-/.style={
        decoration={markings,mark=at position #1 with {\arrow{latex}}},
        postaction={decorate}
    }}

    \filldraw[draw=white,fill=gray!20] (0,0) rectangle (1.5, 1.5);
    \draw[line width=1pt] (0,0)--(0,1.5);
    \draw[line width=1pt] (1.5,0)--(1.5,1.5);
    
    \draw [line width =0.8pt] (0,0.75)--(1.5,0.75);

    \draw [line width =0.8pt] (0,1.2)--(1.5,0.3);
     \draw [line width =0.8pt] 
            (1.5,1.2) -- (0,0.75);
     \draw [line width =0.8pt] 
            (0,0.75) .. controls (0.5,0.3) and (1,0.15) .. (1.5,0.3);

       \node at(0,0.3) {$\bullet$};
    \node[left] at(0,0.3) {$3$};
    \node at(0,0.75) {$\bullet$};
    \node[left] at(0,0.75) {$2$};
    \node at(0,1.2) {$\bullet$};
    \node[left] at(0,1.2) {$1$};

    \node at(1.5,0.3) {$\bullet$};
    \node[right] at(1.5,0.3) {$1$};
    \node at(1.5,0.75) {$\bullet$};
    \node[right] at(1.5,0.75) {$2$};
    \node at(1.5,1.2) {$\bullet$};
    \node[right] at(1.5,1.2) {$3$};

    \node [text=red] at(0.8,0.3) {$1$};
    \node [text=red] at(1.13,0.75) {$3$};
        \node [text=red] at(1.3,0.45) {$2$};
    \node [text=red] at (1.3,1.15) {$4$};
    \end{tikzpicture}
    }
$,\quad
$
\raisebox{-0.25in}{
    \begin{tikzpicture}
    \tikzset{->-/.style={
        decoration={markings,mark=at position #1 with {\arrow{latex}}},
        postaction={decorate}
    }}

    \filldraw[draw=white,fill=gray!20] (0,0) rectangle (1.5, 1.5);
    \draw[line width=1pt] (0,0)--(0,1.5);
    \draw[line width=1pt] (1.5,0)--(1.5,1.5);
    
    \draw [line width =0.8pt] (0,0.75)--(1.5,0.75);

    \draw [line width =0.8pt] (0,0.3)--(1.5,1.2);
     \draw [line width =0.8pt] 
          (0,0.75)  .. controls (0.5,1.2) and (1,1.35) .. (1.5,1.2);
     \draw [line width =0.8pt] 
            (0,0.75) -- (1.5,0.3);

       \node at(0,0.3) {$\bullet$};
    \node[left] at(0,0.3) {$3$};
    \node at(0,0.75) {$\bullet$};
    \node[left] at(0,0.75) {$2$};
    \node at(0,1.2) {$\bullet$};
    \node[left] at(0,1.2) {$1$};

    \node at(1.5,0.3) {$\bullet$};
    \node[right] at(1.5,0.3) {$1$};
    \node at(1.5,0.75) {$\bullet$};
    \node[right] at(1.5,0.75) {$2$};
    \node at(1.5,1.2) {$\bullet$};
    \node[right] at(1.5,1.2) {$3$};

    \node [text=red] at(1.3,0.4) {$1$};
    \node [text=red] at(1.13,0.75) {$2$};
        \node [text=red] at(1.3,1.04) {$3$};
    \node [text=red] at (0.8,1.25) {$4$};
    \end{tikzpicture}
    }
$,\quad
$
\raisebox{-0.25in}{
    \begin{tikzpicture}
    \tikzset{->-/.style={
        decoration={markings,mark=at position #1 with {\arrow{latex}}},
        postaction={decorate}
    }}

    \filldraw[draw=white,fill=gray!20] (0,0) rectangle (1.5, 1.5);
    \draw[line width=1pt] (0,0)--(0,1.5);
    \draw[line width=1pt] (1.5,0)--(1.5,1.5);
    
    \draw [line width =0.8pt] (0,1.2)--(1.5,0.75);
    
    \draw [line width =0.8pt] (0,0.75)--(1.5,1.2);
    
 \draw [line width =0.8pt] (0,0.75)--(1.5,0.3);
 \draw [line width =0.8pt] (0,0.3)--(1.5,0.75);

       \node at(0,0.3) {$\bullet$};
    \node[left] at(0,0.3) {$3$};
    \node at(0,0.75) {$\bullet$};
    \node[left] at(0,0.75) {$2$};
    \node at(0,1.2) {$\bullet$};
    \node[left] at(0,1.2) {$1$};

    \node at(1.5,0.3) {$\bullet$};
    \node[right] at(1.5,0.3) {$1$};
    \node at(1.5,0.75) {$\bullet$};
    \node[right] at(1.5,0.75) {$2$};
    \node at(1.5,1.2) {$\bullet$};
    \node[right] at(1.5,1.2) {$3$};

    \node [text=red] at(0.2,0.35) {$1$};
    \node [text=red] at(0.3,0.65) {$2$};
    \node [text=red] at(1.15,0.85) {$4$};
    \node [text=red] at(1.3,1.15) {$3$};
    \end{tikzpicture}
    }
$.

\medskip

In each of the above $13$ figures, there are two possible orientations for the four labeled arcs. Hence, this subcase yields a total of $26$ clusters.

\medskip

\noindent
\textit{Subcase 2:} The labeled arcs in $S$ do not intersect in the interior of $\mathbb P_2$.

\medskip

$
\raisebox{-0.25in}{
    \begin{tikzpicture}
    \tikzset{->-/.style={
        decoration={markings,mark=at position #1 with {\arrow{latex}}},
        postaction={decorate}
    }}

    \filldraw[draw=white,fill=gray!20] (0,0) rectangle (1.5, 1.5);
    \draw[line width=1pt] (0,0)--(0,1.5);
    \draw[line width=1pt] (1.5,0)--(1.5,1.5);
    
    \draw [line width =0.8pt,decoration={markings, mark=at position 0.5 with {\arrow{<}}},postaction={decorate}]  (0,1.2).. controls (0.5,1.25) and (1,0.85) ..(1.5,0.3);
    \draw [line width =0.8pt,decoration={markings, mark=at position 0.5 with {\arrow{>}}},postaction={decorate}] (0,1.2).. controls (0.5,0.65) and (1,0.25) ..(1.5,0.3);
    
     \draw [line width =0.8pt] 
            (0,1.2) .. controls (0.5,1.35) and (1,1.35) .. (1.5,0.75);

     \draw [line width =0.8pt] 
            (0,0.75) .. controls (0.5,0.15) and (1,0.15) .. (1.5,0.3);

       \node at(0,0.3) {$\bullet$};
    \node[left] at(0,0.3) {$3$};
    \node at(0,0.75) {$\bullet$};
    \node[left] at(0,0.75) {$2$};
    \node at(0,1.2) {$\bullet$};
    \node[left] at(0,1.2) {$1$};

    \node at(1.5,0.3) {$\bullet$};
    \node[right] at(1.5,0.3) {$1$};
    \node at(1.5,0.75) {$\bullet$};
    \node[right] at(1.5,0.75) {$2$};
    \node at(1.5,1.2) {$\bullet$};
    \node[right] at(1.5,1.2) {$3$};

    \node [text=red] at(0.75,0.24) {$1$};
    \node [text=red] at(1.2,0.62) {$3$};
    \node [text=red] at(0.35,0.86) {$2$};
    \node [text=red] at(0.75,1.25) {$4$};
    \end{tikzpicture}
    }
$,\quad
$
\raisebox{-0.25in}{
    \begin{tikzpicture}
    \tikzset{->-/.style={
        decoration={markings,mark=at position #1 with {\arrow{latex}}},
        postaction={decorate}
    }}

    \filldraw[draw=white,fill=gray!20] (0,0) rectangle (1.5, 1.5);
    \draw[line width=1pt] (0,0)--(0,1.5);
    \draw[line width=1pt] (1.5,0)--(1.5,1.5);
    
    \draw [line width =0.8pt,decoration={markings, mark=at position 0.5 with {\arrow{>}}},postaction={decorate}]  (1.5,1.2).. controls (1,1.25) and (0.5,0.85) ..(0,0.3);
    \draw [line width =0.8pt,decoration={markings, mark=at position 0.5 with {\arrow{<}}},postaction={decorate}] (1.5,1.2).. controls (1,0.65) and (0.5,0.25) ..(0,0.3);
    
     \draw [line width =0.8pt] 
            (1.5,1.2) .. controls (1,1.35) and (0.5,1.35) .. (0,0.75);

     \draw [line width =0.8pt] 
            (1.5,0.75) .. controls (1,0.15) and (0.5,0.15) .. (0,0.3);

       \node at(0,0.3) {$\bullet$};
    \node[left] at(0,0.3) {$3$};
    \node at(0,0.75) {$\bullet$};
    \node[left] at(0,0.75) {$2$};
    \node at(0,1.2) {$\bullet$};
    \node[left] at(0,1.2) {$1$};

    \node at(1.5,0.3) {$\bullet$};
    \node[right] at(1.5,0.3) {$1$};
    \node at(1.5,0.75) {$\bullet$};
    \node[right] at(1.5,0.75) {$2$};
    \node at(1.5,1.2) {$\bullet$};
    \node[right] at(1.5,1.2) {$3$};

    \node [text=red] at(0.75,0.24) {$1$};
    \node [text=red] at(0.3,0.62) {$3$};
    \node [text=red] at(1.15,0.86) {$2$};
    \node [text=red] at(0.75,1.25) {$4$};
    \end{tikzpicture}
    }
$.

\medskip
In each of the above two figures, there are four possible orientations for the two unoriented labeled arcs. Therefore, this subcase yields a total of $8$ clusters.

\medskip

\textbf{Case 2:} $A(S)$ contains $\alpha$ or $\cev{\alpha}$. 
Note that, besides $\alpha$ or $\cev{\alpha}$, the set $A(S)$ contains three other exchangeable cluster variables. 
In the following, we describe the labeling of these three variables by elements of $\{1,2,3,4\}$; the remaining label is then assigned to $\alpha$ or $\cev{\alpha}$.

In this case, we further divide into two subcases according to whether $A(S)$ contains $\alpha$ or $\cev{\alpha}$.

\smallskip

\noindent
\textit{Subcase 1:} $A(S)$ contains $\alpha$.
\medskip

$
\raisebox{-0.25in}{
    \begin{tikzpicture}
    \tikzset{->-/.style={
        decoration={markings,mark=at position #1 with {\arrow{latex}}},
        postaction={decorate}
    }}

    \filldraw[draw=white,fill=gray!20] (0,0) rectangle (1.5, 1.5);
    \draw[line width=1pt] (0,0)--(0,1.5);
    \draw[line width=1pt] (1.5,0)--(1.5,1.5);
    
    \draw  [line width =0.8pt,decoration={markings, mark=at position 0.75 with {\arrow{<}}},postaction={decorate}] (0,0.3)--(1.5,1.2);

    \draw [line width =0.8pt,decoration={markings, mark=at position 0.2 with {\arrow{>}}},postaction={decorate}] (0,0.75)--(1.5,0.3);
     \draw  [line width =0.8pt,decoration={markings, mark=at position 0.39 with {\arrow{>}}},postaction={decorate}] 
            (1.5,0.75) -- (0,0.3);

       \node at(0,0.3) {$\bullet$};
    \node[left] at(0,0.3) {$3$};
    \node at(0,0.75) {$\bullet$};
    \node[left] at(0,0.75) {$2$};
    \node at(0,1.2) {$\bullet$};
    \node[left] at(0,1.2) {$1$};

    \node at(1.5,0.3) {$\bullet$};
    \node[right] at(1.5,0.3) {$1$};
    \node at(1.5,0.75) {$\bullet$};
    \node[right] at(1.5,0.75) {$2$};
    \node at(1.5,1.2) {$\bullet$};
    \node[right] at(1.5,1.2) {$3$};

    \node [text=red] at(1.3,0.35) {$1$};
    \node [text=red] at(0.8,0.8) {$4$};
    \node [text=red] at(1.2,0.65) {$2$};
    \end{tikzpicture}
    }
$,\quad
$
\raisebox{-0.25in}{
    \begin{tikzpicture}
    \tikzset{->-/.style={
        decoration={markings,mark=at position #1 with {\arrow{latex}}},
        postaction={decorate}
    }}

    \filldraw[draw=white,fill=gray!20] (0,0) rectangle (1.5, 1.5);
    \draw[line width=1pt] (0,0)--(0,1.5);
    \draw[line width=1pt] (1.5,0)--(1.5,1.5);
    
    \draw [line width =0.8pt,decoration={markings, mark=at position 0.2 with {\arrow{<}}},postaction={decorate}] (0,0.3)--(1.5,1.2);

     \draw [line width =0.8pt,decoration={markings, mark=at position 0.4 with {\arrow{>}}},postaction={decorate}]
            (0,1.2) -- (1.5,0.75);

     \draw [line width =0.8pt,decoration={markings, mark=at position 0.75 with {\arrow{>}}},postaction={decorate}]
            (0,0.75) -- (1.5,0.3);

       \node at(0,0.3) {$\bullet$};
    \node[left] at(0,0.3) {$3$};
    \node at(0,0.75) {$\bullet$};
    \node[left] at(0,0.75) {$2$};
    \node at(0,1.2) {$\bullet$};
    \node[left] at(0,1.2) {$1$};

    \node at(1.5,0.3) {$\bullet$};
    \node[right] at(1.5,0.3) {$1$};
    \node at(1.5,0.75) {$\bullet$};
    \node[right] at(1.5,0.75) {$2$};
    \node at(1.5,1.2) {$\bullet$};
    \node[right] at(1.5,1.2) {$3$};

    \node [text=red] at(1.3,0.35) {$1$};
    \node [text=red] at(1.2,1.05) {$3$};
    \node [text=red] at(0.2,1.15) {$4$};
    \end{tikzpicture}
    }
$,\quad
$
\raisebox{-0.25in}{
    \begin{tikzpicture}
    \tikzset{->-/.style={
        decoration={markings,mark=at position #1 with {\arrow{latex}}},
        postaction={decorate}
    }}

    \filldraw[draw=white,fill=gray!20] (0,0) rectangle (1.5, 1.5);
    \draw[line width=1pt] (0,0)--(0,1.5);
    \draw[line width=1pt] (1.5,0)--(1.5,1.5);
    
    \draw [line width =0.8pt,decoration={markings, mark=at position 0.75 with {\arrow{<}}},postaction={decorate}] (0,0.3)--(1.5,1.2);

     \draw [line width =0.8pt,decoration={markings, mark=at position 0.75 with {\arrow{>}}},postaction={decorate}] 
            (1.5,1.2) -- (0,0.75);
     \draw [line width =0.8pt,decoration={markings, mark=at position 0.3 with {\arrow{>}}},postaction={decorate}]
            (1.5,0.75) -- (0,0.3);

       \node at(0,0.3) {$\bullet$};
    \node[left] at(0,0.3) {$3$};
    \node at(0,0.75) {$\bullet$};
    \node[left] at(0,0.75) {$2$};
    \node at(0,1.2) {$\bullet$};
    \node[left] at(0,1.2) {$1$};

    \node at(1.5,0.3) {$\bullet$};
    \node[right] at(1.5,0.3) {$1$};
    \node at(1.5,0.75) {$\bullet$};
    \node[right] at(1.5,0.75) {$2$};
    \node at(1.5,1.2) {$\bullet$};
    \node[right] at(1.5,1.2) {$3$};

    \node [text=red] at(0.5,0.4) {$1$};
    \node [text=red] at(0.2,0.8) {$4$};
    \node [text=red] at(0.75,0.75) {$3$};
    \end{tikzpicture}
    }
$,\quad
$
\raisebox{-0.25in}{
    \begin{tikzpicture}
    \tikzset{->-/.style={
        decoration={markings,mark=at position #1 with {\arrow{latex}}},
        postaction={decorate}
    }}

    \filldraw[draw=white,fill=gray!20] (0,0) rectangle (1.5, 1.5);
    \draw[line width=1pt] (0,0)--(0,1.5);
    \draw[line width=1pt] (1.5,0)--(1.5,1.5);
    
    \draw [line width =0.8pt,decoration={markings, mark=at position 0.3 with {\arrow{<}}},postaction={decorate}] (0,0.3)--(1.5,1.2);

    \draw [line width =0.8pt,decoration={markings, mark=at position 0.88 with {\arrow{>}}},postaction={decorate}]  (0,1.2)--(1.5,0.75);
     \draw [line width =0.8pt,decoration={markings, mark=at position 0.75 with {\arrow{>}}},postaction={decorate}]
            (1.5,1.2) -- (0,0.75);

       \node at(0,0.3) {$\bullet$};
    \node[left] at(0,0.3) {$3$};
    \node at(0,0.75) {$\bullet$};
    \node[left] at(0,0.75) {$2$};
    \node at(0,1.2) {$\bullet$};
    \node[left] at(0,1.2) {$1$};

    \node at(1.5,0.3) {$\bullet$};
    \node[right] at(1.5,0.3) {$1$};
    \node at(1.5,0.75) {$\bullet$};
    \node[right] at(1.5,0.75) {$2$};
    \node at(1.5,1.2) {$\bullet$};
    \node[right] at(1.5,1.2) {$3$};

    \node [text=red] at(0.2,0.8) {$3$};
    \node [text=red] at(0.75,0.75) {$2$};
    \node [text=red] at (0.2,1.15) {$4$};
    \end{tikzpicture}
    }
$,\quad
$
\raisebox{-0.25in}{
    \begin{tikzpicture}
    \tikzset{->-/.style={
        decoration={markings,mark=at position #1 with {\arrow{latex}}},
        postaction={decorate}
    }}

    \filldraw[draw=white,fill=gray!20] (0,0) rectangle (1.5, 1.5);
    \draw[line width=1pt] (0,0)--(0,1.5);
    \draw[line width=1pt] (1.5,0)--(1.5,1.5);
    
    \draw [line width =0.8pt,decoration={markings, mark=at position 0.28 with {\arrow{>}}},postaction={decorate}] (0,1.2)--(1.5,0.3);

    \draw [line width =0.8pt,decoration={markings, mark=at position 0.37 with {\arrow{>}}},postaction={decorate}] (0,0.75)--(1.5,0.3);
    \draw [line width =0.8pt,decoration={markings, mark=at position 0.88 with {\arrow{<}}},postaction={decorate}] (0,0.3)--(1.5,0.75);
       \node at(0,0.3) {$\bullet$};
    \node[left] at(0,0.3) {$3$};
    \node at(0,0.75) {$\bullet$};
    \node[left] at(0,0.75) {$2$};
    \node at(0,1.2) {$\bullet$};
    \node[left] at(0,1.2) {$1$};

    \node at(1.5,0.3) {$\bullet$};
    \node[right] at(1.5,0.3) {$1$};
    \node at(1.5,0.75) {$\bullet$};
    \node[right] at(1.5,0.75) {$2$};
    \node at(1.5,1.2) {$\bullet$};
    \node[right] at(1.5,1.2) {$3$};

    \node [text=red] at(0.3,0.65) {$2$};
    \node [text=red] at(0.75,0.75) {$4$};

    \node [text=red] at(0.2,0.35) {$1$};
    \end{tikzpicture}
    }
$,

\medskip

$
\raisebox{-0.25in}{
    \begin{tikzpicture}
    \tikzset{->-/.style={
        decoration={markings,mark=at position #1 with {\arrow{latex}}},
        postaction={decorate}
    }}

    \filldraw[draw=white,fill=gray!20] (0,0) rectangle (1.5, 1.5);
    \draw[line width=1pt] (0,0)--(0,1.5);
    \draw[line width=1pt] (1.5,0)--(1.5,1.5);
    
    \draw [line width =0.8pt,decoration={markings, mark=at position 0.37 with {\arrow{>}}},postaction={decorate}] (0,1.2)--(1.5,0.3);
    \draw [line width =0.8pt,decoration={markings, mark=at position 0.7 with {\arrow{>}}},postaction={decorate}] (0,1.2)--(1.5,0.75);
    \draw [line width =0.8pt,decoration={markings, mark=at position 0.37 with {\arrow{>}}},postaction={decorate}] (0,0.75)--(1.5,0.3);

       \node at(0,0.3) {$\bullet$};
    \node[left] at(0,0.3) {$3$};
    \node at(0,0.75) {$\bullet$};
    \node[left] at(0,0.75) {$2$};
    \node at(0,1.2) {$\bullet$};
    \node[left] at(0,1.2) {$1$};

    \node at(1.5,0.3) {$\bullet$};
    \node[right] at(1.5,0.3) {$1$};
    \node at(1.5,0.75) {$\bullet$};
    \node[right] at(1.5,0.75) {$2$};
    \node at(1.5,1.2) {$\bullet$};
    \node[right] at(1.5,1.2) {$3$};

    \node [text=red] at(0.3,0.65) {$1$};
    \node [text=red] at(0.75,0.75) {$3$};
    \node [text=red] at(1.2,0.85) {$4$};
 
    \end{tikzpicture}
    }
$,\quad
$
\raisebox{-0.25in}{
    \begin{tikzpicture}
    \tikzset{->-/.style={
        decoration={markings,mark=at position #1 with {\arrow{latex}}},
        postaction={decorate}
    }}

    \filldraw[draw=white,fill=gray!20] (0,0) rectangle (1.5, 1.5);
    \draw[line width=1pt] (0,0)--(0,1.5);
    \draw[line width=1pt] (1.5,0)--(1.5,1.5);

    \draw [line width =0.8pt,decoration={markings, mark=at position 0.5 with {\arrow{>}}},postaction={decorate}] (0,1.2)--(1.5,0.3);
     \draw [line width =0.8pt,decoration={markings, mark=at position 0.4 with {\arrow{>}}},postaction={decorate}] 
            (1.5,1.2) -- (0,0.75);
     \draw [line width =0.8pt,decoration={markings, mark=at position 0.65 with {\arrow{>}}},postaction={decorate}] 
            (1.5,0.75) -- (0,0.3);

       \node at(0,0.3) {$\bullet$};
    \node[left] at(0,0.3) {$3$};
    \node at(0,0.75) {$\bullet$};
    \node[left] at(0,0.75) {$2$};
    \node at(0,1.2) {$\bullet$};
    \node[left] at(0,1.2) {$1$};

    \node at(1.5,0.3) {$\bullet$};
    \node[right] at(1.5,0.3) {$1$};
    \node at(1.5,0.75) {$\bullet$};
    \node[right] at(1.5,0.75) {$2$};
    \node at(1.5,1.2) {$\bullet$};
    \node[right] at(1.5,1.2) {$3$};

    \node [text=red] at(0.2,0.35) {$1$};
    \node [text=red] at(0.3,1.05) {$3$};

    \node [text=red] at (1.3,1.15) {$4$};
    \end{tikzpicture}
    }
$,\quad
$
\raisebox{-0.25in}{
    \begin{tikzpicture}
    \tikzset{->-/.style={
        decoration={markings,mark=at position #1 with {\arrow{latex}}},
        postaction={decorate}
    }}

    \filldraw[draw=white,fill=gray!20] (0,0) rectangle (1.5, 1.5);
    \draw[line width=1pt] (0,0)--(0,1.5);
    \draw[line width=1pt] (1.5,0)--(1.5,1.5);
    
    \draw [line width =0.8pt,decoration={markings, mark=at position 0.8 with {\arrow{>}}},postaction={decorate}] (0,1.2)--(1.5,0.3);
    \draw [line width =0.8pt,decoration={markings, mark=at position 0.4 with {\arrow{>}}},postaction={decorate}] (0,1.2)--(1.5,0.75);

    \draw [line width =0.8pt,decoration={markings, mark=at position 0.2 with {\arrow{<}}},postaction={decorate}] (0,0.75)--(1.5,1.2);
       \node at(0,0.3) {$\bullet$};
    \node[left] at(0,0.3) {$3$};
    \node at(0,0.75) {$\bullet$};
    \node[left] at(0,0.75) {$2$};
    \node at(0,1.2) {$\bullet$};
    \node[left] at(0,1.2) {$1$};

    \node at(1.5,0.3) {$\bullet$};
    \node[right] at(1.5,0.3) {$1$};
    \node at(1.5,0.75) {$\bullet$};
    \node[right] at(1.5,0.75) {$2$};
    \node at(1.5,1.2) {$\bullet$};
    \node[right] at(1.5,1.2) {$3$};

    \node [text=red] at(0.75,0.75) {$2$};
    \node [text=red] at(1.2,0.85) {$3$};
    \node [text=red] at(1.2,1.15) {$4$};
    \end{tikzpicture}
    }
$.

\medskip

\noindent
\textit{Subcase 2:} $A(S)$ contains $\cev{\alpha}$.

\medskip

$
\raisebox{-0.25in}{
    \begin{tikzpicture}
    \tikzset{->-/.style={
        decoration={markings,mark=at position #1 with {\arrow{latex}}},
        postaction={decorate}
    }}

    \filldraw[draw=white,fill=gray!20] (0,0) rectangle (1.5, 1.5);
    \draw[line width=1pt] (0,0)--(0,1.5);
    \draw[line width=1pt] (1.5,0)--(1.5,1.5);
    
    \draw  [line width =0.8pt,decoration={markings, mark=at position 0.75 with {\arrow{>}}},postaction={decorate}] (0,0.3)--(1.5,1.2);

    \draw [line width =0.8pt,decoration={markings, mark=at position 0.2 with {\arrow{<}}},postaction={decorate}] (0,0.75)--(1.5,0.3);
     \draw  [line width =0.8pt,decoration={markings, mark=at position 0.39 with {\arrow{<}}},postaction={decorate}] 
            (1.5,0.75) -- (0,0.3);

       \node at(0,0.3) {$\bullet$};
    \node[left] at(0,0.3) {$3$};
    \node at(0,0.75) {$\bullet$};
    \node[left] at(0,0.75) {$2$};
    \node at(0,1.2) {$\bullet$};
    \node[left] at(0,1.2) {$1$};

    \node at(1.5,0.3) {$\bullet$};
    \node[right] at(1.5,0.3) {$1$};
    \node at(1.5,0.75) {$\bullet$};
    \node[right] at(1.5,0.75) {$2$};
    \node at(1.5,1.2) {$\bullet$};
    \node[right] at(1.5,1.2) {$3$};

    \node [text=red] at(1.3,0.35) {$1$};
    \node [text=red] at(0.8,0.8) {$4$};
    \node [text=red] at(1.2,0.65) {$2$};
    \end{tikzpicture}
    }
$,\quad
$
\raisebox{-0.25in}{
    \begin{tikzpicture}
    \tikzset{->-/.style={
        decoration={markings,mark=at position #1 with {\arrow{latex}}},
        postaction={decorate}
    }}

    \filldraw[draw=white,fill=gray!20] (0,0) rectangle (1.5, 1.5);
    \draw[line width=1pt] (0,0)--(0,1.5);
    \draw[line width=1pt] (1.5,0)--(1.5,1.5);
    
    \draw [line width =0.8pt,decoration={markings, mark=at position 0.2 with {\arrow{>}}},postaction={decorate}] (0,0.3)--(1.5,1.2);

     \draw [line width =0.8pt,decoration={markings, mark=at position 0.4 with {\arrow{<}}},postaction={decorate}]
            (0,1.2) -- (1.5,0.75);

     \draw [line width =0.8pt,decoration={markings, mark=at position 0.75 with {\arrow{<}}},postaction={decorate}]
            (0,0.75) -- (1.5,0.3);

       \node at(0,0.3) {$\bullet$};
    \node[left] at(0,0.3) {$3$};
    \node at(0,0.75) {$\bullet$};
    \node[left] at(0,0.75) {$2$};
    \node at(0,1.2) {$\bullet$};
    \node[left] at(0,1.2) {$1$};

    \node at(1.5,0.3) {$\bullet$};
    \node[right] at(1.5,0.3) {$1$};
    \node at(1.5,0.75) {$\bullet$};
    \node[right] at(1.5,0.75) {$2$};
    \node at(1.5,1.2) {$\bullet$};
    \node[right] at(1.5,1.2) {$3$};

    \node [text=red] at(1.3,0.35) {$1$};
    \node [text=red] at(1.2,1.05) {$3$};
    \node [text=red] at(0.2,1.15) {$4$};
    \end{tikzpicture}
    }
$,\quad
$
\raisebox{-0.25in}{
    \begin{tikzpicture}
    \tikzset{->-/.style={
        decoration={markings,mark=at position #1 with {\arrow{latex}}},
        postaction={decorate}
    }}

    \filldraw[draw=white,fill=gray!20] (0,0) rectangle (1.5, 1.5);
    \draw[line width=1pt] (0,0)--(0,1.5);
    \draw[line width=1pt] (1.5,0)--(1.5,1.5);
    
    \draw [line width =0.8pt,decoration={markings, mark=at position 0.75 with {\arrow{>}}},postaction={decorate}] (0,0.3)--(1.5,1.2);

     \draw [line width =0.8pt,decoration={markings, mark=at position 0.75 with {\arrow{<}}},postaction={decorate}] 
            (1.5,1.2) -- (0,0.75);
     \draw [line width =0.8pt,decoration={markings, mark=at position 0.3 with {\arrow{<}}},postaction={decorate}]
            (1.5,0.75) -- (0,0.3);

       \node at(0,0.3) {$\bullet$};
    \node[left] at(0,0.3) {$3$};
    \node at(0,0.75) {$\bullet$};
    \node[left] at(0,0.75) {$2$};
    \node at(0,1.2) {$\bullet$};
    \node[left] at(0,1.2) {$1$};

    \node at(1.5,0.3) {$\bullet$};
    \node[right] at(1.5,0.3) {$1$};
    \node at(1.5,0.75) {$\bullet$};
    \node[right] at(1.5,0.75) {$2$};
    \node at(1.5,1.2) {$\bullet$};
    \node[right] at(1.5,1.2) {$3$};

    \node [text=red] at(0.5,0.4) {$1$};
    \node [text=red] at(0.2,0.8) {$4$};
    \node [text=red] at(0.75,0.75) {$3$};
    \end{tikzpicture}
    }
$,\quad
$
\raisebox{-0.25in}{
    \begin{tikzpicture}
    \tikzset{->-/.style={
        decoration={markings,mark=at position #1 with {\arrow{latex}}},
        postaction={decorate}
    }}

    \filldraw[draw=white,fill=gray!20] (0,0) rectangle (1.5, 1.5);
    \draw[line width=1pt] (0,0)--(0,1.5);
    \draw[line width=1pt] (1.5,0)--(1.5,1.5);
    
    \draw [line width =0.8pt,decoration={markings, mark=at position 0.3 with {\arrow{>}}},postaction={decorate}] (0,0.3)--(1.5,1.2);

    \draw [line width =0.8pt,decoration={markings, mark=at position 0.88 with {\arrow{<}}},postaction={decorate}]  (0,1.2)--(1.5,0.75);
     \draw [line width =0.8pt,decoration={markings, mark=at position 0.75 with {\arrow{<}}},postaction={decorate}]
            (1.5,1.2) -- (0,0.75);

       \node at(0,0.3) {$\bullet$};
    \node[left] at(0,0.3) {$3$};
    \node at(0,0.75) {$\bullet$};
    \node[left] at(0,0.75) {$2$};
    \node at(0,1.2) {$\bullet$};
    \node[left] at(0,1.2) {$1$};

    \node at(1.5,0.3) {$\bullet$};
    \node[right] at(1.5,0.3) {$1$};
    \node at(1.5,0.75) {$\bullet$};
    \node[right] at(1.5,0.75) {$2$};
    \node at(1.5,1.2) {$\bullet$};
    \node[right] at(1.5,1.2) {$3$};

    \node [text=red] at(0.2,0.8) {$3$};
    \node [text=red] at(0.75,0.75) {$2$};
    \node [text=red] at (0.2,1.15) {$4$};
    \end{tikzpicture}
    }
$,\quad
$
\raisebox{-0.25in}{
    \begin{tikzpicture}
    \tikzset{->-/.style={
        decoration={markings,mark=at position #1 with {\arrow{latex}}},
        postaction={decorate}
    }}

    \filldraw[draw=white,fill=gray!20] (0,0) rectangle (1.5, 1.5);
    \draw[line width=1pt] (0,0)--(0,1.5);
    \draw[line width=1pt] (1.5,0)--(1.5,1.5);
    
    \draw [line width =0.8pt,decoration={markings, mark=at position 0.28 with {\arrow{<}}},postaction={decorate}] (0,1.2)--(1.5,0.3);

    \draw [line width =0.8pt,decoration={markings, mark=at position 0.37 with {\arrow{<}}},postaction={decorate}] (0,0.75)--(1.5,0.3);
    \draw [line width =0.8pt,decoration={markings, mark=at position 0.88 with {\arrow{>}}},postaction={decorate}] (0,0.3)--(1.5,0.75);
       \node at(0,0.3) {$\bullet$};
    \node[left] at(0,0.3) {$3$};
    \node at(0,0.75) {$\bullet$};
    \node[left] at(0,0.75) {$2$};
    \node at(0,1.2) {$\bullet$};
    \node[left] at(0,1.2) {$1$};

    \node at(1.5,0.3) {$\bullet$};
    \node[right] at(1.5,0.3) {$1$};
    \node at(1.5,0.75) {$\bullet$};
    \node[right] at(1.5,0.75) {$2$};
    \node at(1.5,1.2) {$\bullet$};
    \node[right] at(1.5,1.2) {$3$};

    \node [text=red] at(0.3,0.65) {$2$};
    \node [text=red] at(0.75,0.75) {$4$};

    \node [text=red] at(0.2,0.35) {$1$};
    \end{tikzpicture}
    }
$,

\medskip

$
\raisebox{-0.25in}{
    \begin{tikzpicture}
    \tikzset{->-/.style={
        decoration={markings,mark=at position #1 with {\arrow{latex}}},
        postaction={decorate}
    }}

    \filldraw[draw=white,fill=gray!20] (0,0) rectangle (1.5, 1.5);
    \draw[line width=1pt] (0,0)--(0,1.5);
    \draw[line width=1pt] (1.5,0)--(1.5,1.5);
    
    \draw [line width =0.8pt,decoration={markings, mark=at position 0.37 with {\arrow{<}}},postaction={decorate}] (0,1.2)--(1.5,0.3);
    \draw [line width =0.8pt,decoration={markings, mark=at position 0.7 with {\arrow{<}}},postaction={decorate}] (0,1.2)--(1.5,0.75);
    \draw [line width =0.8pt,decoration={markings, mark=at position 0.37 with {\arrow{<}}},postaction={decorate}] (0,0.75)--(1.5,0.3);

       \node at(0,0.3) {$\bullet$};
    \node[left] at(0,0.3) {$3$};
    \node at(0,0.75) {$\bullet$};
    \node[left] at(0,0.75) {$2$};
    \node at(0,1.2) {$\bullet$};
    \node[left] at(0,1.2) {$1$};

    \node at(1.5,0.3) {$\bullet$};
    \node[right] at(1.5,0.3) {$1$};
    \node at(1.5,0.75) {$\bullet$};
    \node[right] at(1.5,0.75) {$2$};
    \node at(1.5,1.2) {$\bullet$};
    \node[right] at(1.5,1.2) {$3$};

    \node [text=red] at(0.3,0.65) {$1$};
    \node [text=red] at(0.75,0.75) {$3$};
    \node [text=red] at(1.2,0.85) {$4$};
 
    \end{tikzpicture}
    }
$,\quad
$
\raisebox{-0.25in}{
    \begin{tikzpicture}
    \tikzset{->-/.style={
        decoration={markings,mark=at position #1 with {\arrow{latex}}},
        postaction={decorate}
    }}

    \filldraw[draw=white,fill=gray!20] (0,0) rectangle (1.5, 1.5);
    \draw[line width=1pt] (0,0)--(0,1.5);
    \draw[line width=1pt] (1.5,0)--(1.5,1.5);

    \draw [line width =0.8pt,decoration={markings, mark=at position 0.5 with {\arrow{<}}},postaction={decorate}] (0,1.2)--(1.5,0.3);
     \draw [line width =0.8pt,decoration={markings, mark=at position 0.4 with {\arrow{<}}},postaction={decorate}] 
            (1.5,1.2) -- (0,0.75);
     \draw [line width =0.8pt,decoration={markings, mark=at position 0.65 with {\arrow{<}}},postaction={decorate}] 
            (1.5,0.75) -- (0,0.3);

       \node at(0,0.3) {$\bullet$};
    \node[left] at(0,0.3) {$3$};
    \node at(0,0.75) {$\bullet$};
    \node[left] at(0,0.75) {$2$};
    \node at(0,1.2) {$\bullet$};
    \node[left] at(0,1.2) {$1$};

    \node at(1.5,0.3) {$\bullet$};
    \node[right] at(1.5,0.3) {$1$};
    \node at(1.5,0.75) {$\bullet$};
    \node[right] at(1.5,0.75) {$2$};
    \node at(1.5,1.2) {$\bullet$};
    \node[right] at(1.5,1.2) {$3$};

    \node [text=red] at(0.2,0.35) {$1$};
    \node [text=red] at(0.3,1.05) {$3$};

    \node [text=red] at (1.3,1.15) {$4$};
    \end{tikzpicture}
    }
$,\quad
$
\raisebox{-0.25in}{
    \begin{tikzpicture}
    \tikzset{->-/.style={
        decoration={markings,mark=at position #1 with {\arrow{latex}}},
        postaction={decorate}
    }}

    \filldraw[draw=white,fill=gray!20] (0,0) rectangle (1.5, 1.5);
    \draw[line width=1pt] (0,0)--(0,1.5);
    \draw[line width=1pt] (1.5,0)--(1.5,1.5);
    
    \draw [line width =0.8pt,decoration={markings, mark=at position 0.8 with {\arrow{<}}},postaction={decorate}] (0,1.2)--(1.5,0.3);
    \draw [line width =0.8pt,decoration={markings, mark=at position 0.4 with {\arrow{<}}},postaction={decorate}] (0,1.2)--(1.5,0.75);

    \draw [line width =0.8pt,decoration={markings, mark=at position 0.2 with {\arrow{>}}},postaction={decorate}] (0,0.75)--(1.5,1.2);
       \node at(0,0.3) {$\bullet$};
    \node[left] at(0,0.3) {$3$};
    \node at(0,0.75) {$\bullet$};
    \node[left] at(0,0.75) {$2$};
    \node at(0,1.2) {$\bullet$};
    \node[left] at(0,1.2) {$1$};

    \node at(1.5,0.3) {$\bullet$};
    \node[right] at(1.5,0.3) {$1$};
    \node at(1.5,0.75) {$\bullet$};
    \node[right] at(1.5,0.75) {$2$};
    \node at(1.5,1.2) {$\bullet$};
    \node[right] at(1.5,1.2) {$3$};

    \node [text=red] at(0.75,0.75) {$2$};
    \node [text=red] at(1.2,0.85) {$3$};
    \node [text=red] at(1.2,1.15) {$4$};
    \end{tikzpicture}
    }
$.

\end{appendices}

\newpage

\bibliography{ref.bib}

@article {GLSB,
    AUTHOR = {Galashin, Pavel and Lam, Thomas and Sherman-Bennett, Melissa},
     TITLE = {Braid variety cluster structures, {II}: general type},
   JOURNAL = {Invent. Math.},
  FJOURNAL = {Inventiones Mathematicae},
    VOLUME = {243},
      YEAR = {2026},
    NUMBER = {3},
     PAGES = {1079--1127},
      ISSN = {0020-9910,1432-1297},
   MRCLASS = {13F60 (14M15)},
  MRNUMBER = {5008161},
       DOI = {10.1007/s00222-025-01390-5},
       URL = {https://doi.org/10.1007/s00222-025-01390-5},
}

@book{majid2000foundations,
  title={Foundations of quantum group theory},
  author={Majid, Shahn},
  year={2000},
  publisher={Cambridge university press}
}

@article{huang2025quantum,
  title={{Quantum cluster realization for projected stated ${\rm SL}_n$-skein algebras}},
  author={Huang, Min and Wang, Zhihao},
  journal={arXiv preprint arXiv:2509.25938},
  year={2025}
}

@article {CZ,
    AUTHOR = {Chang, Wen and Zhu, Bin},
     TITLE = {On rooted cluster morphisms and cluster structures in
              2-{C}alabi-{Y}au triangulated categories},
   JOURNAL = {J. Algebra},
  FJOURNAL = {Journal of Algebra},
    VOLUME = {458},
      YEAR = {2016},
     PAGES = {387--421},
}

@article {HLY,
    AUTHOR = {Huang, Min and Li, Fang and Yang, Yichao},
     TITLE = {On structure of cluster algebras of geometric type {I}: {I}n
              view of sub-seeds and seed homomorphisms},
   JOURNAL = {Sci. China Math.},
  FJOURNAL = {Science China. Mathematics},
    VOLUME = {61},
      YEAR = {2018},
    NUMBER = {5},
     PAGES = {831--854},
}

@article {FK,
    AUTHOR = {Fu, Changjian and Keller, Bernhard},
     TITLE = {On cluster algebras with coefficients and 2-{C}alabi-{Y}au
              categories},
   JOURNAL = {Trans. Amer. Math. Soc.},
  FJOURNAL = {Transactions of the American Mathematical Society},
    VOLUME = {362},
      YEAR = {2010},
    NUMBER = {2},
     PAGES = {859--895},
}

@article {CHL,
    AUTHOR = {Chang, Wen and Huang, Min and Li, Jian-Rong},
     TITLE = {Quasi-homomorphisms of quantum cluster algebras},
   JOURNAL = {J. Algebra},
  FJOURNAL = {Journal of Algebra},
    VOLUME = {638},
      YEAR = {2024},
     PAGES = {506--534},
      ISSN = {0021-8693,1090-266X},
   MRCLASS = {13F60 (20F36)},
  MRNUMBER = {4656645},
MRREVIEWER = {Hironori\ Oya},
       DOI = {10.1016/j.jalgebra.2023.09.036},
       URL = {https://doi.org/10.1016/j.jalgebra.2023.09.036},
}

@article{Q2024,
  title={{Analogs of the dual canonical bases for cluster algebras from Lie theory}},
  author={Qin, Fan},
  journal={arXiv preprint arXiv:2407.02480},
  year={2024}
}

@article{BR,
  title={{The cluster and dual canonical bases of $\mathbb Z[x_{11}, ..., x_{33}]$ are equal}},
  author={Rhoades, Brendon},
  journal={Discrete Mathematics and Theoretical Computer Science},
  volume={12 (5)},
  year={2020},
  pages={97-124},
}

@article {SW21,
    AUTHOR = {Shen, Linhui and Weng, Daping},
     TITLE = {Cluster structures on double {B}ott-{S}amelson cells},
   JOURNAL = {Forum Math. Sigma},
  FJOURNAL = {Forum of Mathematics. Sigma},
    VOLUME = {9},
      YEAR = {2021},
     PAGES = {Paper No. e66, 89},
}

@article {KWQ,
    AUTHOR = {Kimura, Yoshiyuki and Qin, Fan and Wei, Qiaoling},
     TITLE = {Twist automorphisms and {P}oisson structures},
   JOURNAL = {SIGMA Symmetry Integrability Geom. Methods Appl.},
  FJOURNAL = {SIGMA. Symmetry, Integrability and Geometry. Methods and
              Applications},
    VOLUME = {19},
      YEAR = {2023},
     PAGES = {Paper No. 105, 39},
      ISSN = {1815-0659},
   MRCLASS = {13F60 (17B63)},
  MRNUMBER = {4681480},
MRREVIEWER = {Jie\ Pan},
       DOI = {10.3842/SIGMA.2023.105},
       URL = {https://doi.org/10.3842/SIGMA.2023.105},
}

@article {FZ4,
    AUTHOR = {Fomin, Sergey and Zelevinsky, Andrei},
     TITLE = {Cluster algebras. {IV}. {C}oefficients},
   JOURNAL = {Compos. Math.},
  FJOURNAL = {Compositio Mathematica},
    VOLUME = {143},
      YEAR = {2007},
    NUMBER = {1},
     PAGES = {112--164},
      ISSN = {0010-437X,1570-5846},
   MRCLASS = {16S99 (05E15 14M17 22E46)},
  MRNUMBER = {2295199},
MRREVIEWER = {Christof\ Gei\ss},
       DOI = {10.1112/S0010437X06002521},
       URL = {https://doi.org/10.1112/S0010437X06002521},
}

@article{KimWang,
  title={{Naturality of ${\rm SL}_n$ quantum trace maps for surfaces}},
  author={Kim, Hyun Kyu and Wang, Zhihao},
  journal={arXiv preprint arXiv:2412.16959},
  year={2024}
}

@article {H2,
    AUTHOR = {Huang, Min},
     TITLE = {Positivity for quantum cluster algebras from unpunctured
              orbifolds},
   JOURNAL = {Trans. Amer. Math. Soc.},
  FJOURNAL = {Transactions of the American Mathematical Society},
    VOLUME = {376},
      YEAR = {2023},
    NUMBER = {2},
     PAGES = {1155--1197}}

@article {R,
    AUTHOR = {Rupel, Dylan},
     TITLE = {On a quantum analog of the {C}aldero-{C}hapoton formula},
   JOURNAL = {Int. Math. Res. Not. IMRN},
  FJOURNAL = {International Mathematics Research Notices. IMRN},
      YEAR = {2011},
    NUMBER = {14},
     PAGES = {3207--3236}}

@article{H3,
  title={Positivity for quantum cluster algebras from orbifolds},
  author={Huang, Min},
  journal={arXiv preprint arXiv:2406.03362},
  year={2024}
}

@article{muller2016skein,
  title={{Skein and cluster algebras of marked surfaces}},
  author={Muller, Greg},
  journal={Quantum topology},
  volume={7},
  number={3},
  pages={435--503},
  year={2016}
}

@article{bloomquist2025degenerations,
  title={Degenerations of skein algebras and quantum traces},
  author={Bloomquist, Wade and Karuo, Hiroaki and L{\^e}, Thang},
  journal={Transactions of the American Mathematical Society},
  volume={378},
  number={09},
  pages={6049--6108},
  year={2025}
}

@article{ishibashi2023skein,
  title={{Skein and cluster algebras of unpunctured surfaces for $\mathfrak{sl}_3$}},
  author={Ishibashi, Tsukasa and Yuasa, Wataru},
  journal={Mathematische Zeitschrift},
  volume={303},
  number={3},
  pages={72},
  year={2023},
  publisher={Springer}
}

@article{LY22,
  title={Quantum traces and embeddings of stated skein algebras into quantum tori},
  author={L{\^e}, Thang TQ and Yu, Tao},
  journal={Selecta Mathematica},
  volume={28},
  number={4},
  pages={66},
  year={2022},
  publisher={Springer}
}

@article{le2018triangular,
  title={Triangular decomposition of skein algebras},
  author={L{\^e}, Thang TQ},
  journal={Quantum Topology},
  volume={9},
  number={3},
  pages={591--632},
  year={2018}
}

@article {FZ,
    AUTHOR = {Fomin, Sergey and Zelevinsky, Andrei},
     TITLE = {Cluster algebras. {I}. {F}oundations},
   JOURNAL = {J. Amer. Math. Soc.},
  FJOURNAL = {Journal of the American Mathematical Society},
    VOLUME = {15},
      YEAR = {2002},
    NUMBER = {2},
     PAGES = {497--529},
}

@article {BFZ,
    AUTHOR = {Berenstein, Arkady and Fomin, Sergey and Zelevinsky, Andrei},
     TITLE = {Cluster algebras. {III}. {U}pper bounds and double {B}ruhat cells},
   JOURNAL = {Duke Math. J.},
  FJOURNAL = {Duke Mathematical Journal},
    VOLUME = {126},
      YEAR = {2005},
    NUMBER = {1},
     PAGES = {1--52},
}

@article {GHKK,
    AUTHOR = {Gross, Mark and Hacking, Paul and Keel, Sean and Kontsevich, Maxim},
     TITLE = {Canonical bases for cluster algebras},
   JOURNAL = {J. Amer. Math. Soc.},
  FJOURNAL = {Journal of the American Mathematical Society},
    VOLUME = {31},
      YEAR = {2018},
    NUMBER = {2},
     PAGES = {497--608},
}

@article {LS,
    AUTHOR = {Lee, Kyungyong and Schiffler, Ralf},
     TITLE = {Positivity for cluster algebras},
   JOURNAL = {Ann. of Math. (2)},
  FJOURNAL = {Annals of Mathematics. Second Series},
    VOLUME = {182},
      YEAR = {2015},
    NUMBER = {1},
     PAGES = {73--125},
}

@article {CGGLS,
    AUTHOR = {Casals, Roger and Gorsky, Eugene and Gorsky, Mikhail and Le, Ian and Shen, Linhui and Simental, Jos\'e},
     TITLE = {Cluster structures on braid varieties},
   JOURNAL = {J. Amer. Math. Soc.},
  FJOURNAL = {Journal of the American Mathematical Society},
    VOLUME = {38},
      YEAR = {2025},
    NUMBER = {2},
     PAGES = {369--479},
}

@article {M1,
    AUTHOR = {Muller, Greg},
     TITLE = {{$\mathcal A=\mathcal U$} for locally acyclic cluster algebras},
   JOURNAL = {SIGMA Symmetry Integrability Geom. Methods Appl.},
  FJOURNAL = {SIGMA. Symmetry, Integrability and Geometry. Methods and Applications},
    VOLUME = {10},
      YEAR = {2014},
     PAGES = {Paper 094, 8},
}

@article {GLS1,
    AUTHOR = {Geiss, Christof and Leclerc, Bernard and Schr\"oer, Jan},
     TITLE = {Quantum cluster algebras and their specializations},
   JOURNAL = {J. Algebra},
  FJOURNAL = {Journal of Algebra},
    VOLUME = {558},
      YEAR = {2020},
     PAGES = {411--422},
}

@article {S,
    AUTHOR = {Scott, Joshua S.},
     TITLE = {Grassmannians and cluster algebras},
   JOURNAL = {Proc. London Math. Soc. (3)},
  FJOURNAL = {Proceedings of the London Mathematical Society. Third Series},
    VOLUME = {92},
      YEAR = {2006},
    NUMBER = {2},
     PAGES = {345--380},

}

@article {IY1,
    AUTHOR = {Ishibashi, Tsukasa and Yuasa, Wataru},
     TITLE = {Skein and cluster algebras of unpunctured surfaces for
              {$\mathfrak{sp}_4$}},
   JOURNAL = {Adv. Math.},
  FJOURNAL = {Advances in Mathematics},
    VOLUME = {465},
      YEAR = {2025},
     PAGES = {Paper No. 110149, 68},
}

@article {BMS,
    AUTHOR = {Bucher, Eric and Machacek, John and Shapiro, Michael},
     TITLE = {Upper cluster algebras and choice of ground ring},
   JOURNAL = {Sci. China Math.},
  FJOURNAL = {Science China. Mathematics},
    VOLUME = {62},
      YEAR = {2019},
    NUMBER = {7},
     PAGES = {1257--1266},
}

@article {M,
    AUTHOR = {Muller, Greg},
     TITLE = {Locally acyclic cluster algebras},
   JOURNAL = {Adv. Math.},
  FJOURNAL = {Advances in Mathematics},
    VOLUME = {233},
      YEAR = {2013},
     PAGES = {207--247},
}

@article {CKQ,
    AUTHOR = {Cao, Peigen and Keller, Bernhard and Qin, Fan},
     TITLE = {The valuation pairing on an upper cluster algebra},
   JOURNAL = {J. Reine Angew. Math.},
  FJOURNAL = {Journal f\"ur die Reine und Angewandte Mathematik. [Crelle's
              Journal]},
    VOLUME = {806},
      YEAR = {2024},
     PAGES = {71--114},
}

@book {G,
    AUTHOR = {Ingermanson, Grace},
     TITLE = {Cluster {A}lgebras of {O}pen {R}ichardson {V}arieties},
      NOTE = {Thesis (Ph.D.)--University of Michigan},
 PUBLISHER = {ProQuest LLC, Ann Arbor, MI},
      YEAR = {2019},
     PAGES = {149},
}

@article{goodearl2017quantum,
  title={Quantum cluster algebra structures on quantum nilpotent algebras},
  author={Goodearl, KR and Yakimov, MT},
  journal={Mem. Amer. Math. Soc.},
  volume={247},
  number={1169},
  pages={vii+--119},
  year={2017}
}

@misc{le_sikora2025,
  author       = {L{\^e}, Thang TQ and Sikora, Adam S},
  title        = {Skein Algebras of Surfaces and Quantum Groups},
  year         = {2025},
  howpublished = {\url{https://app.icerm.brown.edu/assets/547/10247/10247_5760_Le_121120251030_Slides.pdf}},
  note         = {Slides presented in several talks by Thang T. Q. L\^e}
}

@article {FST,
    AUTHOR = {Fomin, Sergey and Shapiro, Michael and Thurston, Dylan},
     TITLE = {Cluster algebras and triangulated surfaces. {I}. {C}luster
              complexes},
   JOURNAL = {Acta Math.},
  FJOURNAL = {Acta Mathematica},
    VOLUME = {201},
      YEAR = {2008},
    NUMBER = {1},
     PAGES = {83--146},
}

@book {GSV,
    AUTHOR = {Gekhtman, Michael and Shapiro, Michael and Vainshtein, Alek},
     TITLE = {Cluster algebras and {P}oisson geometry},
    SERIES = {Mathematical Surveys and Monographs},
    VOLUME = {167},
 PUBLISHER = {American Mathematical Society, Providence, RI},
      YEAR = {2010},
     PAGES = {xvi+246},
}

@article {CLS,
    AUTHOR = {\c{C}anak\c{c}{\i}, \.{I}lke and Lee, Kyungyong and Schiffler, Ralf},
     TITLE = {On cluster algebras from unpunctured surfaces with one marked
              point},
   JOURNAL = {Proc. Amer. Math. Soc. Ser. B},
  FJOURNAL = {Proceedings of the American Mathematical Society. Series B},
    VOLUME = {2},
      YEAR = {2015},
     PAGES = {35--49},
}

@article {GY,
    AUTHOR = {Goodearl, K. R. and Yakimov, M. T.},
     TITLE = {The {B}erenstein-{Z}elevinsky quantum cluster algebra
              conjecture},
   JOURNAL = {J. Eur. Math. Soc. (JEMS)},
  FJOURNAL = {Journal of the European Mathematical Society (JEMS)},
    VOLUME = {22},
      YEAR = {2020},
    NUMBER = {8},
     PAGES = {2453--2509},
}

@article {IOS,
    AUTHOR = {Ishibashi, Tsukasa and Oya, Hironori and Shen, Linhui},
     TITLE = {{$\mathcal{A}=\mathcal{U}$} for cluster algebras from moduli spaces of
              {$G$}-local systems},
   JOURNAL = {Adv. Math.},
  FJOURNAL = {Advances in Mathematics},
    VOLUME = {431},
      YEAR = {2023},
     PAGES = {Paper No. 109256, 50},
}

@article{higgins2020triangular,
  title={{Triangular decomposition of ${\rm SL}_3$ skein algebras}},
  author={Higgins, Vijay},
  journal={Quantum Topology},
  volume={14},
  number={1},
  pages={1--63},
  year={2023}
}

@article{thurston2014positive,
  title={Positive basis for surface skein algebras},
  author={Thurston, Dylan Paul},
  journal={Proceedings of the National Academy of Sciences},
  volume={111},
  number={27},
  pages={9725--9732},
  year={2014},
  publisher={National Academy of Sciences}
}

@article{kim2025frobenius,
  title={{Frobenius homomorphisms for stated ${\rm SL}_n$-skein modules}},
  author={Kim, Hyun Kyu and L{\^e}, Thang TQ and Wang, Zhihao},
  journal={arXiv preprint arXiv:2504.08657},
  year={2025}
}

@book{JS,
  title={Algebraic geometry},
  author={Milne, James S},
  year={2012},
  publisher={Allied Publishers}
}

@article{bai2018quantum,
  title={{A quantum analog of generalized cluster algebras}},
  author={Bai, Liqian and Chen, Xueqing and Ding, Ming and Xu, Fan},
  journal={Algebras and Representation Theory},
  volume={21},
  number={6},
  pages={1203--1217},
  year={2018},
  publisher={Springer}
}

@book{KS,
  title={Quantum groups and their representations},
  author={Klimyk, Anatoli and Schm{\"u}dgen, Konrad},
  year={2012},
  publisher={Springer Science \& Business Media}
}

@article{L,
  title={On cluster algebras from once punctured closed surfaces},
  author={Sefi Ladkani},
  journal={arXiv preprint arXiv:1310.4454},
  year={2013}
}

@article {MW,
    AUTHOR = {Moon, Han-Bom and Wong, Helen},
     TITLE = {Consequences of the compatibility of skein algebra and cluster
              algebra on surfaces},
   JOURNAL = {New York J. Math.},
  FJOURNAL = {New York Journal of Mathematics},
    VOLUME = {30},
      YEAR = {2024},
     PAGES = {1648--1682},
}

@article {D,
    AUTHOR = {Davison, Ben},
     TITLE = {Positivity for quantum cluster algebras},
   JOURNAL = {Ann. of Math. (2)},
  FJOURNAL = {Annals of Mathematics. Second Series},
    VOLUME = {187},
      YEAR = {2018},
    NUMBER = {1},
     PAGES = {157--219},
}

@article {Q,
    AUTHOR = {Qin, Fan},
     TITLE = {Bases for upper cluster algebras and tropical points},
   JOURNAL = {J. Eur. Math. Soc. (JEMS)},
  FJOURNAL = {Journal of the European Mathematical Society (JEMS)},
    VOLUME = {26},
      YEAR = {2024},
    NUMBER = {4},
     PAGES = {1255--1312},
}

@article{LY23,
  title={{Quantum traces for $SL_n$-skein algebras}},
  author={L{\^e}, Thang TQ and Yu, Tao},
  journal={arXiv preprint arXiv:2303.08082},
  year={2023}
}

@article{GLS,
  title={Generic Caldero-Chapoton functions with coefficients and applications to surface cluster algebras},
  author={Gei\ss, Christof and Labardini-Fragoso, Daniel and Schr\"oer, Jan},
  journal={arXiv preprint 	arXiv:2007.05483},
  year={2020}
}

@article{frohman20223,
  title={{$SU (3)$-skein algebras and webs on surfaces}},
  author={Frohman, Charles and Sikora, Adam S},
  journal={Mathematische Zeitschrift},
  volume={300},
  number={1},
  pages={33--56},
  year={2022},
  publisher={Springer}
}

@article{LS21,
  title={{Stated $SL(n)$-skein modules and algebras}},
  author={L{\^e}, Thang TQ and Sikora, Adam S},
  journal={Journal of Topology},
  volume={17},
  number={3},
  pages={e12350},
  year={2024},
  publisher={Wiley Online Library}
}

@article{CLL,
  title={{Stated skein algebras of surfaces}},
  author={Costantino, Francesco and L{\^e}, Thang TQ},
  journal={Journal of the European Mathematical Society},
  volume={24},
  number={12},
  pages={4063--4142},
  year={2022}
}

@article{BZ,
  title={{Quantum cluster algebras}},
  author={Berenstein, Arkady and Zelevinsky, Andrei},
  journal={Advances in Mathematics},
  volume={195},
  number={2},
  pages={405--455},
  year={2005},
  publisher={Elsevier}
}

@article{detcherry2025embedding,
  title={{An embedding of skein algebras of surfaces into localized quantum tori from Dehn--Thurston coordinates}},
  author={Detcherry, Renaud and Santharoubane, Ramanujan},
  journal={Geometry \& Topology},
  volume={29},
  number={1},
  pages={313--348},
  year={2025},
  publisher={Mathematical Sciences Publishers}
}

@article{QY,
  title={Partially compactified quantum cluster structures on simple algebraic groups and the full Berenstein--Zelevinsky conjecture},
  author={Qin, Fan and Yakimov, Milen},
  journal={arXiv preprint arXiv:2504.05134},
  year={2025}
}

@article{gaetz2025rotation,
  title={{Rotation-invariant web bases from hourglass plabic graphs}},
  author={Gaetz, Christian and Pechenik, Oliver and Pfannerer, Stephan and Striker, Jessica and Swanson, Joshua P},
  journal={Inventiones mathematicae},
  pages={1--102},
  year={2025},
  publisher={Springer}
}

@inproceedings{douglas2024tropical,
  title={{Tropical Fock--Goncharov coordinates for-webs on surfaces I: construction}},
  author={Douglas, Daniel C and Sun, Zhe},
  booktitle={Forum of Mathematics, Sigma},
  volume={12},
  pages={e5},
  year={2024},
  organization={Cambridge University Press}
}

@article{mandel2021scattering,
  title={Scattering diagrams, theta functions, and refined tropical curve counts},
  author={Mandel, Travis},
  journal={Journal of the London Mathematical Society},
  volume={104},
  number={5},
  pages={2299--2334},
  year={2021},
  publisher={Wiley Online Library}
}

@article{davison2021strong,
  title={Strong positivity for quantum theta bases of quantum cluster algebras},
  author={Davison, Ben and Mandel, Travis},
  journal={Inventiones mathematicae},
  volume={226},
  number={3},
  pages={725--843},
  year={2021},
  publisher={Springer}
}

@article{Kim20,
  title={{${\rm SL}_3$-laminations as bases for ${\rm PGL}_3$ cluster varieties for surfaces}},
  author={Kim, Hyun Kyu},
  journal={arXiv preprint arXiv:2011.14765},
  year={2020}
}

@article{Cohn,
  title={{Skew fields, Theory of general division rings}},
  author={Cohn, Paul Moritz and Strambach, K},
  journal={Jahresbericht der Deutschen Mathematiker Vereinigung},
  volume={100},
  number={1},
  pages={6--8},
  year={1998},
  publisher={Berlin: Georg Reimer, 1892-}
}

@article{Kim21,
  title={Naturality of ${\rm SL}_3$ quantum trace maps for surfaces.},
  author={Kim, Hyun Kyu},
  journal={Quantum Topology},
  volume={16},
  number={1},
  year={2025}
}

@article{Sik05,
  title={{Skein theory for $SU(n)$-quantum invariants}},
  author={Sikora, Adam S},
  journal={Algebraic \& Geometric Topology},
  volume={5},
  number={3},
  pages={865--897},
  year={2005},
  publisher={Mathematical Sciences Publishers}
}

@article{Lab09,
  title={{Quivers with potentials associated to triangulated surfaces}},
  author={Labardini-Fragoso, Daniel},
  journal={Proceedings of the London Mathematical Society},
  volume={98},
  number={3},
  pages={797--839},
  year={2009},
  publisher={Oxford University Press}
}

@article{GS19,
  title={{Quantum geometry of moduli spaces of local systems and representation theory}},
  author={Goncharov, Alexander and Shen, Linhui},
  journal={arXiv preprint arXiv:1904.10491},
  year={2019}
}

@article{FG06,
  title={{Moduli spaces of local systems and higher Teichm{\"u}ller theory}},
  author={Fock, Vladimir and Goncharov, Alexander},
  journal={Publications Math{\'e}matiques de l'IH{\'E}S},
  volume={103},
  pages={1--211},
  year={2006}
}

\end{document}

\subsection{Triangle case}

For the triangle $\mathbb P_3$, let $Q$ denote the associated skew-symmetric matrix (equivalently, quiver). Let $Q^{{\rm qc}}$ be the quiver obtained from $Q$ by deleting all arrows incident to the vertices $\overline {ii}$ and $\overline {jn}$ for $i\in \{1,2,\cdots,n\}$. For any $\overline{ij}\in \mathcal V(\mathbb P_3)$, let 
\begin{equation*}
    \overline A_{\overline{ij}}^{\rm qc}=\begin{cases}
       [\overline A_{\overline{ij}}\cdot \overline A_{\overline {jj}}^{-1}\cdot \overline A_{\overline {n-j+i, n}}^{-1}], & \mbox{ if $j\neq n$ and $i\neq j$}\\
       A_{ij}, & \mbox{ if $j=n$ or $i=j$}
    \end{cases}\in \overline {\mathcal A}(\mathbb P_3). 
\end{equation*}


It is straightforward to verify that for any $ij,i'j'\in \mathcal V(\mathbb P_3)$, the elements $\overline A_{\overline{ij}}^{\rm qc}$ and $\overline A_{\overline{i'j'}}^{\rm qc}$ are quasi-commutative, denote by 
$\overline \Pi^{\rm qc}(\overline{ij},\overline{i'j'})$ the integer satisfying $\overline A_{\overline{ij}}^{\rm qc}\overline A_{\overline{i'j'}}^{\rm qc}=\omega^{\overline \Pi^{\rm qc}(\overline{ij},\overline{i'j'})}\overline A_{\overline{i'j'}}^{\rm qc}\overline A_{\overline{ij}}^{\rm qc}$. 

Define 
\begin{equation*}
\overline M^{\rm qc}:\mathbb Z^{\mathcal V}\to {\rm Frac}(\widetilde\cS_{\omega}(\mathbb P_3)), \qquad 
    {\overline {ij}} \longmapsto \overline A_{\overline{ij}}^{\rm qc}.    
\end{equation*}

\begin{lemma}
The triple $(\overline Q^{\rm qc}, \overline{\Pi}^{\rm qc},  \overline M^{\rm qc})$ is a quantum seed in ${\rm Frac}(\widetilde\cS_{\omega}(\mathbb P_3))$.
\end{lemma}

Let $\overline {\mathcal A}_{\omega}^{\rm qc}(\mathbb P_3)$ denote the quantum cluster algebra associated with the seed
$(\overline Q^{\rm qc}, \overline{\Pi}^{\rm qc},  \overline M^{\rm qc})$.

\begin{proposition}
    The assignments $\overline A^{\rm qc}_{\overline{ij}}\mapsto \overline A_{\overline{ij}}$ induce a quasi-isomorphism (in the sense of \cite[Definition 2.1]{CHL}) 
    \begin{equation}\label{eq:quasi-iso}
        f:\overline {\mathcal A}^{\rm qc}_{\omega}(\mathbb P_3)\to \overline {\mathcal A}_{\omega}(\mathbb P_3).
    \end{equation}
\end{proposition}

{\bf Standard cluster variables} of $\overline{\mathcal A} ^{\rm qc}(\mathbb P_3)$: $A^{s}_t$ for $s,t>0$ with $s+t\leq n$.

For any $k,t$ with $k+t<n$, denote 
\begin{equation}
    \overline \mu^{(k;t)}=\mu_{\overline{k,k+1}}\mu_{\overline{k,k+2}}\ldots \mu_{\overline{k,k+t}}.
\end{equation}

\begin{lemma}
    We have $A^{s}_t=
    \begin{cases}
    \overline A^{\rm qc}_{\overline{n-s-1,n-s}} & \mbox{ if $t=1$},\vspace{1pt}\\
    \overline\mu^{(n+1-s-t;t-1)} \ldots \overline\mu^{(n-t-1;t-1)} \overline\mu^{(n-t;t-1)}(\overline A^{\rm qc}_{\overline{n+1-s-t,n+2-s-t}})  & \mbox{ if $t>1$}.
    \end{cases}$
\end{lemma}

\begin{proposition}
   Under the quasi-isomorphism $f$ in \eqref{eq:quasi-iso}, we have 
   \begin{equation}
       f(A^{s}_t)=
    \begin{cases}
    [\overline A_{\overline{n-s-1,n-s}} \cdot \overline A^{-1}_{\overline{n-s,n-s}}\cdot \overline A^{-1}_{\overline{n-1,n}}] & \mbox{ if $j=1$},\vspace{1pt}\\
    [\overline\mu^{(n+1-s-t;t-1)} \ldots \overline\mu^{(n-t-1;t-1)} \overline\mu^{(n-t;t-1)}(\overline A_{\overline{n+1-s-t,n+2-s-t}}) \cdot \overline A^{-1}_{\overline{n-s-t,n-s-t}}\cdot \overline A^{-1}_{\overline{n-t,n}}]  & \mbox{ if $j>1$}.
    \end{cases}
   \end{equation}
    Consequently, we have
      \begin{equation}
       f(A^{s}_t)=
    \begin{cases}
     [\overline A_{\overline{n-s-1,n-s}} \cdot \overline A^{-1}_{\overline{n-s,n-s}}\cdot \overline A^{-1}_{\overline{n-1,n}}] & \mbox{ if $j=1$},\vspace{1pt}\\
    [\overline C_{n-t+1,n+1-s-t} \cdot \overline A^{-1}_{\overline{n-t,n}} \cdot \overline A_{\overline{n-t+1,n}}]  & \mbox{ if $t>1$}.
    \end{cases}
   \end{equation}
\end{proposition}

Then $\overline A^{\rm qc}_{\omega}(\mathbb P_3)=\mathbb Q\langle A^{s}_t\mid s+t\leq n\}\rangle\langle \overline A_{\overline{ii}}^{-1},\overline A_{\overline{in}}^{-1}, \overline A_{\overline{0i}}^{-1}\mid 1\leq i\leq n-1\rangle$. (add reference).

Denote by $\overline{\mathcal A}^{{\rm frozen(v_1)}}(\mathbb P_3)$ the $\mathbb Z[\omega^{\pm 1}]$-subalgebra of $\overline{\mathcal A}^{{\rm frozen(v_1)}}(\mathbb P_3)$ generated by all cluster variables and $ \overline A_{\overline{ii}}^{\pm1},\overline A_{\overline{in}}^{\pm1}, 1\leq i\leq n-1$. 

Denote by $\widetilde\cS^{{\rm frozen(v_1)}}_{\omega}(\mathbb P_3)$ the $\mathbb Z[\omega^{\pm 1}]$-subalgebra of $\widetilde\cS_{\omega}(\mathbb P_3)$ generated by all corner arcs $\overline C_{ij}$ and $ \overline A_{\overline{ii}}^{\pm1},\overline A_{\overline{in}}^{\pm1}, 1\leq i\leq n-1$. 

\begin{proposition}
  $\overline{\mathcal A}_\omega^{{\rm frozen(v_1)}}(\mathbb P_3)=\widetilde\cS^{{\rm frozen(v_1)}}_{\omega}(\mathbb P_3)$. Moreover, the corner arcs $C_{ij}\in \widetilde\cS^{{\rm frozen(v_1)}}_{\omega}(\mathbb P_3)$.  
\end{proposition}

\begin{proposition}
    Under an embedding of surfaces $\mathbb P_3\hookrightarrow \fS$, we have the following embedding of $\mathbb Z[\omega^{\pm 1}]$-algebras
    \begin{equation*}
     \begin{cases}
        \widetilde\cS^{{\rm frozen(v_1)}}_{\omega}(\mathbb P_3)\hookrightarrow \widetilde\cS_{\omega}(\fS)\\
        \overline{\mathcal A}_\omega^{{\rm frozen(v_1)}}(\mathbb P_3)\hookrightarrow  \overline{\mathcal A}_\omega(\fS).
     \end{cases}  
    \end{equation*}
\end{proposition}

\subsection{$g$-vectors for $\mathbb{P}_{k+2}$} Let $\lambda$ be the star-like triangulation at vertex $1$ of $\mathbb P_{k+2}$. Let $Q$ be the quiver obtained from $\overline Q_\lambda$ by removing the vertices $ns^i$ for $s\in \{1,2,\cdots, n-1\}, i\in \{1,2,\cdots,k\}$ and $j0^1$ for $j\in\{1,2,\cdots,n-1\}$. Let $\widetilde Q$ be the framed quiver of $Q$. Let $\mathcal{A}^{\mathrm{prin}}_{Q}$ denote the commutative cluster algebra with principal coefficients and initial exchange matrix given by the quiver $Q$.  
For each vertex ${js^i}$ of $Q$, let $\widetilde{A}_{js^i}$ be the corresponding initial cluster variable in $\mathcal{A}^{\mathrm{prin}}_{Q}$.

Recall the mutation sequences $\overset{\leftarrow}{\mu}_{\prec js^{i}}$ and $\overset{\leftarrow}{\mu}(\Delta_i)$ defined in \eqref{eq:mustandard} and \eqref{eq:mustandard1}, respectively. 

For integers $0 \leq \widetilde{s} \leq n-1$ and $1 \leq \widetilde{i} \leq k$, for any vertex $js^i$ of $\overline{Q}_\lambda$, define
\begin{equation*}
   \widetilde{A}_{js^i}^{\widetilde{s}}[\widetilde{i}]
   \;=\;
   \mu_{\widetilde{s}}^\Delta(\widetilde{i}) \circ \cdots \circ \mu_2^\Delta(\widetilde{i}) \circ \mu_1^\Delta(\widetilde{i})
   \circ \bigl( \mu^\Delta(\widetilde{i}-1) \circ \cdots \circ \mu^\Delta(1) \bigr)
   \bigl( \widetilde{A}_{js^i} \bigr).
\end{equation*}
\begin{equation*}
  \widetilde E(js^i;\widetilde s,\widetilde i)
   \;=\;
  \overset{\leftarrow}{\mu}_{\prec (n,\widetilde s+1)^{\widetilde i}}
   \bigl( \widetilde{A}_{js^i} \bigr).
\end{equation*}

We have 
\begin{equation}
   \overset{\leftarrow}{\mu}_{\prec (j,\widetilde s+1)^{\widetilde i}}(\widetilde Q) =
   \overset{\leftarrow}{\mu}_{(j+1,\widetilde s+1)^{\widetilde i}}\cdots  
   \overset{\leftarrow}{\mu}_{(n-1,\widetilde s+1)^{\widetilde i}} \overset{\leftarrow}{\mu}_{(n,\widetilde s+1)^{\widetilde i}}\overset{\leftarrow}{\mu}_{\prec (n,\widetilde s+1)^{\widetilde i}}(\widetilde Q)
\end{equation}

Note that for all integers $s, j$ with $1 \leq s \leq j \leq n-1$ and all $i$ with $1 \leq i \leq k$, the cluster variable $\widetilde{A}_{j1^1}^{\,s-1}[i]$ is the commutative counterpart of the standara variable $\overline E^{\rm qc}(js^i)$.

It is straightforward to verify that $\widetilde{A}_{js^i}^{0}[\widetilde{i}]=\widetilde{A}_{js^i}^{n-1}[\widetilde{i}-1]$ and
\[
   \widetilde{A}_{js^i}^{\widetilde{s}}[\widetilde{i}] = \widetilde{A}_{js^i}^{j}[\widetilde{i}] \quad \text{if } j < \widetilde{s},\qquad
   \widetilde{A}_{js^i}^{\widetilde{s}}[\widetilde{i}] = \widetilde{A}_{js^i}^{j}[k+1 - \widetilde{i}] \quad \text{if } i > k + 1 - \widetilde{i}.
\]

In this subsection, we compute the $g$-vector of $\widetilde{A}_{js^i}^{\widetilde{s}}[\widetilde{i}]$, which we denote by $g_{js^i}^{\widetilde{s}}[\widetilde{i}]$. We have the following observation.

\begin{lemma}
 For integers $1\leq \widetilde{s} \leq n-1$ and $1 \leq \widetilde{i} \leq k$, 
 
\begin{enumerate}[label=\textup{(\alph*)}]
    \item  we have the sequence of mutations 
   $\overset{\leftarrow}{\mu}_{\prec (j,\widetilde s+1)^{\widetilde i}}$ 
   is a sequence of green mutation for $\widetilde Q$. 
   \huang{}
  \item Let $ j s^{i} $ be a vertex with $ j \geq \widetilde{s} $ and $ i \leq k + 1 - \widetilde{i} $. Consider the quiver
$$
\begin{aligned}
    Q' = {} & 
    \bigl( \mu_{j s^{i}} \cdots \mu_{j 1^{i}} \bigr)
    \bigl( \mu_{j j^{i-1}} \cdots \mu_{j 1^{i-1}} \bigr)
    \cdots
    \bigl( \mu_{j j^{2}} \cdots \mu_{j 1^{2}} \bigr)
    \bigl( \mu_{j j^{1}} \cdots \mu_{j 1^{1}} \bigr) \\
    & \circ \mu_{j+1}\bigl((\widetilde{i}-1)(\widetilde{s}+1)+\widetilde{s}\bigr)
    \circ \cdots
    \circ \mu_{n-2}\bigl((\widetilde{i}-1)(n-2)+\widetilde{s}\bigr) \circ \mu_{n-1}\bigl((\widetilde{i}-1)(n-1)+\widetilde{s}\bigr)\\
    & 
    \circ \bigl(\overset{\leftarrow}{\mu}(\Delta_{i-1}) \cdots  \overset{\leftarrow}{\mu}(\Delta_1) \bigr)(\widetilde{Q}).
\end{aligned}
$$
Then $ Q'_{v,\, j s^{i}} < 0 $ if and only if  
$$
    v = j (s-1)^{i} \quad (\text{provided } s > 1 \text{ or } i > 1) \quad \text{or} \quad v = j (s+1)^{i}.
$$
In these cases, we have $ Q'_{v,\, j s^{i}} = -1 $.

\item Consequently, for any vertex $ j s^{i} $ with $ j \geq \widetilde{s} $ and $ i \leq k + 1 - \widetilde{i} $, the following recurrence relation holds:
\begin{equation}\label{eq:recg1}
    g_{j s^{i}}^{\widetilde{s}}[\widetilde{i}] =
    \begin{cases}
        g_{j (s+1)^{i}}^{\widetilde{s}-1}[\widetilde{i}] - g_{j s^{i}}^{\widetilde{s}-1}[\widetilde{i}]
        & \text{if } s = 1 \text{ and } i = 1, \\[6pt]
        g_{j (s+1)^{i}}^{\widetilde{s}-1}[\widetilde{i}] + g_{j (s-1)^{i}}^{\widetilde{s}}[\widetilde{i}] - g_{j s^{i}}^{\widetilde{s}-1}[\widetilde{i}]
        & \text{otherwise.}
    \end{cases}
\end{equation}
\end{enumerate}
\end{lemma}

\begin{proof}
Part~(a) follows from \cite[Theorem~4.1]{SW21}.  
Part~(b) follows from Lemmas~\ref{lem:mutation-row1}, \ref{lem:mutation-row2}, and~\ref{lem:mutation-row3}.  
Part~(c) is an immediate consequence of parts~(a) and~(b) together with Lemma~\ref{lem:grec}.
\end{proof}

\begin{lemma}\label{lem:g-vector-formula1}
    Let $0 \leq \widetilde{s} \leq n-1$ and $1 \leq \widetilde{i} \leq k$. For any vertex $js^i$ with $j \geq \widetilde{s}$ and $i \leq k+1 - \widetilde{i}$, we have
    \begin{equation}\label{eq:gpk}
        g_{js^i}^{\widetilde{s}}[\widetilde{i}] = {\bf e}_{j,s+\widetilde{s}^{\,i+\widetilde{i}-1}} - {\bf e}_{j\widetilde{s}^{\widetilde{i}}},
    \end{equation}
    where ${\bf e}_v$ denotes the unit vector associated with the vertex $v$ of $Q$. By convention, ${\bf e}_{j0^1} = {\bf 0}$, and if $s \geq j$, we set ${\bf e}_{js^i} = {\bf e}_{j,s-j^{\,i+1}}$.
\end{lemma}

\begin{proof}
We proceed by induction on $\widetilde{i}$, $\widetilde{s}$, and $j(i-1)+s$.

We first consider the case $\widetilde{i} = 1$.

When $\widetilde{s} = 0$, we have $\widetilde{A}_{js^i}^{\widetilde{s}}[\widetilde{i}] = \widetilde{A}_{js^i}$ for any $s$, and thus $g_{js^i}^{0}[1] = {\bf e}_{js^{i}}={\bf e}_{js^{i}}-{\bf e}_{j0^{1}}$.

Now suppose $\widetilde{s}>0$, and assume that equation \eqref{eq:gpk} holds for all triples with $\widetilde{i} = 1$, $\widetilde{s}-1$, and arbitrary $j(i-1)+s$.

If $j(i-1)+s=1$, then $i=1$ and $s=1$. Thus by \eqref{eq:recg1}
\begin{equation}
\begin{aligned}
    g_{js^i}^{\widetilde{s}}[1]
    & =g_{j1^1}^{\widetilde{s}}[1]  
    = g_{j2^1}^{\widetilde{s}-1}[1] - g_{j1^1}^{\widetilde{s}-1}[1] \\
    &= \bigl({\bf e}_{j,1+\widetilde{s}^{1}} - {\bf e}_{j,\widetilde{s}-1^{1}}\bigr) 
       - \bigl({\bf e}_{j\widetilde{s}^{1}} - {\bf e}_{j,\widetilde{s}-1^{1}}\bigr) \\
    &= {\bf e}_{j,1+\widetilde{s}^{1}} - {\bf e}_{j,\widetilde{s}^{1}}.
\end{aligned}
\end{equation}
This implies  \eqref{eq:gpk} holds for $\widetilde{i}=1$, $\widetilde{s}$, and $j(i-1)+s=1$.

If $j(i-1)+s>1$, we further assume that \eqref{eq:gpk} holds for $\widetilde{i} = 1$, $\widetilde{s}$, and $j(i-1)+s-1$. Then by \eqref{eq:recg1}
\begin{equation}
\begin{aligned}
    g_{js^i}^{\widetilde{s}}[1] 
    &= g_{j,(s-1)^i}^{\widetilde{s}}[1] + g_{j,(s+1)^i}^{\widetilde{s}-1}[1] - g_{j,s^i}^{\widetilde{s}-1}[1] \\
    &= \bigl({\bf e}_{j,s+(\widetilde{s}-1)^{i}} - {\bf e}_{j\widetilde{s}^{1}}\bigr) 
       + \bigl({\bf e}_{j,s+\widetilde{s}^{i}} - {\bf e}_{j,\widetilde{s}-1^{1}}\bigr)- \bigl({\bf e}_{j,s+(\widetilde{s}-1)^{i}} - {\bf e}_{j,\widetilde{s}-1^{1}}\bigr) \\
    &= {\bf e}_{j,s+\widetilde{s}^{i}} - {\bf e}_{j,\widetilde{s}^{1}}.
\end{aligned}
\end{equation}

This establishes \eqref{eq:gpk} for all $\widetilde{i}=1$.

We now consider the case $\widetilde{i}>1$. Assume inductively that \eqref{eq:gpk} holds for $\widetilde{i}-1$, arbitrary $\widetilde{s}$, and all $j(i-1)+s-1$. Since $\widetilde{A}_{js^i}^{0}[\widetilde{i}]= \widetilde{A}_{js^i}^{n-1}[\widetilde{i}-1]$, the formula also holds for $\widetilde{s}=0$, any $js^i$, and the given $\widetilde{i}$.

Now let $\widetilde{s} > 0$, and assume that \eqref{eq:gpk} holds for $\widetilde{i}$, $\widetilde{s}-1$, and all $j(i-1)+s$.

If $j(i-1)+s=1$, then $i=s=1$. Thus by \eqref{eq:recg1}
\begin{equation}
\begin{aligned}
    g_{js^i}^{\widetilde{s}}[\widetilde{i}]  &=g_{j1^1}^{\widetilde{s}}[\widetilde{i}] 
    = g_{j2^i}^{\widetilde{s}-1}[\widetilde{i}] - g_{j1^i}^{\widetilde{s}-1}[\widetilde{i}] \\
    &= \bigl({\bf e}_{j,1+\widetilde{s}^{\widetilde{i}}} - {\bf e}_{j,(\widetilde{s}-1)^{\widetilde{i}}}\bigr)- \bigl({\bf e}_{j,1+(\widetilde{s}-1)^{\widetilde{i}}} - {\bf e}_{j,(\widetilde{s}-1)^{\widetilde{i}}}\bigr) \\
    &= {\bf e}_{j,1+\widetilde{s}^{\widetilde{i}}} - {\bf e}_{j,\widetilde{s}^{\widetilde{i}}}.
\end{aligned}
\end{equation}
This implies  \eqref{eq:gpk} holds for $\widetilde{i}$, $\widetilde{s}$, and $j(i-1)+s=1$.

If $s > 0$, we further assume that \eqref{eq:gpk} holds for $\widetilde{i}$, $\widetilde{s}$, and $j(i-1)+s-1$. Then by \eqref{eq:recg1}
\begin{equation}
\begin{aligned}
    &\quad\;\; g_{js^i}^{\widetilde{s}}[\widetilde{i}] 
    = g_{j,(s-1)^i}^{\widetilde{s}}[\widetilde{i}] + g_{j,(s+1)^i}^{\widetilde{s}-1}[\widetilde{i}] - g_{j,s^i}^{\widetilde{s}-1}[\widetilde{i}] \\
    &= \bigl({\bf e}_{j,s+(\widetilde{s}-1)^{i+\widetilde{i}-1}} - {\bf e}_{j\widetilde{s}^{\;\widetilde{i}}}\bigr)+ \bigl({\bf e}_{j,s+\widetilde{s}^{\;i+\widetilde{i}-1}} - {\bf e}_{j,\widetilde{s}-1^{\widetilde{i}}}\bigr)- \bigl({\bf e}_{j,s+(\widetilde{s}-1)^{i+\widetilde{i}-1}} - {\bf e}_{j,\widetilde{s}-1^{\widetilde{i}}}\bigr) \\
    &= {\bf e}_{j,s+\widetilde{s}^{\,i+\widetilde{i}-1}} - {\bf e}_{j,\widetilde{s}^{\,\widetilde{i}}}.
\end{aligned}
\end{equation}

This completes the induction and the proof.
\end{proof}